\documentclass[12pt,reqno] {amsart}
\usepackage{amsfonts}
\usepackage{amsmath}
\usepackage{amssymb}
\usepackage{mathrsfs}
\usepackage{amsthm}
\usepackage{fancyhdr}
\usepackage{graphicx}
\usepackage{extarrows}
\usepackage{indentfirst,latexsym,bm}
\usepackage{amsmath,amsthm,amsfonts,amssymb,amscd,mathrsfs}
\usepackage{graphicx}
\usepackage[all]{xy}
\usepackage{pst-node}
\usepackage{pstricks}
\usepackage{cite}
\usepackage{enumitem}
\usepackage{subfigure}
\usepackage{pifont}
\usepackage{amsbsy}
\usepackage{bbm}
\usepackage{ulem}
\usepackage{tikz}
\usepackage{float}

\newcommand\uuuline{\bgroup\markoverwith%
   {%
     \textcolor{black}{\rule[-0.3ex]{1pt}{0.3pt}}%
     \llap{\textcolor{black}{\rule[-0.6ex]{1pt}{0.3pt}}}%
     \llap{\textcolor{black}{\rule[-0.9ex]{1pt}{0.3pt}}}%
   }%
   \ULon}

\makeatletter
\renewcommand{\@thesubfigure}{\hskip\subfiglabelskip}
\makeatother

 \theoremstyle{plain}

\newtheorem{theorem}{Theorem}[section]

\newtheorem{corollary}[theorem]{Corollary}
\newtheorem{definition}[theorem]{Definition}
\newtheorem{lemma}[theorem]{Lemma}
\newtheorem{proposition}[theorem]{Proposition}
\newtheorem{remark}[theorem]{Remark}
\newtheorem{example}[theorem]{Example}

\def\bC{\mathbb{C}}

\def\bG{\mathbb{G}}

\def\bN{\mathbb{N}}

\def\bQ{\mathbb{Q}}

\def\bS{\mathbb{S}}

\def\bZ{\mathbb{Z}}

\def\cA{\mathcal{A}}
\def\cB{\mathcal{B}}
\def\cC{\mathcal{C}}
\def\cD{\mathcal{D}}

\def\cF{\mathcal{F}}
\def\cG{\mathcal{G}}
\def\cH{\mathcal{H}}

\def\cP{\mathcal{P}}
\def\cQ{\mathcal{Q}}

\def\cS{\mathcal{S}}

\def\fA{\mathfrak{A}}
\def\fB{\mathfrak{B}}
\def\fC{\mathfrak{C}}

\def\fL{\mathfrak{L}}

\def\fS{\mathfrak{S}}

\def\fW{\mathfrak{W}}

\def\Re{\text{Re}}

\def\i{\mathfrak{i}}

\def\<{\langle}
\def\>{\rangle}

\def\ind{\textsf{ind}}
\def\gcd{\textsf{gcd}}
\def\lcm{\textsf{lcm}}
\def\If{\text{if }}
\def\Otherwise{\text{otherwise}}
\def\Span{\textsf{Span}}
\def\PDM{\textsf{PDM}}
\def\PD{\textsf{PD}}

\def\rlh{\rightleftharpoons}

\newtheoremstyle{mythm}{1.5ex plus 1ex minus .2ex}{1.5ex plus 1ex minus .2ex}
    {\fangsong}{\parindent}{\songti\bfseries}{}{1em}{}

\begin{document}

\begin{center}
\large Natural Monoids and Non-commutative Arithmetics\footnote{MSC Class: 11N99; 20F05; 20F38
}
\end{center}

\begin{center}
Boqing Xue
\end{center}

\begin{center}
\textit{Academy of Mathematics and Systems Science, \\
Chinese Academy of Sciences, Beijing {\rm100190}, China,\\
boqing\_xue@hotmail.com}
\end{center}

\bigskip

\begin{quote}
{\footnotesize {
\begin{center}
\scshape Abstract
\end{center}
}
 \medskip

We introduce several classes of monoids satisfying up to five axioms and establish basic theories on their arithmetics. The one satisfying all the axioms is named natural monoid. Two typical examples are 1) the monoid $\bN$ of natural numbers in the group of positive rationals and 2) a certain monoid $\bS$ in one of Thompson's groups. The latter one is non-abelian, which serves as an important example for non-commutative arithmetics.

Defining primes in a non-abelian monoid $S$ is highly non-trivial, which relies on a concept we called ``castling''. Three types of castlings are essential to grasp the arithmetics on $S$. Multiplicative and completely multiplicative functions are defined. In particular, M\"{o}bius function is multiplicative, and Liouville function on a natural monoid is completely multiplicative. The divisor function has a sub-multiplicative property, which induces a non-trivial quantity $\tau_0(u)=\lim\nolimits_{n\rightarrow \infty}(\tau(u^n))^{1/n}$ in a non-abelian monoid $S$. Moreover, the quantity $\c{C}(S)=\sup\nolimits_{1\neq u\in S}\tau_0(u)/\tau(u)$ describes the complexity for castlings in $S$. We show that $\c{C}(\bN)=1/2$ and $\c{C}(\bS)=1$. The $C^\ast$-algebra obtained from the left regular representation of $S$ on $l^2(S)$, on which a particular trace can be defined, is also studied. Furthermore, we prove that a natural monoid having finitely many primes is amenable.
}
\end{quote}


\tableofcontents

\bigskip

\section{Introduction}

The purpose of this paper is to introduce several classes of monoids satisfying up to five axioms, and to establish basic theories on their arithmetics. The first three axioms are given below.

\begin{definition}
Let $G$ be a countable group and $S$ be a monoid with $S\subseteq G$. We say that $S$ is an integral monoid and $G$ is its fractional group, if the following conditions hold.

\textsc{Axiom} \uppercase\expandafter{\romannumeral1}. It satisfies $S\cap S^{-1}=\{1\}$, where $1$ is the identity of $G$.

\textsc{Axiom} \uppercase\expandafter{\romannumeral2}. For any $u\in G$, there exists a pair of elements $x,y\in S$ with $u=xy^{-1}$ such that, whenever $u=zw^{-1}$ for some $z,w\in S$, then $z=xc$ and $w=yc$ for some $c\in S$.

\textsc{Axiom} \uppercase\expandafter{\romannumeral3}. For any $u\in S$, it satisfies that $\#\{(v,w)\in S\times S:\, u=vw\}<+\infty$.
\end{definition}

Axiom \uppercase\expandafter{\romannumeral1} requires that $G$ have no torsion. In Axiom \uppercase\expandafter{\romannumeral2}, we call $zw^{-1}$ a (right) fraction of $u$ with numerator $z$ and denominator $w$. Combining Axioms \uppercase\expandafter{\romannumeral1} and \uppercase\expandafter{\romannumeral2}, one can deduce that the pair $x,y$ is unique. We call $xy^{-1}$ the (right) fraction of $u$ in lowest terms. In this paper, we always assume that $G,S$ satisfy Axioms \uppercase\expandafter{\romannumeral1}-\uppercase\expandafter{\romannumeral3} as above, and $G\neq \{1\}$.

For $u,w\in S$, we say that $u$ divides $w$, and write $u|w$, if there is some $v\in S$ such that $uv=w$. Indeed, Axiom \uppercase\expandafter{\romannumeral1} ensures ``$|$'' to be a partial order on $S$. Axiom \uppercase\expandafter{\romannumeral2} allows one to define the least common multiple. Axiom \uppercase\expandafter{\romannumeral3} makes the greatest common divisor well-defined, and allows an irreducible decomposition of each element in $S$. All these are foundations for exploring arithmetics.

Two typical examples of integral monoid $S$ with its fractional groups $G$ are \romannumeral1) the set of natural numbers $\bN$ in the positive rational numbers $\bQ^+$ with multiplication, and \romannumeral2) a certain monoid $\bS$ in Thompson's group $\bG$. In particular, the group $\bG$ is non-abelian, which serves as an important example for the non-commutative arithmetics. We will come back to explain the axioms after a brief introduction to these two examples.

\medskip

The natural numbers $\bN$, as a multiplicative monoid, has primes $\cP=\{p_0,p_1,p_2,\ldots\}$ being its generators. Here $p_0=2$, $p_1=3$, $p_2=5$, $p_3=7$, $p_4=11\ldots$. These numbers are irreducible in $\bN$, since each prime has only $1$ and itself as its divisors. They are called primes, since the condition $p_j|mn$ $(m,n\in \bN)$ implies that either $p_j|m$ or $p_j|n$ $(j=0,1,2,\ldots)$. The fundamental theorem of arithmetic states that every natural number greater than $1$ is a product of primes and such decomposition is unique up to reordering.

Around 300 BC, Euclid proved the infinitude of primes by showing that the natural number $p_0p_1\ldots p_{n-1}+1$ has at least one prime divisor other than $p_0,p_1,\ldots,p_{n-1}$ $(n\geq 1)$. By taking the logarithm of the product formula $\sum\nolimits_{n=1}^\infty n^{-s}=\prod\nolimits_{p\in \cP}(1-p^{-s})^{-1}$ $(s>1)$ and letting $s\rightarrow 1^+$, Euler showed that the series $\sum\nolimits_{p\in \cP}1/p$ diverges. Therefore, the primes can not ``too sparsely'' distributed in $\bN$. Around 1800, after mass statistics by hand, Gauss predicted that $\pi(x)$, the number of primes up to $x$, is asymptotic to $li(x)=\int_2^x (1/\log t)dt$ as $x\rightarrow \infty$. In 1859, Riemann \cite{Rie59} studied the function $\zeta(s)=\sum\nolimits_{n=1}^\infty \frac{1}{n^s}$, together with Euler product formula, as a complex function. By proving a functional equation, Riemann extended $\zeta(s)$ to a meromorphic function on the whole complex plane and establishes remarkably a connection between zeros of $\zeta(s)$ and the distribution of primes. More specifically, the function $\zeta(s)$ has no zeros in $1/2<\Re(s)\,(\leq 1)$ if and only if $\pi(x)=li(x)+\textit{O}(x^{1/2+\varepsilon})$ for any $\varepsilon>0$. The statement that all the non-trivial zeros of $\zeta(s)$ lie on the vertical line $\Re(s)=1/2$ is known as Riemann hypothesis. We refer to \cite{Bom06,Ge-Xue18} for surveys on Riemann hypothesis. Till now, people only have knowledge of the non-existence of zeros in the region ``very close'' to the vertical line $\Re(s)=1$ (see \cite{Vin58} for example).

In modern theoretical physics, people usually use operators instead of functions, to explain physical phenomena or demonstrate physical theories. And commutative structures are often lifted to some corresponding non-commutative structures. In \cite{DHX18}, the authors studied the multiplicative structure of natural numbers by operators and operator algebras through the left regular representation of $\bN$ on $l^2(\bN)$. One of the theorems says that the $C^\ast$-algebra generated by $\bN$ in $B(l^2(\bN))$ does not contain non-trivial projections of finite rank. Indeed, this statement is equivalent to the infinitude of primes.

We present some details of the left regular representation here. For $m\in \bN$, let $\delta_m$ be the function taking value $1$ at $m$ and $0$ elsewhere. Then $\{\delta_m:\, m\in \bN\}$ is an orthonormal basis for $l^2(\bN)$. For $k\in \bN$, define $L_k$ to be the operator on $\cH$ given by $L_k \delta_m = \delta_{mk}$ $(m\in \bN)$. Noting that $L_kL_l = L_{kl}=L_{l}L_k$ for all $k,l\in \bN$, the set $L_\bN:=\{L_n:\, n\in \bN\}$ is a monoid in $B(l^2(\bN))$ which keeps the multiplicative structure of $\bN$. The $C^\ast$-algebra is closed under taking adjoints. The adjoints are given by $L_k^\ast \delta_m = \delta_{m/k}$ for $k|m$ and $0$ for $k\nmid m$. At first glance, these operators provide the operation ``division'' and one would obtain the positive rational numbers $\bQ^+$ by combining $L_k$ and $L_k^\ast$ $(k\in \bN)$ together. However, the operators $L_j$ and $L_k^\ast$ $(j,k\in \bN)$ do not always commute. They satisfy
\begin{equation} \label{eq_LjLk_ast_in_natural_number}
L_{j_1} L_{j_2} = L_{j_1j_2},\quad  L_{k_1}^\ast L_{k_2}^\ast = L_{k_1k_2}^\ast,\quad L_k^\ast L_j = L_{j/\gcd(j,k)} L^\ast_{k/\gcd(j,k)}
\end{equation}
for $j,j_1,j_2,k,k_1,k_2\in \bN$. We use $\cQ$ to denote the monoid generated by $\{L_k,L_k^\ast:\, k\in \bN\}$ in $B(\cH)$, and call it the multiplicative monoid of non-commutative rationals. By \eqref{eq_LjLk_ast_in_natural_number}, one can deduce that $\cQ=\left\{L_j L_k^\ast:\, j,k\in \bN\right\}$, while $\bQ^+=\{j/k:\, j,k\in \bN\}$. This lift the commutative structure $\bQ^+$ to a non-commutative structure $\cQ$. This process requires few properties of natural numbers. It can be applied to other monoids or groups, such as one of the Thompson's groups.

\medskip

We use $\bG$ to denote Thompson's group $F$ in this paper, which was defined by Richard Thompson in 1965. It consists of piecewise linear homeomorphisms from the closed unit interval $[0,1]$ to itself with finitely many breakpoints with the following two conditions: (\romannumeral1) each breakpoint is a dyadic rational number; (\romannumeral2) each slope is a power of $2$. The identity element $1$ is the map $[0,1]\rightarrow [0,1],\, x\mapsto x$. Let $A,B$ be two elements in $\bG$ which are given below.
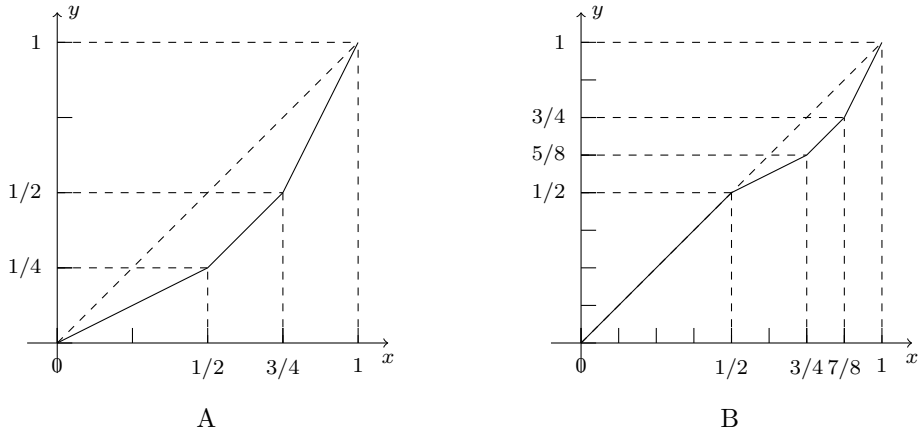
\begin{figure}[H]
\centering
  \subfigure[A]{\begin{tikzpicture}[scale = 4]
	\draw[->](-0.1,0)--(1.1,0)node[left,below,font=\tiny]{$x$};
	\draw[->](0,-0.1)--(0,1.1)node[right,font=\tiny]{$y$};
	\draw[dashed](0,1)--(1,1);
    \draw[dashed](1,0)--(1,1);
    \draw[dashed](0,0)--(1,1);
    \draw[dashed](1/2,0)--(1/2,1/4);
    \draw[dashed](0,1/4)--(1/2,1/4);
    \draw[dashed](3/4,0)--(3/4,1/2);
    \draw[dashed](0,1/2)--(3/4,1/2);
	\foreach \x in {0,1/4,3/4}{\draw(\x,0)--(\x,0.05)node[below,outer sep=2pt,font=\tiny]at(\x,0){};}
\foreach \x in {0,1/2,3/4,1}{\draw(\x,0)--(\x,0.05)node[below,outer sep=2pt,font=\tiny]at(\x,0){\x};}
	\foreach \y in {0,1/4,3/4}{\draw(0,\y)--(0.05,\y)node[left,outer sep=2pt,font=\tiny]at(0,\y){};}
\foreach \y in {1/4,1/2,1}{\draw(0,\y)--(0.05,\y)node[left,outer sep=2pt,font=\tiny]at(0,\y){\y};}
	\draw[domain=0:1/2]plot(\x,{\x/2});
    \draw[domain=1/2:3/4]plot(\x,{\x-1/4});
    \draw[domain=3/4:1]plot(\x,{2*\x-1});
	\end{tikzpicture}}
  \hspace{0.5in}
  \subfigure[B]{\begin{tikzpicture}[scale = 4]
	\draw[->](-0.1,0)--(1.1,0)node[left,below,font=\tiny]{$x$};
	\draw[->](0,-0.1)--(0,1.1)node[right,font=\tiny]{$y$};
	\draw[dashed](0,1)--(1,1);
    \draw[dashed](1,0)--(1,1);
    \draw[dashed](0,0)--(1,1);
    \draw[dashed](1/2,0)--(1/2,1/2);
    \draw[dashed](0,1/2)--(1/2,1/2);
    \draw[dashed](3/4,0)--(3/4,5/8);
    \draw[dashed](0,5/8)--(3/4,5/8);
    \draw[dashed](0,3/4)--(7/8,3/4);
    \draw[dashed](7/8,0)--(7/8,3/4);
	\foreach \x in {0,1/8,1/4,3/8,1/2,5/8,3/4,7/8}{\draw(\x,0)--(\x,0.05)node[below,outer sep=2pt,font=\tiny]at(\x,0){};}
\foreach \x in {0,1/2,3/4,7/8,1}{\draw(\x,0)--(\x,0.05)node[below,outer sep=2pt,font=\tiny]at(\x,0){\x};}
	\foreach \y in {0,1/8,1/4,3/8,1/2,5/8,3/4,7/8}{\draw(0,\y)--(0.05,\y)node[left,outer sep=2pt,font=\tiny]at(0,\y){};}
\foreach \y in {1/2,5/8,3/4,1}{\draw(0,\y)--(0.05,\y)node[left,outer sep=2pt,font=\tiny]at(0,\y){\y};}
	\draw[domain=0:1/2]plot(\x,{\x});
    \draw[domain=1/2:3/4]plot(\x,{\x/2+1/4});
    \draw[domain=3/4:7/8]plot(\x,{\x-1/8});
    \draw[domain=7/8:1]plot(\x,{2*\x-1});
	\end{tikzpicture}}
\end{figure}
The group $\bG$ is finitely-presented,
\[
\bG=\left\<A,B \, |\, [AB^{-1},A^{-1}BA],\, [AB^{-1},A^{-2}BA^2]\right\>.
\]
It was used by McKenzie and Thompson in \cite{McK-Tho73} to solve certain word problems. And it is the first example of a torsion-free infinite-dimensional $FP_\infty$ group, shown by Brown and Geoghegan \cite{Bro-Geo84}. Later Brin and Squier \cite{Bri-Squ85} proved that $\bG$ does not contain a free group of rank greater than one and does not satisfy any laws. Moreover, Geoghegan popularized the interest in knowing whether or not $\bG$ is amenable. This question is still open at present. The notes \cite{CFP96} by Cannon, Floyd and Parry gives a thorough introduction to Thompson's group and related works before the 21st century. Nowadays, Thompson's group is related to many branches of mathematics, and vast topics on Thompson's groups are studied (see \cite{Jon17,Moo13,Wu-Che11} for example).

Putting
\begin{equation} \label{eq_define_p_j_in_Thompson_introduction}
p_0=A, \quad p_1=B, \quad p_j=A^{-(j-1)}BA^{j-1},\,(j\geq 2),
\end{equation}
one obtains another presentation of $\bG$ as
\[
\bG=\<p_0,p_1, p_2,\ldots:\, p_jp_i= p_ip_{j+1} \, (0\leq i<j)\>.
\]
For $0\leq i<j$, one has
\begin{equation} \label{eq_Thompson_relation_expand}
p_i^{-1}p_j=p_{j+1}p_i^{-1},\quad p_j^{-1}p_i=p_ip_{j+1}^{-1},\quad p_jp_i=p_ip_{j+1}.
\end{equation}
Therefore, given an element in $\bG$, one can always move the $p_i$'s with negative powers or larger subscripts rightwards according to \eqref{eq_Thompson_relation_expand}. A carefully analysis leads to the conclusion that every non-trivial element of $\bG$ can be expressed in a unique normal form
\begin{equation} \label{eq_normal_form_introduction}
p_0^{a_0}p_1^{a_1}\ldots p_{n-1}^{a_{n-1}}p_n^{a_n}p_n^{-b_n}p_{n-1}^{-b_{n-1}}\ldots p_1^{-b_1}p_0^{-b_0},
\end{equation}
where $n,a_0,a_1,\ldots,a_n,b_0,b_1,\ldots,b_n$ are nonnegative integers such that (\romannumeral1) exactly one of $a_n$ and $b_n$ is nonzero, and (\romannumeral2) if $a_k>0$ and $b_k>0$ for some integer $k$ with $0\leq k<n$, then either $a_{k+1}>0$ or $b_{k+1}>0$. Moreover, each above normal form is non-trivial. (See Corollary-Definition 2.7 in \cite{CFP96}.)

\medskip

One sees that the normal form gives a right fraction of each element in $\bG$. It inspires us to choose the monoid generated by $\{p_0,p_1,p_2,\ldots\}$, i.e.,
\[
\bS=\{1\}\cup\left\{p_0^{a_0}p_1^{a_1}\ldots p_{n-1}^{a_{n-1}}p_n^{a_n}:\, n\geq 0,\, a_0,a_1,\ldots,a_n\geq 0\right\},
\]
and call it Thompson's monoid. In \cite{CFP96}, elements in $\bS$ are called ``positive elements''.

We will show in Section \ref{subsection_example_of_intergral_monoid} that $\bS$ is an integral monoid and $\bG$ is its fractional group. Now we show some examples about arithmetics on $\bS$. For each $j=0,1,2,\ldots$, it follows from the normal form that the divisors of $p_j$ are exactly $1$ and itself. So $p_0,p_1,p_2,\ldots$ are irreducible elements in $\bS$. The normal form also gives one irreducible decomposition of each element in $\bS$. However, an element many have different irreducible decompositions. For example, it satisfies $p_0p_2=p_1p_0$, which leads to $\lcm[p_0,p_1]=p_0p_2=p_1p_0$. Consider two elements $u=p_0^2p_1p_4$ and $v=p_0p_2p_3$, all the irreducible decompositions are
\begin{align*}
u=p_0^2p_1p_4=p_0^2p_3p_1=p_0p_2p_0p_1=p_1p_0^2p_1, \quad
v=p_0p_2p_3=p_1p_0p_3=p_1p_2p_0.
\end{align*}
The divisors of $u$ are exactly $1,p_0,p_1,p_0^2,p_0p_2,p_0^2p_1,p_0^2p_3$ and $u$. The divisors of $v$ are exactly $1,p_0,p_1,p_0p_2,p_1p_2$ and $v$. One obtains that $\gcd(u,v)=p_0p_2$. If one considers the left regular representation of $\bS$ on $l^2(\bS)$ as previous, and let $\cQ$ be the monoid generated by $\{L_u,L_u^\ast:\, u\in \bS\}$ in $B(l^2(\bS))$, then $\cQ=\{L_uL_v^\ast:\, u,v\in \bS\}$. The relations in \eqref{eq_LjLk_ast_in_natural_number} are replaced by
\[
L_{u_1}L_{u_2}=L_{u_1u_2},\quad L_{v_2}^\ast L_{v_1}^\ast = L_{v_1v_2}^\ast,\quad L_v^\ast L_u = L_{v^{-1}\lcm[u,v]}L^\ast_{u^{-1}\lcm[u,v]},
\]
where $u,v,u_1,u_2,v_1,v_2\in \bS$. Such a structure shares similar nature with that of the natural numbers.

In Section \ref{section_integral_monoids}, we will provide details about the divisors, multiples, irreducible decompositions and the left regular representations for an integral monoid $S$. Moreover, we also introduce the notion of co-divisors, co-multiples and show a duality between common divisors/multiples and common co-divisors/co-multiples. Such a duality is crucial for a non-abelian $S$. We will also prove that the $C^\ast$-algebra $\fA$ generated by $\cQ$ in $B(l^2(S))$ does not contain a certain projection if and only if $S$ has infinitely many irreducible elements. The $C^\ast$-algebra $\fA$ admits a certain trace, from which the GNS construction gives the reduced group $C^\ast$-algebra of $G$. Moreover, some basic properties of arithmetic functions on $S$ are studied in Section \ref{section_integral_monoids}.

\medskip

To explore arithmetics further, one may never avoid the notion of ``prime elements'', which will be abbreviated as ``primes'' in this paper. Recall that a prime $p$ in $\bN$ is defined by
\[
p|uv\quad \Longrightarrow \quad \text{either }p|u,\text{ or }p|v.
\]
What would happen for a non-abelian monoid $S$? When $u,v$ do not commute, there are no direct connections between $p|uv$ and $p|v$.  The idea is to understand the above expressions by ``either $p$ divides $u$, or $p$ is a divisor coming from $v$''. That is to say, we hope that $uv=\widetilde{v}\widetilde{u}$ for some $\widetilde{u},\widetilde{v}\in \bS$, and $p$ divides $\widetilde{v}$ instead of $v$. While it makes sense, the pair of elements $\widetilde{v},\widetilde{u}$ should be uniquely determined by the pair of elements $u,v$, and $\widetilde{v}$ should contain information exactly from $v$. We will call such a process a castling of elements. The word ``castling'' comes from chess, which is a move involving a player's king and one rook to ``jump over'' each other. The two locations of castled chess pieces are slightly different from their original locations (the two chess pieces become closer to each other). Here, the two castled element $\widetilde{u}$ and $\widetilde{v}$ may be ``slightly different'' from $u$ and $v$, respectively. To fulfill such a process mathematically, we need three types of castlings in all to make clear the whole arithmetics on $S$.

Consider $u=p_0p_3$ and $\widetilde{v}=p_1$ in the Thompson's monoid. One has $\gcd(u,\widetilde{v})=1$ and $\lcm[u,\widetilde{v}]=(p_0p_3)p_2=p_1(p_0p_4)$. Putting $v=p_2$ and $\widetilde{u}=p_0p_4$, we have that $p_1|uv$. Note that $\gcd(p_1,u)=1$. The element $p$ contains no information from $u$, so one can claim that $p_1$ is a divisor ``coming from'' $v$. Indeed, it satisfies that $uv=\widetilde{v}\widetilde{u}$ and $p_1|\widetilde{v}$.  We call this process a free castling. The concrete definition will be given in Section \ref{section_homogeneous_monoid} with Axiom \uppercase\expandafter{\romannumeral4}' formulated. At this stage, we can prove that
\begin{equation} \label{eq_tau_submultiplicative_introduction}
\tau(uv) \leq \tau(u)\tau(v),\quad (u,v\in S),
\end{equation}
where $\tau$ is the divisor function. And M\"{o}bius function appears as
\[
\mu(u) =
\begin{cases}
1,\quad &\If u=1,\\
(-1)^k,\quad &\If u=\lcm[q_1,\ldots,q_k] \text{ for distinct }q_1,\ldots,q_k\in \cP,\\
0,\quad &\Otherwise.
\end{cases}
\]
where $\cP$ is the set of irreducible elements.

\medskip

Next, consider $u=p_0p_2p_4$ and $v=p_2p_5$ in Thompson's monoid. In this example, we have $\gcd(u,v)=p_2\neq 1$. A distinct approach is needed to describe the castlings of two elements. Suppose that $v$ may ``jump over'' $u$ and become $\widetilde{v}$. Then $v$ should ``jump over'' $p_4,p_2,p_0$ successively. Indeed, we have
\[
p_0p_2p_4(p_2p_5) =  p_0p_2(p_2 p_5) p_5 = p_0(p_2 p_4) p_2p_5 = (p_1 p_3) p_0 p_2 p_5.
\]
Therefore $\widetilde{v}=p_1p_3$ and $\widetilde{u}=p_0p_2p_5$. Moreover, we also have
\[
(p_0p_2p_4)p_2p_5 = p_1(p_0 p_2 p_5)p_5 = p_1p_3 (p_0 p_2 p_5).
\]
That is to say, the element $u$ may also ``jump over'' $v$ and become $\widetilde{u}$. These processes give hints to define castlings in general.

Unluckily, we meet some difficulties in the following example. Consider $u=p_0$ and $v=p_0p_2$ in Thompson's monoid. We have $(p_0)p_0p_2=p_0(p_0)p_2 = p_0p_1(p_0)$, i.e., $\widetilde{v}=p_0p_1$ and $\widetilde{u}=p_0$. However, the element $v$ has the other irreducible decomposition $v=p_1p_0$. If $u$ may ``jump over'' $v$ in this case, then $u$ should ``jump over'' $p_1$ first. But the element $p_0p_1$ has only one irreducible decomposition, and $u$ is stuck by $p_1$.

In Section \ref{section_castlable_monoid}, we will define strong castlings and weak castlings, with Axiom \uppercase\expandafter{\romannumeral4}, to distinguish such circumstances. Here, Axiom \uppercase\expandafter{\romannumeral4} implies Axiom \uppercase\expandafter{\romannumeral4}'. At this stage, the irreducible elements are turned into primes. We deduce that a prime power $p^m$ has a unique irreducible decomposition, and $\tau(p^m)=m+1$. Multiplicative and completely multiplicative functions will be defined. In particular, M\"{o}bius function is multiplicative. And the convolution of two multiplicative functions is still multiplicative.


With previous axioms, distinct prime divisors will become distinct prime divisors after a castling. However, primes powers might change. In Section \ref{section_natural_monoid}, we put Axiom \uppercase\expandafter{\romannumeral5} to gain the power-preserving property. When Axioms \uppercase\expandafter{\romannumeral1}-\uppercase\expandafter{\romannumeral5} are satisfied, we call $S$ a natural monoid and $G$ its rational group. At this stage, Liouville function is completely multiplicative. We will also build up methods to determine prime divisors of an element with multiplicities from an arbitrary prime decomposition. A special class of natural monoids, which is said to be fully castlable, is investigated. One may regard it as the simplest class of natural monoids, in which the notion of weak and strong castlings coincide. We will prove that any natural monoid having finitely many primes is fully castlable, and is also amenable.

In section \ref{section_Thomspon}, we shall verify Axioms \uppercase\expandafter{\romannumeral4} and \uppercase\expandafter{\romannumeral5} for Thompson's monoid $\bS$. Constructing castlings in a concrete monoid is quite different from the abstract definition of castings in Sections \ref{section_homogeneous_monoid} and \ref{section_castlable_monoid}. We will apply a totally different way as follows. Regarding distinct prime decomposition of an element as distinct words, we first define castlings of words. Second, we establish a partial order on all words of a given element, and prove that castlings of words preserve this partial order. Third, we shall show that maximum and minimum words exist. Fourth, we define strong and weak castlings of elements with minimum and maximum words, respectively. Fifth, we prove the fundamental lemma for arithmetic and define free castling of elements in Thompson's monoid. Sixth, it is shown that these definitions coincide with that given in Sections \ref{section_homogeneous_monoid} and \ref{section_castlable_monoid}. Finally, we verify Axioms \uppercase\expandafter{\romannumeral4} and \uppercase\expandafter{\romannumeral5}, and prove that $\bS$ is a natural monoid.

By \eqref{eq_tau_submultiplicative_introduction}, the sequence $\{\log \tau(u^n)\}_{n=1}^\infty$ is sub-additive in a homogeneous monoid. Thus, for any $u\in S$, the limit
\[
\tau_0(u)=\lim\limits_{n\rightarrow \infty} \left(\tau(u^n)\right)^{1/n}
\]
exists. One may compare it with spectral radius of a bounded operator, or entropy of a dynamical system. The quantity
\[
\c{C}(S)=\sup\limits_{1\neq u\in S}\frac{\tau_0(u)}{\tau(u)}
\]
takes value in $[1/2,1]$ and reflects the complexity for castlings in the whole monoid. We study these quantities in Section \ref{section_complexity}. It is proved that $\c{C}(S)=1/2$ for any natural monoid containing finitely many primes, and $\c{C}(\bS)=1$ for Thompson's monoid.

\medskip

For a finite set $T$, both $|T|$ and $\#T$ stand for the cardinality of $T$. In most situations, the letters $u,v,w,x,y,z$ will denote an element in $S$, the letters $i,j,k,l,m,n$ will denote integers, and the letters $p,q$ may denote irreducible elements or primes. When a letter is used to present an element, without saying which set it belongs to, it always belongs to a corresponding monoid $S$. For example, ``for $u,w\in S$ with $u|w$, we write $w=uv$''. Here $v$ is an element in $S$. For basics in number theory, we refer to \cite{Nat}. For those in operator algebra and functional analysis, see \cite{Kad-Rin}.

\section{General Theory for Integral Monoids}
\label{section_integral_monoids}

\subsection{Examples}
\label{subsection_example_of_intergral_monoid}

In this subsection, we show some examples of integral monoids. Let us begin with Thompson's monoid $\bS$. For an element $u\in \bS$ with normal form $p_0^{a_0}p_1^{a_1}p_2^{a_2}\ldots p_m^{a_m}$, define $\ind(u)=\sum\nolimits_{j=0}^n a_j$, which counts the number of $p_j$'s involved. Also put $\ind(1)=0$. When
\begin{equation} \label{eq_def_representation_of_u}
u=p_{j_1}p_{j_2}\ldots p_{j_k},
\end{equation}
for some $j_1,j_2,\ldots,j_k\in \{0,1,2,\ldots\}$, we call the right-hand side of \eqref{eq_def_representation_of_u} a word of $u$ and each $p_{j_t}$ $(1\leq t\leq k)$ a letter in this word. Any two words may be turned into each other by applying the last equality in \eqref{eq_Thompson_relation_expand}. Therefore, the number of letters occurred remains the same, and the quantity $\ind(u)$ is independent of words chosen. For example, for $u=p_2^2p_4p_5=p_2p_3p_2p_5=p_2p_3p_4p_2$, one has $\ind(u)=4$.

Another way to understand this quantity is to define $\ind(A)=\ind(B)=1$ in the free group $F_{\{A,B\}}$ generated by $\{A,B\}$, and extends $\ind$ to be a group homomorphism from $F_{\{A,B\}}$ to $(\bZ,+)$. Note that
\[
\ind\left([AB^{-1},A^{-1}BA]\right)= \ind\left([AB^{-1},A^{-2}BA^2]\right)=0.
\]
So the normal subgroup $N$ generated by the above two elements in $F_{\{A,B\}}$ is contained in the kernel of $\ind$. So $\ind$ can be naturally defined on the quotient group $\bG=F_{\{A,B\}}/N$. That is to say, the map $\ind$ is a homomorphism from $\bG$ to $(\bZ,+)$, where $\ind(p_0)=\ind(p_1)=1$. Then $\ind(p_j)=\ind\left(p_0^{-(j-1)}p_1p_0^{j-1}\right)=1$ and
\[
\ind\left(p_0^{a_0}p_1^{a_1}p_2^{a_2}\ldots p_m^{a_m}\right)= a_0+a_1+\ldots +a_m.
\]
In particular, we have $\ind(uv)=\ind(u)+\ind(v)$ for $u,v\in S$. Now we shall verify that $\bS$ is an integral monoid.

\begin{theorem}
Thompson's monoid $\bS$ is an integral monoid with its fractional group $\bG$.
\end{theorem}

\begin{proof}
It is apparent that $\bS\cap\bS^{-1}=\{1\}$ and Axiom \uppercase\expandafter{\romannumeral1} holds.

Now we shall verify Axiom \uppercase\expandafter{\romannumeral2}. For $u=1$, the proof is trivial. For $u\neq 1$, it has the normal form as in \eqref{eq_normal_form_introduction}. Put $x=p_0^{a_0}p_1^{a_1}\ldots p_n^{a_n}$ and $p_n^{-b_n}\ldots p_1^{-b_1}p_0^{-b_0}$. Then $xy^{-1}$ is a fraction of $u$ in lowest terms with numerator $x$ and denominator $y$. Suppose that $u=vw^{-1}$ for some $v,w\in S$. We put $w^{(1)}=w$, $v^{(1)}=v$ and iterate as follows. For $k\geq 1$, suppose that $w^{(k)}$ and $v^{(k)}$ has normal form
\[
v^{(k)}=p_0^{a_0^{(k)}}p_1^{a_1^{(k)}}\ldots p_{n_k}^{a_{n_k}^{(k)}},\quad w^{(k)}=p_0^{b_0^{(k)}}p_1^{b_1^{(k)}}\ldots p_{m_k}^{b_{m_k}^{(k)}}.
\]
Without loss of generality, we set $a_j^{(k)}=0$ for $j>n_k$ and $b_j^{(k)}=0$ for $j>m_k$.

Case 1. Assume that for all $j=0,1,\ldots$, we have that $a_{j+1}^{(k)}>0$ or $b_{j+1}^{(k)}>0$ whenever both $a_j^{(k)}>0$ and $b_j^{(k)}>0$ hold. Then we stop the iterating process.

Case 2. Let $j_0$ be some number such that $a_{j_0}^{(k)}>0$, $b_{j_0}^{(k)}>0$ and $a_{j_0+1}^{(k)}=b_{j_0+1}^{(k)}=0$. Then
\[
v^{(k)}= p_0^{a_{0}^{(k)}}\ldots p_{j_0-1}^{a_{j_0-1}^{(k)}}p_{j_0}^{a_{j_0}^{(k)}-1}p_{j_0+1}^{a_{j_0+2}^{(k)}}\ldots p_{n_k-1}^{a_{n_k}^{(k)}}p_{j_0},\quad w^{(k)}= p_0^{b_{0}^{(k)}}\ldots p_{j_0-1}^{b_{j_0-1}^{(k)}}p_{j_0}^{b_{j_0}^{(k)}-1}p_{j_0+1}^{b_{j_0+2}^{(k)}}\ldots p_{m_k-1}^{b_{m_k}^{(k)}}p_{j_0}.
\]
In this case, we put
\[
v^{(k+1)}=v^{(k)}p_{j_0}^{-1},\quad w^{(k+1)}= w^{(k)}p_{j_0}^{-1}.
\]
It satisfies that $v^{(k+1)},w^{(k+1)}\in S$ and
\[
\ind(v^{(k+1)})=\ind(v^{(k)})-1,\quad \ind(w^{(k+1)})=\ind(w^{(k)})-1.
\]
Now we iterate with $k+1$ instead of $k$.

Since $\ind(v)$ is finite, the iterating process will stop, say, at step $K$. Then there exists some $c\in S$ such that $w=w^{(K)}c$ and $v=v^{(K)}c$. We have $u=vw^{-1}=v^{(K)}(w^{(K)})^{-1}$. By the construction of Case 1 in the iterating process, the term $v^{(K)}(w^{(K)})^{-1}$ is the normal form of $u$. By the uniqueness of normal form in Thompson's group, we deduce that $v^{(K)}=x$ and $w^{(K)}=y$. It follows that $v=xc$ and $w=yc$, and Axiom \uppercase\expandafter{\romannumeral2} holds.

Suppose that $u=vw$ for some $v,w\in S$. Write $v$ and $w$ in their normal form $p_0^{a_0}p_1^{a_1}\ldots p_k^{a_k}$ and $p_0^{a_0^\prime}p_1^{a_1^\prime}\ldots p_l^{a_l^\prime}$, respectively. Then $p_0^{a_0}p_1^{a_1}\ldots p_k^{a_k}p_0^{a_0^\prime}p_1^{a_1^\prime}\ldots p_l^{a_l^\prime}$ is one of the words of $u$. To verify Axiom \uppercase\expandafter{\romannumeral3}, it is sufficient to prove that $u$ has only finitely many words.

Suppose that $u\in S$ has normal form $p_0^{a_0}p_1^{a_1}\ldots p_n^{a_n}$ and $\ind(u)=m$. Recall that any word $p_{j_1}p_{j_2}\ldots p_{j_m}$ of $u$ may be transformed into the normal form by applying the relation $p_lp_k=p_lp_{l+1}$ with $0\leq k<l$ for finitely many times. During this process, we can require that the subscripts be non-decreasing. Therefore $j_1,j_2,\ldots,j_m\leq n$. Hence, the number of choices for such words are no larger than $(n+1)^m$. So Axiom \uppercase\expandafter{\romannumeral3} follows. The proof is completed.
\end{proof}

Next, we consider
\[
\cG=\<U,V:\, VU=UV^2\>.
\]
Note that
\begin{equation} \label{eq_relation_UV}
VU=UV^2,\, V^{-1}U = UV^{-2},\quad U^{-1}V^{-1}= V^{-2}U^{-1},\, U^{-1}V = V^2U^{-1}.
\end{equation}
One can always move $U$ to the left of $V,V^{-1}$ and $U^{-1}$ to the right of $V,V^{-1}$. So every element can be written in the form $U^a V^b U^{-c}$ with $a,c\geq 0$ and $b\in \bZ$. Moreover, if $a,c>0$ and $b$ is an even number, then the relations $V^k U=UV^{2k}$ $(k\in \bZ)$ ensures that $U^a V^b U^{-c}=U^{a-1}V^{b/2}U^{-(c-1)}$. So the normal form of an element in $\cG$ is given by $U^a V^b U^{-c}$ with $a,c\geq 0,\, b\in \bZ$, where either $ac=0$, or $ac\neq 0$ and $b$ is odd. We choose the monoid $\cS=\{U^mV^n:\,m,n\geq 0\}$.

\begin{theorem}
The monoid $\cS$ is an integral monoid with its fractional group $\cG$.
\end{theorem}

\begin{proof}
It is not hard to see that $\cS\cap \cS^{-1}=\{1\}$ and Axiom \uppercase\expandafter{\romannumeral1} holds.

For $u\in G$ with normal form $U^a V^b U^{-c}$, we set
\[
\begin{cases}
x=U^a V^b,\, y=U^{c},\quad &\If b\geq 0,\\
x=U^a,\, y=U^{c}V^{-b},\quad &\If b<0.
\end{cases}
\]
Let $\widetilde{x}=U^{a_1}V^{b_1}$ and $\widetilde{y}=U^{a_2}V^{b_2}$, where $a_1,b_1,a_2,b_2\geq 0$. And suppose that $\widetilde{x}\widetilde{y}^{-1}=u= xy^{-1}$. Let $k$ be the maximum number such that $k\leq a_1,a_2$ and $2^k(b_1-b_2)\in \bZ$. Then $U^{a_1}V^{b_1-b_2}U^{-a_2}= U^{a_1-k}V^{2^{-k}(b_1-b_2)}U^{-(a_2-k)}$, and on the right-hand side is the normal form of $u$. So $a=a_1-k$, $c=a_2-k$ and $b=2^{-k}(b_1-b_2)$. When $b_1\geq b_2$, we have
\begin{align*}
& x\cdot U^kV^{b_2} = U^{a_1-k}V^{2^{-k}(b_1-b_2)}\cdot U^kV^{b_2} = U^{a_1-k}U^kV^{b_1-b_2}V^{b_2} = \widetilde{x},\\
& y\cdot U^kV^{b_2} = U^{a_2-k}\cdot U^kV^{b_2} = \widetilde{y}.
\end{align*}
When $b_1<b_2$, we have
\begin{align*}
&x\cdot U^kV^{b_1} = U^{a_1-k} \cdot U^kV^{b_1} = \widetilde{x},\\
&y\cdot U^kV^{b_1} = U^{a_2-k}V^{-2^{-k}(b_1-b_2)}\cdot U^kV^{b_1} = U^{a_2-k}U^kV^{b_2-b_1}V^{b_1}=\widetilde{y}.
\end{align*}
This leads to Axiom \uppercase\expandafter{\romannumeral2}.

Consider the product $U^{m_3}V^{n_3}=U^{m_1}V^{n_1}\cdot U^{m_2}V^{n_2}$, where $m_1,m_2,m_3,n_1,n_2,n_3\geq 0$. Calculation shows that
\begin{equation} \label{eq_VU=UV^2_proof}
m_3=m_1+m_2,\quad n_3=2^{m_2}n_1+n_2.
\end{equation}
For given $m_3$ and $n_3$, there are only finitely many solutions to \eqref{eq_VU=UV^2_proof} for non-negative integers $m_1,m_2,n_1,n_2$. We conclude that $\#\{(x,y)\in \cS\times \cS:\, xy=U^{m_3}V^{n_3}\}<+\infty$ for any given $m_3,n_3\geq 0$. Now Axiom \uppercase\expandafter{\romannumeral3} follows. The proof is completed.
\end{proof}

There are many other examples. Huang \cite{Hua19} shows several classes of natural monoids, which includes the monoid generated by three matrices
\begin{equation} \label{eq_example_Huang}
p_0=
\left(
  \begin{array}{ccc}
    0 & 2 \\
    3 & 0
  \end{array}
\right), \quad p_1=
\left(
  \begin{array}{ccc}
    1 & 0 \\
    0 & 2
  \end{array}
\right), \quad p_2=
\left(
  \begin{array}{ccc}
    2 & 0 \\
    0 & 1
  \end{array}
\right).
\end{equation}

\subsection{Divisors, Multiples, and Irreducible Decompositions}

Now we start to explore general theories on integral monoids. 

\begin{definition}
Let $u,v,w$ be elements in $S$ such that $uv=w$.

We say that $u$ is a divisor of $w$, or $u$ divides $w$, or $w$ is a multiple of $u$, and denote $u|w$. Equivalently, we have $u|w$ if and only if $w\in uS$ if and only if $wS\subseteq uS$.

Moreover, we say that $v$ is a co-divisor of $w$, or $v$ co-divides $w$, or $w$ is a co-multiple of $v$, and denote $v \ddagger w$. Equivalently, we have $v\ddagger w$ if and only if $w\in Sv$ if and only if $Sw \subseteq Sv$.
\end{definition}

\begin{lemma} \label{lem_antisymmetry}
(\romannumeral1) If $u,v$ are two elements in $S$ satisfying $u S=v S$, then $u=v$.

(\romannumeral2) If $u,v$ are two elements in $S$ satisfying $Su=Sv$, then $u=v$.
\end{lemma}

\begin{proof}
(\romannumeral1) It follows from $u S=v S$ and $1\in  S$ that $v^{-1}u \in S$. Similarly, one has $(v^{-1}u)^{-1}=u^{-1}v\in  S$. By Axiom \uppercase\expandafter{\romannumeral1}, we conclude that $v^{-1}u\in S\cap S^{-1}=\{1\}$. So $u=v$. Similar arguments lead to (\romannumeral2).
\end{proof}

\begin{lemma}
The relations ``$|$'' and ``$\ddagger$'' are partial orders over $S$.
\end{lemma}

\begin{proof}
Let $u,v,w\in  S$.

(\romannumeral1) (Reflexivity.) Since $u\in u S$, one has $u|u$.

(\romannumeral2) (Antisymmetry.) If $u|v$ and $v|u$, then $uS=vS$. It follows from Lemma \ref{lem_antisymmetry} that $u=v$.

(\romannumeral3) (Transitivity.) Suppose that $u | v$ and $v|w$. That is to say, we have $w\in v S$ and $v\in u S$. So $w\in u S$, i.e., $u|w$.

As a result, we conclude that ``$|$'' is a partial order. Similar arguments show that ``$\ddagger$'' is also a partial order on $S$.
\end{proof}

Here are some basic properties about the divisibility.

\begin{lemma} \label{lem_basic_properties_of_divisibility}

Let $u,v\in S$, $w\in G$.

(\romannumeral1) Suppose that $wu\in  S$. If $u|v$, then $wu|wv$.

(\romannumeral2) If $uv|u$, then $v=1$.

(\romannumeral3) If $uv|v$, then $u=1$.
\end{lemma}

\begin{proof} \label{lem_basic_properties}
(\romannumeral1) The conclusion follows since $v\in uS$ implies $wv \in wuS$.

(\romannumeral2) It follows from (\romannumeral1) that $v|1$. Combining the fact $1|v$ and Lemma \ref{lem_antisymmetry}, one obtains $v=1$.

(\romannumeral3) Since $u| uv$ and $uv|v$, one has $u|v$. Write $v=uw$ for some $w\in S$. Then $u^2w|v$, which leads to $u^2|v$. By similar argument, one obtains $u^k|v$ for all $k\geq 1$. By  Axiom \uppercase\expandafter{\romannumeral3}, the element $v$ has only finitely many divisors. Then there are some $k_1\neq k_2$ such that $u^{k_1}=u^{k_2}$. Noting that $G$ has no torsion, we conclude that $u=1$.
\end{proof}

The following lemma follows similarly, whose proof is omitted here.

\begin{lemma} \label{lem_basic_properties_of_divisibility}

Let $u,v\in S$, $w\in G$.

(\romannumeral1) Suppose that $uw\in  S$. If $u \ddagger v$, then $uw \ddagger vw$.

(\romannumeral2) If $uv \ddagger v$, then $u=1$.

(\romannumeral3) If $uv \ddagger u$, then $v=1$.
\end{lemma}

Next, we will show that the least common multiple and the greatest common divisor can be well-defined on an integral monoid $S$.

\begin{lemma} \label{lem_lcm_may_be_defined}
For any $u,v\in S$, there is a unique element $w\in S$ such that $uS\cap vS=wS$.
\end{lemma}

\begin{proof}
By Axiom \uppercase\expandafter{\romannumeral2}, the element $v^{-1}u$ has a right fraction $xy^{-1}$ in lowest terms with numerator $x$ and denominator $y$. Put $w=uy=vx$. Then $wS\subseteq uS\cap vS$. On the other hand, let us consider any $z\in uS\cap vS$. Write $z=u\widetilde{y}=v\widetilde{x}$ for some $\widetilde{x},\widetilde{y}\in S$. Then $v^{-1}u= \widetilde{x}\widetilde{y}^{-1}$. The right-hand side of the above equality is also a fraction of $v^{-1}u$. By Axiom \uppercase\expandafter{\romannumeral2}, there is some $c\in S$ such that $\widetilde{x}=xc$ and $\widetilde{y}=yc$. So $z=u\widetilde{y}=uyc=wc\in wS$. We obtain that $uS\cap vS \subseteq wS$. Now the existence of such a $w$ is obtained. The uniqueness follows from Lemma \ref{lem_antisymmetry}. This completes the proof.
\end{proof}

\begin{definition}
For $u,v\in S$, we define the least common multiple of $u$ and $v$ to be $\lcm[u,v]=w$ with $w$ the unique element in $S$ such that $uS\cap vS=wS$.
\end{definition}

Note that
\[
\lcm[u,v]S = uS\cap vS = vS\cap uS = \lcm[v,u]S,\quad (u,v\in S).
\]
So $\lcm[u,v]=\lcm[v,u]$. Similarly,
\begin{equation} \label{eq_def_of_lcm_general}
u_1S\cap u_2S\cap u_3S = \lcm[u_1,u_2]S\cap u_3S = \lcm[\lcm[u_1,u_2],u_3]S,\quad (u_1,u_2,u_3\in S).
\end{equation}
Since an intersection of sets does not depend on the order, we obtain the same if we permute $u_1,u_2$ and $u_3$ in \eqref{eq_def_of_lcm_general}. Therefore, it is natural to define $\lcm[u_1,u_2,u_3]=\lcm[\lcm[u_1,u_2],u_3]$ and
\[
\lcm[u_1,u_2,\ldots,u_k] = \lcm[\lcm[u_1,u_2,\ldots, u_{k-1}],u_k],\quad (k\geq 2, u_1,\ldots,u_k\in S)
\]
in general. We also write $\lcm[u]=u$ for a single element $u\in S$ and $\lcm[F]=\lcm[u_1,u_2,\ldots,u_k]$ for a non-empty finite set $F=\{u_1,u_2,\ldots,u_k\}\subseteq S$. The following lemma follows immediately.

\begin{lemma} \label{lem_numbertheoretic_def_lcm}
Let $k\geq 1$ and $u_1,u_2,\ldots,u_k,v\in S$.

(\romannumeral1) For $1\leq j\leq k$, we have $u_j| \lcm[u_1,u_2,\ldots,u_k]$.

(\romannumeral2) If $u_1,u_2,\ldots,u_k|v$, then $\lcm[u_1,u_2,\ldots,u_k]|v$.
\end{lemma}

\begin{example}
In Thompson's monoid $\bS$, we have $\lcm[p_0,p_1]= p_0p_2=p_1p_0$ and $\lcm[p_0^2,p_1^2]=p_0^2p_3^2=p_1^2p_0^2$.
\end{example}

Next, we turn to consider the notion of greatest common divisor of given elements.

\begin{definition}
For finitely many elements $u_1,u_2,\ldots,u_k\in S$ $(k\geq 2)$, we define their greatest common divisor to be
\begin{equation} \label{eq_def_gcd}
\gcd(u_1,u_2,\ldots,u_k) = \lcm[w\in S:\, w|u_j\, (1\leq j\leq k)].
\end{equation}
Or, equivalently,
\[
\gcd(u_1,u_2,\ldots,u_k)S=\bigcap\limits_{w|u_1,\ldots,u_k}wS.
\]
\end{definition}

By Axiom \uppercase\expandafter{\romannumeral3}, the set $\{w\in S:\, w|u_j\, (1\leq j\leq k)\}$ has finite cardinality. And it is non-empty, since $1$ is the divisor of any element of $S$. So the expression on the right-hand side of \eqref{eq_def_gcd} is well-defined. Moreover, the definition remains the same if we permute $u_1,u_2,\ldots,u_k$ in \eqref{eq_def_gcd}. We also write $\gcd(u)=u$ for a single element $u\in S$ and $\gcd(F)=\gcd(u_1,u_2,\ldots,u_k)$ for a non-empty finite set $F=\{u_1,u_2,\ldots,u_k\}\subseteq S$.

\begin{lemma} \label{lem_numbertheoretic_def_gcd}
Let $k\geq 1$ and $u_1,u_2,\ldots,u_k,v\in S$.

(\romannumeral1) For $1\leq j\leq k$, we have $\gcd(u_1,u_2,\ldots,u_k)|u_j$.

(\romannumeral2) If $v|u_1,u_2,\ldots,u_k$, then $v|\gcd(u_1,u_2,\ldots,u_k)$.
\end{lemma}

\begin{proof}
(\romannumeral1) For any $1\leq j\leq k$, and any $w\in S$ with $w|u_1,u_2,\ldots,u_k$, we always have $u_j\in wS$. Then (\romannumeral1) holds due to the fact that
\[
u_j\in \bigcap\limits_{w|u_1,\ldots,u_k}wS =\gcd(u_1,u_2,\ldots,u_k)S.
\]

(\romannumeral2) The conclusion follows by noticing that
\[
\gcd(u_1,u_2,\ldots,u_k) S = \bigcap\limits_{w|u_1,\ldots,u_k}wS \subseteq vS.
\]
\end{proof}

\begin{lemma}
For $u_1,u_2,\ldots,u_k\in S$ $(k\geq 3)$, we have
\[
\gcd(u_1,u_2,\ldots,u_k)=\gcd(\gcd(u_1,\ldots,u_{k-1}),u_k).
\]
\end{lemma}

\begin{proof}
By Lemma \ref{lem_numbertheoretic_def_gcd}(\romannumeral1), one obtains
\begin{align*}
&\gcd(\gcd(u_1,\ldots,u_{k-1}),u_k)\, |\, \gcd(u_1,\ldots,u_{k-1}),\\
&\gcd(\gcd(u_1,\ldots,u_{k-1}),u_k)| u_k.
\end{align*}
Since $\gcd(u_1,\ldots,u_{k-1})|u_j$ $(1\leq j\leq k-1)$, one obtains $\gcd(\gcd(u_1,\ldots,u_{k-1}),u_k)|u_j$ for $1\leq j\leq k-1$ as well. Now Lemma \ref{lem_numbertheoretic_def_gcd} (\romannumeral2) ensures that
\[
\gcd(\gcd(u_1,\ldots,u_{k-1}),u_k)| \gcd(u_1,u_2,\ldots,u_k).
\]

On the other hand, since $\gcd(u_1,u_2,\ldots,u_k)| u_j$ $(1\leq j\leq k)$, we have
\[
\gcd(u_1,u_2,\ldots,u_k)| \gcd(u_1,u_2,\ldots,u_{k-1})
\]
by Lemma \ref{lem_numbertheoretic_def_gcd} (\romannumeral2). It follows that
\[
\gcd(u_1,u_2,\ldots,u_k)|\gcd(\gcd(u_1,\ldots,u_{k-1}),u_k).
\]
This completes the proof.
\end{proof}

We explore some other properties below. Let $c\in G$ and $u_1,u_2,\ldots,u_k\in S$ $(k\geq 1)$. We have
\[
cu_1S\cap cu_2S\cap \ldots cu_kS = c(u_1S\cap u_2S\cap \ldots u_kS) = c\cdot \lcm[u_1,u_2,\ldots,u_k]S.
\]
When $cu_1,cu_2,\ldots,cu_k\in S$, the above formula becomes
\[
\lcm[cu_1,cu_2,\ldots,cu_k] = c\cdot \lcm[u_1,u_2,\ldots,u_k].
\]
Similar result also holds for the greatest common divisor.

\begin{lemma} \label{lem_c_dot_gcd}
Let $k\geq 1$ and $c,u_1,\ldots,u_k\in S$. Then $\gcd(cu_1,\ldots,cu_k)=c\cdot\gcd(u_1,\ldots,u_k)$.
\end{lemma}

\begin{proof}
For simplicity, we denote $d=\gcd(u_1,\ldots,u_k)$ and $e=\gcd(cu_1,\ldots,cu_k)$. Write $u_j=dw_j$ and $cu_j=ev_j$ for some $u_j,v_j\in S$ $(1\leq j\leq k)$. Note that $cd|cdw_j=cu_j$ $(1\leq j\leq k)$, which implies $cd|\gcd(cu_1,\ldots,cu_k)=e$ by Lemma \ref{lem_numbertheoretic_def_gcd}(\romannumeral2). One the other hand, write $\lcm[e,c]=ey=cx$ for some $x,y\in S$. Note that $ev_j=cu_j$ $(1\leq j\leq k)$, which are all common multiples of $e$ and $c$. So $cx|cu_j$ by Lemma \ref{lem_numbertheoretic_def_lcm}(\romannumeral2), which leads to $x|u_j$ $(1\leq j\leq k)$ by Lemma \ref{lem_basic_properties_of_divisibility}(\romannumeral1). One deduces that $x|d$ by Lemma \ref{lem_numbertheoretic_def_gcd}(\romannumeral2). Now $e|cx$ and $cx|cd$. So $e|cd$. The lemma now follows.
\end{proof}

\begin{corollary} \label{cor_remains_of_gcd_has_gcd_1}
Let $u_1,u_2,\ldots,u_k \in S$. Suppose that $u_j=\gcd(u_1,u_2,\ldots,u_k)\cdot v_j$ $(1\leq j\leq k)$. Then $\gcd(v_1,v_2,\ldots,v_k)=1$.
\end{corollary}

\begin{proof}
Write $c=\gcd(u_1,u_2,\ldots,u_k)$. It follows from Lemma \ref{lem_c_dot_gcd} that
\[
c=\gcd(u_1,u_2,\ldots,u_k) = \gcd(cv_1,cv_2,\ldots,cv_k) = c\cdot \gcd(v_1,v_2,\ldots,v_k).
\]
The corollary then follows.
\end{proof}

Suppose that $k\geq 2$ and $u_1,\ldots,u_k$ are elements in $S$ such that $\gcd(u_1,\ldots,u_k)=1$. Then we say that the elements $u_1,\ldots,u_k$ are free. The reason for not using the term ``coprime'' is because we have not established the notion of ``prime'' yet.

\begin{remark} \label{remark_lcm_to_gcd}
The statement that ``$\lcm[u,v]=uu_1=vv_1$ implies $\gcd(u_1,v_1)=1$'' is false. For example, in Thompson's monoid $\bS$, we have
\[
\lcm[p_0p_1p_2p_8,p_3]=p_0p_1p_2p_8\cdot p_6=p_3\cdot p_0p_1p_2p_9.
\]
Noting that $p_0p_1p_2p_9=p_6p_0p_1p_2$, we also have $\gcd(p_6,p_0p_1p_2p_9)=p_6$.
\end{remark}

\medskip

The divisor function is defined by
\[
\tau(z) = \#\{(z_1,z_2)\in S:\ z=z_1z_2\},
\]
which counts the number of divisors, or co-divisors, of $z$. We call an element $p$ in $S$ irreducible, if $\tau(p)=2$. That is to say, the only divisors of an irreducible element is $1$ and itself. We use $\cP$ to denote the set of all irreducible elements in $S$. For the natural numbers $\bN$, one has $\cP=\{2,3,5,7,11,\ldots\}$. For Thompson's monoid $\bS$, one has $\cP=\{p_0,p_1,p_2,\ldots\}$ as in \eqref{eq_define_p_j_in_Thompson_introduction}. Irreducible elements usually generate the monoid and also the group. However, it should be pointed out that $\bG$ can be generated by only two elements $\{p_0,p_1\}$, while $\bS$ has infinitely many irreducible elements $\cP=\{p_0,p_1,p_2,\ldots\}$.

\begin{lemma} [Irreducible divisors] \label{lem_prime_divisor_exsits}
Suppose that $u$ is an element in $S$ with $u\neq 1$. Then there is some $p\in \cP$ such that $p|u$.
\end{lemma}

\begin{proof}
We write $u^{(1)}=u$ and use iteration. For $k\geq 1$, whenever $\tau(u^{(k)})> 2$, there are some divisor $u^{(k+1)}$ of $u^{(k)}$ with $u^{(k+1)}\neq 1$ and $u^{(k+1)}\neq u^{(k)}$. Noting that each divisor of $u^{(k+1)}$ is a divisor of $u^{(k)}$, one obtains $2\leq \tau(u^{(k+1)})<\tau(u^{(k)})$. By Axiom \uppercase\expandafter{\romannumeral3}, one has $\tau(u)<+\infty$. So the iterating process will stop at some step, say, $K$, with $u^{(K)}=2$. Now $u^{(K)}\in \cP$ and it is a divisor of $u$.
\end{proof}

\begin{lemma}  [Irreducible decompositions] \label{lem_irr_decomposition_of_any_element}
For any $u\neq 1$, there exists some $K\geq 1$ and $q_1,q_2,\ldots,q_K\in \cP$ such that $u=q_1q_2\ldots q_K$.
\end{lemma}

\begin{proof}
We put $u^{(1)}=u$ and use iteration. For any $k\geq 1$, by Lemma \ref{lem_prime_divisor_exsits}, there is some $q_k\in \cP$ such that $q_k| u^{(k)}$. Whenever $\tau(u^{(k)})>2$, we write $u^{(k)}=q_k u^{(k+1)}$ for some $u^{(k+1)}\in S$ with $\tau(u^{(k+1)})\neq 1$. Then we iterate with $k+1$ instead of $k$. Since $\tau(u^{(k+1)})<\tau(u^{(k)})$ for each $k\geq 1$ and $\tau(u)<+\infty$ by \textsc{Axiom} \uppercase\expandafter{\romannumeral3}, the iteration process will stop at step, say $K$, with $\tau(u^{K})=2$. Then we denote $q_k=u^{(K)}$, which belongs to $\cP$. It appears that $u=q_1q_2\ldots q_K$ with $q_j\in \cP$ $(1\leq j\leq K)$.
\end{proof}

There may be many ways to write an element as a product of irreducible elements. When different irreducible decompositions of a given element are considered, we will call an irreducible element a letter, and call a composition of letters a word. 

\subsection{Co-divisors, Co-multiples and a Duality}

The least common co-divisor of two elements may not exist. However, we show below that such notation still works when we put some upper bound on the elements involved.

\begin{lemma} \label{lem_lcm_ddagger_exists}
Let $u_1,u_2,\ldots,u_k,w$ $(k\geq 1)$ be elements in $S$ with $u_1,u_2,\ldots,u_k\ddagger w$. Then there exists an element $z\ddagger w$ with the following two properties.

(\romannumeral1) It satisfies that $u_1,u_2,\ldots,u_k\ddagger z$.

(\romannumeral2) If $v$ is an element in $S$ such that $v\ddagger w$ and $ u_1,u_2,\ldots,u_k \ddagger v$, then $z\ddagger v$.
\end{lemma}

\begin{proof}
(\romannumeral1) We write $w=c_ju_j$ for $1\leq j\leq k$. Let $c=\gcd(c_1,c_2,\ldots,c_k)$ and $c_j=c x_j$ $(1\leq j\leq k)$. Noting that $c|w$, we write $w=cz$. Then $z=x_ju_j$ for all $1\leq j\leq k$. It follows that $z\ddagger w$ and $u_j\ddagger z$ for $1\leq j\leq k$.

(\romannumeral2) For $v\ddagger w$ with $u_1,u_2,\ldots,u_k \ddagger v$, we write $w=dv$ and $v=y_ju_j$. Then $c_ju_j=w=dy_ju_j$, which implies $c_j=dy_j$ $(1\leq j\leq k)$. Now $d|c_1,\ldots,c_k$ and so $d|\gcd(c_1,c_2,\ldots,c_j)=c$. Write $c=de$. Then $dv=w=cz=dez$. Thus, one obtains $v=ez$. The proof is completed.
\end{proof}

\begin{definition}
Let $k\geq 1$ and $u_1,u_2,\ldots,u_k,w$ be the elements in $S$ with $u_1,u_2,\ldots,u_k\ddagger w$. Define the least common co-multiple of $u_1,u_2,\ldots,u_k$ up to $w$ to be
\begin{equation} \label{eq_lcm_ddagger_definition}
\lcm_\ddagger[w;u_1,\ldots,u_k]= \left(\gcd\left(wu_1^{-1},\ldots,wu_k^{-1}\right)\right)^{-1}w.
\end{equation}
\end{definition}

\begin{lemma} \label{lem_gcd_ddagger_exists}
Let $u_1,u_2,\ldots,u_k,w$ $(k\geq 1)$ be elements in $S$ with $u_1,u_2,\ldots,u_k\ddagger w$. Then there exists an element $z\ddagger w$ with the following two properties.

(\romannumeral1) It satisfies that $z\ddagger u_1,u_2,\ldots,u_k$.

(\romannumeral2) If $v$ is an element in $S$ satisfying $v\ddagger u_1,u_2,\ldots,u_k$, then $v\ddagger z$.
\end{lemma}

\begin{proof}
(\romannumeral1) We write $w=c_ju_j$ for $1\leq j\leq k$. Let $c=\lcm[c_1,c_2,\ldots,c_k]=c_jx_j$ $(1\leq j\leq k)$. Since $c_1,c_2,\ldots,c_k|w$, one has $c|w$. We put $w=cz$. Then $z\ddagger w$. In view of $u_j=x_jz$, one obtains $z\ddagger u_j$ for $1\leq j\leq k$.

(\romannumeral2) Since $v\ddagger u_1,u_2,\ldots,u_k$, we write $u_j=y_jv$. Since $w=c_jy_jv$, one gets $c_j| c_jy_j=wv^{-1}$ for $1\leq j\leq k$. Here $wv^{-1}\in S$. It follows that $c=\lcm[c_1,\ldots,c_k]| wv^{-1}$. We write $wv^{-1}=cd$. Then $cz=w=cdv$, which implies that $z=dv$. This completes the proof.
\end{proof}

\begin{definition}
Let $k\geq 1$ and $u_1,u_2,\ldots,u_k,w$ be the elements in $S$ with $u_1,u_2,\ldots,u_k\ddagger w$. Define the greatest common co-divisor of $u_1,u_2,\ldots,u_k$ up to $w$ to be
\begin{equation} \label{eq_gcd_ddagger_definition}
\gcd_\ddagger(w;v_1,\ldots,v_k) = \left(\lcm\left[wv_1^{-1},\ldots,wv_k^{-1}\right]\right)^{-1}w.
\end{equation}
\end{definition}

From now on, whenever we write $\lcm_\ddagger[w;u_1,u_2,\ldots,u_k]$ or $\gcd_\ddagger(w;u_1,u_2,\ldots, u_k)$, we always mean that $k\geq 1$, the elements $u_1,u_2,\ldots,u_k,w$ belongs to $S$, and they satisfy $u_1,u_2,\ldots,u_k\ddagger w$. The following three lemmas can be verified by direct computation, and we omit the proofs here.

\begin{lemma}
Let $k\geq 1$ and $u_1,u_2,\ldots,u_k,w$ be the elements in $S$ with $u_1,u_2,\ldots,u_k\ddagger w$. Then

(\romannumeral1) $\lcm_\ddagger[w;u_1,u_2,\ldots, u_k] = \lcm_\ddagger[w;\lcm_\ddagger[w;u_1,\ldots,u_{k-1}],u_k]$;

(\romannumeral2) $\gcd_\ddagger(w;u_1,u_2,\ldots, u_k) = \gcd_\ddagger(w;\gcd_\ddagger(w;u_1,\ldots,u_{k-1}),u_k)$.
\end{lemma}

\begin{lemma} \label{lem_codivisor_pro}
Let $k\geq 1$. Suppose that $c,u_1,u_2,\ldots,u_k,w$ are elements in $S$ with $u_1,\ldots,u_k\ddagger w$. Then

(\romannumeral1) $\lcm_\ddagger[wc;u_1c,u_2c,\ldots,u_kc] = \lcm_\ddagger[w;u_1,u_2,\ldots,u_k]\cdot c$;

(\romannumeral2) $\gcd_\ddagger(wc;u_1c,u_2c,\ldots,u_kc) = \gcd_\ddagger(w;u_1,u_2,\ldots,u_k)\cdot c$.
\end{lemma}

\begin{lemma} \label{lem_ddagger_upperbound_transfer}
Let $k\geq 1$ and $u_1,u_2,\ldots,u_k,w,w^\prime\in S$. Suppose that $w\ddagger w^\prime$ and $u_1,u_2,\ldots,u_k$ are co-divisors of both $w, w^\prime$. Then

(\romannumeral1) $\lcm_\ddagger[w;u_1,u_2,\ldots,u_k] = \lcm_\ddagger[w^\prime;u_1,u_2,\ldots,u_k]$;

(\romannumeral2) $\gcd_\ddagger(w;u_1,u_2,\ldots,u_k)=\gcd_\ddagger(w^\prime;u_1,u_2,\ldots,u_k)$.
\end{lemma}

However, if there are no information about the relation of $w$ and $w^\prime$ in above lemma, then we do not know the relation of the least common co-multiples and greatest common co-divisors either.





\begin{corollary}
Let $k\geq 1$ and $u_1,u_2,\ldots,u_k,w$ are elements in $S$ with $u_1,\ldots,u_k\ddagger w$. Suppose that $\gcd_\ddagger(w;u_1,\ldots,u_k)=d$. Suppose further that $u_j=u_j^\prime d$ $(1\leq j\leq k)$ and $w=w^\prime d$. Then $\gcd_\ddagger(w^\prime; u_1^\prime,\ldots,u_k^\prime)=1$.
\end{corollary}

\begin{proof}
Applying Lemma \ref{lem_codivisor_pro}(\romannumeral2), we obtain
\[
d=\gcd_\ddagger(w;u_1,\ldots,u_k) = \gcd_\ddagger(w^\prime d;u_1^\prime d,\ldots, u_k^\prime d)= \gcd_\ddagger(w^\prime;u_1^\prime,\ldots, u_k^\prime) \cdot d.
\]
The corollary then follows.
\end{proof}

Next we obtain a duality between common divisors/multiples and common co-divisors/co-multiples.

\begin{lemma} [Duality]\label{lem_gcd_gcddagger_transfer}
(\romannumeral1) Let $u,v,x,y,z$ be elements in $S$ satisfying $\gcd(u,v)=1$ and $z=\lcm[u,v]=uy=vx$. Then $\gcd_\ddagger(z;x,y)=1$ and $\lcm_\ddagger[z;x,y]=z$.

(\romannumeral2) Let $w,u,v,x,y,z$ be elements in $S$ satisfying $x,y\ddagger w$, $\gcd_\ddagger(w;x,y)=1$ and $z=\lcm_\ddagger[w;x,y]=uy=vx$. Then $\gcd(u,v)=1$ and $\lcm[u,v]=z$.
\end{lemma}

\begin{proof}
We prove (\romannumeral2) here. Since $z\ddagger w$ and $x,y\ddagger z$, we obtain by Lemma \ref{lem_ddagger_upperbound_transfer} that $\gcd_\ddagger(w;x,y)=\gcd_\ddagger(z;x,y)$ and $\lcm_\ddagger[w;x,y]=\lcm_\ddagger[z;x,y]$. It follows from \eqref{eq_gcd_ddagger_definition} that
\[
1=\gcd_\ddagger(z;x,y)=\left(\lcm[zx^{-1},zy^{-1}]\right)^{-1}z = \left(\lcm[u,v]\right)^{-1}z,
\]
which shows that $\lcm[u,v]=z$. Similarly, by \eqref{eq_lcm_ddagger_definition}, we have
\[
z=\lcm_\ddagger[z;x,y] = \left(\gcd(zx^{-1},zy^{-1})\right)^{-1}z = \left(\gcd(u,v)\right)^{-1}z.
\]
Therefore $\gcd(u,v)=1$. Similar arguments lead to (\romannumeral1). The proof is completed.
\end{proof}

\begin{remark}
Let us reconsider the proposition in Remark \ref{remark_lcm_to_gcd}. Suppose that $\lcm[u,v]=uu_1=vv_1$. Let $d=\gcd(u,v)$ and $u=du_0$, $v=dv_0$. Then $\gcd(u_0,v_0)=1$ and $\lcm[u,v]=d\cdot \lcm[u_0,v_0]$. It follows that $\lcm[u_0,v_0]=u_0u_1=v_0v_1$. By Lemma \ref{lem_gcd_gcddagger_transfer}, we obtain that $\gcd_\ddagger(uu_1;u_1,v_1)=\gcd_\ddagger(u_0u_1;u_1,v_1)=1$.
\end{remark}

For an abelian $S$, it is not necessary to distinguish between divisors and co-divisors, or multiples and co-multiples. And one obtains above lemma immediately. However, for an non-abelian $S$, the duality in Lemma \ref{lem_gcd_gcddagger_transfer} is crucial. It shows that the information provided by divisors is nearly equivalent to that by co-divisors. However, information from only one side is not enough for arithmetics. This duality will play an important role in Section \ref{section_homogeneous_monoid}.

\subsection{Left Regular Representations of Integral Monoids}

Denote $\cH=l^2(S)$. For $w\in S$, let $\delta_w$ be the function that $\delta_w(w)=1$ and $\delta_w(z)=0$ for $z\neq w$. Then $\{\delta_w:\, w\in S\}$ is an orthonormal basis of the Hilbert space $\cH$. For $u\in S$, let $L_u$ be the operator induced by $L_u \delta_w = \delta_{uw}$. Equivalently,
\[
(L_u f)(z) =
\begin{cases}
f(u^{-1}z),\quad &\If z\in uS,\\
0,\quad &\If z\in S\setminus uS
\end{cases}
\]
for $f\in \cH$. Since
\[
\|L_u f\|^2 = \sum\limits_{z\in uS} |f(u^{-1}z)|^2 =\sum\limits_{z\in S} |f(z)|^2 = \|f\|^2,
\]
one has $\|L_u\|=1$ and $L_u$ is an isometry. In particular, $L_1=I$ is the identity operator. By calculation, one obtains that
\[
(L_v^\ast g)(z) = g(vz), \quad (g\in \cH, \, z\in S).
\]
Or, equivalently, the adjoint operator $L_v^\ast$ is given by $L_v^\ast \delta_w = \delta_{v^{-1}w}$ for $w\in vS$ and $L_v^\ast \delta_w=0$ for $w\in S\setminus vS$. Now, let $\cQ$ be the monoid of $B(\cH)$ generated by $L_u,L_u^\ast$ $(u\in S)$.

\begin{lemma}
We have $\cQ =\{L_uL_v^\ast: \, u,v \in S\}$. It satisfies that
\begin{equation} \label{eq_cQ_relation_section2}
L_{u_1}L_{u_2}=L_{u_1u_2},\quad L_{v_2}^\ast L_{v_1}^\ast = L_{v_1v_2}^\ast,\quad L_v^\ast L_u = L_{v^{-1}\lcm[u,v]}L^\ast_{u^{-1}\lcm[u,v]}
\end{equation}
for $u,v,u_1,u_2,v_1,v_2\in S$. Moreover, we have $L_{u_1}L_{v_1}^\ast=L_{u_2}L_{v_2}^\ast$ if and only if $u_1=u_2$ and $v_1=v_2$.
\end{lemma}

\begin{proof}
The first two equalities in \eqref{eq_cQ_relation_section2} follows immediately. We prove the third one below. For $u,v,w\in S$, calculation reveals that
\[
L_v^\ast L_u \delta_w =
\begin{cases}
\delta_{v^{-1}uw},\quad &\If w\in S\cap u^{-1}vS,\\
0,\quad &\Otherwise,
\end{cases}
\]
and
\[
L_{v^{-1}\lcm[u,v]}L^\ast_{u^{-1}\lcm[u,v]}\delta_w =
\begin{cases}
\delta_{v^{-1}\lcm[u,v](u^{-1}\lcm[u,v])^{-1}w},\quad &\If w\in u^{-1}\lcm[u,v]S,\\
0,\quad &\Otherwise,
\end{cases}
\]
Note that $v^{-1}\lcm[u,v](u^{-1}\lcm[u,v])^{-1}=v^{-1}u$ and
\[
S\cap u^{-1}vS = u^{-1}(uS\cap vS) = u^{-1}\lcm [u,v]S.
\]
We conclude that $L_v^\ast L_u = L_{v^{-1}\lcm[u,v]}L^\ast_{u^{-1}\lcm[u,v]}$.

Now for any finite product of $L_w$'s and $L_w^\ast$'s $(w\in S)$, we can always move the operators with a $\ast$ to the right-hand side, and obtain an operator of the form $L_uL_v^\ast$ with some $u,v\in S$. So $\cQ =\{L_uL_v^\ast: \, u,v \in S\}$.

Moreover, we have
\[
L_uL_v^\ast\delta_w =
\begin{cases}
\delta_{uv^{-1}w},\quad &\If w\in vS,\\
0,\quad &\Otherwise.
\end{cases}
\]
One sees that $L_{u_1}L_{v_1}^\ast=L_{u_2}L_{v_2}^\ast$ if and only if $u_1 v_1^{-1}=u_2 v_2^{-1}$ and $v_1S=v_2S$, if and only if $u_1=u_2,\, v_1=v_2$.
\end{proof}

Let $\fA$ be the $C^\ast$-algebra generated by $\cQ$ in $\cH$. Denote $E_w=L_wL_w^\ast$ for $w\in S$, which is the projection from $\cH$ onto the closed subspace spanned by $\{\delta_z:\,w|z\}$. Note that, for $w_1,w_2\in S$,
\[
E_{w_1}E_{w_2} = L_{w_1}L_{w_1}^\ast L_{w_2}L_{w_2}^\ast = L_{\lcm[w_1,w_2]}L_{\lcm[w_1,w_2]}^\ast = E_{w_2}E_{w_1}.
\]
For an element $u\in S$, we use $P_u$ to denote the projection from $\cH$ onto $\bC\delta_u$.

\begin{theorem}
The following statements are equivalent.

(\romannumeral1) There are infinitely many irreducible elements in $S$.

(\romannumeral2) The projection $P_1$ does not belong to $\fA$.
\end{theorem}

\begin{proof}
We first prove that (\romannumeral2) implies (\romannumeral1). Suppose on the contrary that $\cP=\{q_1,q_2,\ldots,q_k\}$. Note that $(I-E_{q_1})(I-E_{q_2})\ldots (I-E_{q_k})$ is the projection onto the closed subspace spanned by
\[
\{\delta_w:\, q_1\nmid w,\, q_2\nmid w,\, \ldots, q_k\nmid w\}.
\]
By Lemma \ref{lem_prime_divisor_exsits}, the above set equals $\{\delta_1\}$. Then
\[
P_1=(I-E_{q_1})(I-E_{q_2})\ldots (I-E_{q_k}) = \sum\limits_{F\subseteq \{1,2,\ldots,k\}}(-1)^{|F|} E_{\lcm[q_j:\,j\in F]} \, \in \fA.
\]
A contradiction appears.

In the following, we shall prove that (\romannumeral1) implies (\romannumeral2). Assume on the contrary that $P_1\in \fA$. Then there is a finite sum $T=\sum\nolimits_{(u,v)\in F_0} c_{uv} L_uL_v^\ast$ such that $\|T-P_1\|<1/10$. In particular, we have
\[
\|P_1TP_1-P_1\|\leq \|P_1\|\cdot \|T-P_1\|\cdot \|P_1\| <1/10.
\]
Note that $P_1L_uL_v^\ast P_1 = P_1$ if $(u,v)=(1,1)$ and $P_1L_uL_v^\ast P_1=0$ otherwise. So $P_1TP_1=c_{11} P_1$. It follows that $|c_{11}-1|=\|c_{11}P_1-P_1\|<1/10$, which implies $|c_{11}|\geq 9/10$.

One the other hand, recall that the set $F_0$ has finite cardinality. Note that each element in $S$ has only finitely many irreducible divisors, and $\cP$ contains infinitely many elements. So there is some $q\in \cP$ such that $\gcd(q,u)=\gcd(q,v)=1$ for all $(u,v)\in F_0$. Calculations show that $P_q L_u L_v^\ast P_q = P_q$ for $(u,v)=(1,1)$ and $P_q L_u L_v^\ast P_q = 0$ for $(u,v)\in F_0\setminus \{(1,1)\}$. So
\[
|c_{11}|=\|c_{11} P_q -0\| = \|P_qTP_q-P_qP_1P_q\| \leq \|P_q\|\cdot \|T-P_1\|\cdot \|P_q\| <1/10.
\]
Now a contradiction appears.

\end{proof}

\subsection{A Trace on the $C^\ast$-algebra}

Let $\{F_l\}_{l=1}^\infty$ be a sequence of subsets of $S$ with $F_1\subseteq F_2\subseteq F_3\subseteq \ldots$ and $\bigcup\nolimits_{l=1}^\infty F_l = S$. Put $z_l = \lcm[z:\, z\in F_l]$. It is not hard to see that for any given $u\in S$, there is some $L>0$ such that $u|z_l$ for all $l\geq L$.

\begin{lemma}
The function $\tau:\, \fA\rightarrow \bC$ given by
\begin{equation} \label{eq_simplest_trace_definition}
\tau(A) = \lim\limits_{l\rightarrow \infty} \< A \delta_{z_l}, \, \delta_{z_l} \>
\end{equation}
is well-defined, and is a trace on $\fA$.
\end{lemma}

\begin{proof}
First, we shall show that $\tau$ is a well-defined bounded linear functional on $\fA$. There are three steps: showing that $\tau$ can be defined on a dense subspace of $\fA$ by \eqref{eq_simplest_trace_definition}; showing that $\tau$ can be extended to the whole $\fA$; showing that $\tau$ is defined by \eqref{eq_simplest_trace_definition} on the whole $\fA$.

Let $A=\sum\limits_{(u,v)\in F}c_{uv}L_uL_v^\ast$, where $F$ is a finite set and $c_{uv}$'s are complex numbers. Then there exists an $L\geq 1$ such that $v|z_l$ for all $(u,v)\in F$ and $l\geq L$. Hence
\[
A\delta_{z_l} = \sum\limits_{(u,v)\in F}c_{uv}L_uL_v^\ast \delta_{z_l} = \sum\limits_{(u,v)\in F}c_{uv}\delta_{uv^{-1}z_l}.
\]
Since $\<\delta_{uv^{-1}z_l},\, \delta_{z_l}\>=0$ if and only if $u=v$, one deduces that $\<A\delta_{z_l},\, \delta_{z_l}\> =  \sum\nolimits_{(u,u)\in F} c_{uu}$ for $l\geq L$. Now we have
\begin{equation} \label{eq_simplest_trace_finite_sum_form}
\tau(A)= \lim\limits_{l\rightarrow \infty}\<A\delta_{z_l},\, \delta_{z_l}\> =  \sum\limits_{(u,u)\in F} c_uu.
\end{equation}
It is well-defined on $\Span\{L_uL_v^\ast:\, u,v\in S\}$, which is a dense subspace of $\fA$. Note that
\[
\tau(A)\leq \lim\limits_{l\rightarrow\infty} \|A\|\cdot \|\delta_{z_l}\|_\cH^2 \leq \|A\|
\]
for $A$ in this subspace. By Hahn-Banach theorem, we conclude that $\tau$ can be extended to a linear functional on $\fA$.

Next, we will show that \eqref{eq_simplest_trace_definition} holds for all operators in $\fA$. For any operator $A\in \fA$, there exists a sequence $A_n\in \Span\{L_uL_v^\ast:\, u,v\in S\}$ such that $\|A_n-A\|\rightarrow \infty$ as $n\rightarrow \infty$. For any $\varepsilon>0$, there exists some $N>0$ such that $\|A_n-A\|\leq \varepsilon/3$ whenever $n\geq N$. For this given $N$, there exists some $L>0$ such that $\left|\<A_N \delta_{z_l}, \, \delta_{z_l}\>- \<A_N \delta_{z_l^\prime}, \, \delta_{z_l^\prime}\>\right|<\varepsilon/3$ whenever $l,l^\prime \geq L$, since the limit $\lim\limits_{l\rightarrow \infty}\<A_N\delta_{z_l},\delta_{z_l}\>$ exists. Therefore,
\begin{align*}
&\left|\<A \delta_{z_l}, \, \delta_{z_l}\>-\<A \delta_{z_l^\prime}, \, \delta_{z_l^\prime}\>\right| \\
&\leq \left|\<(A-A_N) \delta_{z_l}, \, \delta_{z_l}\>\right|+ \left|\<A_N \delta_{z_l}, \, \delta_{z_l}\>- \<A_N \delta_{z_l^\prime}, \, \delta_{z_l^\prime}\>\right|+ \left|\<(A_N-A) \delta_{z_l^\prime}, \, \delta_{z_l^\prime}\>\right|\\
&\leq 2\|A_N-A\|+\left|\<A_N \delta_{z_l}, \, \delta_{z_l}\>- \<A_N \delta_{z_l^\prime}, \, \delta_{z_l^\prime}\>\right|<\varepsilon.
\end{align*}
So $\{\<A \delta_{z_l}, \, \delta_{z_l}\>\}_{l=1}^\infty$ is a Cauchy sequence. We denote its limit by $\lim\limits_{l\rightarrow\infty}\<A \delta_{z_l}, \, \delta_{z_l}\>=\alpha$. Moreover, one has
\[
|\tau(A_n)-\alpha| \leq \lim\limits_{l\rightarrow \infty} \left|\<A_n \delta_{z_l},\, \delta_{z_l}\> - \<A \delta_{z_l},\, \delta_{z_l}\>\right| \leq \|A_n-A\| \rightarrow 0,\quad (n\rightarrow \infty).
\]
It follows that $\tau(A) = \lim\limits_{n\rightarrow \infty}\tau(A_n) = \alpha$.

Second, if $A$ is a positive operator, then $\<A\delta_{\delta_{z_l},\delta_{z_l}}\>\geq 0$. One conclude further that $\tau(A)\geq 0$. Now we have shown that $\tau$ is a positive bounded linear functional on $\fA$.

Third, we need to prove that $\tau(AB)=\tau(BA)$ for any $A,B\in \fA$. Thanks to the fact that $\Span\{L_uL_v^\ast:\, u,v\in S\}$ is dense in $\fA$, it is sufficient to show that $\tau(L_{u_1}L_{v_1}^\ast L_{u_2}L_{v_2}^\ast)=\tau(L_{u_2}L_{v_2}^\ast L_{u_1}L_{v_1}^\ast)$ for any $u_1,v_1,u_2,v_2\in S$.  Calculation reveals that
\begin{align*}
&B_1:=L_{u_1}L_{v_1}^\ast L_{u_2}L_{v_2}^\ast = L_{u_1v_1^{-1}\lcm[v_1,u_2]}L_{v_2u_2^{-1}\lcm[v_1,u_2]}^\ast,\\
&B_2:=L_{u_2}L_{v_2}^\ast L_{u_1}L_{v_1}^\ast = L_{u_2v_2^{-1}\lcm[v_2,u_1]}L_{v_1u_1^{-1}\lcm[v_2,u_1]}.
\end{align*}
By \eqref{eq_simplest_trace_finite_sum_form}, we deduce that both $\tau(B_1)$ and $\tau(B_2)$ take value $1$ when $u_1v_1^{-1}=v_2u_2^{-1}$ and $0$ otherwise. This completes the proof.

\end{proof}

Let us recall Gelfand-Naimark-Segal construction with the $C^\ast$-algebra $\fA$ and the state $\tau$ (see Chapter 4.5 of \cite{Kad-Rin} for details). The set $\fL_\tau=\{A\in\fA:\, \tau(A^\ast A)=0\}$ is a closed left ideal in $\fA$. For $A\in \cA$, we write $[A]:=A+\fL_\tau$ for simplicity, which is an element in the quotient linear space $\fA/\fL_\tau$. The equation
\[
\<[A],[B]\>_\tau=\tau(B^\ast A),\quad (A,B\in \fA)
\]
defines a definite inner product on $\fA/\fL_\tau$. Denote its completion by $\cH_\tau$, which is a Hilbert space. Define the action $\pi_\tau$ of $\fA$ on $\fA/\fL_\tau$ by $\pi_\tau(A)([B])=[AB]$, which extends to a $\ast$-representation of $\fA$ on $\cH_\tau$ with the cyclic vector $[I]$.

\begin{lemma} \label{lem_GNS_H}
Let $\fS=\{(x,y)\in S\times S:\, xy^{-1} \text{ is the fraction of some } w\in G  \text{ in lowest terms}\}$. Then the set $\fB=
\left\{[L_xL_y^\ast]: \, (x,y)\in \fS\right\}$ is an orthonormal basis of $\cH_\tau$. Moreover, for $u_1,u_2,v_1,v_2\in S$, we have that $[L_{u_1}L_{v_1}^\ast]=[L_{u_2}L_{v_2}^\ast]$ if and only if $u_1v_1^{-1}$ and $u_2v_2^{-1}$ are fractions of a same element in $G$.
\end{lemma}

\begin{proof}
Suppose that $u,v,x,y\in S$ are elements in $S$ satisfying $(x,y)\in \fS$ and $uv^{-1}=xy^{-1}$. Applying \eqref{eq_cQ_relation_section2}, we have
\begin{align*}
&\tau\left((L_uL_v^\ast-L_xL_y^\ast)^\ast (L_uL_v^\ast - L_xL_y^\ast)\right) \\
= &\tau(L_vL_v^\ast) - \tau(L_{vu^{-1}\lcm[u,x]}L^\ast_{yx^{-1}\lcm[u,x]}) -\tau(L_{yx^{-1}\lcm[x,u]}L^\ast_{vu^{-1}\lcm[x,u]})+\tau(L_yL_y^\ast)\\
= & 1-1-1+1 =0.
\end{align*}
So $[L_xL_y^\ast]= [L_uL_v^\ast]$. Since $\Span\{L_uL_v^\ast:\, u,v\in S\}$ is dense in $\fA$, we have that $\Span\{[L_xL_y^\ast]:\, (x,y)\in \fS\}$ is dense in $\cH_\tau$. Moreover, for $(x_1,y_1),(x_2,y_2)\in \fS$, we have
\[
\<[L_{x_1}L_{y_1}^\ast],\, [L_{x_2}L_{y_2}^\ast]\>_{\cH_\tau} = \tau(L_{y_2}L_{x_2}^\ast L_{x_1}L_{y_1}^\ast) = \tau\left(L_{y_2x_2^{-1}\lcm[x_2,x_1]}L^\ast_{y_1x_1^{-1}\lcm[x_2,x_1]}\right),
\]
which equals $1$ when $x_1y_1^{-1}=x_2y_2^{-1}$ and $0$ otherwise. So $\fB$ is an orthonormal basis of $\cH_\tau$.
\end{proof}

By Lemma \ref{lem_GNS_H}, one deduces that
\[
\pi_\tau(L_uL_v^\ast)([L_xL_y^\ast]) = [L_uL_v^\ast L_xL_y^\ast] = [L_{uv^{-1}\lcm[v,x]}L^\ast_{yx^{-1}\lcm[v,x]}].
\]
And $\pi_\tau(L_uL_v^\ast)=\pi_\tau(L_{u^\prime}L_{v^\prime}^\ast)$ if and only if
\[
(uv^{-1}\lcm[v,x])(yx^{-1}\lcm[v,x])^{-1} =(u^\prime v^{\prime-1}\lcm[v^\prime,x])(yx^{-1}\lcm[v^\prime,x])^{-1}
\]
for all $(x,y)\in \fS$, if and only if $uv^{-1}=u^\prime v^{\prime-1}$. Therefore $\pi_\tau(\fA)$ is the closure of $\Span\{\pi(L_xL_y^\ast):\, (x,y)\in \fS\}$ in $B(\cH_\tau)$.

Let $C_r^\ast(G)$ be the reduced group $C^\ast$-algebra of $G$. More concretely, let $\widetilde{\cH}=l^2(G)$ and $\delta_z$ be the function taking value $1$ at $z$ and $0$ elsewhere. The operator $\widetilde{L_x}$ on $\widetilde{\cH}$ is defined by $\widetilde{L_w}\delta_z = \delta_{wz}$ $(z\in G)$. Then $C_r^\ast(G)$ is the $C^\ast$-algebra generated by $\{\widetilde{L_w}:\, w\in G\}$ in $B(\widetilde{\cH})$. Note that $\widetilde{L_w}^\ast \delta_z = \delta_{w^{-1}z}$ $(z\in G)$. If $uv^{-1}=uv^{\prime-1}$, then $\widetilde{L_u}\widetilde{L_v}^\ast = \widetilde{L_{u^\prime}}\widetilde{L_{v^\prime}}^\ast$. So $C_r^\ast(G)$ is the closure of $\Span\{\widetilde{L_x}\widetilde{L_y}^\ast:\, (x,y)\in \fS\}$ in $B(\widetilde{H})$. Now one arrives at the following conclusion immediately.

\begin{theorem}
The map $\phi:\, \pi_\tau(\fA)\rightarrow C_r^\ast(G)$ induced by $\pi_\tau(L_xL_y^\ast)\mapsto \widetilde{L_x}\widetilde{L_y}^\ast$ for $(x,y)\in \fS$ is a $\ast$-isomorphism.
\end{theorem}

\subsection{Arithmetic Functions on Integral Monoids}

For two functions $f$ and $g$ on $S$, the convolution $f\ast g$ is given by
\[
(f\ast g)(z) = \sum\limits_{z_1,z_2\in S\atop z_1z_2=z} f(z_1)g(z_2) = \sum\limits_{z_1|z} f(z_1)g(z_1^{-1}z)=\sum\limits_{z_2 \ddagger z} f(zz_2^{-1})g(z_2).
\]
It is associative, but may be non-commutative. For example, one has $L_u f = \delta_u\ast f$ for $u\in S$ and $f\in \cH$. And the divisor function satisfies that $\tau=1\ast 1$. For $u\in S$, we write $\PD(u)$ and $\PD_\ddagger(u)$ for the set of irreducible divisors and co-divisors of $u$, respectively. Let
\[
\omega(u)=\#\PD(u),\quad \omega_\ddagger(u)=\#PD_\ddagger(u),\quad (u\in S).
\]

By Lemma \ref{lem_irr_decomposition_of_any_element}, it is not hard to obtain the following lemma.

\begin{lemma} \label{lem_di_codi_reduce_one}
Let $u$ be an element in $S$.

(\romannumeral1) For any $d|u$ with $d\neq u$, there is some $q\in \PD_\ddagger(u)$ such that $p|uq^{-1}$.

(\romannumeral2) For any $d\ddagger u$ with $d\neq u$, there is some $q\in \PD(u)$ such that $d\ddagger q^{-1}u$.
\end{lemma}

The function $\delta_1$ is the identity with respect to convolution, i.e., one has $f\ast \delta_1=\delta_1\ast f =f$ for any arithmetic function $f$. Next, we consider the inverse of an arithmetic function with respect to convolution.

\begin{theorem} \label{thm_inverse_formula_wrt_convolution}
Let $f$ be an arithmetic function on $S$ with $f(1)\neq 0$. Then $f$ has a unique inverse $g$, i.e., $f\ast g=g\ast f=1$, which is given by either of the following two iterating formulae,
\begin{equation} \label{eq_convolution_inverse_left}
g(1)=f(1)^{-1},\quad g(z) = - f(1)^{-1}\sum\limits_{v|z\atop v\neq z} g(v)f(v^{-1}z),
\end{equation}
or
\begin{equation} \label{eq_convolution_inverse_right}
g(1)=f(1)^{-1},\quad g(z) = - f(1)^{-1}\sum\limits_{v\ddagger z\atop v\neq z} f(zv^{-1})g(v).
\end{equation}
\end{theorem}

\begin{proof}
For $z\in S$, any proper divisor $v$ of $z$ satisfies $\tau(v)<\tau(z)$. Similarly, any proper co-divisor $v$ of $z$ satisfies $\tau(v)<\tau(z)$. The above iteration is well-defined. It is not hard to see that $g(1) = f(1)^{-1}$. Suppose that the value $f^{-1}(v)$ has been determined for all $v$ with $\tau(v)< k$ for some $k\geq 1$. Now for a $z$ with $\tau(z)=k$, one has
\[
0= (g\ast f)(z) = \sum\limits_{v|z} g(v)f(v^{-1}z) = g(z)f(1) + \sum\limits_{v|z\atop v\neq z} g(v)f(v^{-1}z),
\]
or
\[
0 =(f\ast g)(z) = \sum\limits_{v\ddagger z} f(zv^{-1})g(v) = f(1)g(z) + \sum\limits_{v\ddagger z\atop v\neq z} f(zv^{-1})g(v).
\]
Now \eqref{eq_convolution_inverse_left} and \eqref{eq_convolution_inverse_right} follows.
\end{proof}

The inverse of $1$ is one of the most significant arithmetic functions, which is known as M\"{o}bius function in classical arithmetics. We denote it by $\mu$, i.e., it satisfies $\mu\ast 1=1\ast \mu=\delta_1$.

\begin{theorem} \label{prop_mu_l_calculation}
For $z\in S$ with $z\neq 1$, let $\cF(z)=\{zp^{-1}:\, p\in \PD_\ddagger(z)\}$ and $\cF_\ddagger(z)=\{p^{-1}z:\, p \in \PD(z)\}$. We have $\mu(1)=1$ and
\[
\mu(z)= \sum\limits_{\emptyset \neq F\subseteq \cF(z)\atop \gcd(F)=1}(-1)^{|F|} = \sum\limits_{\emptyset \neq F\subseteq \cF_\ddagger(z)\atop \gcd_\ddagger(z;F)=1}(-1)^{|F|},\quad (z\neq 1).
\]

\end{theorem}

\begin{proof}
It is apparent that $\mu(1)=1$. For $z\neq 1$, by Lemma \ref{lem_di_codi_reduce_one}, any divisor $v$ of $z$ with $v\neq z$ divides at least one of the $y$'s with $y\in \cF(z)$. Applying the inclusion-exclusion principle, we deduce that
\begin{align*}
0&=(\mu\ast 1) (z)= \sum\limits_{v|u}\mu(v) = \mu(z)+ \sum\limits_{\emptyset \neq F\subseteq \cF}(-1)^{|F|+1}\sum\limits_{v|\gcd(F)} \mu(v)\\
 &= \mu(z)+ \sum\limits_{\emptyset \neq F\subseteq \cF}(-1)^{|F|+1} (\mu\ast 1)(\gcd(F))= \mu(z)+ \sum\limits_{\emptyset \neq F\subseteq \cF\atop \gcd(F)=1}(-1)^{|F|+1} .
\end{align*}

Similarly, any co-divisor $v$ of $z$ with $v\neq 1$ co-divides at least one of the $y$'s with $y\in \cF_\ddagger(z)$. Applying the inclusion-exclusion principle, we deduce that
\begin{align*}
0&=(1\ast \mu) (z)= \sum\limits_{v \ddagger u}\mu(v) = \mu(z)+ \sum\limits_{\emptyset \neq F\subseteq \cF_\ddagger(z)}(-1)^{|F|+1}\sum\limits_{v\ddagger \gcd_\ddagger(F)} \mu(v)\\
 &= \mu(z)+ \sum\limits_{\emptyset \neq F\subseteq \cF_\ddagger(z)}(-1)^{|F|+1} (1\ast \mu)(\gcd_\ddagger(z;F))= \mu(z)+ \sum\limits_{\emptyset \neq F\subseteq \cF_\ddagger(z)\atop \gcd_\ddagger(z;F)=1}(-1)^{|F|+1} .
\end{align*}
The proof is completed.
\end{proof}

\section{Axiom \uppercase\expandafter{\romannumeral4}' and Homogenous Monoids}
\label{section_homogeneous_monoid}

Define
\[
\fC_1=\left\{\left(u,u^{-1}\lcm[u,v]\right):\, u,v\in S,\, \gcd(u,v)=1\right\},
\]
which is a subset of $S\times S$, and
\[
\Gamma_1=\left\{\left(\left(u,u^{-1}\lcm[u,v]\right),\left(v,v^{-1}\lcm[u,v]\right)\right):\,u,v\in S,\, \gcd(u,v)=1 \right\},
\]
which is a subset of $\fC_1\times \fC_1$.

\medskip

\textsc{Axiom \uppercase\expandafter{\romannumeral4}'}. The set $\Gamma_1$ is a graph of a map $\eta:\, \fC_1\rightarrow \fC_1$.

\begin{definition}
We call an integral monoid $S$ homogenous, if Axiom \uppercase\expandafter{\romannumeral4}' holds.
\end{definition}

In this section, we always assume that $\Gamma_1$ is a graph of a map $\eta:\, \fC_1\rightarrow \fC_1$. It is not hard to see that $\eta^2=id$ on $\fC_1$. That is to say, if $(u,v)\in \fC_1$ and $\eta(u,v)=(\widetilde{v},\widetilde{u})$, then we also have $(\widetilde{v},\widetilde{u})\in \fC_1$ and $\eta(\widetilde{v},\widetilde{u})=(u,v)$. For simplicity, we will rewrite the above formulas by either of the following four expressions:
\[
\uuuline{u}v \rlh \widetilde{v} \uuuline{\widetilde{u}}, \quad u\uuuline{v} \rlh  \uuuline{\widetilde{v}}\widetilde{u},\quad \widetilde{v} \uuuline{\widetilde{u}} \rlh \uuuline{u}v,\quad \uuuline{\widetilde{v}}\widetilde{u} \rlh u\uuuline{v}.
\]
And whenever we put three underlines under $u$ or $v$ for an ordered pair of elements $u,v$, we always mean that $(u,v)\in \fC_1$. For example, for any $u\in S$, we have $\uuuline{1}u\rlh u\uuuline{1}$.

\begin{definition}
When $(u,v)\in \fC_1$, we say that $u,v$ are castled-free. When $\uuuline{u} v \rlh \widetilde{v} \uuuline{\widetilde{u}}$, we call it, or the element $((u,v),(\widetilde{v},\widetilde{u}))\in \Gamma_1$, a free castling.
\end{definition}

The arithmetic meaning of this notion is interpreted by the following lemma.

\begin{lemma} \label{lem_lcm_unique_from_eta_2=id}
The following statements are equivalent.

(\romannumeral1) The set $\Gamma_1$ is a graph of a map $\eta:\, \fC_1\rightarrow \fC_1$.

(\romannumeral2) Suppose that $u,v,w$ are elements in $S$ such that $\lcm[w,u]=\lcm[w,v]$ and $\gcd(w,u)=\gcd(w,v)=1$. Then $u=v$.

(\romannumeral3) Suppose that $u_0,v_0,w_0$ are elements in $S$ such that $\lcm[w_0,u_0]=\lcm[w_0,v_0]$ and $\gcd(w_0,u_0)=\gcd(w_0,v_0)$. Then $u_0=v_0$.

(\romannumeral4) Suppose that $x,y,z,r$ are elements in $S$ such that $\lcm_\ddagger[r;x,y]=\lcm_\ddagger[r;x,z]$ and $\gcd_\ddagger(r;x,y)=\gcd_\ddagger(r;x,z)=1$, then $y=z$.

(\romannumeral5) Suppose that $x_0,y_0,z_0,r_0$ are elements in $S$ such that $\lcm_\ddagger[r_0;x_0,y_0]=\lcm_\ddagger[r_0;x_0,z_0]$ and $\gcd_\ddagger(r_0;x_0,y_0)=\gcd_\ddagger(r_0;x_0,z_0)$, then $y_0=z_0$.
\end{lemma}

\begin{proof}
First, we prove that (\romannumeral1) implies (\romannumeral2). With the conditions in (\romannumeral2), we assume that
\[
\lcm[w,u]=\lcm[w,v]=wx=uy=vz
\]
for some $x,y,z\in S$. By the definition of $\fC_1$, we have $(w,x),(u,y),(v,z)\in \fC_1$, and $\eta((w,x))=(u,y)=(v,z)$. Since $\eta$ is a well-defined map, we have $(u,y)=(v,z)$. So $u=v$.

Second, we show that (\romannumeral2) implies (\romannumeral4). Under the conditions in (\romannumeral4), we set
\[
\lcm_\ddagger[r;x,y]=\lcm_\ddagger[r;x,z]=wx=uy=vz
\]
for some $w,u,z\in S$. Since $\gcd_\ddagger(r;x,y)=\gcd_\ddagger(r;x,z)=1$, one deduces by Lemma \ref{lem_gcd_gcddagger_transfer} that
\[
\gcd(w,u)=\gcd(w,v)=1, \quad \lcm[w,u]=\lcm_\ddagger[r;x,y]=\lcm_\ddagger[r;x,z]=\lcm[w,v].
\]
So $u=v$ by (\romannumeral2), which leads to $y=z$.

Third, we show that (\romannumeral4) leads to (\romannumeral1). Suppose that both $((w,x),(u,y))$ and $((w,x),(v,z))$ belongs to $\Gamma_1$. Then $\gcd(w,u)=\gcd(w,v)=1$ and $\lcm[w,u]=wx=uy$, $\lcm[w,v]=wx=vz$. Put $r=\lcm[w,u]=\lcm[w,v]$. By Lemma \ref{lem_gcd_gcddagger_transfer}, we have $\gcd_\ddagger(r;x,y)=\gcd_\ddagger(r;x,z)=1$ and
\[
\lcm_\ddagger[r;x,y]=\lcm[w,u]=\lcm[w,v]=\lcm_\ddagger[r;x,z].
\]
It follows that $y=z$ by (\romannumeral4) and then $u=v$. Hence the map $\eta$ is well-defined.

Next, we show that (\romannumeral4) is equivalent to (\romannumeral5). Let $\gcd_\ddagger(r_0;x_0,y_0)=\gcd_\ddagger(r_0;x_0,z_0)=d$ and $r_0=rd,\,x_0=xd,\, y_0=yd$, then $\gcd_\ddagger(r;x,y)=\gcd(r;x,z)=1$ and
\[
\lcm_\ddagger[r;x,y] = \lcm_\ddagger[r_0;x_0,y_0] \cdot d^{-1} =\lcm_\ddagger[r_0;x_0,z_0]\cdot d^{-1} =\lcm[r;x,z].
\]
Note that $y_0=z_0$ if and only if $y=z$. The statements (\romannumeral4) and (\romannumeral5) are equivalent. Similarly, one can deduce that (\romannumeral2) and (\romannumeral3) are equivalent. The proof is completed.
\end{proof}

From the proof of Lemma \ref{lem_lcm_unique_from_eta_2=id}, we obtain the following corollary.

\begin{corollary} \label{cor_C_1_defined_by_ddagger}
We have $\fC_1=\left\{\left(\lcm_\ddagger[w;x,y]y^{-1},y\right):\, x,y,w\in S,\, \gcd_\ddagger(w;x,y)=1\right\}$. Moreover, it satisfies that
\[
\eta\left(\left(\lcm_\ddagger[w;x,y]y^{-1},y\right)\right) = \left(\lcm_\ddagger[w;x,y]x^{-1},x\right).
\]
\end{corollary}

\begin{remark}
Let $u,v,x,y$ be elements in a homogenous monoid $S$. The following statements are equivalent.

(\romannumeral1) We have the free castling $\uuuline{u}y\rlh v\uuuline{x}$.

(\romannumeral2) It satisfies $\gcd(u,v)=1$ and $\lcm[u,v]=uy=vx$.

(\romannumeral3) It satisfies $\gcd_\ddagger(ux;y,x)=1$ and $\lcm_\ddagger[ux;y,x]=uy=vx$.
\end{remark}

\begin{remark} \label{remark_castled_free_determine_relation}
For a free castling $\uuuline{u}y\rlh v\uuuline{x}$, we have (\romannumeral1) $u,y$ uniquely determine $v,x$; (\romannumeral2) $v,x$ uniquely determine $u,y$; (\romannumeral3) $u,v$ uniquely determine $y,x$; (\romannumeral4) given a $w\in S$, under the condition that $uy,vx\ddagger w$, the elements $y,x$ uniquely determine $u,v$.
\end{remark}

\subsection{Index of an Element}

\begin{definition}
For $u\in S$, define
\[
\ind(u)= \min\{k: \, u= q_1q_2\ldots q_k \text{ with }q_1,\ldots,q_k\in \cP\}.
\]
\end{definition}

Here $\ind(1)=0$. One can verify that $\ind(uv)\leq \ind(u)+\ind(v)$ for $u,v\in S$. The following lemma shows that the number of letters in a word of a given element is an invariant, and the equality holds in above formula.

\begin{lemma} \label{lem_ind_eq}
Let $u\in S$. Suppose that $u=q_1q_2\ldots q_k$ for some $q_1,q_2,\ldots q_k\in \cP$. Then $k=\ind(u)$.
\end{lemma}

\begin{proof}
We use induction on $\ind(u)$, First consider the case $\ind(u)=0$, i.e., $1=q_1q_2\ldots q_k$. It is immediate that $k\neq 1$. Assume that $k\geq 2$. Then $q_1=(q_2\ldots q_k)^{-1}\in S^{-1}$. So $q_1\in S\cap S^{-1}=\{1\}$, which is a contradiction. As a result, we have $k=0$.

When $\ind(u)=1$, one has $u\in \cP$ and $\tau(u)=2$. Assume on the contrary that $k\geq 2$. Then $q_1|u$ and $q_1\neq u$. The elements $1,q_1,u$ are distinct divisors of $u$. It follows that $\tau(u)\geq 3$, which is a contradiction.

Suppose that the result has been obtained for $\ind(u)\leq m-1$ for some $m\geq 2$. Now we deal with the case $\ind(u)=m$. By the definition of $\ind(u)$, there are some $r_1,r_2,\ldots,r_m\in \cP$ such that $u=r_1r_2\ldots r_m$, and $k\geq m$. Write $w=q_1q_2\ldots q_{k-1}$ and $v=r_1r_2\ldots r_{m-1}$.

\textsc{Case 1}. Suppose that $d:=\gcd(w, v)\neq 1$. Write $w = dx$ and $v=dy$. Moreover, write
\[
d=p_1p_2\ldots p_h,\quad x=\widetilde{q_1}\widetilde{q_2}\ldots \widetilde{q_l},\quad y=\widetilde{r_1}\widetilde{r_2}\ldots \widetilde{r_n}
\]
for some $h\geq 1, l,n\geq 0$ and $p_1,\ldots, p_h, \widetilde{q_1},\ldots \widetilde{q_l},\widetilde{r_1},\ldots \widetilde{r_n}\in \cP$. Noting that $\ind(v)\leq m-1$, we deduce by inductive hypothesis that $\ind(v)=m-1= h+n$, which implies $n\leq m-2$. Note that
\[
d^{-1}u=\widetilde{q_1}\widetilde{q_2}\ldots \widetilde{q_l} q_k = \widetilde{r_1}\widetilde{r_2}\ldots \widetilde{r_n}r_m,
\]
where $\ind(d^{-1}u)\leq n + 1\leq m-1$. By inductive hypothesis, one gets $\ind(d^{-1}u)=n+1 = l+1$. Thus, we have $\ind(w)\leq \ind(d)+\ind(x)\leq h+l =h+n=m-1$. By inductive hypothesis again, one concludes that $k-1=\ind(w)\leq m-1$, which implies $k=m$.

\textsc{Case 2}. Suppose that $r_m=q_k$. Then $w=v$ and $\ind(v)\leq m-1$. By inductive hypothesis, we have $\ind(w)=\ind(v)=m-1$. It follows that $k-1=m-1$. So $k=m$.

\textsc{Case 3}. Suppose that $\gcd(w, v)= 1$ and $r_m\neq q_k$. Note that $\lcm[v,w]|u$. Write $u=\lcm[v,w]d$ and $\lcm[v,w]=va=wb$ for some $a,b,d\in S$. Then $u=vad=wbd$, which leads to $ad=q_k$ and $bd=r_{m-1}$. In view of $r_m\neq q_k$, one deduces that $a=q_k$, $b=r_m$ and $d=1$. Hence $\lcm[v,w]=u$. Assume on the contrary that $k>m$. Note that $k\geq m+1\geq 3$. Let $w_0=q_1q_2\ldots q_{k-2}$. Then $w_0\neq 1$. We denote $\lcm[w_0,v]=vc$ for some $c\in S$. Since $\gcd(v,w_0)=1$, one has $w_0\nmid v$ and so $c\neq 1$. Moreover, one has $w_0|u$ and $v|u$, which implies $vc=\lcm[v,w_0]|u=vr_m$. It follows that $c=r_m$ and $\lcm[v,w_0]=vr_m=u=\lcm[v,w]$. By Lemma \ref{lem_lcm_unique_from_eta_2=id}, one has $w=w_0$, which is a contradiction. Thus, we conclude that $k=m$.

By induction, the lemma follows.
\end{proof}

Now, we know that the integral monoid $\cS$ in Section \ref{subsection_example_of_intergral_monoid} is not homogeneous. For Thompson's group $\bS$, we have already shown that $\ind(\cdot)$ can be extended to a group homomorphism from $(\bG,\cdot)$ to $(\bZ,+)$. Does this holds for any homogeneous monoid $S$?

The following two corollary follows immediately.

\begin{corollary} \label{cor_ind_eq}
For any $u,v\in S$, it satisfies $\ind(uv)=\ind(u)+\ind(v)$.
\end{corollary}

\begin{corollary}
(\romannumeral1) Suppose that $u,v$ are elements in $S$ with $u|v$ and $\ind(u)=\ind(v)$. Then $u=v$. (\romannumeral2) Suppose that $u,v$ are elements in $S$ with $u\ddagger v$ and $\ind(u)=\ind(v)$. Then $u=v$.
\end{corollary}

\begin{lemma} \label{lem_gcd_lcm_u_v}
(\romannumeral1) For any $u,v\in S$, it satisfies
\[
\ind\left(\gcd(u,v)\right)+\ind\left(\lcm[u,v]\right)=\ind(u)+\ind(v).
\]

(\romannumeral2) Let $u,v,w$ be elements in $S$ satisfying $u,v\ddagger w$. Then
\[
\ind\left(\gcd_\ddagger(w;u,v)\right)+\ind\left(\lcm_\ddagger[w;u,v]\right)=\ind(u)+\ind(v).
\]
\end{lemma}

\begin{proof}
(\romannumeral1) We first prove it under the condition $\gcd(u,v)=1$. Induction on $\ind(u)$ is applied. For $\ind(u)=0$ or $\ind(v)=0$, the proof is trivial. In the following, we always assume that $\ind(v)\geq 1$.

Write $z:=\lcm[u,v]=uy=vx$. Suppose that the result has been proved for $\ind(u)\leq m-1$ with some $m\geq 1$. Now we handle the case $\ind(u)=m$. Let $u=u_1q$ for some $q\in \cP$. Then $\gcd(u_1,v)=1$. Write $z_1:=\lcm[u_1,v]=u_1y_1=vx_1$. By inductive hypothesis, it satisfies $\ind(z_1)=\ind(u_1)+\ind(v)$. If $z_1=z$, then we obtain by Lemma \ref{lem_lcm_unique_from_eta_2=id} that $u_1=u$, which is a contradiction. Thus, we have $z_1|z$ and $z_1\neq z$. So $\ind(z)\geq \ind(z_1)+1= \ind(u)+\ind(v)$.

Since $u\nmid v$, one has $\ind(x)\geq 1$ and we write $x=px_0$ for some $p\in \cP$ and $x_0\in S$. Since $u|z$ and $vp|z$, one has $\lcm[u,vp]|z$. In view of $u|\lcm[u,vp]$ and $v|\lcm[u,vp]$, one deduces that $z=\lcm[u,v]|\lcm[u,vp]$. As a result, we have $\lcm[u,v]=\lcm[u,vp]$. Denote $d=\gcd(u,vp)$. If $d=1$, then we deduce by Lemma \ref{lem_lcm_unique_from_eta_2=id} that $vp=v$, which is a contradiction. So $d\neq 1$. Write $u=da$ and $vp=db$. Then $\gcd(a,b)=1$. And $\ind(a)=\ind(u)-\ind(d)\leq m-1$. By inductive hypothesis, we have
\[
\ind\left(\lcm[a,b]\right)= \ind(a)+\ind(b)=\ind(u)+\ind(vp)-2\cdot \ind(d)\leq \ind(u)+\ind(v)-\ind(d).
\]
It follows that
\[
\ind\left(\lcm[u,vp]\right) = \ind\left(d\cdot \lcm[a,b]\right)= \ind(d)+ \ind(\lcm[a,b])\leq  \ind(u)+\ind(v).
\]
Hence $\ind(\lcm[u,v])\leq \ind(u)+\ind(v)$.

We have shown that $\ind(\lcm[u,v])= \ind(u)+\ind(v)$ in the case $\gcd(u,v)=1$. Now, we turn to the general case that $\gcd(u,v)=e$. Write $u=eu_1$, $v=ev_1$. Then $\gcd(u_1,v_1)=1$. By above discussions, one gets $\ind\left(\lcm[u_1,v_1]\right)=\ind(u_1)+\ind(v_1)$. It follows that
\[
\ind\left(\lcm[u,v]\right) = \ind\left(e\cdot \lcm[u_1,v_1]\right) = \ind(e) + \ind(u_1)+\ind(v_1) = \ind(u)+\ind(v) - \ind(e).
\]
The proof is completed.

(\romannumeral2) Write $w=c_1u=c_2v$. Put $c=\gcd(c_1,c_2)$ and $d=\lcm[c_1,c_2]$. Write $w=cz=dz_0$. Then $\lcm_\ddagger [w;u,v]=z$ and $\gcd_\ddagger (w;u,v)= z_0$. Note that
\[
\ind(c)+\ind(z)=\ind(d)+\ind(z_0) = \ind(w) = \ind(c_1)+\ind(u)=\ind(c_2)+\ind(v).
\]
By (\romannumeral1), one has $\ind(c)+\ind(d)=\ind(c_1)+\ind(c_2)$. Then
\[
\ind(z_0)+\ind(z) = 2\ind(w)-\ind(c) - \ind(d) = 2\ind(w)-\ind(c_1)-\ind(c_2)= \ind(u)+\ind(v).
\]
This completes the proof.
\end{proof}



\begin{corollary} \label{cor_C_1_ind_transfer}
Suppose that $\uuuline{u} v \rlh \widetilde{v} \uuuline{\widetilde{u}}$. Then $\ind(u)=\ind(\widetilde{u})$, $\ind(v)=\ind(\widetilde{v})$ and $uv=\widetilde{v}\widetilde{u}$.
\end{corollary}

\begin{proof}
By the construction of $\fC_1$ and the definition of free castlings, we have $\gcd(u,\widetilde{v})=1$ and $\lcm[u,\widetilde{v}]=uv=\widetilde{v}\widetilde{u}$. Hence
\[
\ind(u)+\ind(\widetilde{v}) = \ind(\lcm[u,\widetilde{v}]) = \ind(u)+\ind(v)=\ind(\widetilde{v})+\ind(\widetilde{u}).
\]
Now the corollary follows.
\end{proof}

\begin{corollary}
Let $u,v,w\in S$. Then
\begin{align*}
&\ind\left(\lcm[uw,v]\right)\leq \ind\left(\lcm[u,v]\right)+\ind(w),\\
&\ind\left(\gcd(uw,v)\right)\leq \ind\left(\gcd(u,v)\right)+\ind(w).
\end{align*}
\end{corollary}

\begin{proof} \label{cor_ind_uw_v}
Write $a=\lcm[uw,v]$ and $b=\lcm[u,v]$ for simplicity. Note that
\[
a=\lcm[\lcm[uw,u],v] = \lcm[uw,\lcm[u,v]] = \lcm[uw,b].
\]
Since $u|b$ and $u|uw$, we have $u|\gcd(uw,b)$ and $\ind(u)\leq \ind\left(\gcd(uw,b)\right)$. By Lemma \ref{lem_gcd_lcm_u_v}, we deduce that
\[
\ind(a) = \ind(uw)+\ind(b) - \ind\left(\gcd(uw,b)\right) \leq \ind(b)+\ind(w).
\]

Write $a^\prime=\gcd(uw,v)$ and $b^\prime=\gcd(u,v)$ for simplicity. Then
\[
b^\prime = \gcd(\gcd(u,uw),v) = \gcd(u,\gcd(uw,v)) = \gcd(u,a^\prime).
\]
Since $u|uw$ and $a^\prime| uw$, one has $\lcm[u,a^\prime]|uw$ and $\ind\left(\lcm[u,a^\prime]\right)\leq \ind(uw)$. By Lemma \ref{lem_gcd_lcm_u_v}, we obtain
\[
\ind(b^\prime) = \ind(u)+\ind(a^\prime) - \ind\left(\lcm[u,a^\prime]\right) \geq \ind(a^\prime)-\ind(w).
\]

Now the lemma follows.
\end{proof}

\subsection{Composition and Decomposition of Free Castlings}

In this subsection, we show some basic properties of free castlings, which will play a large part in the rest of this paper.

\begin{lemma}  [Fundamental lemma for arithmetic]  \label{lem_divisor_castling}
\quad

(\romannumeral1) Let $u,v\in S$. Suppose that $w$ is a divisor of $uv$ satisfying $\gcd(w,u)=1$. Then there exists some $v_1|v$ and $\widetilde{u}\in S$ such that $\uuuline{w}\widetilde{u} \rlh u\uuuline{v_1}$. Moreover, if $w^\prime$ is also a divisor of $uv$ such that $\gcd(w^\prime,u)=1$ and $\uuuline{w^\prime}\widetilde{u}^\prime \rlh u \uuuline{v_1}$ for some $\widetilde{u}^\prime\in S$, then $w^\prime=w$.

(\romannumeral2) Let $u,v\in S$. Suppose that $w$ is a co-divisor of $uv$ satisfying $\gcd_\ddagger(uv;w,v)=1$. Then there exists some $u_1\ddagger u$ and $\widetilde{v}\in S$ such that $\uuuline{\widetilde{v}}w \rlh u_1\uuuline{v}$. Moreover, if $w^\prime$ is also a co-divisor of $uv$ such that $\gcd_\ddagger(uv;w^\prime,v)=1$ and $\uuuline{\widetilde{v}^\prime}w^\prime \rlh u_1 \uuuline{v}$ for some $\widetilde{v}^\prime\in S$, then $w^\prime=w$.
\end{lemma}

\begin{proof}
The uniqueness results from Axiom \uppercase\expandafter{\romannumeral4}'. It is sufficient to prove the existence of corresponding elements.

(\romannumeral1) Denote $\lcm[w,u]=w\widetilde{u}=uv_1$ for some $\widetilde{u},v_1\in S$. Combining $\gcd(w,u)=1$, we see that $\uuuline{w}\widetilde{u}\rlh u \uuuline{v_1}$. Note that $w|uv$ and $u|uv$. One obtains $uv_1=\lcm[u,w]|uv$, which implies $v_1|v$.

(\romannumeral2) Denote $\lcm_\ddagger[uv;w,v]=\widetilde{v}w=u_1v$ for some $\widetilde{v},u_1\in S$. Combining $\gcd_\ddagger(uv;w,v)=1$, we see that $\uuuline{\widetilde{v}}w\rlh u_1 \uuuline{v}$. Note that $w\ddagger uv$ and $v\ddagger uv$. One obtains $u_1v=\lcm_\ddagger[uv;w,v]\ddagger uv$, which implies $u_1\ddagger u$.
\end{proof}

\begin{remark}
This lemma gives a first hint to turn irreducible elements into primes. Suppose that $p|uv$. Then either $p|u$ or $p\nmid u$. In the latter case, we have $\uuuline{p}\widetilde{u}\rlh u\uuuline{q}$ for some $q|v$, $q\in \cP$ and $\widetilde{u}\in S$. That is to say, the element $p$ either comes from $u$, or comes from $v$. We will give the concrete definition of a prime in Section \ref{section_castlable_monoid}, after we put into consider the elements that are not free.
\end{remark}

\begin{lemma} [Decomposition of free castlings] \label{lem_C1_subseteq_C_0}
\quad

Let $u,v,\widetilde{u},\widetilde{v}$ be elements in $S$ such that $\uuuline{u}v \rlh \widetilde{v} \uuuline{\widetilde{u}}$.

(\romannumeral1). For any $u_1,u_2\in S$ with $u_1u_2=u$, there exist elements $\widehat{u_1},\widehat{u_2},\widehat{v}$ in $S$ with $\widehat{u_1}\widehat{u_2}=\widetilde{u}$ such that
\[
\uuuline{u_2}v \rlh \widehat{v} \uuuline{\widehat{u_2}},\quad \uuuline{u_1} \widehat{v}  \rlh \widetilde{v} \uuuline{\widehat{u_1}}.
\]

(\romannumeral2). For any $v_1,v_2\in S$ with $v_1v_2=v$, there exist elements $\widehat{u},\widehat{v_1},\widehat{v_2}$ with $\widehat{v_1}\widehat{v_2}=\widetilde{v}$ such that
\begin{align*}
\uuuline{u}v_1 =\widehat{v_1}\uuuline{\widehat{u}},\quad \uuuline{\widehat{u}}v_2  =  \widehat{v_2} \uuuline{\widetilde{u}}.
\end{align*}
\end{lemma}

\begin{proof}
(\romannumeral1). Write $z:=uv=\widetilde{v}\widetilde{u}$. We know from $\uuuline{u}v \rlh \widetilde{v} \uuuline{\widetilde{u}}$ that
\[
\gcd(u,\widetilde{v})=\gcd_\ddagger(z;v,\widetilde{u})= 1,\quad \lcm[u,\widetilde{v}]= \lcm_\ddagger[z;v,\widetilde{u}]=z
\]
and $\ind(u)=\ind(\widetilde{u})$, $\ind(v)=\ind(\widetilde{v})$. It follows from $\gcd(u,\widetilde{v})=1$ that $\gcd(u_1,\widetilde{v})=1$. Write $\lcm[u_1,\widetilde{v}]= u_1 \widehat{v} = \widetilde{v}\widehat{u_1}$. Then $\uuuline{u_1}\widehat{v}\rlh \widetilde{v}\uuuline{\widehat{u_1}}$ and $\ind(u_1)=\ind(\widehat{u_1})$. Since $\lcm[u_1,\widetilde{v}]|\lcm[u,\widetilde{v}]$, i.e., $\widetilde{v}\widehat{u_1}|\widetilde{v}\widetilde{u}$, one obtains $\widehat{u_1}|\widetilde{u}$. Put $\widetilde{u}=\widehat{u_1}\widehat{u_2}$. Then
\[
u_1u_2v=uv=\widetilde{v}\widetilde{u}= \widetilde{v}\widehat{u_1}\widehat{u_2}=u_1\widehat{v}\widehat{u_2}.
\]
We deduce that $u_2v=\widehat{v}\widehat{u_2}$, and $\ind(u_2)=\ind(\widehat{u_2})$. In view of $\gcd_\ddagger(z;v,\widetilde{u})=1$, one has $\gcd_\ddagger(z;v,\widehat{u_2})=1$. Then
\begin{equation} \label{eq_decomposition_C_1_1}
\ind\left(\lcm_\ddagger[z;v,\widehat{u_2}]\right)=\ind(v)+\ind(\widehat{u_2})=\ind(v)+\ind(u_2) = \ind(u_2v).
\end{equation}
Moreover, since both $\widehat{u_2}$ and $v$ are co-divisors of $u_2v$, and $u_2v\ddagger z$, we obtain $\lcm_\ddagger[z;v,\widehat{u_2}]\ddagger u_2v$. Combining \eqref{eq_decomposition_C_1_1}, one obtains $\lcm_\ddagger[z;v,\widehat{u_2}]= u_2v$. Now, we conclude that $\uuuline{u_2}v \rlh \widehat{v} \uuuline{\widehat{u_2}}$.

(\romannumeral2). The conclusion follows from similar arguments as above.
\end{proof}

For simplicity, we abbreviate the formulae in Lemma \ref{lem_C1_subseteq_C_0} as follows. When $\uuuline{u}v \rlh \widetilde{v} \uuuline{\widetilde{u}}$, for any $u_1,u_2$ with $u_1u_2=u$, we have
\[
\uuuline{u}v=\uuuline{u_1u_2}v \rlh \uuuline{u_1}\widehat{v} \uuuline{\widehat{u_2}} \rlh \widetilde{v}\uuuline{\widehat{u_1}\widehat{u_2}} = \widetilde{v}\uuuline{\widetilde{u}}
\]
for some $\widehat{u_1},\widehat{u_2},\widehat{v}\in S$; and for any $v_1,v_2$ with $v_1v_2=v$, we have
\[
\uuuline{u}v=\uuuline{u}v_1v_2 \rlh \widehat{v_1}\uuuline{\widehat{u}}v_2  \rlh \widehat{v_1}\widehat{v_2}\uuuline{\widetilde{u}} = \widetilde{v}\uuuline{\widetilde{u}}.
\]
for some $\widehat{u},\widehat{v_1},\widehat{v_2}\in S$. Here $\uuuline{a}b = \uuuline{a^\prime}b^\prime$ always means that $a=a^\prime$ and $b=b^\prime$ as elements in $S$.

\begin{lemma} [Composition of free castlings] \label{lem_composition_C_1}
\quad

(\romannumeral1) Suppose $u_1,u_2,v,\widetilde{u_1},\widetilde{u_2},\widetilde{v},\widetilde{\widetilde{v}}$ are elements in $S$ such that
\[
\uuuline{u_2}v = \widetilde{v}\uuuline{\widetilde{u_2}},\quad \uuuline{u_1}\widetilde{v} = \widetilde{\widetilde{v}} \uuuline{\widetilde{u_1}}.
\]
Then $\uuuline{u_1u_2}v  = \widetilde{\widetilde{v}}\uuuline{\widetilde{u_1}\widetilde{u_2}}$.

(\romannumeral2) Suppose $u,v,v_1,v_2,\widetilde{u},\widetilde{\widetilde{u}},\widetilde{v_1},\widetilde{v_2}$ are elements in $S$ such that
\[
\uuuline{u} v_1 = \widetilde{v_1} \uuuline{\widetilde{u}}, \quad \uuuline{\widetilde{u}}v_2  = \widetilde{v_2}\uuuline{\widetilde{\widetilde{u}}}.
\]
Then $\uuuline{u}v_1v_2  =  \widetilde{v_1}\widetilde{v_2}\uuuline{\widetilde{\widetilde{u}}}$.
\end{lemma}

\begin{proof}
(\romannumeral1) From the given conditions, we deduce that
\[
u_1u_2v=u_1\widetilde{v}\widetilde{u_2}=\widetilde{\widetilde{v}}\widetilde{u_1}\widetilde{u_2}, \quad \gcd(u_1,\widetilde{\widetilde{v}})=\gcd(u_2,\widetilde{v})=1,
\]
and $\ind(v)=\ind(\widetilde{v})=\ind(\widetilde{\widetilde{v}})$. Therefore $\lcm[u_1u_2,\widetilde{\widetilde{v}}]|u_1u_2v$. If we may prove that $\gcd(u_1u_2,\widetilde{\widetilde{v}})=1$, then
\[
\ind(\lcm[u_1u_2,\widetilde{\widetilde{v}}])=\ind(u_1u_2)+\ind(\widetilde{\widetilde{v}})=  \ind(u_1u_2)+\ind(v) =\ind(u_1u_2v),
\]
which implies that $\lcm[u_1u_2,\widetilde{\widetilde{v}}]=u_1u_2v$ and $\uuuline{u_1u_2}v  = \widetilde{\widetilde{v}}\uuuline{\widetilde{u_1}\widetilde{u_2}}$.

Suppose on the contrary that there is some $p\in\cP$ such that $p|\gcd(u_1u_2,\widetilde{\widetilde{v}})$. Since $\gcd(u_1,\widetilde{\widetilde{v}})=1$, we have $p\nmid u_1$ and $p|u_1u_2$. By Lemma \ref{lem_divisor_castling}, there are some $q\in \cP$ with $q|u_2$ and $\breve{u_1}\in S$ such that $\uuuline{p}\breve{u_1}\rlh u_1 \uuuline{q}$. Similarly, in view of the facts that $p|\widetilde{\widetilde{v}}\widetilde{u_1}= u_1\widetilde{v}$ and $p\nmid u_1$, there are some $r|\widetilde{v}$ and $\check{u_1}\in S$ such that $\uuuline{p}\check{u_1}\rlh u_1 \uuuline{r}$. However, we deduce by Lemma \ref{lem_lcm_unique_from_eta_2=id} that $r=q$. Now $1\neq q| \gcd(u_2,\widetilde{v})$, which is a contradiction.

(\romannumeral2) The conclusion follows from (\romannumeral1) by changing the variables from $\widetilde{v_1},\widetilde{v_2},\widetilde{\widetilde{u}}, v_1,v_2,\widetilde{u},u$ to $u_1,u_2,v,\widetilde{u_1},\widetilde{u_2},\widetilde{v},\widetilde{\widetilde{v}}$, respectively.
\end{proof}

For simplicity, we abbreviate the formulae in Lemma \ref{lem_composition_C_1} in the following way. The composition of free castlings $\uuuline{u_2}v = \widetilde{v}\uuuline{\widetilde{u_2}}$ and $\uuuline{u_1}\widetilde{v} = \widetilde{\widetilde{v}} \uuuline{\widetilde{u_1}}$ gives

\[
\uuuline{u_1u_2}v \rlh \uuuline{u_1}\widetilde{v} \uuuline{\widetilde{u_2}} \rlh \widetilde{\widetilde{v}}\uuuline{\widetilde{u_1}\widetilde{u_2}}.
\]
Also, the composition of $\uuuline{u} v_1 = \widetilde{v_1} \uuuline{\widetilde{u}}$ and $\uuuline{\widetilde{u}}v_2  = \widetilde{v_2}\uuuline{\widetilde{\widetilde{u}}}$ leads to
\[
\uuuline{u}v_1v_2 \rlh \widetilde{v_1}\uuuline{\widetilde{u}}v_2  \rlh \widetilde{v_1}\widetilde{v_2}\uuuline{\widetilde{\widetilde{u}}}.
\]

\begin{remark} \label{remark_decomposition_chain}
Let $k,l\geq 1$ be given integers. A (de)composition-chain of depth $(k,l)$ is a sequence $\{(i_h,j_h)\}_{h=1}^{kl}$
such that
\[
\{i_h,\ldots,k-1,k\}\times \{1,2,\ldots,j_h\} \subseteq \{(i_1,j_1),(i_2,j_2)\ldots, (i_h,j_h)\}
\]
for all $1\leq h\leq kl$.

A stronger form of Lemma \ref{lem_C1_subseteq_C_0} can be stated as follows. Suppose that $\uuuline{u}v\rlh \widetilde{v}\uuuline{\widetilde{u}}$. Then for any $k,l\geq 1$, any decomposition chain $\{(i_h,j_h)\}_{h=1}^{kl}$ of depth $(k,l)$, and any decomposition of elements $u=u_{1,1}u_{2,1}\ldots u_{k,1},\, v=v_{1,k}v_{2,k}\ldots v_{l,k}$, there exist elements $u_{i,j}$ $(1\leq i\leq k,\, 2\leq j\leq l+1)$ and $v_{j,i}$ $(1\leq j\leq l , \,0\leq i\leq k-1)$ in $S$ such that $\uuuline{u_{i_h,j_h}}v_{j_h,i_h}\rlh v_{j_h,i_h-1}\uuuline{u_{i_h,j_h+1}}$. Moreover, we have $\widetilde{v}=v_{1,0}v_{2,0}\ldots,v_{l,0}$ and $\widetilde{u}=u_{1,l+1}u_{2,l+1}\ldots u_{k,l+1}$.

Abbreviations of above formulae are possible. For example, let $k=3$, $l=2$, $u=u_{1,1}u_{2,1}u_{3,1}$, $v=v_{1,3}v_{2,3}v_{3,3}$ and take the decomposition chain
\[
(3,1),\, (2,1),\,(3,2),\, (1,1),\, (2,2),\,(1,2),
\]
we can write
\begin{align*}
\uuuline{u}v = &\uuuline{u_{1,1}u_{2,1}u_{3,1}}v_{1,3}v_{2,3}\rlh \uuuline{u_{1,1}u_{2,1}}v_{1,2}\uuuline{u_{3,2}}v_{2,3} \rlh \uuuline{u_{1,1}}v_{1,1}\uuuline{u_{2,2}u_{3,2}}v_{2,3} \\
\rlh &\uuuline{u_{1,1}}v_{1,1}\uuuline{u_{2,2}}v_{2,2}\uuuline{u_{3,3}}v_{3,3}
\rlh v_{1,0}\uuuline{u_{1,2}u_{2,2}}v_{2,2}\uuuline{u_{3,3}} \rlh
v_{1,0}\uuuline{u_{1,2}}v_{2,1}\uuuline{u_{2,3}u_{3,3}} \\
\rlh &v_{1,0}v_{2,0}\uuuline{u_{1,3}u_{2,3}u_{3,3}} =\widetilde{v}\uuuline{\widetilde{u}}.
\end{align*}

A stronger form of Lemma \ref{lem_composition_C_1} may appear in the following way. Suppose that $\{(i_h,j_h)\}_{h=1}^{kl}$ is a composition chain of depth $(k,l)$ for some $k,l\geq 1$. Let $u_{i,j}$ $(1\leq i\leq k,\, 2\leq j\leq l+1)$ and $v_{j,i}$ $(1\leq j\leq l , \,0\leq i\leq k-1)$ be elements in $S$ such that $\uuuline{u_{i_h,j_h}}v_{j_h,i_h}\rlh v_{j_h,i_h-1}\uuuline{u_{i_h,j_h+1}}$ for $1\leq h\leq kl$. Then
\[
\uuuline{u_{1,1}u_{2,1}\ldots u_{k,1}}v_{1,k}v_{2,k}\ldots v_{l,k} \rlh v_{1,0}v_{2,0}\ldots,v_{l,0} \uuuline{u_{1,l+1}u_{2,l+1}\ldots u_{k,l+1}}.
\]
\end{remark}

\subsection{A Sub-multiplicative Property of Divisor Function}

Our main purpose of this subsection is to prove the following theorem.

\begin{theorem} \label{thm_divisor_submul}
For any $u,v\in S$, we have
\[
\tau(uv)\leq \tau(u)\tau(v).
\]
The equality holds if and only if $u,v$ are castled-free.
\end{theorem}

Several lemmas and corollaries are needed before we prove the above theorem. The following one follows immediately from the fundamental lemma for arithmetic, i.e., Lemma \ref{lem_divisor_castling}.

\begin{corollary} \label{coro_divisor_inequality}
For given $u,v\in S$, we have
\[
\#\{w\in S:\, w|uv,\, \gcd(w,u)=1\}\leq \tau(v).
\]
\end{corollary}

\begin{corollary} \label{cor_divisor_decomposition_uv}
Let $u,v\in S$. Suppose that $w$ is a divisor of $uv$. Then there exist $d,w_1,u_1,\widetilde{u_1},v_1\in S$ such that
\[
w=dw_1,\quad u=du_1,\quad v_1|v, \quad \uuuline{w_1}\widetilde{u_1}\rlh u_1\uuuline{v_1} .
\]
\end{corollary}

\begin{proof}
Let $d=\gcd(w,u)$ and $w=dw_1$, $u=du_1$. Then $w_1|u_1v$ and $\gcd(w_1,u_1)=1$. By Lemma \ref{lem_divisor_castling}, there exists some $v_1|v$ and $\widetilde{u_1}\in S$ such that $\uuuline{w_1}\widetilde{u_1}\rlh u_1\uuuline{v_1}$.
\end{proof}

The next lemma gives another hint of a ``prime''.

\begin{lemma} \label{lem_p_divides_lcm_implies_p_divides_one_of_them}
(\romannumeral1) Let $k\geq 2$ and $u_1,u_2,\ldots,u_k\in S$. Suppose $p$ is an element in $\cP$ such that $p|\lcm[u_1,u_2,\ldots,u_k]$, then $p|u_j$ for some $1\leq j\leq k$.

(\romannumeral2) Let $k\geq 2$ and $u_1,u_2,\ldots,u_k,w\in S$. Suppose $p$ is an element in $\cP$ such that $p\ddagger \lcm[w;u_1,u_2,\ldots,u_k]$, then $p\ddagger u_j$ for some $1\leq j\leq k$.
\end{lemma}

\begin{proof}
Proof of (\romannumeral2) are shown below. Similar arguments lead to (\romannumeral1) and we omit the details here.

We use induction on $k$. For $k=2$, let $d=\gcd_\ddagger(w;u_1,u_2)$, $w=w^\prime d$, $u_1=u_1^\prime d$, $u_2=u_2^\prime d$ and $e=\lcm_\ddagger[w^\prime;u_1^\prime,u_2^\prime]$. Then $\gcd_\ddagger(w^\prime;u_1^\prime,u_2^\prime)=1$ and $\lcm_\ddagger[w;u_1,u_2]=ed$. If $p\ddagger d$, then both $p\ddagger u_1$ and $p\ddagger u_2$ hold. In the following, we assume that $p\not \ddagger d$. In view of the fact $p\ddagger ed$ and Lemma \ref{lem_divisor_castling}, there is some $q\ddagger e$ and $w\in S$ such that $\uuuline{w}p\rlh q\uuuline{d}$.

Since $\gcd_\ddagger(w^\prime; u_1^\prime,u_2^\prime)=1$, we write $\uuuline{y}u_1^\prime\rlh x\uuuline{u_2^\prime}$ for some $x,y\in S$, where $e= y u_1^\prime =x u_2^\prime $. We shall prove below that either $q\ddagger u_1^\prime$ or $q\ddagger u_2^\prime$. Suppose that $q\not \ddagger u_1^\prime$. Recall that $q\ddagger e= y u_1^\prime$.
Then ,by Lemma \ref{lem_divisor_castling}, there are some $r\ddagger y$ and $\widetilde{u_1}^\prime \in S$ such that $\uuuline{\widetilde{u_1}^\prime}q\rlh r\uuuline{u_1^\prime}$. Put $y=zr$. Combining Lemma \ref{lem_C1_subseteq_C_0}, we have
\[
y\uuuline{u_1^\prime} = zr\uuuline{u_1^\prime} \rlh z\uuuline{\widetilde{u_1}^\prime}q \rlh x\uuuline{ \widetilde{z}q} = x\uuuline{u_2^\prime}
\]
for some $\widetilde{z}\in S$. In particular, we have $u_2^\prime=\widetilde{z}q$. So $q\ddagger u_2^\prime$. Indeed, the above arguments show that $q\ddagger u_j^\prime$ for $j=1$ or $j=2$. Now $wp=qd\ddagger u_j^\prime d = u_j$. Thus, one has $p\ddagger u_j$ for $j=1$ or $j=2$.

Suppose that the lemma has been proved for $k\leq K-1$ with some $K\geq 3$. Now we consider the case $k=K$. Noting that $\lcm[w;u_1,\ldots,u_K]= \lcm_\ddagger[w;\lcm_\ddagger[w;u_1,\ldots,u_{K-1}],u_K]$. By inductive hypothesis, one deduces that either $p\ddagger u_K$ or $p\ddagger\lcm_\ddagger[w;u_1,\ldots,u_{K-1}]$. For the latter case, one obtains from inductive hypothesis again that $p\ddagger u_j$ for some $1\leq j\leq K-1$. This completes the proof.
\end{proof}

\begin{corollary} \label{cor_pairwise_gcd}
(\romannumeral1) Let $k\geq 2$ and $u_1,u_2,\dots,u_k,v\in S$. Suppose that $\gcd(u_i,v)=1$ for $1\leq i\leq k$. Then $\gcd\left(\lcm[u_1,u_2,\ldots,u_k],v\right)=1$.

(\romannumeral2) Let $k\geq 2$ and $u_1,u_2,\dots,u_k,v,w\in S$. Suppose that $\gcd_\ddagger(w;u_i,v)=1$ for $1\leq i\leq k$. Then $\gcd_\ddagger\left(w;\lcm_\ddagger[w;u_1,u_2,\ldots,u_k],v\right)=1$.
\end{corollary}

\begin{proof}
(\romannumeral1) Assume on the contrary that there is some $p\in \cP$ such that $p|\gcd\left(\lcm[u_1,u_2,\ldots,u_k],v\right)$. Then $p|v$ and $p|\lcm[u_1,u_2,\ldots,u_k]$. By Lemma \ref{lem_p_divides_lcm_implies_p_divides_one_of_them}, we have $p|u_j$ for some $1\leq j\leq k$. Then $p|\gcd(u_j,v)$, which is a contradiction.

(\romannumeral2) The proof is similar as above and we omit it here.
\end{proof}

\begin{lemma} \label{lem_stronger_Axiom_1}
Suppose that $w_1,w_2,z_1,z_2$ are elements in $S$ such that $\lcm[w_1,z_1]=\lcm[w_2,z_2]$ and $\gcd(w_i,z_j)=1$ $(1\leq i,j\leq 2)$. Then $w_1=w_2$ and $z_1=z_2$.
\end{lemma}

\begin{proof}
Since $\lcm[w_1,z_1]=\lcm[w_2,z_2]$, one can deduce that
\[
\lcm[w_1,z_1] = \lcm[w_1,w_2,z_1] = \lcm[\lcm[w_1,w_2],z_1],
\]
\[
\lcm[w_2,z_2] = \lcm[w_1,w_2,z_2] = \lcm[\lcm[w_1,w_2],z_2].
\]
Moreover, it follows from $\gcd(w_i,z_j)=1$ $(1\leq i,j\leq 2)$ and Corollary \ref{cor_pairwise_gcd} that
\[
\gcd(\lcm[w_1,w_2],z_1)= \gcd(\lcm[w_1,w_2],z_2)=1.
\]
Thus, we conclude by Lemma \ref{lem_lcm_unique_from_eta_2=id} that $z_1=z_2$. By similar arguments, one also obtains that $w_1=w_2$.
\end{proof}

\begin{lemma} \label{lem_divisor_decomposition_castlable_uv}
Let $u,v\in S$. Denote $\cA=\{(d,v_1): \, d|u,\, v_1|v\}$ and $\cB=\{w\in S:\, w|uv\}$. Suppose that $u,v$ are castled-free. Then
\begin{equation} \label{eq_set_tau_u_tau_v}
\cA=\left\{(d,v_1): \, u=du_1, \, v=v_1v_2,\, \uuuline{u_1}v_1 \rightleftharpoons \breve{v_1} \uuuline{\breve{u_1}} \text{ for some }\breve{v_1},\breve{u_1}\in S\right\}.
\end{equation}
Moreover, the map $\rho:\, \cA\rightarrow \cB,\, (d,v_1)\mapsto d\breve{v_1}$ is bijective. Furthermore, the elements $d,\breve{v_1}$ are castled-free, and so are $\breve{u_1},v_2$.
\end{lemma}

\begin{proof}
Suppose that $\uuuline{u}v\rlh \widetilde{v}\uuuline{\widetilde{u}}$ for some $\widetilde{v},\widetilde{u}\in S$, where $\gcd(u,\widetilde{v})=1$. It is easy to see that the right-hand side of \eqref{eq_set_tau_u_tau_v} is a subset of $\cA$. On the other hand, for any $(d,v_1)\in \cA$, let $u=du_1$, $v=v_1v_2$. By Lemma \ref{lem_C1_subseteq_C_0}, the elements $u_1,v_1$ are also castled-free. So there are some $\breve{v_1},\breve{u_1}\in S$ such that $\uuuline{u_1}v_1 \rightleftharpoons \breve{v_1} \uuuline{\breve{u_1}}$. Then \eqref{eq_set_tau_u_tau_v} follows. Moreover, since $uv=du_1v_1v_2=d\breve{v_1}\breve{u_1}v_2$, one has $d\breve{v_1}|uv$. So $\rho(\cA) \subseteq \cB$.

Recalling Lemma \ref{lem_C1_subseteq_C_0} and Remark \ref{remark_decomposition_chain}, we take the decomposition chain $(2,1),\,(1,1),\,(1,2),\,(1,2)$ of depth $(2,2)$ and $u=du_1$, $v=du_2$. Then
\begin{equation} \label{eq_3correspondence}
\uuuline{u}v = \uuuline{du_1}v_1v_2 \rlh \uuuline{d}\breve{v_1}\uuuline{\breve{u_1}} v_2 \rlh  \widehat{v_1} \uuuline{\breve{d} \breve{u_1}} v_2 \rlh \widehat{v_1} \uuuline{\breve{d} } \dot{v_2} \uuuline{\dot{u_1}} \rlh  \widehat{v_1}\widehat{v_2} \uuuline{\dot{d}\dot{u_1}} = \widetilde{v} \uuuline{\widetilde{u}}
\end{equation}
for some $\widehat{v_1},\widehat{v_2},\breve{d},\dot{d},\dot{u_1},\dot{v_2}\in S$. So $d,\breve{v_1}$ are castled-free and $d\breve{v_1}=\lcm[d,\widehat{v_1}]$. And $\breve{u_1},v_2$ are also castled-free.

Suppose that $(d^\prime, v_1^\prime)$ is an element in $\cA$ such that $\rho((d^\prime,v_1^\prime))=\rho((d,v_1))$. Write $u=d^\prime u_1^\prime,\, v=v_1^\prime v_2^\prime$ and $\uuuline{u_1^\prime} v_1^\prime \rightleftharpoons \breve{v_1}^\prime\uuuline{\breve{u_1}^\prime}$ for some $\breve{v_1}^\prime,\breve{u_1}^\prime\in S$. And one obtains in a similar way that
\begin{align*}
\uuuline{u}v = \uuuline{d^\prime u_1^\prime}v_1^\prime v_2^\prime \rlh \uuuline{d^\prime}\breve{v_1}^\prime\uuuline{\breve{u_1}^\prime} v_2^\prime \rlh  \widehat{v_1}^\prime \uuuline{\breve{d}^\prime \breve{u_1}^\prime} v_2^\prime \rlh \widehat{v_1}^\prime \uuuline{\breve{d}^\prime } \dot{v_2}^\prime \uuuline{\dot{u_1}^\prime} \rlh  \widehat{v_1}^\prime\widehat{v_2}^\prime \uuuline{\dot{d}^\prime\dot{u_1}^\prime} = \widetilde{v} \uuuline{\widetilde{u}}
\end{align*}
for some $\widehat{v_1}^\prime,\widehat{v_2}^\prime,\breve{d}^\prime,\dot{d}^\prime,\dot{u_1}^\prime,\dot{v_2}^\prime\in S$. Then
\begin{equation} \label{eq_tau_uv=tau_u_tau_v_for_free_1}
\lcm[d^\prime,\widehat{v_1}^\prime]= d^\prime \breve{v_1}^\prime=d\breve{v_1} = \lcm[d,\widehat{v_1}].
\end{equation}
Note that $d,d^\prime$ are both divisors of $u$, and $\breve{v_1},\breve{v_1}^\prime$ are both divisors of $\widetilde{v}$. In view of the condition $\gcd(u,\widetilde{v})=1$, we have .
\begin{equation} \label{eq_tau_uv=tau_u_tau_v_for_free_2}
\gcd(d,\widehat{v_1}) = \gcd(d,\widehat{v_1}^\prime) = \gcd(d^\prime,\widehat{v_1}) = \gcd(d^\prime,\widehat{v_1}^\prime) =1.
\end{equation}
Combining \eqref{eq_tau_uv=tau_u_tau_v_for_free_1}, \eqref{eq_tau_uv=tau_u_tau_v_for_free_2} and Lemma \ref{lem_stronger_Axiom_1}, we conclude that $d=d^\prime$ and then $v_1=v_1^\prime$. Hence, the map $\rho$ is injective.

Finally, we deduce from Corollary \ref{cor_divisor_decomposition_uv} that $\rho$ is surjective. The proof is completed.
\end{proof}

Recall that, for $u\in S$, the sets $\PD(u)$ and $\PD_\ddagger(u)$ stand for the set of irreducible divisors and co-divisors of $u$, respectively. And $\omega(u)=\#\PD(u)$, $\omega_\ddagger(u)=\#PD_\ddagger(u)$.

\begin{corollary} \label{cor_divisor_decomposition_castlable_uv}
Suppose that $\uuuline{u}v \rlh \widetilde{v}\uuuline{\widetilde{u}}$. (\romannumeral1) There is a one-to-one correspondence between $\cC=\{v_1:\, v_1|v\}$ and $\cD=\{\breve{v_1}:\,\breve{v_1}|\widetilde{v}\}$ by $\rho^\prime: \cC\rightarrow \cD, v_1\mapsto \rho((1,v_1))$. (\romannumeral2) There is a one-to-one correspondence between $\PD(\widetilde{v})$ and $\PD(v)$. (\romannumeral3) There is a one-to-one correspondence between $\PD_\ddagger(\widetilde{v})$ and $\PD_\ddagger(v)$. In particular, we have
\[
\tau(\widetilde{v})=\tau(v),\quad \omega(\widetilde{v})=\omega(v),\quad \omega_\ddagger(\widetilde{v})=\omega_\ddagger(v).
\]
\end{corollary}

\begin{proof}
Recall \eqref{eq_3correspondence} from the proof of Lemma \ref{lem_divisor_decomposition_castlable_uv}. When $d=1$, we have $u_1=u$, $\rho^\prime(v_1)=\rho((1,v_1))=\breve{v_1}$, and $\breve{v_1}=\widehat{v_1}$, which is a divisor of $\widetilde{v}$. So $\rho^\prime(\cC)\subseteq \cD$. For any $\breve{v_1}\in \cD$, we have $\gcd(u,\breve{v_1})=1$. Then there are some $v_1$ and $\breve{u}$ in $S$ such that $\uuuline{u}v_1 \rightleftharpoons \breve{v_1} \uuuline{\breve{u}}$. Indeed, we have $\lcm[u,\breve{v_1}]= uv_1=\breve{v_1}\breve{u}$. Since $\lcm[u,\breve{v_1}]|\lcm[u,\widetilde{v}]$, one has $uv_1|uv$. So $v_1\in \cC$ and $\rho^\prime(v_1)=\breve{v_1}$. Now we have shown that $\rho^\prime$ is surjective. In view of the fact that $\rho$ is injective, one concludes that $\rho^\prime$ is also injective. The first correspondence follows.

Since $\ind(v_1)=\ind(\breve{v_1})$, we have that $v_1\in \cP$ if and only if $\breve{v_1}\in \cP$. The second correspondence also holds. Similarly, we have that $\ind(v_1)=\ind(v)-1$ if and only if $\ind(\breve{v})=\ind(\widetilde{v})-1$. The third correspondence follows.
\end{proof}

\begin{remark}
Note that $\uuuline{u}v \rlh \widetilde{v}\uuuline{\widetilde{u}}$ is equivalent to $\widetilde{v}\uuuline{\widetilde{u}} \rlh \uuuline{u}v$. We also have
\[
\tau(\widetilde{u})=\tau(u),\quad \omega(\widetilde{u})=\omega(u),\quad \omega_\ddagger(\widetilde{u})=\omega_\ddagger(u).
\]
\end{remark}

Now we shall prove  Theorem \ref{thm_divisor_submul}.

\begin{proof} [Proof of Theorem \ref{thm_divisor_submul}]
For any divisor $w|uv$, suppose that $\gcd(w,u)=d$ and write $w=dw_1$ and $u=du_1$. Then $\gcd(w_1,u_1)=1$ and $w_1|u_1v$. It follows from Corollary \ref{coro_divisor_inequality} that
\begin{equation} \label{eq_w_1_inequality}
\#\{w_1\in S:\, w_1|u_1v, \gcd(w_1,u_1)=1\}\leq \tau(v).
\end{equation}
As a result, we deduce that
\begin{equation} \label{eq_tau_uv_inequality}
\tau(uv)= \sum\limits_{w|uv}1 =  \sum\limits_{d|u}\sum\limits_{w|uv\atop \gcd(w,u)=d} 1 =\sum\limits_{d|u}\sum\limits_{w_1|u_1 v\atop \gcd(w_1,u_1)=1} 1 \leq \sum\limits_{d|u} \tau(v) = \tau(u)\tau(v).
\end{equation}

When the equality in \eqref{eq_tau_uv_inequality} holds, the equality in \eqref{eq_w_1_inequality} also holds. In particular, we have $\#\{w\in S:\, w|uv, \, \gcd(w,u)=1\}= \tau(v)$. Combining the uniqueness stated in Lemma \ref{lem_divisor_castling}, there is some $w_0\in S$ with $w_0|uv$ and $\gcd(w_0,u)=1$ such that $\uuuline{w_0}\widetilde{u}\rlh u\uuuline{v}$ for some $\widetilde{u}\in S$. So $u,v$ are castled-free. On the other hand, when $u,v$ are castled-free, one deduces from Lemma \ref{lem_divisor_decomposition_castlable_uv} that $\tau(uv)=\tau(u)\tau(v)$. The theorem then follows.
\end{proof}

\begin{corollary}
For any $u\in S$, we have $\tau(u)\leq 2^{\ind(u)}$.
\end{corollary}

Next, we consider the relation of number of divisors of $u,v$ and $\lcm[u,v]$.

\begin{corollary} \label{cor_tau_lcm_coprime_u_j}
(\romannumeral1) Suppose that $k\geq 2$ and $u_1,u_2,\ldots,u_k$ are elements in $S$ with $\gcd(u_i,u_j)=1$ $(1\leq i<j\leq k)$. Then
\[
\tau(\lcm[u_1,u_2,\ldots, u_k])=\tau(u_1)\tau(u_2)\ldots \tau(u_k).
\]

(\romannumeral2) Suppose that $k\geq 2$ and $u_1,u_2,\ldots,u_k,w$ are elements in $S$ with $\gcd_\ddagger(w;u_i,u_j)=1$ $(1\leq i< j\leq k)$. Then
\[
\tau(\lcm_\ddagger[w;u_1,u_2,\ldots, u_k])=\tau(u_1)\tau(u_2)\ldots \tau(u_k).
\]
\end{corollary}

\begin{proof}
(\romannumeral1) Since $\gcd(u_1,u_2)=1$, we have $\uuuline{u_1}y\rlh u_2\uuuline{x}$ for some $x,y\in S$. By Theorem \ref{thm_divisor_submul} and Corollary \ref{cor_divisor_decomposition_castlable_uv}, we deduce that
\[
\tau(\lcm[u_1,u_2]) = \tau(u_1y) = \tau(u_1)\tau(y) = \tau(u_1)\tau(u_2).
\]
Suppose the lemma has been proved for $k-1$ elements $u_1,\ldots,u_{k-1}$. By Corollary \ref{cor_pairwise_gcd}, we have $\gcd(\lcm[u_1,\ldots,u_{k-1}],u_k)=1$. Then
\begin{align*}
&\tau(\lcm[u_1,\ldots,u_{k-1},u_k]) = \tau(\lcm[\lcm[u_1,\ldots,u_{k-1}],u_k]) \\
&\quad \quad = \tau(\lcm[u_1,\ldots,u_{k-1}])\tau(u_k) = \tau(u_1)\ldots \tau(u_{k-1})\tau(u_k).
\end{align*}

(\romannumeral2) Similar arguments as in (\romannumeral1) work.
\end{proof}

\begin{corollary} \label{cor_tau_lcm_p}
(\romannumeral1) Let $k\geq 1$ and $q_1,\ldots,q_k$ be distinct elements in $\cP$. Then $\tau(\lcm[q_1,\ldots,q_k])=2^k$. (\romannumeral2) Let $k\geq 1$, $w\in S$ and $q_1,\ldots,q_k$ be distinct elements in $\cP$ such that $q_1,\ldots,q_k\ddagger w$. Then $\tau(\lcm_\ddagger[w;q_1,\ldots,q_k])=2^k$.

\end{corollary}

\medskip

We end this section by calculating the M\"{o}bius function.

\begin{theorem}
We have
\begin{equation} \label{eq_mu_expression}
\mu(u) =
\begin{cases}
1,\quad &\If u=1,\\
(-1)^k,\quad &\If u=\lcm[q_1,\ldots,q_k] \text{ for distinct }q_1,\ldots,q_k\in \cP,\\
0,\quad &\Otherwise.
\end{cases}
\end{equation}
In particular, we have $\mu(u)=(-1)^k$ if and only if $\omega(u)=k$, $\tau(u)=2^k$.
\end{theorem}

\begin{proof}
Let $\mu$ be given as in \eqref{eq_mu_expression}. It is sufficient to prove that $1\ast \mu=\delta_1$. For $u=1$, it satisfies $(1\ast \mu)(1)=1\cdot\mu(1)=1=\delta_1(1)$.

Now consider the case $u\neq 1$. Suppose that $\PD(u)=\{q_1,\ldots,q_k\}$. Then
\begin{equation} \label{eq_divisor_of_lcm_q}
\lcm[q_1^{l_1},\ldots,q_k^{l_k}], \quad (l_1,\ldots,l_k)\in \{0,1\}^k,
\end{equation}
are $2^k$ distinct divisors of $u$, which satisfy
\[
\mu\left(\lcm[q_1^{l_1},\ldots,q_k^{l^k}]\right)=(-1)^{l_1+\ldots +l_k}.
\]
And $\mu(d)=0$ for all other divisors $d|u$. It follows that
\begin{align*}
(\mu\ast 1)(u) &= \sum\limits_{d|u}\mu(d) = \sum\limits_{(l_1,\ldots,l_k)\in \{0,1\}^k} \mu\left(\lcm[q_1^{l_1},\ldots,q_k^{l^k}]\right)\\
& = \sum\limits_{(l_1,\ldots,l_k)\in \{0,1\}^k} (-1)^{l_1+\ldots+l_k} = (1-1)^k =0 = \delta_1(u).
\end{align*}
The proof of \eqref{eq_mu_expression} is completed.

If $u=\lcm[q_1,\ldots,q_k]$, then $\omega(u)=k$ and $\tau(u)=2^k$ by Corollary \ref{cor_tau_lcm_p}. On the other hand, suppose that $\omega(u)=k$ and $\tau(u)=2^k$. Write $q_1,\ldots,q_k$ for the $k$ distinct irreducible divisors of $u$. Then $\lcm[q_1,\ldots,q_k]$ is a divisor of $u$ and has $2^k$ distinct divisors as in \eqref{eq_divisor_of_lcm_q}. Since $\tau(u)=2^k$, one has $u= \lcm[q_1,\ldots,q_k]$ and then $\mu(u)=(-1)^k$.
\end{proof}

The following lemma shows that the least common multiple appeared in \eqref{eq_mu_expression} can be replaced by least common co-multiple.

\begin{lemma} \label{lem_lcm_dagger_power=1}
Suppose that $k\geq 2$ and $q_1,q_2,\ldots,q_k$ be distinct irreducible elements. Let $u=\lcm[q_1,q_2,\ldots,q_k]$. Then there exist distinct irreducible elements $r_1,r_2,\ldots,r_k$ such that $u=\lcm_\ddagger[u;r_1,r_2,\ldots,r_k]$.
\end{lemma}

\begin{proof}
For $1\leq j\leq k$, let $v_j=\lcm[q_1,\ldots,q_{j-1},q_{j+1},\ldots,q_k]$. By Corollary \ref{cor_pairwise_gcd}, one deduces that $\gcd(q_j,v_j)=1$. Write $u=\lcm[q_j,v_j]=q_jz_j=v_jr_j$ for some $z_j,r_j\in S$. Then $\uuuline{q_j}v_j\rlh v_j\uuuline{r_j}$, and $r_j\in \cP$. 
To prove that $r_1,\ldots,r_k$ are distinct, we assume on the contrary that $r_i=r_l$ for some $1\leq i\neq l\leq k$. Then $v_i=v_l$. By Axiom \uppercase\expandafter{\romannumeral4}, we conclude that $q_i=q_l$, which is a contradiction. Now we have $r_1,\ldots,r_k\ddagger u$, and
\[
\ind(u)=\ind(\lcm[q_1,\ldots,q_k])=k=\ind(\lcm_\ddagger[u;r_1,\ldots,r_k]).
\]
So $u=\lcm_\ddagger[u;r_1,\ldots,r_k]$.
\end{proof}

\section{Axiom \uppercase\expandafter{\romannumeral4} and Castlable Monoids}
\label{section_castlable_monoid}

Let $\fC\subseteq S\times S$ and $\Gamma\subseteq \fC\times \fC$ be sets generated by the following rules.
\begin{itemize}
\item The set $\fC$ contains $\fC_1$, and $\Gamma$ contains $\Gamma_1$.

\item For any $p\in \cP$, it satisfies
\begin{equation} \label{eq_C_definition_p_p}
(p,p)\in \fC,\quad ((p,p),(p,p))\in \Gamma.
\end{equation}

\item Suppose that $u_1,u_2,v,\widetilde{u_1},\widetilde{u_2},\widetilde{v},\widetilde{\widetilde{v}}$ are elements in $S$ such that $u_1,u_2\neq 1$ and $((u_2,v),(\widetilde{v},\widetilde{u_2}))\in \Gamma$, $((u_1,\widetilde{v}),(\widetilde{\widetilde{v}}, \widetilde{u_1}))\in \Gamma$. Then
\begin{equation} \label{eq_C_castling_composition_1}
(u_1u_2,v),\,(\widetilde{\widetilde{v}},\widetilde{u_1}\widetilde{u_2})\in \fC, \quad \left((u_1u_2,v), (\widetilde{\widetilde{v}},\widetilde{u_1}\widetilde{u_2})\right)\in \Gamma.
\end{equation}

\item Suppose that $u,v_1,v_2,\widetilde{u},\widetilde{v_1},\widetilde{v_2},\widetilde{\widetilde{u}}$ are elements in $S$ such that $v_1,v_2\neq 1$ and $((u,v_1),(\widetilde{v_1},\widetilde{u}))\in \Gamma$, $((\widetilde{u},v_2), (\widetilde{v_2}, \widetilde{\widetilde{u}}))\in \Gamma$. Then
\begin{equation} \label{eq_C_castling_composition_2}
(u,v_1v_2),\, (\widetilde{v_1}\widetilde{v_2}, \widetilde{\widetilde{u}})\in \fC, \quad \left((u,v_1v_2),(\widetilde{v_1}\widetilde{v_2}, \widetilde{\widetilde{u}})\right)\in \Gamma.
\end{equation}
\end{itemize}
That is to say, the set $\fC$ and $\Gamma$ are the minimum ones among the sets satisfying the above properties. Indeed, one can determine whether $(u,v)$ belongs to $\fC$ or not, by induction on $\tau(u)+\tau(v)$.

\medskip

\textsc{Axiom \uppercase\expandafter{\romannumeral4}.} The set $\Gamma$ is a graph of a map $\eta:\, \fC\rightarrow \fC$.

\begin{definition}
We call an integral monoid $S$ castlable, if Axiom \uppercase\expandafter{\romannumeral4} holds.
\end{definition}

For the natural numbers, if $\gcd(n,p)=1$ for $n\in \bN$ and $p$ a prime, then $\gcd(n,p^m)=1$ for all $m\geq 0$. This property plays an important role in classical arithmetics. However, this property may fail for general homogenous groups. For example, if $p^2 = q_1q_2$ for some $p,q_1,q_2\in \cP$ with $p\neq q_1$, then $\gcd(q_1,p)=1$ and $\gcd(q_1,p^2)\neq 1$. The construction of $\fC$ and the map $\eta$ prevent such situations. Indeed, when $p^2=q_1q_2$ with $q_1\neq p$, we have $(p,p)\in \fC_1$ and $\eta((p,p))= (q_1,q_2)$. While we also have $(p,p)\in \fC$ and $\eta((p,p))=(p,p)$. A contradiction appears since $\eta$ is a map.

In this section, we always assume that $\Gamma$ is a graph of a map $\eta:\, \fC\rightarrow \fC$. Then $\eta|_{\fC_1}$ is a well-defined map, which shows that $S$ is homogenous. Lemma \ref{lem_composition_C_1} shows that free castlings also satisfy \eqref{eq_C_castling_composition_1} and \eqref{eq_C_castling_composition_2}. This fact is  compatible with the requirements in the definition of $\fC$ and $\Gamma$. Note that $\eta^2|_{\fC_1}$ is the identity map. And $\eta^2((p,p))=(p,p)$. Combining \eqref{eq_C_castling_composition_1} and \eqref{eq_C_castling_composition_2}, one can verify that $\eta^2=id$ on $\fC$.

Let $(u,v)\in \fC$ and $\eta(u,v)=(\widetilde{v},\widetilde{u})$. Then we also have $(\widetilde{v},\widetilde{u})\in \fC$ and $\eta(\widetilde{v},\widetilde{u})=(u,v)$. For simplicity, we will rewrite the above formulas by either of the following four expressions.
\[
\underline{u}v \rlh \widetilde{v} \underline{\widetilde{u}}, \quad u\underline{v} \rlh  \underline{\widetilde{v}}\widetilde{u},\quad \widetilde{v} \underline{\widetilde{u}} \rlh \underline{u}v,\quad \underline{\widetilde{v}}\widetilde{u} \rlh u\underline{v}.
\]
Moreover, whenever we write $\uline{u} v \rlh \widetilde{v} \uline{\widetilde{u}}$, we mean that $(u,v),\, (\widetilde{v},\widetilde{v})\in \fC$. Now \eqref{eq_C_castling_composition_1} and \eqref{eq_C_castling_composition_2} may be simplified as the following. The composition of weak castlings $\uline{u_2}v\rlh \widetilde{v}\uline{\widetilde{u_2}}$ and $\uline{u_1}\widetilde{v}\rlh \widetilde{\widetilde{v}}\uline{\widetilde{u_1}}$ results in
\[
\uline{u_1u_2}v \rlh \uline{u_1} \widetilde{v} \uline{\widetilde{u_2}} \rlh \widetilde{\widetilde{v}}\uline{\widetilde{u_1}\widetilde{u_2}}.
\]
Similarly, the composition of weak castlings $\uline{u}v_1\rlh \widetilde{v_1}\uline{\widetilde{u}}$ and $\uline{u}v_2\rlh \widetilde{v_2}\uline{\widetilde{\widetilde{u}}}$ leads to
\[
\uline{u}v_1v_2\rlh \widetilde{v_1} \uline{\widetilde{u}} v_2 \rlh \widetilde{v_1}\widetilde{v_2}\uline{\widetilde{\widetilde{u}}}.
\]
The expression $\uline{a}b=\uline{a^\prime}b^\prime$ will always mean $a=a^\prime$ and $b=b^\prime$ as elements in $S$.

\medskip

Next, let $\fC_0$ be a subset of $\fC$ generated by the following rules.

\begin{itemize}
\item The set $\fC_0$ contains $\fC_1$.

\item The set $\fC_0$ contains $(p,p)$ for all $p\in \cP$.

\item If $\uline{u}v \rlh \widetilde{v}\uline{\widetilde{u}}$ and the following two statements both hold, then $(u,v)\in \fC_0$.

    (\romannumeral1) For any $u_1,u_2\neq 1$ with $u_1u_2=u$, there are elements $\widehat{u_1},\widehat{u_2},\widehat{v}$ with $\widehat{u_1}\widehat{u_2}=\widetilde{u}$ such that $(u_2,v),(u_1,\widehat{v})\in \fC_0$ and
\begin{equation}\label{eq_C0_castling_decomposition_1}
\uline{u_2} v \rlh \widehat{v} \uline{\widehat{v_2}}, \quad \uline{u_1}\widehat{v} \rlh \widetilde{v}\uline{\widehat{u_1}}.
\end{equation}

    (\romannumeral2) For any $v_1,v_2\neq 1$ with $v_1v_2=v$, there are elements $\widehat{v_1},\widehat{v_2},\widehat{u}$ with $\widehat{v_1}\widehat{v_2}=\widetilde{v}$ such that $(u,v_1),(\widehat{u},v_2)\in \fC_0$ and
\begin{equation}\label{eq_C0_castling_decomposition_2}
\uline{u} v_1 \rlh \widehat{v_1}\uline{\widehat{u}},\quad \uline{\widehat{u}} v_2  = \widehat{v_2} \uline{\widetilde{u}}.
\end{equation}
\end{itemize}
Moreover, let us put $\Gamma_0 = \Gamma \cap (\fC_0\times \fC_0)$.

The difference between the definition of $\fC_0$ and that of $\fC$ is that the latter requires existence of $u_1,u_2,v_1,v_2$, while the former requires arbitrariness of $u_1,u_2,v_1,v_2$. Indeed, one can also generate $\fC_0$, i.e., determine whether $(u,v)$ belongs to $\fC_0$ or not, by induction on $\ind(u)+\ind(v)$.

If $\uline{u}v\rlh \widetilde{v}\uline{\widetilde{u}}$ and we have further that $(u,v)\in \fC_0$. Then we put double underlines on the side involving $u,v$, i.e.,
\[
\uuline{u}v \rlh \widetilde{v}\uline{\widetilde{u}},\quad \text{or}\quad u\uuline{v}\rlh \uline{\widetilde{v}}\widetilde{u}.
\]
Now \eqref{eq_C0_castling_decomposition_1} and \eqref{eq_C0_castling_decomposition_2} may be abbreviated as
\[
\uuline{u}v = \uuline{u_1u_2}v\rlh \uuline{u_1}\widehat{v} \uline{\widehat{u_2}}\rlh \widetilde{v} \uline{\widehat{u_1}\widehat{u_2}}  = \widetilde{v}\uline{\widetilde{u}}, \quad u\uuline{v} = u\uuline{v_1v_2}\rlh \uline{\widehat{v_1}} \widehat{u}\uuline{v_2}\rlh \uline{\widehat{v_1}\widehat{v_2}} \widetilde{u} = \uline{\widetilde{v}}\widetilde{u}.
\]
For the second formula, we should avoid to write $\uuline{u}v_1v_2\rlh \widehat{v_1}\uuline{\widehat{u}}v_2$, since it would be confusing whether the double underlines under $\widehat{u}$ is paired with $\widehat{v_1}$ or $v_2$. When $(\widetilde{v},\widetilde{u})$ also belongs to $\fC_0$, we also draw double underlines on the other side of the above expressions. For example, we have $\uuline{p^k}p^l \rlh p^l \uuline{p^k}$ for any $p\in \cP$ and $k,l\geq 0$, which results from Lemma \ref{lem_uniqueness_prime_power}.

Note that Lemma \ref{lem_C1_subseteq_C_0} ensures that elements in $\fC_1$ satisfy \eqref{eq_C0_castling_decomposition_1} and \eqref{eq_C0_castling_decomposition_2}. Combining the construction of $\fC_0$, we deduce that any pair $(u,v)$ in $\fC_0$ satisfies \eqref{eq_C0_castling_decomposition_1} and \eqref{eq_C0_castling_decomposition_2}. That is to say, we have $\uuline{u}v \rlh \widetilde{v}\uline{\widetilde{u}}$ if and only if, for any $u_1,u_2$ with $u_1u_2=u$, there are elements $\widehat{u_1},\widehat{u_2},\widehat{v}$ such that $\uuline{u}v = \uuline{u_1u_2}v\rlh \uuline{u_1}\widehat{v} \widehat{u_2}\rlh \widetilde{v} \uline{\widehat{u_1}\widehat{u_2}}  = \widetilde{v}\uline{\widetilde{u}}$; and for any $v_1,v_2$ with $v_1v_2=v$, there are elements $\widehat{u},\widehat{v_1},\widehat{v_2}$ such that $u\uuline{v} = u\uuline{v_1v_2}\rlh \uline{\widehat{v_1}} \widehat{u}\uuline{v_2}\rlh \uline{\widehat{v_1}\widehat{v_2}} \widetilde{u} = \uline{\widetilde{v}}\widetilde{u}$.

\begin{definition}
(\romannumeral1) When $(u,v)\in \fC$, we say that $u,v$ are weakly castlable. When $\uline{u}v\rlh \widetilde{v}\uline{\widetilde{u}}$, we call it, or $((u,v),(\widetilde{v},\widetilde{u}))\in \Gamma$, a weak castling.

(\romannumeral2) When $(u,v)\in \fC_0$, we say that $u,v$ are strongly castlable. When $\uuline{u}v\rlh \widetilde{v}\uuline{\widetilde{u}}$, we call it, or $((u,v),(\widetilde{v},\widetilde{u}))\in \Gamma_0$, a strong castling.

(\romannumeral3) Let $w,u,v\in S$. When $\uline{w} \widetilde{u} \rlh u\uline{v_1}$ for some $v_1|v$ and $\widetilde{u}\in S$, we call $w$ a castled divisor of $v$ over $u$. When $\uline{\widetilde{v}}w  \rlh u_1\uline{v}$ for some $u_1\ddagger u$ and $\widetilde{v}\in S$, we call $w$ a castled co-divisor of $u$ over $v$.

(\romannumeral4) We call an element $p$ in $S\setminus \{1\}$ a prime, when for any $u,v\in S$ with $p|uv$, either $p|u$ or $p$ is a castled divisor of $v$ over $u$.
\end{definition}

\subsection{Basic Properties of Castlings}

\begin{lemma} \label{lem_ind_invariant}
Suppose that $\uline{u} v \rlh \widehat{v} \uline{\widehat{u}}$. Then $uv=\widehat{v}\widehat{u}$, and $\ind(\widehat{u})=\ind(u)$, $\ind(\widehat{v})=\ind(v)$.
\end{lemma}

\begin{proof}
For $\ind(u)=0$ or $\ind(v)=0$, the proof is trivial. In the following, we assume that $\ind(u),\ind(v)\geq 1$. We use induction on $\ind(u)+\ind(v)$. Suppose that the lemma has been proved for $\ind(u)+\ind(v)\leq m-1$ with some $m\geq 2$. Now we consider the case $\ind(u)+\ind(v)=m$.

For $(u,v)\in \fC_1$, the expected results follow from Corollary \ref{cor_C_1_ind_transfer}. For $(u,v)=(p,p)$ for some prime $p$, we have $(\widehat{v},\widehat{u})=(p,p)$ and the expected results hold. For $(u,v)\in \fC$ obtained from \eqref{eq_C_castling_composition_1} with some $u_1,u_2\neq 1$, i.e.,
\[
\underline{u}v=\underline{u_1u_2}v \rlh \underline{u_1} \widetilde{v} \underline{\widetilde{u_2}} \rlh \widetilde{\widetilde{v}}\underline{\widetilde{u_1}\widetilde{u_2}} = \underline{\widehat{v}}\widehat{u}
\]
for some $\widetilde{u_1},\widetilde{u_2},\widetilde{v},\widetilde{\widetilde{v}}\in S$, one deduces by inductive hypothesis that
\[
u_2v=\widetilde{v}\widetilde{u_2},\quad \ind(u_2)=\ind(\widetilde{u_2}),\quad \ind(\widetilde{v})=\ind(\widetilde{\widetilde{v}}),
\]
\[
u_1\widetilde{v}= \widetilde{\widetilde{v}} \widetilde{u_1}, \quad \ind(u_1)=\ind(\widetilde{u_1}), \quad \ind(v)=\ind(\widetilde{v}).
\]
It follows that $uv=u_1u_2v=u_1 \widetilde{v}\widetilde{v_2}=\widetilde{\widetilde{v}}\widetilde{u_1}\widetilde{u_2}=\widehat{v}\widehat{u}$ and
\[
\ind(\widehat{u})=\ind(\widetilde{u_1})+\ind(\widetilde{u_2})= \ind(u_1)+\ind(u_2)=\ind(u),\quad \ind(\widehat{v})=\ind(\widetilde{\widetilde{v}})=\ind(v).
\]
If $(u,v)\in \fC$ is obtained by \eqref{eq_C_castling_composition_2} with some $v_1,v_2\neq 1$, i.e.,
\[
\uline{u}v=\uline{u}v_1v_2\rlh \widetilde{v_1} \uline{\widetilde{u}} v_2 \rlh \widetilde{v_1}\widetilde{v_2}\uline{\widetilde{\widetilde{u}}}=\widehat{v}\uline{\widehat{u}}
\]
for some $\widetilde{v_1},\widetilde{v_2},\widetilde{u},\widetilde{\widetilde{u}}\in S$, then one deduces by inductive hypothesis that $uv_1=\widetilde{v_1}\widetilde{u}$, $\widetilde{u}v_2=\widetilde{v_2}\widetilde{\widetilde{u}}$ and
\[
\ind(v_1)=\ind(\widetilde{v_1}),\quad \ind(u)=\ind(\widetilde{u}),\quad \ind(v_2)=\ind(\widetilde{v_2}),\quad \ind(\widetilde{u})=\ind(\widetilde{\widetilde{u}}).
\]
One deduces similarly that $uv= uv_1v_2=\widetilde{v_1}\widetilde{u}v_2=\widetilde{v_1}\widetilde{v_2}\widetilde{\widetilde{u}}=\widehat{v}\widehat{u}$ and
\[
\ind(\widehat{v})=\ind(\widetilde{v_1})+\ind(\widetilde{v_2})= \ind(v_1)+\ind(v_2)=\ind(v),\quad \ind(\widehat{u})=\ind(\widetilde{\widetilde{u}})=\ind(u).
\]
By induction, the proof is completed.
\end{proof}

\begin{remark}
Suppose that $\uline{p}q\rlh r\uline{t}$ for $p,q,r,t\in \cP$. If $p\neq r$, then $\gcd(p,r)=1$ and $\uuuline{p}q\rlh r\uuuline{t}$. If $p=r$, then $q=t=p$ and $\uline{p}p\rlh p\uline{p}$. For all other cases, we have $((p,q),(r,t))\notin \Gamma$. Similarly, if $q\neq t$, then $\gcd_\ddagger(pq;q,t)=1$ and $\uuuline{p}q\rlh r\uuuline{t}$. If $q=t$, then $p=r=q$ and $\uline{q}q\rlh q\uline{q}$. For all other cases, we have $((p,q),(r,t))\notin \Gamma$.
\end{remark}

Recall that ``$w$ is a castled divisor of $v$ over $u$'' means that $\uline{w}\widetilde{u}\rlh u\uline{v_1}$ for some $v_1|v$ and $\widetilde{u}\in S$. Here $\ind(w)=\ind(v_1)$.

\begin{corollary} \label{coro_prime_insteadof_irr}
The set $\cP$ consists of all the primes.
\end{corollary}

\begin{proof}
Suppose that $w\in \cP$ and $w|uv$ for some $u,v\in S$. Denote $d=\gcd(w,u)$. Since $\tau(w)=2$, we have either $d=w$ or $d=1$. For the former case, we have $w|u$. For the latter case, an application of Corollary \ref{cor_divisor_decomposition_uv} shows that $w$ is a castled divisor of $v$ over $u$. On the other hand, if $w\notin\cP\cup \{1\}$, then $\ind(w)\geq 2$. We write $w=uv$, where $1<\ind(u),\ind(v)<\ind(w)$. Then neither $w|u$, nor $w$ is a castled divisor of $v$ over $u$.
\end{proof}

The following corollary can be proved with similar arguments as above. We omit the details here.

\begin{corollary}
An element $p$ in $S\setminus \{1\}$ is a prime if and only if, whenever $p\ddagger uv$, either $p\ddagger v$ or $p$ is a castled co-divisor of $u$ over $v$.
\end{corollary}

From now on, one may say that $u,v$ are coprime when $\gcd(u,v)=1$, and say ``prime decomposition'' instead of ``irreducible decomposition''. The following corollary explains why we say $u,v$ castled-free when $(u,v)\in \fC_1$.

\begin{corollary}
Suppose that $\uline{u} v \rlh \widetilde{v} \uline{\widetilde{u}}$. If $u,v$ are not castled-free, then $\gcd(u,\widetilde{v})\neq 1$ and $\gcd_\ddagger(uv; v,\widetilde{u})\neq 1$.
\end{corollary}

\begin{proof}
Suppose on the contrary that $\gcd(u,\widetilde{v})=1$. Then, by Corollary \ref{lem_gcd_lcm_u_v} and Lemma \ref{lem_ind_invariant}, we have
\[
\ind(\lcm[u,\widetilde{v}])=\ind(u)+\ind(\widetilde{v})=\ind(u)+\ind(v)=\ind(uv).
\]
Moreover, we have $uv=\widetilde{v}\widetilde{u}$, which is a common divisor of $u$ and $\widetilde{v}$. So $\lcm[u,\widetilde{v}]|uv$. It follows that $\lcm[u,\widetilde{v}]=uv=\widetilde{v}\widetilde{u}$ and $u,\widetilde{v}$ are castled-free, which a contradiction. We conclude that $\gcd(u,\widetilde{v})\neq 1$. Similar arguments lead to the conclusion that $\gcd_\ddagger(uv; v,\widetilde{u})\neq 1$.
\end{proof}

\medskip

\begin{lemma} [Decomposition of strong castlings]
Let $u,v,\widetilde{u},\widetilde{v}$ be elements in $S$ such that $\uuline{u}v \rlh \widetilde{v} \uuline{\widetilde{u}}$.

(\romannumeral1). For any $u_1,u_2\in S$ with $u_1u_2=u$, there exist elements $\widehat{u_1},\widehat{u_2},\widehat{v}$ in $S$ with $\widehat{u_1}\widehat{u_2}=\widetilde{u}$ such that
\[
\uuline{u_2}v \rlh \widehat{v} \uuline{\widehat{u_2}},\quad \uuline{u_1} \widehat{v}  \rlh \widetilde{v} \uuline{\widehat{u_1}}.
\]

(\romannumeral2). For any $v_1,v_2\in S$ with $v_1v_2=v$, there exist elements $\widehat{u},\widehat{v_1},\widehat{v_2}$ with $\widehat{v_1}\widehat{v_2}=\widetilde{v}$ such that
\begin{align*}
\uuline{u}v_1 =\widehat{v_1}\uuline{\widehat{u}},\quad \uuline{\widehat{u}}v_2  =  \widehat{v_2} \uuline{\widetilde{u}}.
\end{align*}
\end{lemma}

\begin{proof}
(\romannumeral1) Since $u,v$ are strongly castlable, we have
\begin{equation} \label{eq_one_to_one_correspondence_11}
\uuline{u_2}v \rlh \widehat{v} \uline{\widehat{u_2}},\quad \uuline{u_1}\widehat{v} \rlh \widetilde{v}\uline{\widehat{u_1}}
\end{equation}
for some $\widehat{u_1},\widehat{u_2},\widehat{v}\in S$ with $\widehat{u_1}\widehat{u_2}=\widetilde{u}$. Note that $\widetilde{v},\widetilde{u}$ are also strongly castlable. It follows that $\widetilde{v_1},\widehat{u_1}$ are strongly castlable, and so are $\widehat{v},\widehat{u_2}$. Now \eqref{eq_one_to_one_correspondence_11} becomes
\[
\uuline{u_2}v \rlh \widehat{v} \uuline{\widehat{u_2}},\quad \uuline{u_1}\widehat{v} \rlh \widetilde{v}\uuline{\widehat{u_1}}.
\]
The proof is completed.

(\romannumeral2) Similar arguments as in (\romannumeral1) work.
\end{proof}

\begin{remark}
The decomposition of strong castlings and the composition of weak castlings may appear in a stronger form as in Remark \ref{remark_decomposition_chain}, with a (de)composition-chain.
\end{remark}

\begin{lemma} \label{lem_skew_uniqueness}
(\romannumeral1) Suppose that $u,\widetilde{v},v_1,v_2,\widetilde{u_1},\widetilde{u_2}$ are elements in $S$ such that $\uuline{u}v_i\rlh \widetilde{v}\uuline{\widetilde{u_i}}$ $(i=1,2)$. Then $v_1=v_2$ and $\widetilde{u_1}=\widetilde{u_2}$.

(\romannumeral2) Suppose that $v,\widetilde{u},u_1,u_2,\widetilde{v_1},\widetilde{v_2},w$ are elements in $S$ such that $\uuline{u_i}v \rlh \widetilde{v_i}\uuline{\widetilde{u}}$ $(i=1,2)$ and $u_1v, u_2v \ddagger w$. Then $u_1=u_2$ and $\widetilde{v_1}=\widetilde{v_2}$.
\end{lemma}

\begin{proof}
(\romannumeral1) We use induction on $\ind(u)$ and $\ind(\widetilde{v})$. For $\ind(u)=0$ or $\ind(\widetilde{v})=0$, the proof is trivial. For $\ind(u)=\ind(\widetilde{v})=1$, if $u\neq \widetilde{v}$, then $\uuuline{u}v_i\rlh \widetilde{v}\uuuline{\widetilde{u_i}}$ and the expected results follows. Otherwise one has $u=\widetilde{v}=p$ for some $p\in \cP$, and then $v_i=\widetilde{u_i}=p$ $(i=1,2)$.

Now we suppose that the lemma has been proved for $\ind(u)\leq m-1$ and $\ind(\widetilde{v})\leq n$ for some $m\geq 2, n\geq 1$. Now we deal with the case $\ind(u)=m$, $\ind(v)=n$. Write $u=xy$, where $x,y\neq 1$. Then there are elements $\widehat{v_1},\widehat{v_2},\widehat{x_1},\widehat{x_2},\widehat{y_1},\widehat{y_2}\in S$ such that
\begin{align*}
&\uuline{u}v_1= \uuline{xy}v_1 \rlh \uuline{x}\widehat{v_1}\uuline{\widehat{y_1}}\rlh \widetilde{v} \uuline{\widehat{x_1}\widehat{y_1}}=\widetilde{v} \uuline{\widetilde{u_1}},\\
&\uuline{u}v_2= \uuline{xy}v_2 \rlh \uuline{x}\widehat{v_2}\uuline{\widehat{y_2}}\rlh \widetilde{v} \uuline{\widehat{x_2}\widehat{y_2}}=\widetilde{v} \uuline{\widetilde{u_2}}.
\end{align*}
Combining $\uuline{x}\widehat{v_1}\rlh \widetilde{v} \uuline{\widehat{x_1}}$, $\uuline{x}\widehat{v_2}\rlh \widetilde{v} \uuline{\widehat{x_2}}$ and the inductive hypothesis, one arrives at the conclusion that $\widehat{v_1}=\widehat{v_2}$ and $\widehat{x_1}=\widehat{x_2}$. It follows that $\uuline{y}v_1 \rlh \widehat{v_1}\uuline{\widehat{y_1}}$ and $\uuline{y}v_2 \rlh \widehat{v_1}\uuline{\widehat{y_2}}$. By inductive hypothesis again, one obtains $v_1=v_2$ and $\widehat{y_1}=\widehat{y_2}$. Now $\widetilde{u_1}=\widetilde{u_2}$.

Assume that the lemma has been proved for $\ind(u)\leq m$ and $\ind(\widetilde{v})\leq n-1$ for some $m\geq 1, n\geq 2$. We exchange the role of $u,v_1,v_2$ and $\widetilde{v},\widetilde{u_1},\widetilde{u_2}$, respectively, and consider the strong castlings $\uuline{\widetilde{v}}\widetilde{u_i}\rlh u\uuline{v_i}$ $(i=1,2)$. Similar arguments as previous ensure that the lemma also holds for the case $\ind(u)=m$, $\ind(v)=n$. The proof is completed.

(\romannumeral2) Similar argument also work in this case.
\end{proof}

\subsection{Uniqueness of Prime Powers}
\label{subsection_uniqueness_prime_powers}

The next lemma is an important consequence of Axiom \uppercase\expandafter{\romannumeral4}.

\begin{lemma} [Uniqueness of prime powers] \label{lem_uniqueness_prime_power}
If $k\geq 1$ and $p,q_1,q_2,\ldots,q_k$ are elements in $\cP$ such that $p^k=q_1q_2\ldots q_k$, then $q_1=q_2=\ldots=q_k=p$. In particular, we have $\tau(p^k)=k+1$.
\end{lemma}

\begin{proof}
For $k=1$, we of course have $q_1=p$. For $k=2$, assume on the contrary that $p\neq q_1$. Then $\gcd(p,q_1)=1$ and $\lcm[p,q_1]| p^2=q_1q_2$. By Corollary \ref{lem_gcd_lcm_u_v}, one has  $\ind(\lcm[p,q_1])=\ind(p)+\ind(q_1)=2=\ind(p^2)$. So $\lcm[p,q]=p^2=q_1q_2$. Then $\uuuline{p}p\rlh q_1\uuuline{q_2}$. But we also have $\uline{p}p\rlh p\uline{p}$, which contradicts Axiom \uppercase\expandafter{\romannumeral4}.

Assume that the lemma has been proved for $k\leq K-1$ with some $K\geq 3$. Now we consider the case $k=K$.

\textsc{Case 1}. Suppose that $\gcd(q_1,p^{K-1})\neq 1$. Then $q_1|p^{K-1}$. Write $p^{K-1}=q_1q_2^\prime \ldots q_{K-1}^\prime$ for some $q_2^\prime, \ldots , q_{K-1}^\prime\in \cP$. By inductive hypothesis, we have $q_1=p$. Now $q_2q_3\ldots q_k=p^{K-1}$. By inductive hypothesis again, one gets that $q_2=q_3=\ldots =q_k = p$.

\textsc{Case 2}. Suppose that $\gcd(q_1,p^{K-1})= 1$. Note that $q_1|p^{K-1}\cdot p$. By Lemma \ref{lem_divisor_castling}, we have $\uuuline{q_1}z\rlh p^{K-1}\uuuline{p}$, where $z=q_1^{-1} \cdot p^{K-1}\cdot p = q_2\ldots q_K$. By Lemma \ref{lem_C1_subseteq_C_0}, one deduces that
\[
\uuuline{q_1}q_2(q_3\ldots q_K) \rlh p_2\uuuline{p_1}(q_3\ldots q_K) \rlh p_2 w \uuuline{p} = p^{K-1}\uuuline{p}
\]
for some $p_1,p_2\in \cP$ and $w\in S$. Now we have $p_2w=p^{K-1}$. By inductive hypothesis, one gets $p_2=p$. It follows from $p_2p_1q_3\ldots q_k= p^{K}$ that $p_1q_3\ldots q_k=p^{K-1}$. By inductive hypothesis again, we have $p_1=q_3=\ldots = q_k = p$. Now $q_1q_2=p_2p_1=p^2$, which implies $q_1=q_2=p$.

The divisors of $p^k$ are exactly $1,p,p^2,\ldots, p^k$. This completes the proof.
\end{proof}

Now the following corollary follows immediately.

\begin{corollary} \label{cor_p_k_q_l_coprime}
(\romannumeral1) Let $k,l\geq 0$ and $p,q$ be distinct elements in $\cP$. Then $\gcd(p^k,q^l)=1$. (\romannumeral1) Let $k,l\geq 0$, $w\in S$ and $p,q$ be distinct elements in $\cP$. Suppose that $p^k\ddagger w$ and $q^l \ddagger w$. Then $\gcd_\ddagger(w;p^k,q^l)=1$.
\end{corollary}

\begin{corollary} \label{cor_gcd_p_m}
(\romannumeral1) Let $u\in S$ and $p\in \cP$. Suppose that $\gcd(u,p)=1$, then $\gcd(u,p^m)=1$ for any $m\geq 0$. (\romannumeral2) Let $m\geq 1$, $u,w\in S$ and $p\in \cP$. Suppose that $p^m\ddagger w$, $u\ddagger w$ and $\gcd_\ddagger(w;u,p)=1$, then $\gcd_\ddagger(w;u,p^m)=1$.
\end{corollary}

\begin{corollary} \label{cor_tau_lcm_p_m}
(\romannumeral1) Let $k\geq 1$ and $q_1,q_2,\ldots,q_k$ be distinct elements in $\cP$. Let $m_1,m_2,\ldots,m_k\geq 0$. Then
\begin{equation} \label{eq_tau_p_m_(m+1)}
\tau\left(\lcm[q_1^{m_1},q_2^{m_2},\ldots, q_k^{m_k}]\right)= (m_1+1)(m_2+1)\ldots (m_k+1).
\end{equation}
Moreover, the corresponding divisors are exactly
\[
\lcm[q_1^{l_1},q_2^{l_2},\ldots, q_k^{l_k}],\quad (0\leq l_j \leq m_j,\quad 1\leq j\leq k).
\]

(\romannumeral2) Let $k\geq 1$, $m_1,m_2,\ldots,m_k\geq 0$, $w\in S$, and $q_1,q_2,\ldots,q_k$ be distinct elements in $\cP$. Suppose that $q_j^{m_j}\ddagger w$ for $1\leq j\leq k$. Then
\[
\tau\left(\lcm_\ddagger[w;q_1^{m_1},q_2^{m_2},\ldots, q_k^{m_k}]\right)= (m_1+1)(m_2+1)\ldots (m_k+1).
\]
Moreover, the corresponding co-divisors are exactly
\[
\lcm_\ddagger[w;q_1^{l_1},q_2^{l_2},\ldots, q_k^{l_k}],\quad (0\leq l_j \leq m_j,\quad 1\leq j\leq k).
\]
\end{corollary}

\begin{proof}
(\romannumeral1) The equation \eqref{eq_tau_p_m_(m+1)} follows by combining Corollaries \ref{cor_tau_lcm_coprime_u_j}, \ref{cor_p_k_q_l_coprime} and Lemma \ref{lem_uniqueness_prime_power}. It is sufficient to prove that $\lcm[q_1^{l_1},q_2^{l_2},\ldots, q_k^{l_k}]\neq \lcm[q_1^{l_1^\prime},q_2^{l_2^\prime},\ldots, q_k^{l_k^\prime}]$ whenever $(l_1,l_2,\ldots,l_k)\neq (l_1^\prime,l_2^\prime,\ldots,l_k^\prime)$. Without loss of generality, we assume on the contrary that
\[
v=\lcm[q_1^{l_1},q_2^{l_2},\ldots, q_k^{l_k}]= \lcm[q_1^{l_1^\prime},q_2^{l_2^\prime},\ldots, q_k^{l_k^\prime}]
\]
and $l_1< l_1^\prime$. Then $q_1^{l_1^\prime}|v$. One deduces that
\[
v= \lcm\left[q_1^{l_1}, \lcm[q_2^{l_2},\ldots,q_k^{l_k}]\right] = \lcm\left[q_1^{l_1^\prime}, \lcm[q_2^{l_2},\ldots,q_k^{l_k}]\right].
\]
By corollaries \ref{cor_pairwise_gcd} and \ref{cor_p_k_q_l_coprime}, we have
\[
\gcd(q_1^{l_1},\lcm[q_2^{l_2},\ldots,q_k^{l_k}])=\gcd(q_1^{l_1^\prime},\lcm[q_2^{l_2},\ldots,q_k^{l_k}])=1. \]
It follows from Lemma \ref{lem_lcm_unique_from_eta_2=id} that $q_1^{l_1}=q_1^{l_1^\prime}$, which is a contradiction.

(\romannumeral2) The conclusion follows from similar arguments as above.
\end{proof}

\begin{corollary} \label{cor_prime_keeping_in_lcm}
(\romannumeral1) Let $k\geq 1$. Suppose that $q_1,\ldots,q_k$ are distinct primes and $m_1,\ldots,m_k\geq 1$. Let $u=\lcm[q_1^{m_1},\ldots,q_k^{m_1}]$. If $p\in \cP$ and $m\geq 1$ satisfy that $p^m|u$, then $p=p_j$ for some $1\leq j\leq k$ and $m\leq m_j$.

(\romannumeral2) Let $k\geq 1$ and $w\in S$. Suppose that $m_1,\ldots,m_k\geq 1$ and $q_1,\ldots,q_k$ are distinct primes satisfying $q_j^{m_j}\ddagger w$ $(1\leq j\leq k)$. Let $u=\lcm_\ddagger[w;q_1^{m_1},\ldots,q_k^{m_1}]$. If $p\in \cP$ and $m\geq 1$ satisfy that $p^m\ddagger u$, then $p=p_j$ for some $1\leq j\leq k$ and $m\leq m_j$.
\end{corollary}

\begin{proof}
We prove (\romannumeral2) below. One can prove (\romannumeral1) with similar arguments.

Since $p\ddagger u$, it follows from Lemma \ref{lem_p_divides_lcm_implies_p_divides_one_of_them} that $p\ddagger q_j^{m_j}$ for some $1\leq j\leq k$. By Lemma \ref{lem_uniqueness_prime_power}, we have $p=q_j$. Write $v=\lcm_\ddagger[w;q_1^{m_1},\ldots,q_{j-1}^{m_{j-1}}, q_j^m, q_{j+1}^{m_{j+1}},\ldots,q_k^{m_k}]$. Since $q_j^m\ddagger u$, we have $v\ddagger u$. By Corollary \ref{cor_tau_lcm_p_m}, one deduces that
\[
(m_1+1)\ldots (m_{j-1}+1)(m+1)(m_{j+1}+1)\ldots(m_k+1)=\tau(v)\leq \tau(u)= (m_1+1)\ldots (m_k+1).
\]
Thus, we have $m\leq m_j$.
\end{proof}

\begin{corollary} \label{cor_gcd_lcm_of_two_lcms}
(\romannumeral1) Let $k\geq 1$ and $q_1,\ldots, q_k\in \cP$. Let $n_1,\ldots,n_k,m_1,\ldots,m_k$ be non-negative integers. Let $u=\lcm[q_1^{n_1},\ldots, q_k^{n_k}]$ and $v=\lcm[q_1^{m_1},\ldots, q_k^{m_k}]$. Then
\begin{align*}
&\lcm[u,v] = \lcm\left[q_1^{\max\{n_1,m_1\}},\ldots, q_k^{\max\{n_k,m_k\}}\right],\\ &\gcd(u,v)= \lcm\left[q_1^{\min\{n_1,m_1\}},\ldots, q_k^{\min\{n_k,m_k\}}\right].
\end{align*}

(\romannumeral2) Let $k\geq 1$, $w\in S$ and $q_1,\ldots, q_k\in \cP$. Let $n_1,\ldots,n_k,m_1,\ldots,m_k$ be non-negative integers. Let $u=\lcm[w;q_1^{n_1},\ldots, q_k^{n_k}]$ and $v=\lcm[w;q_1^{m_1},\ldots, q_k^{m_k}]$. Then
\begin{align*}
&\lcm_\ddagger[w;u,v] = \lcm_\ddagger\left[w;q_1^{\max\{n_1,m_1\}},\ldots, q_k^{\max\{n_k,m_k\}}\right],\\
&\gcd_\ddagger(w;u,v)= \lcm_\ddagger\left[w;q_1^{\min\{n_1,m_1\}},\ldots, q_k^{\min\{n_k,m_k\}}\right].
\end{align*}
\end{corollary}

With the unique irreducible decomposition of prime powers, we are able to count the multiplicities of prime divisors, or co-divisors, of $u$. Let $\PDM(u)$ and $\PDM_\ddagger(u)$ be the multi-set of prime divisors and prime co-divisors of $u$, respectively. Define
\begin{align*}
&\Omega(u)=\#PDM(u)=\sum\nolimits_{p\in \cP} \max\{k\geq 0:\, p^k|u\}, \\
&\Omega_\ddagger(u)=\#PDM_\ddagger(u)=\sum\nolimits_{p\in \cP} \max\{k\geq 0:\, p^k \ddagger u\}.
\end{align*}
Also write $\lambda(u)=(-1)^{\Omega(u)}$, which is known as Liouville function in classical arithmetics. We also define $\lambda_\ddagger(u) = (-1)^{\Omega_\ddagger(u)}$. Then
\begin{align*}
&\Omega\left(\lcm[q_1^{m_1},q_2^{m_2},\ldots, q_k^{m_k}]\right)= m_1+m_2+\ldots m_k,\\
&\Omega_\ddagger\left(\lcm[w;q_1^{m_1},q_2^{m_2},\ldots, q_k^{m_k}]\right)= m_1+m_2+\ldots m_k,
\end{align*}
where $w\in S$, $q_1,\ldots,q_k$ are distinct primes and $m_1,\ldots,m_k\geq 0$.

\medskip

The following lemma will be applied in later sections.

\begin{lemma} \label{lem_q_m_r_requires_q=r}
Let $q\in \cP$, $w,z\in S$ and $m,n\geq 1$.

(\romannumeral1) Suppose that $\uline{q^m}w\rlh q^n\uline{z}$. Then $w=q^n$ and $z=q^m$.

(\romannumeral2) Suppose that $\uline{w} q^m\rlh z\uline{q^n}$. Then $w=q^n$ and $z=q^m$.
\end{lemma}

\begin{proof}
(\romannumeral1) For $m=n=1$, we have $\ind(w)=\ind(z)=1$ and $qw=qz$. So $w=z$ and $\uline{q}w\rlh q\uline{w}$. By the construction of $\fC$, the only possibility is $w=q$ and $\uline{q}q\rlh q\uline{q}$.

Suppose that the lemma has been proved for $m+n\leq K-1$ with some $K\geq 3$. Now we consider the case $m+n=K$. Note that $\gcd(q^m,q^n)\neq 1$. So $q^m,w$ are not castled-free. By the constriction of $\fC$, the fact that $(q^m,w)\in \fC$ comes from either \eqref{eq_C_castling_composition_1} or \eqref{eq_C_castling_composition_2}.

For the former case, there are some $u_1,u_2\neq 1$ with $u_1u_2=q^m$ and $\widehat{w},\widehat{w_1},\widehat{w_2}\in S$ such that
\[
\uline{q^m}w = \uline{u_1u_2} w \rlh \uline{u_1} \widehat{w} \uline{\widehat{u_2}} \rlh q^n\uline{\widehat{u_1}\widehat{u_2}} = q^n\uline{z}.
\]
Suppose that $\ind(u_1)=l$ and $\ind(u_2)=m-l$, where $0<l<m$. Since $u_1u_2=q^m$, one deduces by Lemma \ref{lem_uniqueness_prime_power} that $u_1=q^l$ and $u_2=q^{m-l}$. By inductive hypothesis, we deduce from $q^l\widehat{w}\rlh q^n\uline{\widehat{u_1}}$ that $\widehat{w} = q^n$ and $\widehat{u_1}=q^l$. Then the fact $\uline{q^{m-l}}w\rlh q^n\uline{\widehat{u_2}}$ results in $w=q^n$ and $\widehat{u_2}=q^{m-l}$. Now $z=\widehat{u_1}\widehat{u_2}=q^m$.

For the latter case, there are some $w_1,w_2\neq 1$ with $w_1w_2=w$ and $\widehat{w_1},\widehat{w_2},y\in S$ such that
\[
\uline{q^m}w = \uline{q^m}w_1w_2 \rlh \widehat{w_1}\uline{y} w_2 \rlh \widehat{w_1}\widehat{w_2} \uline{z} \rlh q^n \uline{z}.
\]
Suppose that $\ind(w_i)=l_i$ $(i=1,2)$, where $l_1+l_2=n$. Since $q^n=\widehat{w_1}\widehat{w_2}$, one has $\widehat{w_1}=q^{l_1},\, \widehat{w_2}=q^{l_2}$. By inductive hypothesis, it follows from $\uline{q^m}w_1\rlh q^{l_1}\uline{y}$ that  $w_1=q^{l_1}$ and $y=q^m$. And the condition $\uline{q^m}w_2\rlh q^{l_2}\uline{z}$ implies that $w_2=q^{l_2}$ and $z=q^m$. Now $w=w_1w_2= q^n$.
The conclusion follows by induction.

(\romannumeral2) Similar arguments as above also work.
\end{proof}

\subsection{Multiplicative and Completely Multiplicative Functions}

In a strong castling $\uuline{u} v \rightleftharpoons \widetilde{v} \uuline{\widetilde{u}}$, the elements $u$ and $\widetilde{u}$ share ``same'' information. Therefore, it is natural to define the following.

\begin{definition}
We say that a (complex-valued) arithmetic function $f$ on $S$ is castled-invariant, if $\uuline{u} v \rightleftharpoons \widetilde{v} \uuline{\widetilde{u}}$ implies $f(u)=f(\widetilde{u})$.
\end{definition}

Note that the condition $\uuline{u} v \rightleftharpoons \widetilde{v} \uuline{\widetilde{u}}$ is same with $\uuline{\widetilde{v}}\widetilde{u}\rlh u\uuline{v}$. When $f$ is castled-invariant, we have $f(v)=f(\widetilde{v})$ as well. When $S$ is commutative, all weak castlings have the form $\uline{u}v \rlh v\uline{u}$. All arithmetic functions are castled-invariant in this case.

\begin{definition}
We say that an arithmetic function $f$ on $S$ is multiplicative, if $f$ is castled-invariant and $f(uv)=f(u)f(v)$ for each pair of castled-free elements $u,v\in S$. We say that $f$ is completely multiplicative, if $f$ is castled-invariant and $\uuline{u} v \rightleftharpoons \widetilde{v} \uuline{\widetilde{u}}$ implies $f(uv)=f(u)f(v)$ $(u,v\in S)$.
\end{definition}

Note that $\uuuline{1}u\rlh u\uuuline{1}$ for any $u\in S$. For a non-zero multiplicative function $f$, we have $f(u)=f(1\cdot u)=f(1)f(u)$ for all $u\in S$. It follows that $f(1)=1$. As a result, a multiplicative function $f$ is invertible with respect to convolution.

A completely multiplicative function on $\bN$ is determined by its value on all the primes. In particular, the group $\bQ^+$ is generated by all these primes. For Thompson's monoid $\bS$, it follows from $\uuline{p_i}p_0\rlh p_0\uuline{p_{i+1}}$ $(i\geq 1)$ that $f(p_i)=f(p_1)$ for a castled-invariant function $f$ and all $i\geq 1$. So a completely multiplicative function on $\bS$ is determined by its value on $p_0$ and $p_1$. Indeed, the group $\bG$ can be generated by these two elements. Completely multiplicative functions on $S$ characterize the structure of underling group $G$.

\begin{lemma} \label{lem_1-1_strong_castling}
Suppose that $\uuline{u} v \rightleftharpoons \widetilde{v} \uuline{\widetilde{u}}$. Let $\cA=\{(u_1,u_2)\in S\times S:\, u_1u_2=u\}$ and $\cB=\{(\widetilde{u_1},\widetilde{u_2})\in S\times S:\, \widetilde{u_1}\widetilde{u_2}=\widetilde{u}\}$. Then $\rho:\,\cA\rightarrow \cB,\, (u_1,u_2)\rightarrow (\widehat{u_1},\widehat{u_2})$ is a bijection, where
\begin{equation} \label{eq_bijection_u1u2_u1u2tilede}
\uuline{u}v = \uuline{u_1u_2}v \rlh \uuline{u_1}\widehat{v}\uuline{\widehat{u_2}} \rlh \widetilde{v}\uuline{\widehat{u_1}\widehat{u_2}} = \widetilde{v} \uuline{\widetilde{u}}.
\end{equation}
Moreover, there is a one-to-one correspondence between $\{r\in \cP:\, r|\widetilde{u}\}$ and $\{q\in \cP:\, q|u\}$. Furthermore, there is also a one-to-one correspondence between $\{r\in \cP:\, r\ddagger\widetilde{u}\}$ and $\{q\in \cP:\, q\ddagger u\}$. In particular, we have
\[
\tau(u)=\tau(\widetilde{u}),\quad \omega(u)=\omega(\widetilde{u}),\quad \omega_\ddagger(u)=\omega(\widetilde{u}).
\]
\end{lemma}

\begin{proof}
For any $(u_1,u_2)\in \cA$, the decomposition of strong castlings gives
\[
\uuline{u_2}v \rlh \widehat{v} \uuline{\widehat{u_2}},\quad \uuline{u_1}\widehat{v} \rlh \widetilde{v}\uuline{\widehat{u_1}}
\]
for some $(\widehat{u_1},\widehat{u_2})\in \cB$ and $\widehat{v}\in S$. Conversely, for any $(\widehat{u_1},\widehat{u_2})\in \cB$, one has
\[
\widetilde{v} \uuline{\widetilde{u_1}} \rlh \uuline{u_1} \widehat{v},\quad \widehat{v}\uuline{\widehat{u_2}}\rlh \uuline{u_2} \widetilde{v}
\]
for some $(u_1,u_2)\in \cA$ and $\widehat{v}\in S$. So $\rho$ is a bijection. Noting that $\ind(u_1)=\ind(\widehat{u_1})$ and $\ind(u_2)=\ind(\widehat{u_2})$, one can also verify the remaining two correspondences.
\end{proof}

\begin{theorem} \label{thm_castled_invariant_convolution}
(\romannumeral1) Let $f,g$ be two castled-invariant functions on $S$. Then so is $f\ast g$.

(\romannumeral2) Suppose that $f$ is a castled-invariant function on $S$. Also suppose that $f(1)\neq 0$ and $h$ is its inverse with respect to convolution. The $h$ is also castled-invariant.
\end{theorem}

\begin{proof}
Let $u,v,\widetilde{u},\widetilde{v}$ be elements in $S$ satisfying $\uuline{u} v \rightleftharpoons \widetilde{v} \uuline{\widetilde{u}}$. Recall the bijection given in Lemma \ref{lem_1-1_strong_castling}, and $u_1,u_2,\widetilde{u_1},\widetilde{u_2}$ given in \eqref{eq_bijection_u1u2_u1u2tilede}.

(\romannumeral1) Since $f,g$ are castled-invariant, we have $f(u_1)=f(\widehat{u_1})$ and $g(u_2)=g(\widehat{u_2})$. It follows that
\[
(f\ast g)(u) = \sum\limits_{u_1u_2=u} f(u_1)g(u_2) = \sum\limits_{\widehat{u_1} \widehat{u_2} = \widetilde{u}} f(\widehat{u_1})g(\widehat{u_2}) = (f\ast g)(\widetilde{u}).
\]
Therefore $f\ast g$ is castled-invariant.

(\romannumeral2) By Theorem \ref{thm_inverse_formula_wrt_convolution}, the function $h$ is given by the iterating formulae
\[
h(1)=f(1)^{-1},\quad h(u) = - f(1)^{-1}\sum\limits_{u_1u_2=u\atop u_1\neq u} h(u_1)f(u_2).
\]
We use induction on $\ind(u)$. For $u=1$, the proof is trivial. Suppose that it has been verified in all strong castlings $\uuline{u_0}v_0\rlh \widetilde{v_0}\uuline{\widetilde{u_0}}$ with $\ind(u_0)\leq k-1$ for some $k\geq 1$. Consider $\uuline{u}v\rlh \widetilde{v}\uuline{\widetilde{u}}$ with $\ind(u)=k$. Since $u_1\neq u$, one has $\ind(u_1)\leq k-1$. Applying inductive hypothesis on $\uuline{u_1}\widehat{v}\rlh \widetilde{v}\uuline{\widehat{u_1}}$, we obtain $h(u_1)=h(\widehat{u_1})$. Moreover, it follows from the fact that $f$ is castled-invariant and $\uuline{u_2}v\rlh\widehat{v}\uuline{\widehat{v_2}}$ that $f(u_2)=f(\widehat{u_2})$. Now we have
\[
h(u) = - f(1)^{-1}\sum\limits_{u_1u_2=u\atop u_1\neq u} h(u_1)f(u_2)
= - f(1)^{-1}\sum\limits_{\widehat{u_1}\widehat{u_2}=\widetilde{u}\atop \widehat{u_1}\neq \widetilde{u}} h(\widehat{u_1})f(\widehat{u_2}) = h(\widetilde{u}).
\]
Hence $h$ is castled-invariant.
\end{proof}

\begin{theorem} \label{thm_pre-mulitplicative_convolution}
Suppose that $f$ and $g$ are two multiplicative functions on $S$. Then so is $f\ast g$.
\end{theorem}

\begin{proof}
By Theorem \ref{thm_castled_invariant_convolution}, the function $f\ast g$ is castled-invariant. Let $u,v$ be castled-free elements with $\uuuline{u} v \rightleftharpoons \widetilde{v}\uuuline{\widetilde{u}}$. By Lemma \ref{lem_divisor_decomposition_castlable_uv}, there is a one-to-one correspondence between a pair $(d,v_1)$ of divisors of $u$,$v$ and a divisor $w$ of $uv$ by $(d,v_1) \, \mapsto\, w$, where $u=du_1$, $v=v_1v_2$, $w=dw_1$ and $\uuuline{u_1}v_1 \rightleftharpoons w_1\uuuline{z}$ for some $z\in S$. Since $f,g$ are castled-invariant, we have $f(w_1)=f(v_1)$ and $g(u_1)=g(z)$.
It follows from Lemma \ref{lem_divisor_decomposition_castlable_uv} that $d,w_1$ are castled-free, and so are $z,v_2$. Since $f$ is multiplicative, one deduces that
\[
f(w)=f(dw_1)=f(d)f(w_1) = f(d)f(v_1).
\]
Similarly, we have
\[
g(w^{-1}uv)= g(w_1^{-1}u_1v_2)=g(z v_2)= g(z)g(v_2)=g(u_1)g(v_2).
\]

Now we have
\begin{align*}
(f\ast g)(uv) &= \sum\limits_{w|uv}f(w)g(w^{-1}uv) =  \sum\limits_{du_1=u} \sum\limits_{v_1v_2=v}
f(d)f(v_1)g(u_1)g(v_2) \\
&= \sum\limits_{du_1=u}f(d)g(u_1)\sum\limits_{v_1v_2=v} f(v_1)g(v_2) = (f\ast g)(u)(f\ast g)(v).
\end{align*}
The theorem follows.
\end{proof}

The function $1$ is definitely completely multiplicative. Since $\tau=1\ast 1$, it follows from the Theorem \ref{thm_pre-mulitplicative_convolution} that $\tau$ is multiplicative.

\begin{proposition} \label{prop_multiplicative_fh_implies_g}
Let $f,g$ be castled-invariant functions on $S$. Let $h=f\ast g$.

(\romannumeral1) Suppose that $g,h$ are multiplicative. Then so is $f$.

(\romannumeral2) Suppose that $f,h$ are multiplicative. Then so is $g$.
\end{proposition}

\begin{proof}
(\romannumeral1) Suppose on the contrary that $f$ is not multiplicative. Then there exists some pair of castled-free elements $u,v$ such that $f(uv)\neq f(u)f(v)$. We can choose above $u,v$ such that $\ind(u)+\ind(v)=k$ attains minimum among all such pairs. That is to say, $f(dw_1)=f(d)f(w_1)$ for all castled-free pair of elements $d,w_1$ with $\ind(d)+\ind(w_1)<k$.

Suppose that $\uuuline{u} v \rightleftharpoons \widetilde{v}\uuuline{\widetilde{u}}$. By the proof of Theorem \ref{thm_pre-mulitplicative_convolution}, we have already shown that each divisor $w$ of $uv$ gives
\[
w=dw_1,\quad u=du_1,\quad v=v_1v_2,\quad \uuuline{w_1} z \rightleftharpoons u_1 \uuuline{v_1}
\]
for some $z\in S$, and $f(w_1)=f(v_1)$, $g(u_1)=g(z)$. Moreover, the elements $d,w_1$ are castled-free, and so are $z,v_2$. Now
\begin{align*}
h(uv) &= \sum\limits_{w|uv} f(w)g(w^{-1}uv) =  \sum\limits_{du_1=u\atop v_1v_2=v}
f(dw_1)g(zv_2) \\
&= \sum\limits_{{du_1=u\atop v_1v_2=v}\atop \ind(d)+\ind(v_1)<k}
f(d)f(w_1)g(z)g(v_2) + f(uv)g(1)\\
&= \sum\limits_{du_1=u}f(d)g(u_1)\sum\limits_{v_1v_2=v} f(v_1)g(v_2) -f(u)f(v)+f(uv) \\
&= h(u)h(v)+f(uv)-f(u)f(v).
\end{align*}
Since $f(uv)\neq f(u)f(v)$, one obtains that $h(uv)\neq h(u)h(v)$, which is a contradiction.

(\romannumeral2) The proof is similar to that in (\romannumeral1) and we omit it here.
\end{proof}

It is easy to verify that the function $\delta_1$ is multiplicative. Then we obtain the following corollary immediately.

\begin{corollary} \label{cor_multiplicative_inverse_function}
Suppose that $f$ is a multiplicative function on $S$ and $g$ is its inverse with respect to convolution. Then $g$ is also multiplicative. In particular, the M\"{o}bius function $\mu$ is multiplicative.
\end{corollary}

\section{Axiom \uppercase\expandafter{\romannumeral5} and Natural Monoids}
\label{section_natural_monoid}

The previous axioms preserve prime divisors in castlings. However, prime powers are not maintained. For example, suppose that $u=p^2x=pwq=rzq$ for $p,x,w,q,r,z$ distinct elements in $\cP$ and $u$ has no other irreducible decompositions. We have $\uuuline{p^2}x\rlh r\uuuline{zq}$, while $\Omega(p^2)=2$ and $\Omega(zq)=1$. The function $\Omega$ is not castled-invariant in this example. For arithmetic interests, we add the following axiom to exclude such situations.

\textsc{Axiom \uppercase\expandafter{\romannumeral5}.} Let $p,q\in\cP$ and $k,l\geq 0$. If $p^k,q^l$ are weakly castlable, then $\uline{p^k}q^l \rlh r^l \uline{t^k}$ for some $r,t\in \cP$.

\begin{definition}
We say that $S$ is a natural monoid with $G$ its rational group, when Axioms \uppercase\expandafter{\romannumeral1}-\uppercase\expandafter{\romannumeral5} all hold.
\end{definition}

In the statement of Axiom \uppercase\expandafter{\romannumeral5}, the primes $r,t$ may depend on $p,q,k,l$.  This axiom may be called ``power-preserving property''. In this section, the monoid $S$ is assumed to be natural unless we point it out explicitly.

\subsection{Castling of Prime Powers}

\label{subsection_Liouville}

Our main purpose of this subsection is to obtain the following theorem.

\begin{theorem} \label{thm_Liouville_function}
The functions $\lambda$ and $\lambda_\ddagger$ are completely multiplicative.
\end{theorem}

We will turn back to prove Theorem \ref{thm_Liouville_function} after exploring Axiom \uppercase\expandafter{\romannumeral5} and proving some lemmas.

\begin{lemma} \label{lem_power_preserving_strong}
(\romannumeral1) Let $k\geq 1$, $p\in \cP$ and $u\in S$. Suppose that $p^k,u$ are strongly castlable, then $\uuline{p^k}u \rlh v \uline{q^k}$ for some $v\in S$ and $q\in \cP$.

(\romannumeral2) Let $k\geq 1$, $q\in \cP$ and $v\in S$. Suppose that $v,q^k$ are strongly castlable, then $\uuline{v}q^k \rlh p^k\uline{u}$ for some $u\in S$ and $p\in \cP$.
\end{lemma}

\begin{proof}
(\romannumeral1) We use induction on $\ind(u)$. For $\ind(u)=0$, the proof is trivial. For $\ind(u)=1$, the result follows from Axiom \uppercase\expandafter{\romannumeral5}. Assume that the lemma has been proved for $\ind(u)\leq m-1$ with some $m\geq 2$. Now we consider the case $\ind(u)=m$. Write $u=u_1u_2$ with for some $u_1,u_2\in S$ with $u_1,u_2\neq 1$. Noting that $p^k,u$ are strongly castlable, we have that
\[
p^k\uuline{u} = p^k  \uuline{u_1u_2} \rlh \uline{\widetilde{u_1}}a \uuline{u_2} \rlh \uline{\widetilde{u_1}\widetilde{u_2}}b
\]
for some $\widetilde{u_1},\widetilde{u_2},a,b\in \cP$. By inductive hypothesis, we have $a=r^k$ for some $r\in \cP$, and then $b=q^k$ for some $q\in \cP$.

(\romannumeral2) The proof is similar as in (\romannumeral1).
\end{proof}

Axiom \uppercase\expandafter{\romannumeral5} states that if $p^k$, $q^l$ are weakly castlable, then $\uline{p^k}q^l \rlh r^l \uline{t^k}$ for some $r,t\in \cP$. In this statement, the prime $r,t$ may depend on $p,q,k,l$. The following lemma gives a stronger result.

\begin{lemma} \label{lem_for_amenable_proof}
Let $k\geq 1$ be given and $p,q\in \cP$ be two different primes such that $p^k,q$ are weakly castlable. Then for any $l\geq 1$, there are some $r,t_l\in \cP$ such that $\uuuline{p^k}q^l\rlh r^l \uuuline{t_l^k}$. Here, besides on $p,q$, the primes $r,t_l$ may also depend on $k$, while $r$ is independent of $l$.
\end{lemma}

\begin{proof}
By Axiom \uppercase\expandafter{\romannumeral5}, there is some $r,s\in \cP$ such that $\uline{p^k}q\rlh r\uline{s^k}$. Since $p\neq q$, one deduces by Lemma \ref{lem_q_m_r_requires_q=r} that $r\neq p$. Hence $\gcd(p^k,r^l)=1$ for any $l\geq 1$. Suppose that $\uuuline{p^k}y_l\rlh r^l\uuuline{x_l}$ for some $y_l,x_l\in S$. By Lemma \ref{lem_prime_power_castling}, we have $y_l=q_l^l$ for some $q_l\in \cP$, and $x_l=t_l^k$ for some $t_l\in \cP$. Now a decomposition of free castling leads to
\[
\uuuline{p^k}q_lq_l^{l-1} \rlh \widehat{q_l} \uuuline{z_l} q_l^{l-1} \rlh \widehat{q_l} w_l \uuuline{x_l} = r^l \uuuline{x_l}
\]
for some $\widehat{q_l}\in \cP$ and $z_l,w_l\in S$. Since $\widehat{q_l}w_l= r^l$, one has $\widehat{q_l}=r$ and $w_l=r^{l-1}$. Now both $\uuuline{p^k}q_l\rlh r\uuuline{z_l}$ and $\uuuline{p^k}q\rlh r\uuuline{s^k}$ hold. Hence $q_l=q$. It follows that $\uuuline{p^k}q^l\rlh r^l \uuuline{t_l^k}$. The proof is completed.
\end{proof}

The following lemma strengthens Lemma \ref{lem_lcm_dagger_power=1}.

\begin{lemma} \label{lem_lcm_prime_powers_equivalent_lcm_ddagger_prime_powers}
Let $k\geq 1$, $m_1,\ldots,m_k\geq 1$ and $q_1,\ldots, q_k$ be distinct primes. Let $w=\lcm[q_1^{m_1},\ldots, q_k^{m_k}]$. Then there exist distinct primes $r_1,\ldots,r_k$ such that $w=\lcm_\ddagger[w;r_1^{m_1},\ldots,r_k^{m_k}]$. Here $r_1,\ldots,r_k$ may depend on $q_1,\ldots,q_k$ and $m_1,\ldots,m_k$.
\end{lemma}{

\begin{proof}
For $1\leq j\leq k$, let $v_j=\lcm[q_1^{m_1},\ldots,q_{j-1}^{m_{j-1}},q_{j+1}^{m_{j+1}},\ldots,q_k^{m_k}]$. Then $\gcd(q_j^{m_j},v_j)=1$. Suppose that $w=\lcm[q_j^{m_j},v_j]=q_j^{m_j}y_j=v_jx_j$, i.e., $\uuuline{q_j^{m_j}}y_j\rlh v_j\uuuline{x_j}$ for some $x_j,y_j\in S$. By Lemma \ref{lem_power_preserving_strong}, one has $x_j=r_j^{m_j}$ for some $r_j\in \cP$ $(1\leq j\leq k)$.

Next, we shall show that $r_1,r_2,\ldots,r_k$ are distinct. Assume on the contrary that $r_i=r_l$ for some $1\leq i\neq l\leq k$. For the case $m_i\neq m_l$, we suppose without loss of generality that $m_i> m_l$. Since $v_lr_l^{m_l}=v_ir_i^{m_i}$,  we have $v_l=v_ir_l^{m_i-m_l}$ and then $r_l\ddagger v_l$. Note that $v_l,r_l^{m_l}$ are castled-free. We deduce that $r_l$, $r_l^{m_l}$ are also castled-free, which is a contradiction. For the case $m_i=m_l$, we have $v_i=v_l$ and then $q_i^{m_i}=q_l^{m_l}$, which is also a contradiction. As a result, the primes $r_1,r_2,\ldots,r_k$ are distinct.

Combining the facts $\lcm_\ddagger[w;r_1^{m_1},\ldots,r_k^{m_k}]\ddagger w$, and
\[
\ind(w)= \ind(\lcm[q_1^{m_1},\ldots, q_k^{m_k}])=m_1+\ldots+m_k = \ind(\lcm_\ddagger[w;r_1^{m_1},\ldots,r_k^{m_k}]),
\]
we conclude that $w=\lcm_\ddagger[w;r_1^{m_1},\ldots,r_k^{m_k}]$.
\end{proof}

The following lemma establishes corresponding of prime powers in a strong castling.

\begin{lemma} \label{lem_prime_power_castling}
Let $\uuline{u}v\rlh \widetilde{v}\uuline{\widetilde{u}}$ and $\rho$ be the bijection given in Lemma \ref{lem_1-1_strong_castling}.

(\romannumeral1) Suppose that $q^l|u$, where $q\in \cP$ and $l\geq 0$. Then there is an $r\in \cP$ such that $r^l|\widetilde{u}$. Moreover, for any $0\leq h\leq l$, there are $u_{2,h},\widehat{u}_{2,h}\in S$ with $u=q^h u_{2,h}$, $\widetilde{u}=r^h \widehat{u}_{2,h}$ such that $\rho((q^h,u_{2,h}))=(r^h,\widehat{u}_{2,h})$.

(\romannumeral2) Suppose that $q^l\ddagger u$, where $q\in \cP$ and $l\geq 0$. Then there is an $r\in \cP$ such that $r^l\ddagger \widetilde{u}$. Moreover, for any $0\leq h\leq l$, there are $u_{1,h},\widehat{u_{1,h}}\in S$ with $u= u_{1,h}q^h$, $\widetilde{u}=\widehat{u}_{1,h} r^h$ such that $\rho((u_{1,h},q^h))=(\widehat{u}_{1,h},r^h)$.
\end{lemma}

\begin{proof}
(\romannumeral1) When $l=0$, the proof is trivial. We suppose that $l\geq 1$. For any $1\leq h\leq l$, Write $u=q^hu_{2,h}$. By Lemma \ref{lem_1-1_strong_castling}, we have
\begin{equation} \label{eq_power_preserving_1}
\uuline{u}v = \uuline{q^h u_{2,h}}v \rlh \uuline{q^h}\widehat{v_h}\uuline{\widehat{u_{2,h}}} \rlh \widetilde{v}\uuline{\widehat{u_{1,h}}\widehat{u_{2,h}}} = \widetilde{v} \uuline{\widetilde{u}}
\end{equation}
for some $\widehat{u_{1,h}},\widehat{u_{2,h}},\widehat{v}_h\in S$. By Lemma \ref{lem_power_preserving_strong}, we deduce that $\widehat{u_{1,h}}=r_h^h$ for some $r_h\in \cP$. Then $\rho((q^h,u_{2,h}))=(r_h^h, \widehat{u_{2,h}})$.

Denote $r=r_l$. Now we decompose the strong castlings in \eqref{eq_power_preserving_1} further. For any $0\leq h<l$, we have
\[
\uuline{u_{2,h}}v\rlh \uuline{q^{l-h}u_{2,l}}v \rlh \uuline{q^{l-h}}\widehat{v_l}\uuline{\widehat{u_{2,l}}} \rlh \widehat{v_h} \uuline{w_h\widehat{u_{2,l}}} = \widehat{v_h}\uuline{\widehat{u_{2,h}}}
\]
for some $w_h\in S$, and
\[
\uuline{q^l }\widehat{v_l} = \uuline{q^h q^{l-h}}\widehat{v_l} \rlh \uuline{q^h}\widehat{v_h}\uuline{w_h} \rlh \widetilde{v} \uuline{r_h^h w_h}=\widetilde{v}\uuline{r^l}.
\]
In view of Lemma \ref{lem_uniqueness_prime_power} and $r_h^h w_h=r^l$, we have $r_h=r$ and $w_h=r_h^{l-h}$. The proof is completed.

(\romannumeral2) The conclusion follows from similar arguments.
\end{proof}

The arithmetic meaning of Lemma \ref{lem_prime_power_castling} is that, after a strong castling, prime powers becomes prime powers with same multiplicity, and powers of a same prime becomes powers of a same prime. In particular, when $\uuline{u}v \rlh \widetilde{v}\uuline{\widetilde{u}}$, we have
\[
\Omega(u)=\Omega(\widetilde{u}), \quad \Omega(v)=\Omega(\widetilde{v}),\quad \Omega_\ddagger(u)=\Omega_\ddagger(\widetilde{u}),\quad \Omega_\ddagger(v)=\Omega_\ddagger(\widetilde{v}).
\]

In a strong castling $\uuline{u}v \rlh \widetilde{v}\uuline{\widetilde{u}}$. Suppose that $u=p^ku_2$ and $r^l|v$ for some $p,k\in \cP$, $u_2\in S$, $k,l\geq 1$. We have
\[
\uuline{u}r^l= \uuline{p^ku_2}r^l \rlh \uuline{p^k}\breve{r}^l \uuline{\breve{u_2}} \rlh \breve{\breve{r}}^l \uuline{\breve{p}^k \breve{u_2}}
\]
for some $\breve{r},\breve{\breve{r}},\breve{p}\in \cP$ and $\breve{u_2}\in S$. Here $\breve{\breve{r}}^l$ is a divisor of $\widetilde{v}$. Notice that, during the above castling process, a prime power $p^k$ may become a power $\breve{p}^k$ of another prime $\breve{p}$. Luckily, Lemma \ref{lem_q_m_r_requires_q=r} ensures that, when $\breve{\breve{r}}=p$, we have $\breve{p}=p$ and do not encounter another prime. Then $p$ has multiplicity no smaller than $k+l$ as a divisor of $uv$. This observation helps us to prove Theorem \ref{thm_Liouville_function}.

\begin{proof} [Proof of Theorem \ref{thm_Liouville_function}]
In the following, we will show that $\lambda$ is completely multiplicative. Similar arguments work for the function $\lambda_\ddagger$ and we omit the details here.

Let $\widetilde{v}\uuline{\widetilde{u}}\rlh \uuline{u}v$. It follows from Lemma \ref{lem_prime_power_castling} that $\Omega(u)=\Omega(\widetilde{u})$ and then $\lambda(u)=\lambda(\widetilde{u})$. So both $\Omega$ and $\lambda$ are castled-invariant. Suppose that all the distinct prime divisors of $u$ and $\widetilde{v}$ are $q_1,\ldots,q_{k}$, with multiplicity $n_1,\ldots,n_k$ and $m_1,\ldots,m_k$, respectively. Here $k\geq 0$, $n_1,\ldots,n_k,m_1,\ldots,n_k\geq 0$ and $\max\{n_i,m_i\}\geq 1$ for $1\leq i\leq k$. By Lemma \ref{lem_prime_power_castling}, there are distinct primes $r_1,r_2,\ldots,r_k$ such that all the prime divisors of $v$ belongs to $\{r_1,r_2,\ldots,r_k\}$, with multiplicity $m_1,\ldots,m_k$ (one can ignore those $r_j$ with $m_j=0$), respectively. More precisely, we have
\[
\rho\left((q_j^{m_j}, q_j^{-m_j}\widetilde{v})\right)= (r_j^{m_j},r_j^{-m_j}v),\quad (1\leq j\leq k),
\]
where $\rho$ is the map given in Lemma \ref{lem_1-1_strong_castling}. Then $\Omega(v)=\Omega(\widetilde{v})= m_1+\ldots+m_k$ and
\[
\lambda(u)\lambda(v)=  (-1)^{n_1+\ldots+n_k+m_1+\ldots+m_k}.
\]
For any prime divisor $p$ of $uv$, we have either $p|u$ or $p$ is a castled-divisor of $v$ over $u$. So $\PD(uv)\subseteq \{q_1,\ldots,q_k\}$. In order to show that $\lambda(uv)=\lambda(u)\lambda(v)$, we shall prove below that $q_j^{m_j+n_j}|uv$ and $q_j^{m_j+n_j+1}\nmid uv$ for $1\leq j\leq k$.

Given $1\leq j\leq k$ and $q_j^{m_j}|\widetilde{v}$, by Lemma \ref{lem_prime_power_castling},
\[
 \uuline{\widetilde{v}}\widetilde{u} = \uuline{q_j^{m_j} \widehat{v_2}}\widetilde{u} \rlh \uuline{q_j^{m_j}} \widehat{u} \uuline{v_2} \rlh u \uuline{r_j^{m_j}v_2} = u\uuline{v}
\]
for some $\widehat{u},\widehat{v_2},v_2\in S$. Putting $u=q_j^{n_j}u_2$, one has
\begin{equation} \label{eq_Liouville_multiplicative_1}
u\uuline{r_j^{m_j}} = q_j^{n_j}u_2 \uuline{r_j^{m_j}} \rlh q_j^{n_j} \uuline{w} \breve{u_2} \rlh \uuline{q_j^{m_j}}\breve{u_1} \breve{u_2} = \uuline{q_j^{m_j}}\widehat{u}
\end{equation}
for some $w,\breve{u_1},\breve{u_2}\in S$. In particular, we have $q_j^{n_j} \uuline{w} \rlh \uuline{q_j^{m_j}}\breve{u_1}$. Applying Lemma \ref{lem_q_m_r_requires_q=r}, one deduces that $w=q_j^{m_j}$ and $\breve{u_1}=q_j^{n_j}$. That is to say, the expression \eqref{eq_Liouville_multiplicative_1} becomes
\begin{equation} \label{eq_Liouville_multiplicative_2}
u\uuline{r_j^{m_j}} = q_j^{n_j}u_2 \uuline{r_j^{m_j}} \rlh q_j^{n_j} \uuline{q_j^{m_j}} \breve{u_2} \rlh \uuline{q_j^{m_j}}q_j^{n_j} \breve{u_2} = \uuline{q_j^{m_j}}\widehat{u}.
\end{equation}
Now $q_j^{m_j+n_j}| q_j^{m_j}q_j^{n_j} \breve{u_2} = ur_j^{m_j}$ and $ur_j^{m_j}|uv$, which leads to $q_j^{m_j+n_j}|uv$.

Assume on the contrary that $q_j^{m_j+n_j+1}|uv$ for some $1\leq j\leq k$. Noting that $u=q_j^{n_j}u_2$ with $q_j\nmid u_2$, one has $q_j^{m_j+1}|u_2v$ and $\gcd(q_j^{m_j+1},u_2)=1$. By Lemma \ref{lem_divisor_castling}, there is some $v_1|v$ and $\widetilde{u_2}\in S$ such that $\uuuline{q_j^{m_j+1}}\widetilde{u_2}\rlh u_2\uuuline{v_1}$. By Lemma \ref{lem_power_preserving_strong}, one deduces that $v_1=t^{m_j+1}$ for some $t\in \cP$. Applying a decomposition of this free castling, we obtain
\[
u_2\uuuline{t^{m_j+1}} = u_2 \uuuline{t^{m_j}t} \rlh \uuuline{w_1} \dot{u_2} \uuuline{t} \rlh \uuuline{w_1w_2} \widetilde{u_2} = \uuuline{q_j^{m_j+1}}\widetilde{u_2}
\]
for some $w_1,\dot{u_2}\in S$ and $w_2\in \cP$. Since $w_1w_2=q_j^{m_j+1}$, one deduces that $w_1=q_j^{m_j}$ and $w_2=q_j$. In particular, we have $u_2\uuuline{t_j^{m_j}}\rlh \uuuline{q_j^{m_j}}\dot{u_2}$. From \eqref{eq_Liouville_multiplicative_2}, we also have $u_2\uuline{r_j^{m_j}}\rlh \uuline{q_j^{m_j}}\breve{u_2}$. By Lemma \ref{lem_skew_uniqueness}, we conclude that $t=r_j$. Now $r_j^{m_j+1}=v_1|v$, which is a contradiction.

The proof is completed.
\end{proof}

\subsection{Capturing Prime Divisors}

In Section \ref{subsection_Liouville}, we have studied the castling of prime powers in a strong castling. However, locating prime powers from weak castlings requires much harder work. Recall that the multi-set of prime divisors and prime co-divisors of $u$ with multiplicity are denoted by $\PDM(u)$ and $\PDM_\ddagger(u)$, respectively. The main purpose of this section is to prove the following theorem.

\begin{theorem} \label{thm_prime_multiplicity_from}
Suppose that $u=q_1 q_2\ldots q_k$ for some $k\geq 1$ and $q_1,q_2,\ldots,q_k\in \cP$.

(\romannumeral1)  We have
\[
\PDM(u)=\bigcup\limits_{j=1}^{k} \left\{p\in \cP: \underline{p} w_j \rlh q_1\ldots q_{j-1}\uuline{q_j} \text{ for some }w_j\in S \right\}.
\]
Here an empty product, $q_1\ldots q_{j-1}$ with $j=1$, is defined to be $1$. And a set in the union is non-empty if and only if the corresponding castling is satisfied. In particular, we have $\Omega(u)\leq \ind(u)$.

(\romannumeral2) We also have
\[
\PDM_\ddagger(u)=\bigcup\limits_{r=1}^{k} \left\{p\in \cP:  \uuline{q_r}q_{r+1}\ldots q_k \rlh z_r\uline{p } \text{ for some }z_r\in S \right\}.
\]
In particular, we have $\Omega_\ddagger (u)\leq \ind(u)$.
\end{theorem}

\begin{remark}
(1) The castlings occurred in this theorem involve both a strong one and a weak one. One can not change the strong one to a weak one, or change the weak one to a strong one. Examples will be given in Section \ref{subsection_verifying_axioms_thompson} to illustrate this fact.

(2) The statements do not require a particular irreducible decomposition of $u$. That is to say, to get prime divisors of $u$ with multiplicity, one can start with any irreducible decomposition $u=q_1q_2\ldots q_k$.
\end{remark}

The difficulty of this theorem is to make clear the relation of prime powers in strong and weak castlings. We will come back to prove Theorem \ref{thm_prime_multiplicity_from} after the following lemma.

\begin{lemma} \label{lem_power_preserving_strong_castling}
(\romannumeral1) Suppose that $p,q\in \cP$ and $w,u\in S$. If $\uuline{w}q \rlh p\uline{u}$, then $\uuline{p^kw}q \rlh p\uline{p^ku}$ for all $k\geq 1$.
(\romannumeral2) Suppose that $p,q\in \cP$ and $w,u\in S$. If $\uline{w}q \rlh p\uuline{u}$, then $\uline{wq^k}q \rlh p\uuline{uq^k}$ for all $k\geq 1$.
\end{lemma}

\begin{proof}
(\romannumeral1) Note that $\uline{p^k w}q \rlh \uline{p^k}p \uline{u} \rlh p\uline{p^k u}$. In order to show that $p^kw,q$ are strongly castlable, we need to prove that for any $w_1,w_2\neq 1$ with $w_1w_2=p^k w$, there are elements  $\acute{q},\acute{w_1},\acute{w_2}\in S$ with $\acute{w_1}\acute{w_2}=p^k u$ such that $\uuline{w_2} q \rlh  \acute{q} \uline{\acute{w_2}}$ and $\uuline{w_1} \acute{q}  \rlh p \uline{\acute{w_1}}$.

For $\ind(w)=0$, the proof is trivial. In the following, we always assume that $\ind(w)\geq 1$. Now we deal with the case that $k=1$ and $\ind(w)=1$. If $\tau(pw)=3$, then the only non-trivial decomposition of $pw=w_1w_2$ is $w_1=p,w_2=w$. Indeed, we have
\[
\uuline{w}q \rlh p\uline{u},\quad \uuline{p}p\rlh p\uline{p}.
\]
If $\tau(pw)=4$, then $pw=ab$ for some $a,b\in \cP$ with $p\neq a, w\neq b$. Now we have the other non-trivial decomposition of $pw=w_1w_2$, i.e., $w_1=a,w_2=b$. Indeed, we have $\uuuline{p}w\rlh a\uuuline{b}$. Note that $p^2u = pwq =abq$. We have $\gcd(p^2,a)=1$ and $\uuuline{p^2}u\rlh a\uuuline{bq}$. By Axiom \uppercase\expandafter{\romannumeral5}, one concludes that $bq=t^2$ for some $t\in \cP$, which leads to $b=q=t$. Now $\uuuline{p}w\rlh a\uuuline{q}$. We have $wq=pu$ and
\[
\uuline{q}q\rlh q\uline{q},\quad \uuline{a}q\rlh p\uline{w}.
\]
So $pw,q$ are strongly castlable.

Suppose that the lemma has been proved for $\ind(w)+k \leq m-1$ for some $m\geq 3$. Now we consider the case that $\ind(w)+k=m$. Suppose that $w_1,w_2\neq 1$ are elements in $S$ such that $w_1w_2=p^k w$.

\textsc{Case 1.} Suppose that $p| w_1$. Write $w_1=p z$. Then $zw_2=p^{k-1}w$. Note that $k-1+\ind(w)\leq m-1$. By inductive hypothesis, we have $\uuline{zw_2}q=\uuline{p^{k-1}w}q \rlh p\uline{p^{k-1}u}$. Since the left-hand side is a strong castling, we can decompose it to obtain
\[
\uuline{zw_2}q \rlh \uuline{z}\widehat{q}\uline{\widehat{w_2}} \rlh p \uline{\widehat{z}\widehat{w_2}} = p\uline{p^{k-1}u}
\]
for some $\widehat{q}\in \cP$ and $\widehat{w_2},\widehat{z}\in S$. Moreover, noting that $\ind(z)+1=\ind(w_1)\leq m-1$, we deduce by inductive hypothesis and $\uuline{z}\widehat{q}\rlh p\uline{\widehat{z}}$ that $\uuline{w_1}\widehat{q} = \uuline{p z}\widehat{q} \rlh p \uline{p \widehat{z}}$. To sum up, in this case we have
\[
\uuline{w_2}q \rlh \widehat{q}\uline{\widehat{w_2}},\quad \uuline{w_1}\widehat{q} \rlh p \uline{p\widehat{z}},
\]
where $p \widehat{z}\widehat{w_2}= p p^{k-1}u=p^k u$.

\textsc{Case 2.} Suppose that $p\nmid w_1$. Combining Corollary \ref{cor_gcd_p_m}, we have $\gcd(p^k,w_1)=1$. Since $p^k|w_1w_2$, by Lemma \ref{lem_divisor_castling} we have $\uuuline{p^k}v\rlh w_1\uuuline{w_3}$ for some $w_3|w_2$ and $v\neq 1$. Write $w_2=w_3 e$. We have $p^k w=w_1w_2=w_1w_3e=p^k v e$, which implies $w=v e$. Note that $\uuline{w}q\rlh p\uline{u}$. We have
\[
\uuline{w}q = \uuline{ve}q \rlh \uuline{v}\breve{q}\uline{\breve{e}} \rlh p \uline{\breve{v}\breve{e}} = p\uline{u}
\]
for some $\breve{q}\in \cP$ and $\breve{v},\breve{e}\in S$.

\textsc{Case 2-1.} Suppose that $e\neq 1$. Then $\ind(v)+k  \leq m-1$. By inductive hypothesis, we deduced from $\uuline{v}\breve{q} \rlh p\uline{\breve{v}}$ that $\uuline{w_1w_3}\breve{q}=\uuline{p^kv}\breve{q}\rlh p\uline{p^k\breve{v}}$. Since the left-hand side is a strong castling, we have that
\begin{equation} \label{eq_power_needing_1}
\uuline{w_1w_3}\breve{q} \rlh \uuline{w_1} \dot{q} \uline{\dot{w_3}} \rlh p \uline{\dot{w_1}\dot{w_3}} = p \uline{p^k\breve{v}}
\end{equation}
for some $\dot{q}\in \cP$ and $\dot{w_1},\dot{w_3}\in S$. Note that $p^{k+1}|p^{k+1}\breve{v}=w_1(w_3\breve{q})$ and $\gcd(p^{k+1},w_1)=1$ by corollary \ref{cor_gcd_p_m}. Also note that $\ind(w_3\breve{q})=k+1$, which is due to $\ind(w_3)=\ind(p^k)=k$. There is some $c\in S$ such that $\uuuline{p^{k+1}}c\rlh w_1\uuuline{w_3\breve{q}}$. By Lemma \ref{lem_power_preserving_strong}, one deduces that $w_3\breve{q}=t^{k+1}$ for some $t\in \cP$. So $t=\breve{q}$ and $w_3=\breve{q}^k$ by Lemma \ref{lem_uniqueness_prime_power}. Now \eqref{eq_power_needing_1} becomes
\[
\uuline{w_1\breve{q}^k}\breve{q} \rlh \uuline{w_1} \breve{q} \uline{\breve{q}^k} \rlh p \uline{\dot{w_1}\breve{q}^k} = p \uline{p^k\breve{v}}.
\]
Moreover, since $e\neq w$, one has $k+\ind(e)\leq m-1$. It follows by $\uuline{e}q\rlh \breve{q}\uline{\breve{e}}$ and the inductive hypothesis that $\uuline{w_2}q = \uuline{\breve{q}^k e}q \rlh \breve{q}\uline{\breve{q}^{k}\breve{e}}$. To sum up, in this case we have
\[
\uuline{w_2}q \rlh \breve{q}\uline{\breve{q}^{k}\breve{e}},\quad \uuline{w_1}\breve{q} \rlh p \uline{\dot{w_1}},
\]
where $\dot{w_1}\breve{q}^k\breve{e}=p^k\breve{v}\breve{e}=p^k u$.

\textsc{Case 2-2.} Suppose that $e=1$. Then $\uuuline{p^k}w\rlh w_1\uuuline{w_2}$. We have $\ind(w_2)=k$ and then $\ind(w_2q)=k+1$. Note that $p^{k+1}|p^{k+1}u=w_1(w_2q)$ and $\gcd(w_1,p^{k+1})=1$ by corollary \ref{cor_gcd_p_m}. There is some $c\in S$ such that $\uuuline{p^{k+1}}c\rlh w_1\uuuline{w_2q}$. By Lemma \ref{lem_power_preserving_strong}, one deduces that $w_2q=t^{k+1}$ for some $t\in \cP$. So $t=q$ and $w_2=q^k$ by Lemma \ref{lem_uniqueness_prime_power}. Now we have
\[
w_1\uuuline{w_2} = w_1\uuuline{q\cdot q^{k-1}} \rlh \uuuline{\widetilde{q}} \widetilde{w_1} \uuuline{q^{k-1}} \rlh \uuuline{\widetilde{q}y} w = \uuuline{p^k}w
\]
for some $\widetilde{q}\in \cP$ and $\widetilde{w_1},y\in S$. Since $p^k=\widetilde{q}y$, we have $\widetilde{q}=p$ and $y=p^{k-1}$. Now
\[
\uuline{w_2}q = \uuline{q^k}q \rlh q\uline{q^k}, \quad \uuline{w_1}q \rlh \widetilde{q}\uline{\widetilde{w_1}} = p\uline{\widetilde{w_1}},
\]
where $\widetilde{w_1}q^k = (\widetilde{w_1}q^{k-1})q= ywq = p^{k-1}pu=p^k u$. By inductive hypothesis, the proof is completed.

(\romannumeral2) The arguments as above work, and we omit the details here.
\end{proof}

\medskip

\begin{proof} [Proof of Theorem \ref{thm_prime_multiplicity_from}]
(\romannumeral1) If $\underline{p} w_j \rlh q_1\ldots q_{j-1}\uuline{q_j}$ for some $1\leq j\leq k$ and some element $w_j$, then $p$ is definitely a prime divisor of $u$. Suppose that a same prime $p$ is obtained exactly from castlings $\underline{p} w_j \rlh q_1\ldots q_{j-1}\uuline{q_j}$ with $j\in \{i_1,i_2,\ldots, i_l\}$, where $1\leq i_1<i_2<\ldots< i_l\leq k$. We need to show that $p^l| u$. In the following, induction is used to show that $p^{t}| q_1 q_2\ldots q_{i_t}$ for all $1\leq t\leq l$. For $t=1$, one already has $p|q_1q_2\ldots q_{i_1}$. Suppose the conclusion for $t-1$ has been proved, i.e., $p^{t-1}|q_1 q_2\ldots q_{i_{t-1}}$. Write $q_1 q_2\ldots q_{i_{t-1}} = p^{t-1}a$ and $b= q_{i_{t-1}+1}\ldots q_{i_t-1}$.
We have
\begin{equation} \label{eq_prime_power_proof}
\uuline{q_1\ldots q_{i_t-1}}q_{i_t} =\uuline{p^{t-1}ab}q_{i_t} \rlh \uuline{p^{t-1}a} \widehat{q_{i_t}} \underline{\widehat{b}} \rlh \uuline{p^{t-1}}\widehat{\widehat{q_{i_t}}} \uline{\widehat{a} \widehat{b}} \rlh  p \uline{c\widehat{a} \widehat{b}} = p \underline{w_{i_t}}
\end{equation}
for some prime $\widehat{q_{i_t}},\widehat{\widehat{q_{i_t}}}\in \cP$ and $\widehat{a},\widehat{b},c\in S$. Applying Lemma \ref{lem_q_m_r_requires_q=r} with $\uuline{p^{t-1}}\widehat{\widehat{q_{i_t}}}\rlh p\uline{c}$, one deduces that $\widehat{\widehat{q_{i_t}}}=p$. So $p^t| q_1 q_2\ldots q_{i_t}$. Now we obtain that $p^l|u$.

On the other hand, suppose that $p^m | u$ and $p^{m+1}\nmid u$ for some $m\geq 1$. We prove by induction on $K=\ind(u)$ that there exist at least $m$ numbers of $r$'s such that $p$ is obtained from the castlings $\underline{p} w_r \rlh q_1\ldots q_{r-1}\uuline{q_r}$ with $1\leq r\leq K$. For $K=1$, the proof is trivial. Suppose that the result has been proved when $K\leq k-1$ for some $k\geq 2$. Next we consider the case $K=k$ with $u=q_1q_2\ldots q_k$. Let $m_1$ be the non-negative integer satisfying $p^{m_1}|q_1\ldots q_{k-1}$ and $p^{m_1+1}\nmid q_1\ldots q_{k-1}$.

Case 1. Suppose that $m_1=m$. By inductive hypothesis, the prime $p$ is induced by at least $m$ castlings of type $\underline{p} w_r \rlh q_1\ldots q_{r-1}\uuline{q_r}$ for $1\leq r\leq k-1$.

Case 2. Suppose that $m_1<m$. We write $q_1\ldots q_{k-1} = p^{m_1} u_1$, where $p\nmid u_1$. Now one has $p^{m-m_1}| u_1q_k$ and $\gcd(p^{m-m_1},u_1)= 1$. It follows from Lemma \ref{lem_divisor_castling} that $\uuuline{u_1} q_k \rlh p^{m-m_1} \uuuline{z}$ for some $z\in S$. Since $\ind(p^{m-m_1})=\ind(q_k)$, one has $m_1=m-1$. Moreover, the castlings $\uuuline{u_1} q_k \rlh p\uuuline{z}$ and Lemma \ref{lem_power_preserving_strong_castling} results in the conclusion that
\[
\uuline{q_1\ldots q_{k-1}} q_k = \uuline{p^{m-1} u_1} q_k \rlh p\underline{p^{m-1}z}.
\]
So the prime $p$ is induced by one of the castlings expected. In view of the fact that $p^{m-1}|q_1\ldots q_{k-1}$, it follows from inductive hypothesis that the prime $p$ is induced by at least another $m-1$ castlings of expected type. The proof is completed.

(\romannumeral2) The conclusion follows from similar arguments as in (\romannumeral1).
\end{proof}

\subsection{Fully Castlable Elements}
\label{subsection_fully_castlable_elements}

In this subsection, we will introduce a special class of monoids with good properties.

\begin{definition} \label{def_fully_castlable}
We call an element $u$ in a castlable monoid $S$ fully castlable, if for any $u_1,u_2,u_3,u_4\in S$ with $u=u_1u_2u_3u_4$, the elements $u_2,u_3$ are weakly castlable.
\end{definition}

The purpose of this definition is as follows: if $u$ is fully castlable, then each part of $u$ would contribute to a divisor and a co-divisor of $u$.

\begin{lemma} \label{lem_fully_ca_di_co_also_fully_ca}
In a castlable monoid, suppose that $v$ is fully castlable and $v_1v_2=v$. Then both $v_1$ and $v_2$ are fully castlable.
\end{lemma}

\begin{proof}
Let $w_1,w_2,w_3,w_4$ be any elements in $S$ such that $w_1w_2w_3w_4=v_1$. Take $u_i=w_i$ $(1\leq i\leq 3)$ and $u_4=w_4v_2$. Since $v$ is fully castlable, the elements $u_2,u_3$ are weakly castlable. So $v_1$ is fully castlable. Similar arguments show that $v_2$ is also fully castlable.
\end{proof}

\begin{lemma} \label{lem_fully_castlable_equivalent_definition}
Suppose that $S$ is castlable. An element $u$ in $S$ is fully castlable, if and only if $w_1,w_2$ are strongly castlable for any $w_1,w_2\in S$ with $w_1w_2=u$.
\end{lemma}

\begin{proof}
We first prove the ``$\Rightarrow$''-part by induction on $\ind(u)$. For $\ind(u)\leq 1$, the prove is trivial. Suppose that the result has been proved for $\ind(u)\leq m-1$ with some $m\geq 2$. Consider the case that $\ind(u)=m$. For any $x,y\in S$ with $xy=w_1$ and $x,y\neq 1$, the element $yw_2$ is a co-divisor of $u$. By Lemma \ref{lem_fully_ca_di_co_also_fully_ca}, it is fully castlable. By inductive hypothesis, we have that $y,w_2$ are strongly castlable, which we denote by $\uuline{y}w_2\rlh \widehat{w_2}\uline{\widehat{y}}$ for some $\widehat{w_2},\widehat{y}\in S$. Now $u=xyw_2=x\widehat{w_2}\widehat{y}$. By Lemma \ref{lem_fully_ca_di_co_also_fully_ca}, the element $x\widehat{w_2}$ is fully castlable. By inductive hypothesis again, we have $x,\widehat{w_2}$ are strongly castlable, which we denote by $\uuline{x}\widehat{w_2}\rlh \widehat{\widehat{w_2}}\uline{\widehat{x}}$ for some $\widehat{\widehat{w_2}},\widehat{x}\in S$. Similarly, for any $a,b\in S$ with $ab=w_2$ and $a,b\neq 1$, one can also prove that $\uuline{w_1}a\rlh \widehat{a}\uline{\widehat{w_1}}$ and $\uuline{\widehat{w_1}}b\rlh \widehat{b}\uline{\widehat{\widehat{w_1}}}$ for some $\widehat{a},\widehat{b},\widehat{w_1},\widehat{\widehat{w_1}}\in S$. Therefore, the elements $w_1,w_2$ are strongly castlable.

Now we deal with the ``$\Leftarrow$''-part. For any $u_1,u_2,u_3,u_4$ with $u_1u_2u_3u_4=u$, we have that $u_1u_2, u_3u_4$ are strongly castlable. By decomposition of strong castlings, one deduces that $u_2,u_3$ are strongly castlable. This completes the proof.
\end{proof}

Suppose that $u$ is fully castlable and $w_1w_2=u$. Then $\uuline{w_1}w_2\rlh \widetilde{w_2}\uline{\widetilde{w_1}}$ for some $\widetilde{w_1},\widetilde{w_2}\in S$ by Lemma \ref{lem_fully_castlable_equivalent_definition}. Here $\widetilde{w_2}\widetilde{w_1}=u$. So $\widetilde{w_2},\widetilde{w_1}$ are also strongly castlable. We have $\uuline{w_1}w_2\rlh \widetilde{w_2}\uuline{\widetilde{w_1}}$. Lemma \ref{lem_fully_castlable_equivalent_definition} gives another definition of fully castlable elements. Both have advantages. In the following, we turn back to natural monoids.

\begin{theorem} \label{thm_equivalence_fully_castlable}
Suppose that $S$ is a natural monoid.

(\romannumeral1) Let $u$ be an element in $S$, which has prime divisors $q_1,q_2,\ldots,q_k$ with multiplicities $m_1,m_2,\ldots,m_k$. Then $u$ is fully castlable if and only if $u=\lcm[q_1^{m_1},q_2^{m_2},\ldots,q_k^{m_k}]$.

(\romannumeral2) Let $u$ be an element in $S$, which has prime co-divisors $q_1,q_2,\ldots,q_k$ with multiplicities $m_1,m_2,\ldots,m_k$. Then $u$ is fully castlable if and only if $u=\lcm_\ddagger[u;q_1^{m_1},q_2^{m_2},\ldots,q_k^{m_k}]$.
\end{theorem}

\begin{proof}
(\romannumeral1) We first prove the ``$\Rightarrow$''-part. Denote $u_0=\lcm[q_1^{m_1},q_2^{m_2},\ldots,q_k^{m_k}]$. Then $u_0|u$. Assume on the contrary that $u_0\neq u$. Then $u=u_0w$ for some $w\neq 1$. Write $w=rw_0$, where $r\in \cP$. Since $u$ is fully castlable, then $u_0,w$ are strongly castlable by Lemma \ref{lem_fully_castlable_equivalent_definition}. It follows that $u_0,r$ are also strongly castlable. We write $\uuline{u_0}r\rlh \widetilde{r}\uline{\widetilde{u_0}}$ for some $\widetilde{r}\in \cP$ and $\widetilde{u_0}\in S$. Now $\widetilde{r}$ is a prime divisor of $u_0r$, which is also a prime divisor of $u$. So $\widetilde{r}=q_j$ for some $1\leq j\leq k$. Noting that $q_j^{m_j}|u_0$, we can write $u_0=q_j^{m_j}u_0^\prime$ for some $u_0^\prime \in S$. A decomposition of the castling gives
\[
\uuline{u_0}r = \uuline{q_j^{m_j} u_0^\prime}r\rlh \uuline{q_j^{m_j}}\widehat{r}\uline{\widehat{u_0}^\prime} \rlh q_j\uline{c\widehat{u_0}^\prime}=q_j\uline{\widetilde{u_0}}
\]
for some $\widehat{r},\widehat{u_0}^\prime,c\in S$. It follows from  $\uline{q_j^{m_j}}\widehat{r}\rlh q_j\uline{c}$ and Lemma \ref{lem_q_m_r_requires_q=r} that $\widehat{r}=q_j$. So $q_j^{m_j+1}|u_0r$, which is a divisor of $u$. This contradicts the fact that $q_j$ is a prime divisor of $u$ with multiplicity $m_j$.

\medskip

Now we shall prove the ``$\Leftarrow$''-part. Let $l=\ind(u)=m_1+\ldots+m_k$. Let $u_1,u_2,u_3,u_4$ be any elements in $S$ with $u_1u_2u_3u_4=u$. Let $u=r_1r_2\ldots r_l$ be an irreducible decomposition, where
\[
u_1=r_1\ldots r_{l_1}, \quad u_2=r_{l_1+1}\ldots r_{l_2},\quad u_3=r_{l_2+1}\ldots r_{l_3},\quad u_4=r_{l_3+1}\ldots r_l,
\]
with $0\leq l_1\leq l_2\leq l_3\leq l$.

Note that $\ind(u)=m_1+m_2+\ldots+m_k = \Omega(u)$. By Theorem \ref{thm_prime_multiplicity_from}, we have $\uuline{r_1r_2\ldots r_{j-1}}r_j\rlh p_j\uline{w_j}$ for some $w_j\in S$ and $p_j\in \cP$, for all $1\leq j\leq l$. Consider the $(l_2+1)$-th castling appeared above. Since $r_1r_2\ldots r_{l_2}$, $r_{l_2+1}$ are strongly castlable and $r_1r_2\ldots r_{l_2}=u_1u_2$, one obtains that $u_2,r_{l_2+1}$ are strongly castlable. We denote $\uuline{u_2^{(l_2)}}r_{l_2+1}\rlh \widehat{r_{l_2+1}}\uline{u_2^{(l_2+1)}}$ for some $\widehat{r_{l_2+1}}\in \cP$ and $u_2^{(l_2+1)}\in S$, where $u_2^{(l_2)}=u_2$. Next, consider the $(l_2+2)$-th castling. Since $r_1r_2\ldots r_{l_2+1}$, $r_{l_2+2}$ are strongly castlable and $r_1r_2\ldots r_{l_2+1}=u_1\widehat{r_{l_2+1}}u_2^{(l_2+1)}$, one concludes that $u_2^{(l_2+1)}, r_{l_2+2}$ are strongly castlable, which we denote $\uuline{u_2^{(l_2+1)}}r_{l_2+2}\rlh \widehat{r_{l_2+2}}\uline{u_2^{(l_2+2)}}$ for some $\widehat{r_{l_2+2}}\in \cP$ and $u_2^{(l_2+2)}\in S$. We repeat the above process for $l_2+1\leq j\leq l_3$. Since $r_1r_2\ldots r_{j-1}$, $r_j$ are strongly castlable and $r_1r_2\ldots r_{j-1}= u_1 \widehat{r_{l_2+1}}\ldots \widehat{r_{j-1}} u_2^{(j-1)}$, the elements $u_2^{(j-1)}$, $r_j$ are strongly castlable. We assume that $\uuline{u_2^{(j-1)}}r_j\rlh \widehat{r_j}\uline{u_2^{(j)}}$ for some $\widehat{r_j}\in \cP$ and $u_2^{(j)}\in S$. Now a composition of the above weak castlings shows that
\begin{align*}
\uline{u_2} u_3 &= \uline{u_2^{(l_2)}} q_{l_2+1}q_{l_2+2}\ldots q_{l_3} \rlh \widehat{q_{l_2+1}}\uline{u_2^{(l_2+1)}} q_{l_2+2}\ldots q_{l_3} \\
&\rlh \ldots \rlh \widehat{q_{l_2+1}}\ldots \widehat{q_{l_3-1}}\uline{u_2^{(l_3-1)}}q_{l_3} \rlh \widehat{q_{l_2+1}}\ldots \widehat{q_{l_3}} \uline{u_2^{(l_3)}}.
\end{align*}
So $u_2,u_3$ are weakly castlable. As a result, the element $u$ is fully castlable.

(\romannumeral2) Similar arguments as above work.

\end{proof}

\begin{corollary}
Let $S$ be a natural monoid. (\romannumeral1) Suppose that $u$ and $v$ are both fully castlable elements in $S$. Then so is $\lcm[u,v]$. (\romannumeral2) Suppose that $u$ and $v$ are both fully castlable elements in $S$ and $u,v\ddagger w$ for some $w\in S$. Then so is $\lcm_\ddagger[w;u,v]$.
\end{corollary}

\begin{proof}
Combining Corollary \ref{cor_gcd_lcm_of_two_lcms} and Theorem \ref{thm_equivalence_fully_castlable}, the corollary follows.
\end{proof}

\begin{theorem} \label{thm_Omega=Omega_ddagger}
For a fully castlable element $u$ in a natural monoid, we have that $\Omega(u)=\Omega_\ddagger(u)$.
\end{theorem}

\begin{proof}
The theorem can be proved by combining Theorem \ref{thm_equivalence_fully_castlable} and Lemma \ref{lem_lcm_prime_powers_equivalent_lcm_ddagger_prime_powers}. In the following, we give another proof based on Theorem \ref{thm_prime_multiplicity_from}.

we use induction on $\ind(u)$. For $\ind(u)\leq 1$, the proof is trivial. Suppose that the theorem has been proved for $\ind(u)=m-1$ with some $m\geq 2$. Now we consider the case $\ind(u)=m$. Write $u=q_1u_0$ and $u_0=q_2q_3\ldots q_m$ for some $q_1,\ldots,q_m\in \cP$. Lemma \ref{lem_fully_ca_di_co_also_fully_ca} shows that $u_0$ is also fully castlable. By inductive hypothesis, we have $\Omega(u_0)=\Omega_\ddagger(u_0)$. Note that $\Omega_\ddagger(u_0)=\#\PDM_\ddagger(u_0)$ and $\Omega_\ddagger(u)=\#\PDM_\ddagger(u)$, where
\begin{align*}
&\PDM_\ddagger(u_0)=\bigcup\limits_{j=2}^{m} \left\{p\in \cP:  \uuline{q_j}q_{j+1}\ldots q_m \rlh z_j\uline{p } \text{ for some }z_j\in S \right\},\\
&\PDM_\ddagger(u)=\PDM_\ddagger(u_0)\cup \left\{p\in \cP:  \uuline{q_1}q_2q_3\ldots q_m \rlh z_r\uline{p } \text{ for some }z_1\in S \right\}
\end{align*}
(recalling that they are multi-sets) by Theorem \ref{thm_prime_multiplicity_from}.

Since $u$ is fully castlable, the elements $q_1,u_0$ are strongly castlable. Therefore $\Omega_\ddagger(u)=\Omega_\ddagger(u_0)+1$. Suppose that $r_1,\ldots,r_k$ are exactly all the distinct prime divisors of $u_0$, with multiplicities $l_1,\ldots,l_k$, respectively. Here $l_1,\ldots, l_k\geq 1$. Then $q_1,r_i^{l_i}$ are strongly castlable for all $1\leq i\leq k$. By Axiom \uppercase\expandafter{\romannumeral5}, we write $\uuline{q_1}r_i^{l_i}\rlh \widehat{r_i}^{l_i}\uline{\widehat{q_1}}$ for some $\widehat{r_1},\ldots,\widehat{r_k},\widehat{q_1}\in \cP$. It is apparent that $\widehat{r_1},\ldots,\widehat{r_k}$ are distinct, and $\widehat{r_i}^l$ are divisors of $u$.

Suppose that $q_1\notin \{r_1,\ldots, r_k\}$. Recalling Lemma \ref{lem_q_m_r_requires_q=r}, we have $q_1\notin \{\widehat{r_1},\ldots,\widehat{r_k}\}$. Then $\Omega(u)\geq 1+l_1+\ldots+l_k$. Suppose that $q_1=r_{i_0}$ for some $1\leq i_0\leq k$, one has $\widehat{r_{i_0}}=q_1$. Therefore $q_1^{l_{i_0}+1}|u$. We also have
\begin{align*}
\Omega(u)\geq l_1+\ldots+l_{i_0-1}+(l_{i_0}+1)+l_{i_0+1}+\ldots +l_k=l_1+\ldots+l_k+1.
\end{align*}
On the other hand, let us consider any prime power $p^h$ dividing $u$. If $p=q_1$, then $p^{h-1}|u_0$. Therefore $h\leq 1$ when $q_1\notin\{\widehat{r_1},\ldots,\widehat{r_k}\}$ and $h\leq l_{i_0}+1$ when $q_1=\widehat{r_{i_0}}$. If $p\neq q_1$, then we have $p^h|q_1u_0$ and $\gcd(p^h,q_1)=1$. Combining Lemma \ref{lem_divisor_castling} and Axiom \uppercase\expandafter{\romannumeral5}, there are some $\breve{p},\breve{q}\in \cP$ with $\breve{p}^h|u$ such that $\uuuline{p^h}\breve{q_1}\rlh q_1\uuuline{\breve{p}^h}$. It follows that $\breve{p}= r_j$ for some $1\leq j\leq k$ and $h\leq l_j$. As a result, we have $\Omega(u)\leq l_1+\ldots+l_k+1$. Therefore
\[
\Omega(u)=l_1+\ldots+l_k+1= \Omega(u_0)+1=\Omega_\ddagger(u_0)+1=\Omega_\ddagger(u).
\]
This completes the proof.
\end{proof}

Theorem \ref{thm_equivalence_fully_castlable} provide us with approaches to represent an element in $S$ in a unique way. For an element $u\in S$, suppose that $u$ has prime divisors $q_1,q_2,\ldots,q_k$ with multiplicities $m_1,m_2,\ldots,m_k$. Then $u_1=\lcm[q_1^{m_1},q_2^{m_2},\ldots,q_k^{m_k}]$ is a divisor of $u$. Indeed, it is the maximum one among all the fully castlable divisors of $u$. We call it the greatest fully castlable divisor of $u$, since any fully castlable divisor of $u$ divides $u_1$. Now write $u=u_1v_1$. Next, suppose that $v_1$ has prime divisors $r_1,r_2,\ldots,r_l$ with multiplicities $n_1,n_2,\ldots,n_l$. We pick $u_2=\lcm[r_1^{n_1},r_2^{n_2},\ldots,r_l^{n_l}]$ and write $u=u_1u_2v_2$. Iterating this process, each element $u$ can be uniquely written as $u=u_1u_2\ldots u_t$, where $u_j$ is the greatest fully castlable divisor of $(u_1u_2\ldots u_{j-1})^{-1}u$ $(1\leq j\leq t)$. Similarly, we can make use of prime co-divisors and least common co-multiples to uniquely represent $u\in S$ as $u=u_t\ldots u_2u_1$ such that $u_j$ is the greatest fully castlable co-divisor of $u(u_{j+1}\ldots u_2u_1)^{-1}$ $(1\leq j\leq t)$.

Next, we consider those natural monoids $S$ consisting of fully castlable elements.

\begin{definition}
If all elements of a castlable monoid $S$ are fully castlable, then we say that $S$ is fully castlable.
\end{definition}

In a fully castlable monoid, the definition of a strong castling and that of a weak castling coincide. The example given in \eqref{eq_example_Huang} is a fully castlable natural monoid.

Define $\beta:\cP\times \cP\rightarrow \cP$ by $\beta(p,r)=q$, where $\uline{p}q \rlh r\uline{t}$ for some $t\in \cP$. For $p\in \cP$, define $\beta_p:\cP\rightarrow \cP$ by $\beta_p(r)=\beta(p,r)$ $(r\in \cP)$. Indeed, we have
\begin{equation} \label{eq_beta_form}
\beta_p(r)=\beta(p,r)=
\begin{cases}
p^{-1}\lcm[p,r], \quad &\If p\neq r,\\
p,\quad &\If p=r.
\end{cases}
\end{equation}
The maps $\beta$ and $\beta_p$ are well-defined, and $\beta_p$ is injective for any $p\in \cP$. When $S$ is abelian, it is fully castlable and $\beta_p$ is the identity map on $\cP$ for any $p\in \cP$.

\begin{lemma} \label{lem_permutation_fully_castlable}
The natural monoid $S$ is fully castlable if and only if the map $\beta_p$ is a bijection for any $p\in \cP$.
\end{lemma}

\begin{proof}
Suppose that $S$ is fully castlable. For any $p,q\in \cP$, one has $pq=1\cdot p\cdot q\cdot 1$. So $\uline{p}q\rlh r\uline{t}$ for some $r,t\in \cP$, which implies $\beta_p(r)=q$. Hence, the map $\beta_p$ is surjective. Since it is also injective, one concludes that $\beta_p$ is a bijection.

Suppose that $\beta_p$ is bijective for all $p\in \cP$. Then for any $p,q\in \cP$, it satisfies that $\uline{p}q\rlh \beta_p^{-1}(q)\uline{t}$ for some $t\in \cP$. So $p,q$ are weakly castlable. By induction on $\ind(u)$ and $\ind(v)$, it is not hard to prove that $u,v$ are weakly castlable for all $u,v\in S$. Now for $u=u_1u_2u_3u_4$ with $u,u_1,u_2,u_3,u_4\in S$, the elements $u_2,u_3$ are weakly castlable. As a result, any element $u$ in $S$ is fully castlable. The proof is completed.
\end{proof}

\subsection{Natural Monoids Containing Finitely Many Primes}
\label{subsection_natural_monoind_finite_primes}

The main purpose of this subsection is to show that a natural monoid with finitely many primes are fully castlable and the corresponding rational group is amenable.

\begin{theorem} \label{thm_finite_primes_implies_fully_castlable}
Suppose that $S$ is a natural monoid containing finitely many primes. Then $S$ is fully castlable. \end{theorem}

\begin{proof}
For any $p\in \cP$, since $\beta_p:\cP\rightarrow \cP$ is injective and $|\cP|<\infty$, the map $\beta_p$ is also surjective. The theorem follows from Lemma \ref{lem_permutation_fully_castlable}.
\end{proof}

Let $T$ be a monoid or a group. If there is a F{\o}lner sequence $\{F_n\}_{n=1}^\infty$ with $F_n\subseteq F_{n+1}$ $(n\geq 1)$, $\bigcup\nolimits_{n=1}^\infty F_n =T$ such that
\[
\lim\limits_{n\rightarrow\infty} \frac{|u\cdot F_n \bigtriangleup F_n|}{|F_n|} =0
\]
for any given $u\in T$, then we say that $T$ is (left) amenable. Here $A\bigtriangleup B=(A\setminus B)\cup (B\setminus A)$ for two sets $A,B$. Equivalently, we have that $T$ is (left) amenable if and only if for any $\varepsilon>0$ and any finite set $E\subseteq T$, there is some finite set $F\subseteq T$ such that $\sup\nolimits_{u\in E}|u\cdot F \bigtriangleup F|/|F| < \varepsilon$.

Indeed, a monoid satisfying Axiom \uppercase\expandafter{\romannumeral2} is right reversible (see \cite[Page 194]{CP} for related definitions and results). In such situation, the structure of the monoid $S$ and that of the group $G$ are closely related. One may show that an integral monoid $S$ is (left) amenable if and only if its fractional group $G$ is amenable (see \cite{Gri} or \cite[(1.28)]{Pat} for example).

\begin{theorem} \label{thm_amenable_finitely_many_primes}
Let $G$ be a rational group with $S$ its natural monoid. Suppose that $S$ has finitely many primes. Then $G$ is amenable.
\end{theorem}

\begin{proof}
Let $k\geq 1$ and $\cP=\{p_0,p_1,\ldots,p_{k-1}\}$. For any $n\geq 0$, put
\[
F_n=\left\{\lcm[p_0^{m_0},p_1^{m_1},\ldots, p_{k-1}^{m_{k-1}}]:\,0\leq m_0,m_1,\ldots,m_{k-1}\leq n-1\right\}.
\]
We have $|F_n|=k^n$.

By theorem \ref{thm_finite_primes_implies_fully_castlable}, the monoid $S$ is fully castlable. Lemma \ref{lem_permutation_fully_castlable} shows that the map $\beta_{p_i}$ are bijective for $0\leq i\leq k-1$. Suppose that $\uline{p_i}p_j\rlh \beta_{p_i}^{-1}(p_j)\uline{t_{i,j,1}}$ for some $t_{i,j,1}\in \cP$. Applying Lemma \ref{lem_for_amenable_proof}, we obtain that
\[
\uline{p_i}p_j^{m_j}\rlh (\beta_{p_i}^{-1}(p_j))^{m_j}\uline{t_{i,j,m_j}}
\]
for some $t_{i,j,m_j}\in \cP$. Here $\beta_{p_i}^{-1}(p_i)=p_i$ and $t_{i,i,m_i}=p_i$. Then
\begin{align*}
p_i\cdot &\lcm[p_0^{m_0},p_1^{m_1},\ldots, p_{k-1}^{m_{k-1}}] =\lcm[p_ip_j^{m_j}:\, 0\leq j\leq k-1]\\
 &\quad \quad \quad =\lcm\left[(\beta_{p_i}^{-1}(p_j))^{m_j}t_{i,j,m_j}:\, 0\leq j\leq k-1\right].
\end{align*}
So the element
\[
\lcm\left[(\beta_{p_i}^{-1}(p_0))^{m_0},\ldots,(\beta_{p_{i}}^{-1}(p_{i-1}))^{m_i}, p_i^{m_i+1}, (\beta_{p_i}^{-1}(p_{i+1}))^{m_{i+1}},\ldots,(\beta_{p_{i}}^{-1}(p_{k-1}))^{m_{k-1}}\right]
\]
divides $p_i\cdot \lcm[p_0^{m_0},p_1^{m_1},\ldots, p_{k-1}^{m_{k-1}}]$. Moreover, both of these two elements have index $m_0+m_1+\ldots+m_{k-1}+1$. We conclude that
\begin{align*}
&p_i\cdot \lcm[p_0^{m_0},p_1^{m_1},\ldots, p_{k-1}^{m_{k-1}}] \\
&\quad =\lcm\left[(\beta_{p_i}^{-1}(p_0))^{m_0},\ldots,(\beta_{p_{i}}^{-1}(p_{i-1}))^{m_i}, p_i^{m_i+1}, (\beta_{p_i}^{-1}(p_{i+1}))^{m_{i+1}},\ldots,(\beta_{p_{i}}^{-1}(p_{k-1}))^{m_{k-1}}\right].
\end{align*}

Now we have
\begin{align*}
&F_n\setminus (p_i\cdot F_n) = \left\{\lcm[p_0^{m_0},p_1^{m_1},\ldots, p_{k-1}^{m_{k-1}}]\in F_n:\, m_i=0\right\},\\
&F_n\setminus (p_i^{-1}\cdot F_n) =\{\lcm[p_0^{m_0},p_1^{m_1},\ldots, p_{k-1}^{m_{k-1}}]\in F_n:\, m_i=n-1\}.
\end{align*}
Therefore
\begin{align*}
|(p_i\cdot F_n)\setminus F_n| = n^{k-1},\quad |F_n\setminus (p_i\cdot F_n)| = |(p_i^{-1}\cdot F_n)\setminus F_n|= n^{k-1}.
\end{align*}
Then, for $0\leq i\leq k$, we have
\[
\frac{\left|(p_i\cdot F_n)\bigtriangleup F_n\right|}{|F_n|} \leq \frac{2n^{k-1}}{n^{k-2}} = \frac{2}{n}\rightarrow 0,\quad (n\rightarrow \infty).
\]
It follows that $S$ is (left) amenable, and then $G$ is also amenable.
\end{proof}

For a homogenous monoid $S$ containing finitely many irreducible elements, is it (left) amenable?

\section{Construction of Castlings in Thompson's Monoid}
\label{section_Thomspon}

In this section, we will set up the system of castlings in Thompson's Monoid $\bS$. Indeed, the concrete constructions need different approaches. Let us forget those definitions and axioms appeared in Sections \ref{section_homogeneous_monoid}-\ref{section_natural_monoid} at this stage. We will define weak castlings, strong castlings and free castlings in another way. That is to say, we use same terminologies and notations, which are temporarily independent of those occurred previously. In Section \ref{subsection_verifying_axioms_thompson}, we will show that these new definitions of castlings in $\bS$ coincides with the previous ones and Axioms \uppercase\expandafter{\romannumeral4} and \uppercase\expandafter{\romannumeral5} are satisfied.

\subsection{Castling of Words}

We shall define castlings of words before we may define castlings of elements in $\bS$. Recall that when $u=q_1q_2\ldots q_k$ for some $q_1,q_2,\ldots, q_k\in \cP$, we call the right-hand side a word of $u$, and call each $q_t$ $(1\leq t\leq k)$ a letter of this word (we regard $q_t$ as a symbol instead of an element of $S$). For clarity, we will use capital letter to represent a word or its letters. In particular, we will always use $P_j$ to mean the only word of $p_j$ $(j=0,1,2,\ldots)$. If a word $Y$ consists of consecutive letters occurred in a word $X$, then we call $Y$ a subword of $X$. If $X=X_1X_2\ldots X_k$ with $X_j$ subwords of $X$ $(1\leq j\leq k)$, we say that $X_1X_2\ldots X_k$ is a subword-decomposition of $X$. The number of letters in a word $X$ is called the length of this word, which is denoted by $\ind(X)$. A subword-decomposition $X=X_1X_2\ldots X_k$ is said to be proper if $1<\ind(X_j)<\ind(X)$ for all $1\leq j\leq k$.

Applying the relation
\begin{equation} \label{eq_castle_relation_basic}
p_jp_i=p_ip_{j+1},\quad (0\leq i<j),
\end{equation}
it is possible to castle two letters to gain new letters and new words. We say an ordered pair of letters $P_i, P_j$ are castlable, or $P_i,P_j$ can be castled, or $P_i$ can be castled with $P_j$, when $i-j\neq -1$. When $P_i,P_j$ are castlable, we put
\[
\begin{cases}
\widetilde{j}=j, \,\widetilde{i}=i,\quad &\If i-j=0;\\
\widetilde{j}=j, \,\widetilde{i}=i+1,\quad &\If i-j\geq 1;\\
\widetilde{j}=j-1, \,\widetilde{i}=i,\quad &\If i-j\leq -2,
\end{cases}
\]
write $\underline{\underline{P_i}} P_j\rightleftharpoons P_{\widetilde{j}}\underline{\underline{P_{\widetilde{i}}}}$ and call it a castling of letters. For the empty word $\emptyset$ and any word $U$, both $\emptyset, U$ and $U,\emptyset$ are defined to be castlable, and $\underline{\underline{\emptyset}}U \rightleftharpoons U\underline{\underline{\emptyset}}$. Next, we define castling of an ordered pair of words $U,V$ by iteration according to the length of $U,V$. Suppose that castling has been defined for words $U,V$ of length $\ind(U)+\ind(V)\leq m-1$ for some $m\geq 3$. For non-empty words $U,V$ with $\ind(U)+\ind(V)=m$, at least one of $U,V$ has length no smaller than $2$, which ensures a proper subword-decomposition. We say $U,V$ are castlable, or $U,V$ can be castled, or $U$ can be castled with $V$, when at least one of the following two situations hold.

Type (\uppercase\expandafter{\romannumeral1}). If $U$ has a proper subword-decomposition $U=U_1U_2$ such that $\underline{\underline{U_2}} V \rightleftharpoons \widetilde{V}\underline{\underline{\widetilde{U_2}}}$ and $\underline{\underline{U_1}}\widetilde{V} \rightleftharpoons \widetilde{\widetilde{V}}\underline{\underline{\widetilde{U_1}}}$, then $U,V$ are castlable and we define the castling to be $\underline{\underline{U}}V \rightleftharpoons \widetilde{\widetilde{V}}\underline{\underline{\widetilde{U_1}\widetilde{U_2}}}$. For simplicity, we abbreviate the above expressions as
\[
\underline{\underline{U}}V = \underline{\underline{U_1U_2}}V \rlh \underline{\underline{U_1}} \widetilde{V} \underline{\underline{\widetilde{U_2}}} \rlh \widetilde{\widetilde{V}}\underline{\underline{\widetilde{U_1}\widetilde{U_2}}}.
\]

Type (\uppercase\expandafter{\romannumeral2}). If $V$ has a proper subword-decomposition $V=V_1V_2$ such that $\underline{\underline{U}}V_1 \rightleftharpoons \widetilde{V_1}\underline{\underline{\widetilde{U}}}$ and $\underline{\underline{\widetilde{U}}} V_2 \rightleftharpoons \widetilde{V_2}\underline{\underline{\widetilde{\widetilde{U}}}}$, then $U,V$ are also castlable and we define $\underline{\underline{U}}V \rightleftharpoons \widetilde{V_1}\widetilde{V_2}\underline{\underline{\widetilde{\widetilde{U}}}}$. For simplicity, we abbreviate the above expressions as
\[
\underline{\underline{U}}V = \underline{\underline{U}}V_1V_2 \rlh \widetilde{V_1}\underline{\underline{\widetilde{U}}}V_2 \rlh \widetilde{V_1}\widetilde{V_2}\underline{\underline{\widetilde{\widetilde{U}}}}.
\]

For words $U,U^\prime,V,V^\prime$, the expression $U=U^\prime$ means that they are the same words and $\underline{\underline{U}}V=\underline{\underline{U^\prime}}V^\prime$ means $U=U^\prime$ and $V=V^\prime$. The following lemma ensures that the notion of castling of words is well-defined, i,e., the definition of castling in two words $U, V$ does not depend on subword-decompositions, and does not depend on the castling type (\uppercase\expandafter{\romannumeral1}) or (\uppercase\expandafter{\romannumeral2}) either. Moreover, it is also shown that a castling of words can be decomposed into castlings of words according to any subword-decomposition.

\begin{lemma} \label{lem_word_trans_independent_of_word}
Suppose that $U,V$ are castlable, and $\underline{\underline{U}}V\rlh \widetilde{V}\underline{\underline{\widetilde{U}}}$ for some words $\widetilde{V},\widetilde{U}$. Then

(\romannumeral1) for any subword-decomposition $U=U_1^\prime U_2^\prime$, there are words $\widehat{U_1}^\prime,\widehat{U_2}^\prime,\widehat{V}^\prime$ with $\widehat{U_1}^\prime \widehat{U_2}^\prime=\widetilde{U}$ such that
\[
\underline{\underline{U_2^\prime}}V=\widehat{V}^\prime\underline{\underline{\widehat{U_2}^\prime}},\quad \underline{\underline{U_1^\prime}}\widehat{V}^\prime=\widetilde{V}\underline{\underline{\widehat{U_1}^\prime}};
\]

(\romannumeral2) for any subword-decomposition $V=V_1^\prime V_2^\prime$, there are words $\widehat{U}^\prime,\widehat{V_1}^\prime,\widehat{V_2}^\prime$ with $\widehat{V_1}^\prime \widehat{V_2}^\prime=\widetilde{V}$ such that
\[
\underline{\underline{U}}V_1^\prime=\widehat{V_1}^\prime\underline{\underline{\widehat{U}^\prime}},\quad \underline{\underline{\widehat{U}^\prime}}V_1^\prime=\widehat{V_2}^\prime\underline{\underline{\widetilde{U}}}.
\]
\end{lemma}

\begin{proof}
We use induction on $m=\ind(U)+\ind(V)$. For the cases $m\leq 2$ or $\ind(U)=0$ or $\ind(V)=0$, the results follows immediately. In the following, we always assume that $\ind(U),\ind(V)\geq 1$. Suppose that the lemma holds for $m\leq M-1$ with some $M\geq 3$. Now we consider the case $m=M$.

The fact that $U,V$ are castlable results from either type (\uppercase\expandafter{\romannumeral1}) or type (\uppercase\expandafter{\romannumeral2}). Without loss of generality, we deal with type (\uppercase\expandafter{\romannumeral1}) here. For type (\uppercase\expandafter{\romannumeral2}), similar arguments work. That is to say, there is some proper subword decomposition $U=U_1U_2$ and words $\widehat{U_1},\widehat{U_2},\widehat{V}$ with $\widehat{U_1}\widehat{U_2}=\widetilde{U}$ such that
\[
\underline{\underline{U}}V = \underline{\underline{U_1U_2}}V \rlh \underline{\underline{U_1}} \widehat{V} \underline{\underline{\widehat{U_2}}} \rlh \widetilde{V}\underline{\underline{\widehat{U_1}\widehat{U_2}}}=\widetilde{V}\underline{\underline{\widetilde{U}}}.
\]

(\romannumeral1) Without loss of generality, we assume that $U=U_1^\prime U_2^\prime$ is a proper subword-decomposition, and let us suppose that $\ind(U_1)>\ind (U_1^\prime)$ and
\[
U_1=XY,\quad U_2=Z,\quad U_1^\prime=X,\quad U_2^\prime=YZ
\]
for some non-empty subwords $X,Y,Z$ of $U$. Note that $\ind(U_1)+\ind(\widehat{V})< M$. By inductive hypothesis, the castling of $U_1$ and $\widehat{V}$ does not depend on the subword-decomposition and can be decomposed according to any subword-decomposition. For the subword-decomposition $U_1=XY$, a decomposition of castling of words gives
\[
\underline{\underline{U_1}}\widehat{V} = \underline{\underline{XY}}\widehat{V} \rlh \underline{\underline{X}}\breve{V} \underline{\underline{\breve{Y}}} \rlh \widetilde{V}\underline{\underline{\breve{X}\breve{Y}}}= \widetilde{V}\underline{\underline{\widehat{U_1}}}
\]
for some words $\breve{X},\breve{Y}$ and $\breve{V}$. Note that $\ind(U_2^\prime)+\ind(V)<M$. By inductive hypothesis, a composition of castlings of words leads to
\[
\underline{\underline{U_2^\prime}}V = \underline{\underline{ Y U_2}} V  \rlh \underline{\underline{Y}}\widehat{V}\underline{\underline{\widehat{U_2}}}   \rightleftharpoons \breve{V} \underline{\underline{\breve{Y}\widehat{U_2}}},\quad \underline{\underline{U_1^\prime}} \breve{V} =\underline{\underline{X}} \breve{V}  \rightleftharpoons \widetilde{V}\underline{\underline{\breve{X}}}.
\]
Putting $\widehat{V}^\prime=\breve{V}$, $\widehat{U_2}^\prime=\breve{Y}\widehat{U_2}$ and $\widehat{U_1}^\prime=\breve{X}$, we obtain $\widehat{U_1}^\prime\widehat{U_2}^\prime=\breve{X}\breve{Y}\widehat{U_2}=\widehat{U_1}\widehat{U_2}=\widetilde{U}$.
For the case $\ind(U_1)<\ind(U_1^\prime)$, similar arguments also hold. For the case $\ind(U_1)=\ind(u_1^\prime)$, we have $U_1=U_1^\prime$, $U_2=U_2^\prime$ and the expected result also follows.

(\romannumeral2) Without loss of generality, we assume that $V=V_1^\prime V_2^\prime$ is a proper subword-decomposition. Note that $\ind(U_2)+\ind(V)<M$. By inductive hypothesis, the castling of $U_2$ and $V$ does not depend on specific subword-decomposition and can be decomposed according to any subword-decomposition. A decomposition of castling gives
\[
\underline{\underline{U_2}}V=\underline{\underline{U_2}}V_1^\prime V_2^\prime \rlh \dot{V_1} \underline{\underline{\dot{U_2}}}V_2^\prime \rlh \dot{V_1} \dot{V_2} \underline{\underline{\widehat{U_2}}}=\widehat{V}\underline{\underline{\widehat{U_2}}}.
\]
for some words $\dot{V_1},\dot{V_2},\dot{U_2}$. By inductive hypothesis again, the castling of $U_1$ and $\widehat{V}= \dot{V_1} \dot{V_2}$ does not depend on subword-decompositions. So there are words $\ddot{V_1},\ddot{V_2},\dot{U_1}$ such that
\[
\underline{\underline{U_1}}\widehat{V} = \underline{\underline{U_1}} \dot{V_1} \dot{V_2} \rlh \ddot{V_1}\underline{\underline{\dot{U_1}}}\dot{V_2} \rlh \ddot{V_1}\ddot{V_2}\underline{\underline{\widehat{U_1}}}=\widetilde{V}\underline{\underline{\widehat{U_1}}}.
\]
Since $\ind(U)+\ind(V_1)<M$ and $\ind(\dot{U_1} \dot{U_2})+\ind(V_2)<M$, we deduce by inductive hypothesis that
\[
\underline{\underline{U}} V_1^\prime=\underline{\underline{U_1U_2}} V_1^\prime \rlh \underline{\underline{U_1}} \dot{V_1} \underline{\underline{\dot{U_2}}} \rightleftharpoons \ddot{V_1}\underline{\underline{ \dot{U_1} \dot{U_2}}},\quad \underline{\underline{\dot{U_1}\dot{U_2}}}  V_2^\prime \rlh \underline{\underline{\dot{U_1}}}\dot{V_2} \underline{\underline{\widehat{U_2}}} \rightleftharpoons \ddot{V_2} \underline{\underline{\widehat{U_1}\widehat{U_2}}}=\ddot{V_2} \underline{\underline{\widetilde{U}}}.
\]
Putting $\widehat{V_1}^\prime=\ddot{V_1}$, $\widehat{U}^\prime=\dot{U_1}\dot{U_2}$, $\widehat{V_2}^\prime= \ddot{V_2}$, we have $\widehat{V_1}^\prime \widehat{V_2}^\prime = \ddot{V_1}\ddot{V_2}= \widetilde{V}$. The proof is completed.
\end{proof}

The above lemma ensures that the castling of two words $U,V$ is well-defined. Now for a castlable pair of words $U,V$ and a subword decomposition $U=U_1U_2\ldots U_k$, we write
\[
\underline{\underline{U}}V = \underline{\underline{U_1U_2\ldots U_k}}V \rlh \underline{\underline{U_1U_2\ldots U_{k-1}}} V_1 \underline{\underline{\widetilde{U_k}}} \rlh \ldots \rlh  \underline{\underline{U_1}} V_{k-1} \underline{\underline{\widetilde{U_2}\widetilde{U_3}\ldots \widetilde{U_k}}} \rlh  V_k\underline{\underline{\widetilde{U_1}\widetilde{U_2}\ldots \widetilde{U_k}}}
\]
for some words $\widetilde{U_1},\ldots, \widetilde{U}_k$ and $V_1,\ldots, V_k$. The double underline is used to recognize the final words from the initial words during the castling. We emphasis that the pair of words $U,V$ involved are ordered. A castling in $U,V$ does not ensure that $V,U$ can be castled. However, one can prove by induction that if $\underline{\underline{U}}V \rlh \widetilde{V} \underline{\underline{\widetilde{U}}}$, then $\underline{\underline{\widetilde{V}}} \widetilde{U} \rlh U\underline{\underline{V}}$. Indeed, the following four expressions are regarded same.
\[
\underline{\underline{U}}V \rlh \widetilde{V} \underline{\underline{\widetilde{U}}}, \quad U\underline{\underline{V}} \rlh  \underline{\underline{\widetilde{V}}}\widetilde{U},\quad \widetilde{V} \underline{\underline{\widetilde{U}}} \rlh \underline{\underline{U}}V,\quad \underline{\underline{\widetilde{V}}}\widetilde{U} \rlh U\underline{\underline{V}}.
\]
Moreover, it is easy to see by induction that $\ind(U)=\ind(\widetilde{U})$ and $\ind(V)=\ind(\widetilde{V})$.
We end this subsection with the following interesting example.

\begin{example}
Let $U=P_2P_3$ and $V=P_2P_4$. Then
\[
\underline{\underline{U}}V = \underline{\underline{P_2P_3}}P_2P_4 \rlh \underline{\underline{P_2}}P_2P_4 \underline{\underline{P_4}} \rlh P_2P_3\underline{\underline{P_2P_4}} = U\underline{\underline{V}}.
\]
\end{example}

\subsection{Order Preserving of Words in Castlings}

For $u\in \bS$, we use $\fW(u)$ to denote the set of all words of $u$. For two words $U,U^\prime$ of an element $u\in \bS$, they can be transformed into normal form of $u$ by castling a pair of adjacent castlable letters for finitely many times. So $U$ can also be transformed into $U^\prime$ by such castlings of adjacent letters. Now we establish a partial order on all words of a given element.

Let $u\in S$. If $U,U^\prime$ words of $u$ that are same, then we write $U=U^\prime$ as previous. Consider the situation that $U$ and $U^\prime$ differ from exactly one castling of a pair of adjacent castlable letters. Write $U=XP_iP_jY$ and $U^\prime=XP_{\widetilde{j}}P_{\widetilde{i}}Y$ with $\underline{\underline{P_i}}P_j \rlh P_{\widetilde{j}}\underline{\underline{P_{\widetilde{i}}}}$ for some $i,j\geq 0$. We define
\[
\begin{cases}
U=U^\prime,  \quad &\If i-j=0;\\
U \prec U^\prime,\quad &\If i-j\geq 1;\\
U \succ U^\prime,\quad &\If i-j\leq -2.
\end{cases}
\]
For a word $X=P_{j_1}P_{j_2}\ldots P_{j_k}$, let $\Sigma(X)=\sum\nolimits_{i=1}^k j_i$. Note that when $U\prec U^\prime$, one always has $\Sigma(U)<\Sigma(U^\prime)$. Now suppose that $U,U^\prime,U^{\prime\prime}$ are three words of $u$ satisfying $U\prec U^\prime$ and $U^\prime \prec U^{\prime\prime}$. Then $\Sigma(U)<\Sigma(U^{\prime\prime})$. The situations $U=U^{\prime\prime}$ and $U\succ U^{\prime\prime}$ never happen. So, it is reasonable to define $U\prec U^{\prime\prime}$ in this case. We write $X\preccurlyeq Y$ (and $X\succcurlyeq Y$) if either $X=Y$ or $X\prec Y$ (and $X\succ Y$, respectively). It is not hard to see that ``$\preccurlyeq$'' can be extended to a partial order on $\fW(u)$.

For $u\in \bS$, we use $U_\sharp$ to denote the word of $u$ in normal form. It is not hard to see that $U\preccurlyeq U_\sharp$ for any $U\in \fW(u)$. So we call $U_\sharp$ the maximum word of $u$.

\begin{example}
Let $u=p_2p_4p_6$. Words of $u$ are listed below.
\[
P_4P_3P_2 \quad \preccurlyeq \quad \left.
\begin{matrix}
P_3P_5P_2 & &\preccurlyeq &  & P_3P_2P_6 \\
P_4P_2P_4 & &\preccurlyeq &  & P_2P_5P_4
\end{matrix}
\right. \quad \preccurlyeq \quad P_2P_4P_6.
\]
\end{example}

Let $U,U^\prime$ be words of $u$ such that $U\preccurlyeq U^\prime$. Let $X,Y$ be words of $x,y$, respectively. It is not hard to see that $XU \preccurlyeq XU^\prime$ as words of $xu$, and $UY \preccurlyeq U^\prime Y$ as words of $uy$. This partial order is defined on all words of a given element. Whenever we write $U \preccurlyeq U^\prime$, we always mean that $U,U^\prime$ represent the same element. The following proposition is a key for constructing castlings in Thompson's monoid.

\begin{proposition} \label{prop_element_trans_independent_of_word}
Let $u,v\in S$. Suppose that $U$ is a word of $u$ and $V$ is a word of $v$ such that $\underline{\underline{U}} V \rightleftharpoons \widetilde{V} \underline{\underline{\widetilde{U}}}$ for some words $\widetilde{U},\widetilde{V}$. Then for any word $U^\prime$ of $u$ and any word $V^\prime$ of $v$ with $V^\prime \succcurlyeq V$, the words $U^\prime$ and $V^\prime$ are castlable.

Moreover, write $\underline{\underline{U^\prime}} V^\prime \rightleftharpoons  \widetilde{V}^\prime\underline{\underline{\widetilde{U}^\prime}}$ for some words $\widetilde{U}^\prime,\widetilde{V}^\prime$. Then $\widetilde{V}$ and $\widetilde{V}^\prime$ are words of the same element in $S$, and so are $\widetilde{U}$ and $\widetilde{U}^\prime$.

Furthermore, we have $\widetilde{V}^\prime \succcurlyeq \widetilde{V}$. We also have $\widetilde{U}^\prime \succcurlyeq \widetilde{U}$ if and only if $U^\prime \succcurlyeq U$, and $\widetilde{U}^\prime \preccurlyeq \widetilde{U}$ if and only if $U^\prime \preccurlyeq U$.
\end{proposition}

We will prove Proposition \ref{prop_element_trans_independent_of_word} after several more lemmas.

\begin{lemma} \label{lem_U_2_V_1}
Proposition \ref{prop_element_trans_independent_of_word} is true when $\ind(u)=2$ and $\ind(v)=1$.
\end{lemma}

\begin{proof}
Let $u=p_ip_j$ in normal form and $v=p_k$, where $i,j,k\geq 0$. Then $v$ has only one word $P_k$. Note that $u$ has two different words if and only if $i-j\leq -2$. Write $U=P_iP_j$ and $U^\prime= P_{j-1}P_i$. Then $U\succ U^\prime$. We list all the possibilities below.

\begin{itemize}
\item For $j-k=-1$, neither $U$ nor $U^\prime$ is castlable with $P_k$.

\item For $j-k=0$, one has $\underline{\underline{U}}P_k \rlh P_{k-1}\underline{\underline{P_iP_j}}$ and $\underline{\underline{U^\prime}} P_k  \rlh P_{k-1}\underline{\underline{P_{j-1}P_i}}$. Here $P_iP_j$ and $P_{j-1}P_i$ are words of the same element and $P_iP_j \succ P_{j-1}P_i$.

\item For $j-k\leq -2$, it satisfies $i-k\leq -4$. Then $\underline{\underline{U}} P_k \rlh P_{k-2}\underline{\underline{P_iP_j}}$ and $\underline{\underline{U^\prime}} P_k \rlh P_{k-2}\underline{\underline{P_{j-1}P_i}}$. Here $P_iP_j$ and $P_{j-1}P_i$ are words of a same element and $P_iP_j \succ P_{j-1}P_i$.

\item For $j-k\geq 1$ and $k=i+1$, neither $U$ nor $U^\prime$ is castlable with $P_k$.

\item For $j-k\geq 1$ and $k=i$, one has $\underline{\underline{U}} P_k \rlh P_i\underline{\underline{P_iP_{j+1}}}$ and $\underline{\underline{U^\prime}} P_k \rlh P_i\underline{\underline{P_jP_i}}$. Here $P_iP_{j+1}$ and $P_jP_i$ represent the same element and $P_iP_{j+1} \succ P_jP_i$.

\item For $j-k\geq 1$ and $k\leq i-1$, one has $\underline{\underline{U}} P_k \rlh P_k\underline{\underline{P_{i+1}P_{j+1}}}$ and $\underline{\underline{U^\prime}} P_k \rlh P_k \underline{\underline{P_jP_{i+1}}}$. Here $P_{i+1}P_{j+1}$ and $P_jP_{i+1}$ are words of a same element and $P_{i+1}P_{j+1} \succ P_j P_{i+1}$.

\item For $j-k\geq 1$ and $k\geq i+2$, one has $\underline{\underline{U}} P_k \rlh P_{k-1} \underline{\underline{P_iP_{j+1}}}$ and $\underline{\underline{U^\prime}} P_k \rlh P_{k-1}\underline{\underline{P_jP_i}}$. Here $P_iP_{j+1}$ are $P_jP_i$ represent the same element and $P_iP_{j+1} \succ P_j P_i$.
\end{itemize}
The lemma now follows.
\end{proof}

\begin{lemma} \label{lem_U_1_V_2}
Proposition \ref{prop_element_trans_independent_of_word} is true when $\ind(u)=1$ and $\ind(v)=2$.
\end{lemma}

\begin{proof}
Let $v=p_jp_k$ in normal form and $u=p_i$ a prime. Then $u$ has only one word $P_i$. Note that  $v$ has two different words if and only if $j-k\leq -2$. Write $V=P_jP_k$ and $V^\prime= P_{k-1}P_j$. Then $V\succ V^\prime$. We list all the possibilities below.

\begin{itemize}
\item For $i-j=-1$, the letter $P_i$ is castlable with neither $V$ nor $V^\prime$.

\item For $i-j=0$ and $j-k=-2$, one has $\underline{\underline{P_i}} V \rlh P_jP_{k-1}\underline{\underline{P_i}}$, but $P_i$ is not castlable with $V^\prime$.

\item For $i-j=0$ and $j-k\leq -3$, one has $\underline{\underline{P_i}} V \rlh P_jP_{k-1}\underline{\underline{P_i}}$ and $\underline{\underline{P_i}}V^\prime \rlh  P_{k-2}P_j \underline{\underline{P_i}}$. Here $P_jP_{k-1}$ and $P_{k-2}P_j$ are words of a same element and $P_jP_{k-1} \succ P_{k-2}P_j$.

\item For $i-j \leq -2$, one has $i-k\leq -4$. Then $\underline{\underline{P_i}}V \rlh P_{j-1}P_{k-1}\underline{\underline{P_i}}$ and $\underline{\underline{P_i}}V^\prime \rlh P_{k-2}P_{j-1}\underline{\underline{P_i}}$. Here $P_{j-1}P_{k-1}$ and $P_{k-2}P_{j-1}$ are words of a same element and $P_{j-1}P_{k-1} \succ P_{k-2}P_{j-1}$.

\item For $i-j\geq 1$ and $i=k-2$, the letter $P_i$ is castlable with neither $V$ nor $V^\prime$.

\item For $i-j\geq 1$ and $i=k-1$, one has $\underline{\underline{P_i}}V \rlh P_jP_k \underline{\underline{P_{i+1}}}$ and $\underline{\underline{P_i}}V^\prime \rlh P_{k-1}P_j \underline{\underline{P_{i+1}}}$. Here $P_jP_k$ and $P_{k-1}P_j$ represent the same element and $P_jP_k \succ P_{k-1}P_j$.

\item For $i-j\geq 1$ and $i\leq k-3$, one has $j\leq k-4$. Then $\underline{\underline{P_i}}V \rlh P_jP_{k-1}\underline{\underline{P_{i+1}}}$ and $\underline{\underline{P_i}} V^\prime \rlh P_{k-2}P_j \underline{\underline{P_{i+1}}}$. Here $P_jP_{k-1}$ and $P_{k-2}P_j$ represent the same element and $P_jP_{k-1} \succ P_{k-2}P_j$.

\item For $i-j\geq 1$ and $i\geq k$, one has $\underline{\underline{P_i}} V \rlh P_jP_k \underline{\underline{P_{i+2}}}$ and $\underline{\underline{P_i}} V^\prime \rlh P_{k-1}P_j\underline{\underline{P_{i+2}}}$. Here $P_jP_k$ and $P_{k-1}P_j$ represent the same element and $P_jP_k \succ P_{k-1}P_j$.
\end{itemize}
The proof is completed.
\end{proof}

\begin{lemma} \label{lem_U_2_V_2}
Proposition \ref{prop_element_trans_independent_of_word} is true when $\ind(u)=\ind(v)=2$.
\end{lemma}

\begin{proof}
When $u,v$ both have only one word respectively, the proof is trivial. Suppose that $v$ has two different words $V,V^\prime$ with $V^\prime \succcurlyeq V$. Let $U$ be a word of $u$ such that $U,V$ are castlable. Write $U=U_1U_2$, where $U_1$, $U_2$ are both letters. Then
\[
\underline{\underline{U}}V = \underline{\underline{U_1U_2}}V \rlh \underline{\underline{U_1}}\widehat{V}\underline{\underline{\widehat{U_2}}} \rlh \widehat{\widehat{V}}\underline{\underline{\widehat{U_1}\widehat{U_2}}}
\]
for some words $\widehat{U_1},\widehat{U_2},\widehat{V}$ and $\widehat{\widehat{V}}$. Since $V^\prime \succcurlyeq V$, one deduces by Lemma \ref{lem_U_1_V_2} that $U_2$ and $V^\prime$ are also castlable. Write $\underline{\underline{U_2}}V^\prime \rlh \widehat{V}^\prime \underline{\underline{\widehat{U_2}}}$ for some word $\widehat{V}^\prime$. Here $\widehat{V}, \widehat{V}^\prime$ represent a same element and $\widehat{V}^\prime \succcurlyeq \widehat{V}$. Since $U_1$ and $\widehat{V}$ are castlable, one also deduce that $U_1$ and $\widehat{V}^\prime$ are castlable by Lemma \ref{lem_U_1_V_2}. Write $\underline{\underline{U_1}}\widehat{V}^\prime \rlh \widehat{\widehat{V}}^\prime \underline{\underline{\widehat{U_1}}}$ for some word $\widehat{\widehat{V}}^\prime$. Here $\widehat{\widehat{V}}, \widehat{\widehat{V}}^\prime$ represent a same element and $\widehat{\widehat{V}}^\prime\succcurlyeq \widehat{\widehat{V}}$. Now we conclude that
\[
\underline{\underline{U}}V^\prime = \underline{\underline{U_1U_2}}V^\prime \rlh \underline{\underline{U_1}}\widehat{V}^\prime\underline{\underline{\widehat{U_2}}} \rlh \widehat{\widehat{V}}^\prime\underline{\underline{\widehat{U_1}\widehat{U_2}}},
\]
where $\widehat{\widehat{V}}, \widehat{\widehat{V}}^\prime$ represent a same element and $\widehat{\widehat{V}}^\prime\succcurlyeq \widehat{\widehat{V}}$.

Next, by applying Lemma \ref{lem_U_2_V_1} and similar arguments as above, the following conclusion holds. Suppose that $u$ has two different words $U,U^\prime$ and $V$ is a word of $v$ such that $U,V$ are castlable with $\underline{\underline{U}}V \rlh \widehat{V}\underline{\underline{\widehat{U}}}$. Then $U^\prime$ and $V$ are also castlable, which we denote $\underline{\underline{U^\prime}}V \rlh \widehat{V}\underline{\underline{\widehat{U}^\prime}}$ for some $\widehat{U}^\prime$. Moreover, the words $\widehat{U},\widehat{U}^\prime$ represent a same element. We have $\widehat{U} \preccurlyeq \widehat{U}^\prime$ if and only if $U \preccurlyeq U^\prime$, and $\widehat{U} \succcurlyeq \widehat{U}^\prime$ if and only if $U \succcurlyeq U^\prime$.

Finally, let us suppose that $U,U^\prime$ are two different words of $u$, and $V,V^\prime$ are two different words of $v$ with $V^\prime \succcurlyeq V$, and also suppose that $\underline{\underline{U}}V \rlh \widehat{V}\underline{\underline{\widehat{U}}}$. Without loss of generality, we assume that $U^\prime \succcurlyeq U$. By above discussions in the first paragraph, the words $U$ and $V^\prime$ are castlable. Write $\underline{\underline{U}}V^\prime \rlh \widehat{V}^\prime\underline{\underline{\widehat{U}}}$, where $\widehat{V},\widehat{V}^\prime$ are words of a same element and $\widehat{V}^\prime \succcurlyeq \widehat{V}$. By above discussions in the second paragraph, the words $U^\prime$ and $V^\prime$ are castlable. Write $\underline{\underline{U^\prime}}V^\prime \rlh \widehat{V}^\prime\underline{\underline{\widehat{U}^\prime}}$, where $\widehat{U},\widehat{U}^\prime$ are words of a same element and $\widehat{U}^\prime \succcurlyeq \widehat{U}$. This completes the proof.

\end{proof}

Now we are ready to prove Lemma \ref{prop_element_trans_independent_of_word}.

\begin{proof} [Proof of Proposition \ref{prop_element_trans_independent_of_word}]
For $\ind(u)=0$ or $\ind(v)=0$, the proof is trivial. For $\ind(u),\ind(v)\leq 2$, the result follows from Lemmas \ref{lem_U_2_V_1}, \ref{lem_U_1_V_2} and \ref{lem_U_2_V_2}.  In the following, we always assume that either $\ind(u)\geq 3$, $\ind(v)\geq 1$, or $\ind(u)\geq 1$, $\ind(v)\geq 3$.

Suppose that the lemma has been proved for $\ind(u)\leq m-1$ and $\ind(v)\leq n$ with some $m\geq 3, \, n\geq 1$. We proceed with $\ind(u)=m$ and $\ind(v)\leq n$. It is sufficient to deal with the condition that $U^\prime$ differ from $U$ by exactly one castling of a pair of adjacent castlable letters.

\textsc{Case 1.} We consider the case that the adjacent letters that are castled are the first two letters of $U$. Write $U=Q_1Q_2U_1$ and $U^\prime = Q_2^\prime Q_1^\prime U_1$, where $\ind(U_1)=\ind(U)-2$ and $Q_1,Q_2,Q_1^\prime,Q_2^\prime$ are letters such that $\underline{\underline{Q_1}} Q_2\rlh Q_2^\prime \underline{\underline{Q_1^\prime}}$. Without loss of generality, we assume that $ Q_2^\prime Q_1^\prime\succcurlyeq Q_1Q_2$ (for $ Q_2^\prime Q_1^\prime\preccurlyeq Q_1Q_2$, similar arguments hold as well) and $U^\prime \succcurlyeq U$. Inserting the subword-decomposition $U=Q_1Q_2U_1$ into the castling of $U,V$, we obtain
\[
\underline{\underline{U}}V=\underline{\underline{(Q_1Q_2)U_1}}V \rlh \underline{\underline{Q_1Q_2}}\check{V} \underline{\underline{\check{U_1}}} \rlh \widetilde{V}\underline{\underline{\check{Q_1}\check{Q_2}\check{U_1}}} = \widetilde{V}\underline{\underline{\widetilde{U}}}
\]
for some words $\check{Q_1},\check{Q_2},\check{U_1},\check{V}$. Note that $1\leq \ind(U_1)<m$, the inductive hypothesis says that the castling does not depend on the words chosen. So for $V^\prime \succcurlyeq V$, we also have $\underline{\underline{U_1}}V^\prime \rlh \check{V}^\prime \underline{\underline{\check{U_1}}}$ for some words $\check{V}^\prime$. Here $\check{V}$ and $\check{V}^\prime$ are words of the same element and $\check{V}^\prime\succcurlyeq \check{V}$. Note that $\ind(Q_1Q_2)=2<m$ and $Q_1Q_2, Q_2^\prime Q_1^\prime$ are words of a same element. By inductive hypothesis again, we deduce that $Q_2^\prime Q_1^\prime$ and $\check{V}^\prime$ are also castlable. Write $\underline{\underline{Q_2^\prime Q_1^\prime}} \check{V}^\prime \rlh \widetilde{V}^\prime \underline{\underline{\check{Q_2}^\prime\check{Q_1}^\prime}}$, where $\widetilde{V}$ and $\widetilde{V}^\prime$ represent the same element and so does $\check{Q_1}\check{Q_2}$ and $\check{Q_1}^\prime\check{Q_2}^\prime$. Moreover, one has $\widetilde{V}^\prime \succcurlyeq \widetilde{V}$ and $\check{Q_2}^\prime\check{Q_1}^\prime \succcurlyeq Q_1Q_2$. To sum up, we have
\[
\underline{\underline{U^\prime}} V^\prime = \underline{\underline{(Q_2^\prime Q_1^\prime) U_1}}V^\prime \rlh \underline{\underline{Q_2^\prime Q_1^\prime}} \check{V}^\prime \underline{\underline{\check{U_1}}} \rlh \widetilde{V}^\prime \underline{\underline{\check{Q_2}^\prime\check{Q_1}^\prime \check{U_1}}} = \widetilde{V}^\prime \underline{\underline{\widetilde{U}^\prime}},
\]
where $\widetilde{V}$ and $\widetilde{V}^\prime$ represent the same element, and so does $\check{Q_1}\check{Q_2}\check{U_1}$ and $\check{Q_2}^\prime\check{Q_1}^\prime\check{U_1}$. Moreover, one has $\widetilde{V}^\prime \succcurlyeq \widetilde{V}$ and $\widetilde{U}^\prime=\check{Q_2}^\prime\check{Q_1}^\prime\check{U_1} \succcurlyeq Q_1Q_2\check{U_1}=\widetilde{U}$.

Second, we consider the case that the adjacent letters that are castled do not involve the first letter of $U$. Write the subword-decomposition $U=QU_2$ with $\ind(Q)=1$ and $U^\prime=QU_2^\prime$ after castling the adjacent letters. Note that $U_2$ and $U_2^\prime$ represent the same element. Without loss of generality, we still assume that $U^\prime \succcurlyeq U$, i.e., $U_2^\prime \succcurlyeq U_2$. By the castling of $U,V$, one obtains
\[
\underline{\underline{U}}V = \underline{\underline{QU_2}}V \rlh \underline{\underline{Q}} \breve{V} \underline{\underline{\breve{U_2}}} \rlh \widetilde{V}\underline{\underline{\breve{Q}\breve{U_2}}} = \widetilde{V}\underline{\underline{\widetilde{U}}}
\]
for some words $\breve{Q},\breve{U_2},\breve{V}$. Note that $\ind(U_2)<m$. For $V^\prime \succcurlyeq V$, it follows from inductive hypothesis that $U_2^\prime$ and $V^\prime$ are castlable. Write $\underline{\underline{U_2^\prime}}  V^\prime \rlh \breve{V}^\prime \underline{\underline{\breve{U_2}^\prime}}$, where $\breve{V}$ and $\breve{V}^\prime$ are words of same element and so are $\breve{U_2}$ and $\breve{U_2}^\prime$. Moreover, one has $\breve{V}^\prime \succcurlyeq \breve{V}$ and $\breve{U_2}^\prime \succcurlyeq \breve{U_2}$. By inductive hypothesis again, one deduces that $Q$ and $\breve{V}^\prime$ are also castlable. Write $\underline{\underline{Q}}\breve{V}^\prime \rlh  \widetilde{{V}}^\prime\underline{\underline{\breve{Q}}}$, where $\widetilde{{V}}$ and $\widetilde{{V}}^\prime$ represent the same element and $\widetilde{{V}}^\prime\succcurlyeq \widetilde{{V}}$. Now we have
\[
\underline{\underline{U^\prime}}V^\prime = \underline{\underline{QU_2^\prime}}V^\prime \rlh \underline{\underline{Q}} \breve{V}^\prime \underline{\underline{\breve{U_2}^\prime}} \rlh \widetilde{{V}}^\prime\underline{\underline{\breve{Q}\breve{U_2}^\prime}}=\widetilde{{V}}^\prime\underline{\underline{\widetilde{U}^\prime}}
\]
where $\widetilde{{V}}$ and $\widetilde{{V}}^\prime$ represent the same element and so are $\breve{Q}\breve{U_2}$ and $\breve{Q}\breve{U_2}^\prime$. Moreover, one has $\widetilde{{V}}^\prime\succcurlyeq \widetilde{{V}}$ and $\widetilde{U}^\prime=\breve{Q}\breve{U_2}^\prime\succcurlyeq\breve{Q}\breve{U_2}=\widetilde{U}$.

One the other hand, suppose that the lemma has been proved for $\ind(u)\leq m$ and $\ind(v)\leq n-1$ with some $m\geq 1, \, n\geq 3$. We proceed with $\ind(u)\leq m$ and $\ind(v)= n$. Similarly, it is sufficient to deal with the condition that $V^\prime$ differ from $V$ by exactly one castling of a pair of adjacent castlable letters and $V^\prime \succcurlyeq V$. Similar arguments as above also work.

By induction, the proposition follows.
\end{proof}

\subsection{Existence of a Minimum Word}

We say a word $U$ of $u$ minimal, if $U^\prime \preccurlyeq U$ implies $U^\prime=U$ for any $U^\prime\in \fW(u)$. We say a word $U$ of $u$ minimum if $U^\prime \succcurlyeq U$ for all $U^\prime \in \fW(u)$. Since $\fW(u)$ is a finite set, a minimal word always exists. If a minimum word exists, then it is minimal.

For a minimal word, one can verify the following lemma immediately.

\begin{lemma} \label{lem_condition_for_minimal}
Let $U=P_{j_1}P_{j_2}\ldots P_{j_k}$ be word of $u$. Then $U$ is a minimal word if and only if $j_r-j_{r+1}\geq -1$ for all $1\leq r\leq k-1$.
\end{lemma}

\begin{corollary}
Let $U=U_1U_2$. If $U$ is minimal, then $U_1,U_2$ are also minimal.
\end{corollary}

\begin{lemma} \label{lem_minimal_word_castling}
Let $t\geq 2$. Suppose that $P_{i_1}\ldots P_{i_{t-1}}P_{i_t}$ is a minimal word. And suppose that $\underline{\underline{P_{i_1}\ldots P_{i_{t-1}}}}P_{i_t} \rlh P_k \underline{\underline{Y}}$ for some word $Y$. Then $i_1\geq i_t=k$.
\end{lemma}

\begin{proof}
By Lemma \ref{lem_condition_for_minimal}, one has $i_r-i_{r+1}\geq -1$ for $1\leq r\leq t-1$. We use induction on $t$. When $t=2$, the castling of $P_{i_1}$ and $P_{i_2}$ shows that either $i_1=i_2$ or $i_1>i_2$. In the former case, one has $\underline{\underline{P_{i_1}}}P_{i_1}\rlh P_{i_1}\underline{\underline{P_{i_1}}}$, which leads to $i_1=i_2=k$. In the latter case, one has $\underline{\underline{P_{i_1}}}P_{i_2}\rlh P_{i_2}\underline{\underline{P_{i_1+1}}}$, which results in $i_1>i_2=k$. As a result, we have $i_1\geq i_2=k$.

Assume that the lemma has been proved for $t\leq T-1$ for some $T\leq 3$. For $t=T$, it follows from $\underline{\underline{P_{i_1}\ldots P_{i_{T-1}}}}P_{i_T} \rlh P_k \underline{\underline{Y}}$ that the letters $P_{i_{T-1}},P_{i_T}$ are castlable. Write $\underline{\underline{P_{i_{T-1}}}}P_{i_T} \rlh
P_{k^\prime}\underline{\underline{Y^\prime}}$ for some words $P_{k^\prime}$ and $Y^\prime$. Since $P_{i_{T-1}}P_{i_T}$ is a minimal word, the inductive hypothesis shows that $i_{T-1}\geq i_T =k^\prime$. Now one has $i_{T-2} - k^\prime \geq i_{T-2}-i_{T-1}\geq -1$. So $P_{i_1}\ldots P_{i_{T-2}}P_{k^\prime}$ is also a minimal word. Write
\[
\underline{\underline{(P_{i_1}P_{i_2}\ldots P_{i_{T-2}})P_{i_{T-1}}}}P_{i_T} \rlh
\underline{\underline{P_{i_1}P_{i_2}\ldots P_{i_{T-2}}}}P_{k^\prime}\underline{\underline{Y^\prime}} \rlh P_k \underline{\underline{Y^{\prime\prime}Y^\prime}} = P_k \underline{\underline{Y}}
\]
for some word $Y^{\prime\prime}$. By applying inductive hypothesis, we obtain that $i_1\geq k^\prime = k$. So $i_1\geq i_T = k$. The lemma follows by induction.
\end{proof}

\begin{lemma} \label{lem_divisor_shift_is_castling}
Let $x\in\bS$ and $p\in \cP$. Suppose that $p|x$. Write $P$ for the only word of $p$. Then for any word $X$ of $x$, there exists a subword-decomposition $X=YQZ$ such that $Q$ is a letter and $\underline{\underline{Y}}Q \rlh P \underline{\underline{\widetilde{Y}}}$ for some word $\widetilde{Y}$.
\end{lemma}

\begin{proof}
Write $x=pv$, let $V$ be any word of $v$ and put $X_0=PV$. Notice that $X_0$ can be transformed into $X$ by castling a pair of adjacent castlable letters for finitely many times. Write $Y_0=\emptyset$, $Q_0=P$ and $Z_0=V$. Then $X_0=Y_0Q_0Z_0$ and $\underline{\underline{\emptyset}} P \rlh P \underline{\underline{\emptyset}}$. We use iterations on number of castlings of letters involved. Suppose the lemma gives similar results after $k$ times of castlings with some $k\geq 1$, i.e., one obtains a word $X_k=Y_kQ_kZ_k$ of $x$ with $\underline{\underline{Y_k}}Q_k \rlh P \underline{\underline{\widetilde{Y_k}}}$ for some word $\widetilde{Y_k}$. Now we apply another castling of a pair of adjacent letters.

\textsc{Case 1.} If the two letters castled are letters of $Y_k$, then we get a new word $Y_{k+1}$ that represents the same element as $Y_k$ does, and $X_{k+1}=Y_{k+1}Q_kZ_k$. Set $Q_{k+1}=Q_k$ and $Z_{k+1}=Z_k$. By Lemma \ref{prop_element_trans_independent_of_word}, we have $\underline{\underline{Y_{k+1}}}Q_{k+1} \rlh P \underline{\underline{\widetilde{Y_{k+1}}}}$ for some word $\widetilde{Y_{k+1}}$ that represent the same element as $\widetilde{Y_{k}}$ does.

\textsc{Case 2.} If the two letters castled are letters of $Z_k$, then we get a new word $Z_{k+1}$ that represents the same element as $Z_k$ does, and $X_{k+1}=Y_kQ_kZ_{k+1}$. Set $Q_{k+1}=Q_k$ and $Y_{k+1}=Y_k$. The conclusion also follows.

\textsc{Case 3.} If the two letters castled are the last letter of $Y_k$ and $Q_k$. Write $Y_k=Y_k^\prime R$ with $\ind(R)=1$ and $\underline{\underline{R}}Q_k \rlh \widetilde{Q}\underline{\underline{\widetilde{R}}}$ for some letters $\widetilde{Q}$, $\widetilde{R}$. Then $X_{k+1}=Y_k^\prime \widetilde{Q}\widetilde{R}Z_k$. It follows that $Y_k^\prime$ is castlable with $\widetilde{Q}$, and $\underline{\underline{Y_k^\prime}}\widetilde{Q}\rlh P\underline{\underline{\widetilde{Y_k}^{\prime}}}$ with $\widetilde{Y_k}= \widetilde{Y_k}^\prime \widetilde{R}$. Now we set $Y_{k+1}=Y_k^\prime$, $Q_{k+1}= \widetilde{Q}$ and $Z_{k+1}=\widetilde{R}Z_k$. The conclusion follows.

\textsc{Case 4.} If the two letters castled are $Q_k$ and the first letter of $Z_k$. Write $Z_k=R Z_k^\prime$ with $\ind(R)=1$ and $Q_k\underline{\underline{R}} \rlh \underline{\underline{\widetilde{R}}}\widetilde{Q}$ for some letters $\widetilde{Q}$, $\widetilde{R}$. Then $X_{k+1}=Y_k\widetilde{R}\widetilde{Q}Z_k^\prime$. Now we set $Y_{k+1}=Y_k \widetilde{R}$, $Q_{k+1}= \widetilde{Q}$, $Z_{k+1}=Z_k^\prime$. Then
\[
\underline{\underline{Y_{k+1}}}Q_{k+1} = \underline{\underline{Y_k\widetilde{R}}}\widetilde{Q} \rlh \underline{\underline{Y_k}} Q_k \underline{\underline{R}} \rlh P \underline{\underline{\widetilde{Y_k}R}}.
\]

Suppose that $X_0$ is transformed to $X$ at the step $K$. We put $Y=Y_K$, $Q=Q_K$ and $Z=Z_K$. Then the lemma follows.
\end{proof}

\begin{theorem}
Any element $u$ in $S$ has a minimum word $U_\flat$.
\end{theorem}

\begin{proof}
We use induction on $\ind(u)$. For $\ind(u)\leq 1$, the proof is trivial. Suppose the theorem has been proved for $\ind(u)\leq m-1$ with some $m\geq 2$. Now we consider the case $\ind(u)=m$. Let
\begin{equation} \label{eq_min_subscript_prime_co-divisor}
k= \max\{l:\, p_l | u\}.
\end{equation}
Write $u=p_k w$. By inductive hypothesis, there is a minimum word $W_\flat=P_{j_2}P_{j_3}\ldots P_{j_{m}}$ of $w$. By Lemma \ref{lem_condition_for_minimal}, one deduces that $j_{r}-j_{r+1}\geq -1$ for $2\leq r\leq m-1$. We consider the word $U_\flat=P_k W_\flat$. If $k-j_{2}\leq -2$, then we have a word $P_{j_2-1}P_kP_{j_3}\ldots P_{j_{m}}$ of $u$ with $j_2-1 > k$, which contradicts \eqref{eq_min_subscript_prime_co-divisor}. So it satisfies $k-j_{2}\geq -1$. Now we conclude that $U_\flat$ is a minimal word of $u$.

It is sufficient to prove that $U_\flat\preccurlyeq U$ for any minimal word $U$. Suppose that $U=P_{i_1}P_{i_2}\ldots P_{i_m}$ is a minimal word of $u$. One has $i_r-i_{r+1}\geq -1$ $(1\leq r\leq m-1)$. If $i_1 = k$, then $P_{i_2}\ldots P_{i_m}$ represents the same word as $W_\flat$ does. Since $W_\flat$ is a minimum word, one sees that $W_\flat \preccurlyeq P_{i_2}\ldots P_{i_m}$. Therefore $U_\flat \preccurlyeq U$.

In the following, we assume that $i_1 < k$. Note that $p_k|u$. By Lemma \ref{lem_divisor_shift_is_castling}, there exists a subword-decomposition $U=YQZ$ such that $Q$ is a letter and $\underline{\underline{Y}}Q \rlh P_k \underline{\underline{\widetilde{Y}}}$ for some word $\widetilde{Y}$. Assume that $Y=P_{i_1}\ldots P_{i_{t-1}}$ and $Q=P_{i_t}$ for some $2\leq t\leq m$. Then $\underline{\underline{P_{i_1}\ldots P_{i_{t-1}}}}P_{i_t} \rlh P_k \underline{\underline{\widetilde{Y}}}$. Notice that $P_{i_1}\ldots P_{i_{t-1}}P_{i_t}$ is a minimal word. Lemma \ref{lem_minimal_word_castling} shows that $i_1\geq i_t=k$. Now a contradiction appears. This completes the proof.
\end{proof}

\begin{lemma} \label{lem_max_min_preserving}
Suppose that $u,v$ have words $U,V$, maximum words $U_\sharp,V_\sharp$ and minimum words $U_\flat,V_\flat$, respectively.

(\romannumeral1) If $\underline{\underline{U_\sharp}} V \rlh \widetilde{V} \underline{\underline{\widetilde{U}}}$ for some words $\widetilde{V}, \widetilde{U}$. Then $\widetilde{U}$ is also a maximum word.

(\romannumeral2) If $\underline{\underline{U}} V_\sharp \rlh \widetilde{V} \underline{\underline{\widetilde{U}}}$ for some words $\widetilde{V}, \widetilde{U}$. Then $\widetilde{V}$ is also a maximum word.

(\romannumeral3) If $\underline{\underline{U}} V_\flat \rlh \widetilde{V} \underline{\underline{\widetilde{U}}}$ for some words $\widetilde{V}, \widetilde{U}$. Then $\widetilde{V}$ is also a minimum word.
\end{lemma}

\begin{proof}
(\romannumeral1) Let $\widetilde{U}_\sharp$ be the maximum word of $\widetilde{u}$. Then $\widetilde{U}_\sharp\succcurlyeq \widetilde{U}$. By Proposition \ref{prop_element_trans_independent_of_word}, the words $\widetilde{V}, \widetilde{U_\sharp}$ are castlable. Write $\widetilde{V} \underline{\underline{\widetilde{U_\sharp}}} \rlh \underline{\underline{U^\prime}} V$ for some $U^\prime\in \fW(u)$. Then we have $U^\prime \succcurlyeq U_\sharp$. It follows that $U^\prime = U_\sharp$, which leads to $\widetilde{U}=\widetilde{U}_\sharp$.

(\romannumeral2) The proof is similar as in (\romannumeral1).

(\romannumeral3) Let $\widetilde{V}_\flat$ be the minimum word of the element having word $\widetilde{V}$. Then $\widetilde{V}_\flat\preccurlyeq \widetilde{V}$. By Proposition \ref{prop_element_trans_independent_of_word}, the words $\widetilde{V}_\flat, \widetilde{U}$ are castlable. Write $\widetilde{V_\flat} \underline{\underline{\widetilde{U}}} \rlh \underline{\underline{U}} V^\prime$ for some word $V^\prime\in \fW(v)$. Then we have $V^\prime\preccurlyeq V_\flat$. It follows that $V^\prime = V_\flat$ and then $\widetilde{V}=\widetilde{V}_\flat$.
\end{proof}

\begin{remark}
Note that $\underline{\underline{P_0P_1}}P_0\rlh P_0\underline{\underline{P_0P_2}}$. Here $P_0P_1$ is the minimum word of $p_0p_1$, while $P_0P_2$ is not the minimum word of $p_0p_2$, since $P_0P_2\succ P_1P_0$.
\end{remark}

\subsection{Castling of Elements}

\begin{definition}
Let $u,v$ be two elements in $S$ with maximum words $U_\sharp,V_\sharp$ and minimum words $U_\flat,V_\flat$, respectively.

(\romannumeral1) If the words $U_\sharp,V_\sharp$ are castlable, then we say that the elements $u,v$ are weakly castlable, or $u$ is weakly castlable with $v$. Suppose that the castling of words is given by $\underline{\underline{U_\sharp}}V_\sharp \rightleftharpoons \widetilde{V}\underline{\underline{\widetilde{U}}}$ for some words $\widetilde{V},\widetilde{U}$. Let $\widetilde{v}, \widetilde{u}$ be the elements in $\bS$ having words $\widetilde{V}, \widetilde{U}$, respectively.  Then we denote the weak castling of words by $\uline{u}v \rightleftharpoons \widetilde{v} \uline{\widetilde{u}}$. And we write $\fC^\prime=\{(u,v)\in \bS\times \bS:\, u,v \text{ are weakly castlable}\}$ and $\Gamma^\prime=\{((u,v),(\widetilde{v},\widetilde{u}))\in \fC\times \fC:\, \uline{u}v \rightleftharpoons \widetilde{v} \uline{\widetilde{u}}\}$.

(\romannumeral2) If the words $U_\flat,V_\flat$ are castlable, then we say that the elements $u,v$ are strongly castlable, or $u$ is strongly castlable with $v$. Suppose that the castling of words is given by $\underline{\underline{U_\flat}}V_\flat \rightleftharpoons \widetilde{V}\underline{\underline{\widetilde{U}}}$ for some words $\widetilde{V},\widetilde{U}$. Let $\widetilde{v}, \widetilde{u}$ be the elements in $\bS$ having words $\widetilde{V}, \widetilde{U}$, respectively.  Then we denote the strong castling of elements by $\uuline{u}v \rightleftharpoons \widetilde{v} \uline{\widetilde{u}}$. And we write $\fC_0^\prime=\{(u,v)\in \fC:\, u,v \text{ are strongly castlable}\}$ and $\Gamma_0^\prime=\{((u,v),(\widetilde{v},\widetilde{u}))\in \Gamma^\prime:\, \uuline{u}v \rightleftharpoons \widetilde{v} \uuline{\widetilde{u}}\}$.
\end{definition}

\begin{remark} \label{remark_element_castling}
By Proposition \ref{prop_element_trans_independent_of_word}, we have the following equivalences between statements.

The elements $u,v$ are weakly castlable, if and only if there exist some words $U\in \fW(u)$ and $V\in \fW(v)$ such that $U,V$ can be castled, if and only if there is some word $V\in \fW(v)$ such that any word $U\in \fW(u)$ can be castled with $V\in \fW(v)$.

Similarly, the elements $u,v$ are strongly castlable if and only if any word $U\in \fW(u)$ and any word $V\in \fW(v)$ can be castled, if and only if there is some $U\in \fW(u)$ such that $U$ can be castled with any word $V\in \fW(v)$. In particular, a strong castling implies a weak castling.

Moreover, it follows from Lemma \ref{lem_max_min_preserving} that $\uline{u}v\rlh \widetilde{v}\uline{\widetilde{u}}$ if and only if $\uline{\widetilde{v}}\widetilde{u}\rlh u\uline{v}$. And one sees that $\ind(u)=\ind(\widetilde{u})$ and $\ind(v)=\ind(\widetilde{v})$, where $\ind$ is the homomorphism shown in Section \ref{subsection_example_of_intergral_monoid}. By the definitions, we have $\fC_0^\prime \subseteq\fC^\prime$ and $\Gamma_0^\prime \subseteq \Gamma^\prime$.
\end{remark}

\begin{lemma} \label{lem_union_and_decomposition_for_element_weak_castling}
(\romannumeral1) Let $u,v,\widetilde{u},\widetilde{v}$ be elements in $\bS$ such that $\underline{u}v \rightleftharpoons \widetilde{v}\underline{\widetilde{u}}$. Then for any $u_1,u_2\in \bS$ with $u_1u_2=u$, we have that $\underline{u_2}v \rightleftharpoons \widehat{v}\underline{\widehat{u_2}}$ for some elements $\widehat{u_2},\widehat{v}$, and $\underline{u_1}\widehat{v} \rightleftharpoons \widetilde{v}\underline{\widehat{u_1}}$ for some elements $\widehat{u_1}$, where $\widetilde{u}=\widehat{u_1}\widehat{u_2}$.

(\romannumeral2) Let $u_1,u_2,v$ be elements in $\bS$. If $\underline{u_2}v \rightleftharpoons \widetilde{v}\underline{\widetilde{u_2}}$ for some elements $\widetilde{u_2},\widetilde{v}$, and $\underline{u_1}\widetilde{v} \rightleftharpoons \widetilde{\widetilde{v}}\underline{\widetilde{u_1}}$ for some elements $\widetilde{u_1}, \widetilde{\widetilde{v}}$, then $\underline{u_1u_2}v \rightleftharpoons \widetilde{\widetilde{v}}\underline{\widetilde{u_1}\widetilde{u_2}}$.

(\romannumeral3) Let $u,v_1,v_2$ be elements in $\bS$. If $\underline{u}v_1 \rightleftharpoons \widetilde{v_1}\underline{\widetilde{u}}$ for some elements $\widetilde{u},\widetilde{v_1}$, and $\underline{\widetilde{u}}v_2 \rightleftharpoons \widetilde{v_2}\underline{\widetilde{\widetilde{u}}}$ for some elements $\widetilde{v_2}, \widetilde{\widetilde{u}}$, then $\underline{u}v_1v_2 \rightleftharpoons \widetilde{v_1}\widetilde{v_2}\underline{\widetilde{\widetilde{u}}}$.
\end{lemma}

\begin{proof}
(\romannumeral1) Let $U_\sharp, U_{1\sharp},{U_2}_\sharp,V_\sharp,\widetilde{V}_\sharp,\widetilde{U}_\sharp$ be the maximum words of $u,u_1,u_2,v,\widetilde{v},\widetilde{u}$, respectively. Since $\uline{u}v\rlh \widetilde{v}\uline{\widetilde{u}}$, one has $\underline{\underline{U_\sharp}}V_\sharp \rlh \widetilde{V}_\sharp\underline{\underline{\widetilde{U}_\sharp}}$ by Lemma \ref{lem_max_min_preserving}(\romannumeral1,\romannumeral2). Thanks to $U_{1\sharp}U_{2\sharp}\in \fW(u)$, it follows from Proposition \ref{prop_element_trans_independent_of_word} that $U_{1\sharp}U_{2\sharp}$ can be castled with $V_\sharp$. Suppose that $\underline{\underline{U_{1\sharp}U_{2\sharp}}}V_\sharp \rlh \widetilde{V}_\sharp\underline{\underline{\widetilde{U}}}$, where $\widetilde{U}\in \fW(\widetilde{u})$. Now a decomposition of castlings of words leads to
\[
\underline{\underline{U_{1\sharp}U_{2\sharp}}}V_\sharp \rlh \underline{\underline{U_{1\sharp}}}\widehat{V_\sharp}\underline{\underline{\widehat{U_{2\sharp}}}} \rlh  \widetilde{V_\sharp}\underline{\underline{\widehat{U_{1\sharp}}\widehat{U_{2\sharp}}}}
\]
for some words $\widehat{V}_\sharp, \widehat{U_1}_\sharp, \widehat{U_{2}}_{\sharp}$ with $\widehat{U_1}_\sharp \widehat{U_{2}}_{\sharp}\in \fW(\widetilde{u})$. Indeed, the words $\widehat{V}_\sharp, \widehat{U_1}_\sharp, \widehat{U_{2}}_{\sharp}$ are all maximum by Lemma \ref{lem_max_min_preserving}(\romannumeral1,\romannumeral2). Denote by  $\widehat{v}, \widehat{u_1}, \widehat{u_2}$ the elements in $\bS$ having words $\widehat{V}_\sharp, \widehat{U_1}_\sharp, \widehat{U_{2}}_{\sharp}$, respectively. It follows that $\underline{u_1u_2}v \rightleftharpoons \underline{u_1}\widehat{v}\underline{\widehat{u_2}}\rlh \widetilde{v}\underline{\widehat{u_1}\widetilde{u_2}}=\widetilde{v}\underline{\widetilde{u}}$. The proof is completed.

(\romannumeral2) Let $U_{1\sharp},{U_2}_\sharp,V_\sharp,\widetilde{V}_\sharp,\widetilde{\widetilde{V}}_\sharp$ be the maximum words of $u_1,u_2,v,\widetilde{v},\widetilde{\widetilde{v}}$, respectively. One has $\underline{\underline{U_{2\sharp}}}V_\sharp\rlh \widetilde{V}_\sharp \underline{\underline{\widetilde{U_{2}}_{\sharp}}}$ and $\underline{\underline{U_{1\sharp}}}\widetilde{V}_\sharp\rlh \widetilde{\widetilde{V}}_\sharp \underline{\underline{\widetilde{U_1}_{\sharp}}}$. A composition of the above castlings shows that
\[
\underline{\underline{U_{1\sharp}U_{2\sharp}}}V_\sharp \rlh \underline{\underline{U_{1\sharp}}}\widetilde{V}_\sharp \underline{\underline{\widetilde{U_{2\sharp}}}}\rlh \widetilde{\widetilde{V_\sharp}}\underline{\underline{\widetilde{U_{1\sharp}}\widetilde{U_{2\sharp}}}}. \]
The conclusion now follows by applying Remark \ref{remark_element_castling}.

(\romannumeral3) The proof is similar to that of (\romannumeral2).
\end{proof}

\begin{lemma} \label{lem_union_and_decomposition_for_element_castling}
(\romannumeral1) Let $u,v,\widetilde{v},\widetilde{u}$ be elements in $\bS$ such that $\underline{\underline{u}}v \rightleftharpoons \widetilde{v}\underline{\widetilde{u}}$. Then for any $u_1,u_2\in \bS$ with $u_1u_2=u$, we have that $\underline{\underline{u_2}}v \rightleftharpoons \widehat{v}\underline{\widehat{u_2}}$ for some elements $\widehat{u_2},\widehat{v}$, and $\underline{\underline{u_1}}\widehat{v} \rightleftharpoons \widetilde{v}\underline{\widehat{u_1}}$ for some elements $\widehat{u_1}$, where $\widetilde{u}=\widehat{u_1}\widehat{u_2}$.

(\romannumeral2) Let $u_1,u_2,v$ be elements in $\bS$. If $\underline{\underline{u_2}}v \rightleftharpoons \widetilde{v}\underline{\widetilde{u_2}}$ for some elements $\widetilde{u_2},\widetilde{v}$, and $\underline{\underline{u_1}}\widetilde{v} \rightleftharpoons \widetilde{\widetilde{v}}\underline{\widetilde{u_1}}$ for some elements $\widetilde{u_1}, \widetilde{\widetilde{v}}$, then $\underline{\underline{u_1u_2}}v \rightleftharpoons \widetilde{\widetilde{v}}\underline{\widetilde{u_1}\widetilde{u_2}}$.

(\romannumeral3) Let $u,v,\widetilde{v},\widetilde{u}$ be elements in $\bS$ such that $\underline{\underline{u}}v \rightleftharpoons \widetilde{v}\underline{\widetilde{u}}$. Then for any $v_1,v_2\in \bS$ with $v_1v_2=v$, we have that $\underline{\underline{u}}v_1 \rightleftharpoons \widehat{v_1}\underline{\widehat{u}}$ for some elements $\widehat{u},\widehat{v_1}$, and $\underline{\underline{\widehat{u}}}v_2 \rightleftharpoons \widehat{v_2}\underline{\widetilde{u}}$ for some elements $\widehat{v_2}$, where $\widetilde{v}=\widehat{v_1}\widehat{v_2}$.
\end{lemma}

\begin{proof}
(\romannumeral1) Let $U_\flat,U_{1\flat},U_{2\flat},V_\flat,\widetilde{V}_\flat$ be the minimum words of $u,u_1,u_2,v,\widetilde{v}$, respectively. Since $U_{1\flat}U_{2\flat}$ is a word of $u$ and the words $U_\flat,V_\flat$ are castlable, we have that the words $U_{1\flat}U_{2\flat},V_\flat$ are also castlable. Let us denote $\underline{\underline{U_{1\flat}U_{2\flat}}}V_{\flat}\rlh \widetilde{V}\underline{\underline{U}}$, where $\widetilde{V}\in \fW(v)$ and $\widetilde{U}\in \fW(u)$. Now a decomposition of such a castling implies that
\[
\underline{\underline{U_{1\flat}U_{2\flat}}}V_\flat \rightleftharpoons \underline{\underline{U_{1\flat}}} \widehat{V_\flat} \underline{\underline{\widehat{U_2}}} \rlh \widetilde{V} \underline{\underline{\widehat{U_1} \widehat{U_2}}}
\]
for some words $\widehat{U_1}$, $\widehat{U_2}$ and $\widehat{V_\flat}$. By Lemma \ref{lem_max_min_preserving}(\romannumeral3), one deduces that $\widehat{V_\flat}$ is a minimum word here. Let $\widehat{u_1}$, $\widehat{u_2}$ and $\widehat{v_\flat}$ be elements in $\bS$ having words $\widehat{U_1}$, $\widehat{U_2}$ and $\widehat{V_\flat}$, respectively. Combining Remark \ref{remark_element_castling}, we conclude that $\underline{\underline{u_2}}v \rightleftharpoons \widehat{v}\underline{\widehat{u_2}}$ and $\underline{\underline{u_1}}\widehat{v} \rightleftharpoons \widetilde{v}\underline{\widehat{u_1}}$.

(\romannumeral2) Let $U_{1\flat},{U_2}_\flat,V_\flat,\widetilde{V}_\flat$ be the minimum words of $u_1,u_2,v,\widetilde{v}$, respectively. Combining the castlings $\uuline{u_2}v \rightleftharpoons \widetilde{v}\underline{\widetilde{u_2}}$,  and $\uuline{u_1}\widetilde{v} \rightleftharpoons \widetilde{\widetilde{v}}\underline{\widetilde{u_1}}$ and Lemma \ref{lem_max_min_preserving}(\romannumeral3), we have  $\underline{\underline{U_{2\flat}}}V_\flat \rightleftharpoons \widetilde{V}_\flat \underline{\underline{\widetilde{U_2}}}$ for some $\widetilde{U_2}\in \fW(\widetilde{u_2})$, and $\underline{\underline{U_{1\flat}}}\widetilde{V}_\flat \rightleftharpoons \widetilde{\widetilde{V}}\underline{\underline{\widetilde{U_1}}}$ for some $\widetilde{\widetilde{V}}\in \fW(\widetilde{\widetilde{v}})$ and $\widetilde{U_1}\in \fW(\widetilde{u_1})$. So $\underline{\underline{U_{1\flat}U_{2\flat}}}V_\flat \rightleftharpoons \widetilde{\widetilde{V}}\underline{\underline{\widetilde{U_1}\widetilde{U_2}}}$, where $U_{1\flat}U_{2\flat}\in \fW(u)$. Combining Remark \ref{remark_element_castling}, we conclude that $\underline{\underline{u_1u_2}}v \rightleftharpoons \widetilde{\widetilde{v}}\underline{\widetilde{u_1}\widetilde{u_2}}$.

(\romannumeral3) Let $U_\flat,V_\flat,V_{1\flat},V_{2\flat}$ be the minimum words of $u,v,v_1,v_2$, respectively. Since $V_{1\flat}V_{2\flat}\succcurlyeq V_\flat$ and $U_\flat,V_\flat$ are castlable, the words $U_\flat,V_{1\flat}V_{2\flat}$ are also castlable. Suppose that $\underline{\underline{U_\flat}}V_{1\flat}V_{2\flat} \rlh \widetilde{V}\underline{\underline{\widetilde{U}}}$ for some $\widetilde{V}\in \fW(\widetilde{v})$ and $\widetilde{U}\in \fW(\widetilde{u})$. A decomposition of such a castling shows that
\[
\underline{\underline{U_\flat}}V_{1\flat}V_{2\flat} \rightleftharpoons \widehat{V_1}\underline{\underline{\widehat{U}}}V_{2\flat} \rlh \widehat{V_1}\widehat{V_2}\underline{\underline{\widetilde{U}}}=\widetilde{V}\underline{\underline{\widetilde{U}}}
\]
for some words $\widehat{U}, \widehat{V_1}, \widehat{V_2}$. Let $\widehat{u}, \widehat{v_1}, \widehat{v_2}$ be elements in $\bS$ having words $\widehat{U}, \widehat{V_1}, \widehat{V_2}$, respectively. Combining Remark \ref{remark_element_castling}, we conclude that $\underline{\underline{u}}v_1 \rightleftharpoons \widehat{v_1}\underline{\widehat{u}}$ and $\underline{\underline{\widehat{u}}}v_2 \rightleftharpoons \widehat{v_2}\underline{\widetilde{u}}$.
\end{proof}

\subsection{Free Castlings and Fundamental Lemma for Arithmetic}

In this subsection, we will define free castlings of elements, and prove the fundamental lemma for arithmetic in the context of castlings defined in this section.

\begin{lemma} \label{lem_not_castlable_implies_divides}
Let $w,u,\widetilde{u},\widetilde{w}$ be elements in $\bS$ satisfying $\underline{w}u\rlh \widetilde{u}\underline{\widetilde{w}}$. Suppose that $w,u$ are not strongly castlable. Then $\gcd(w,\widetilde{u})\neq 1$, and $\gcd_\ddagger(wu; u,\widetilde{w}) \neq 1$.
\end{lemma}

\begin{proof}
We use induction on $\ind(w)$. For $\ind(w)=0$, the proof is trivial.

Now we deal with the case $w=p_k$ for some $k\geq 0$. Note that $P_k$ is castlable with $U_\sharp$, but is not castlable with $U_\flat$. And there is a totally ordered chain, with respect to $\succcurlyeq$, between $U_\sharp$ and $U_\flat$. So there exist two distinct words $U$ and $U^\prime$ of $u$, differing from each other by exactly one castling of a pair of adjacent letters, such that $P_k$ is castlable with $U$ but not castlable with $U^\prime$. Here $U\succ U^\prime$. Write $U=U_0P_iP_j U_1$ and $U^\prime= U_0P_{j-1} P_i U_1$, where $i-j\leq -2$. It follows from the castling of $P_k$ and $U$ that
\[
\underline{\underline{P_k}}U = \underline{\underline{P_k}}U_0P_iP_jU_1\rlh U_0^\prime\underline{\underline{P_{k^\prime}}}P_iP_j U_1 \rlh U_0^\prime P_{i^\prime}P_{j^\prime} \underline{\underline{P_{k^{\prime\prime}}}}U_1 \rlh U_0^\prime P_{i^\prime}P_{j^\prime} U_1^\prime \underline{\underline{P_{k^{\prime\prime\prime}}}}
\]
for some words $U_0^\prime, U_1^\prime$ and letters $P_{i^\prime},P_{j^\prime}, P_{k^\prime},P_{k^{\prime\prime}}, P_{k^{\prime\prime\prime}}$. Here $U_0^\prime P_{i^\prime}P_{j^\prime} U_1^\prime$ is a word of $\widetilde{u}$ and $P_{k^{\prime\prime\prime}}$ is the letter of $\widetilde{w}$.

If $P_{k^\prime}$ is castlable with $P_{j-1}P_{i}$, then $\underline{\underline{P_{k^\prime}}}P_{j-1}P_i \rlh Z\underline{\underline{P_{k^{\prime\prime}}}}$ for some word $Z$ by Proposition \ref{prop_element_trans_independent_of_word}. It follows that
\[
\underline{\underline{P_k}}U^\prime = \underline{\underline{P_k}}U_0P_{j-1}P_iU_1\rlh U_0^\prime\underline{\underline{P_{k^\prime}}}P_{j-1}P_i U_1 \rlh U_0^\prime Z \underline{\underline{P_{k^{\prime\prime}}}}U_1 \rlh U_0^\prime Z U_1^\prime \underline{\underline{P_{k^{\prime\prime\prime}}}},
\]
which contradicts the fact that $P_k$, $U^\prime$ are not castlable.

Thus, the letter $P_{k^\prime}$ is castlable with $P_iP_j$, but not castlable with $P_{j-1}P_i$. By the proof of Lemma \ref{lem_U_1_V_2}, one sees that $k^\prime = i = j-2$. It follows that $\underline{\underline{P_i}}P_iP_{i+2}\rlh P_iP_{i+1}\underline{\underline{P_i}}$, i.e., $i^\prime=i, j^\prime=i+1, k^{\prime\prime}=i$. Now $U_0^\prime P_iP_{i+1}U_1^\prime=U_0^\prime P_{i^\prime}P_{j^\prime} U_1^\prime$, and it is a word of $\widetilde{u}$. Combing $\underline{\underline{P_k}}U_0 \rlh U_0^\prime\underline{\underline{P_{k^\prime}}}=U_0^\prime\underline{\underline{P_i}}$, we deduce that $P_kU_0P_{i+1}U_1^\prime$ is also a word of $\widetilde{u}$, which implies that $p_k|\widetilde{u}$. Similarly, since $U_0P_{i+1}P_iU_1=U_0P_{j-1}P_iU_1\in \fW(u)$ and $\underline{\underline{P_i}}U_1=\underline{\underline{P_{k^{\prime\prime}}}}U_1\rlh U_1^\prime\underline{\underline{P_{k^{\prime\prime\prime}}}}$, one deduces that $U_0P_{i+1}U_1^\prime P_{k^{\prime\prime\prime}}$ is also a word of $u$. Thus $\widetilde{w}\ddagger u$. We conclude that $\gcd(w,\widetilde{u})\neq 1$ and $\gcd_\ddagger(wu;u,\widetilde{w})\neq 1$.

Next, let us suppose that the lemma has been proved for $\ind(w)\leq m-1$ with some $m\geq 2$. Now we handle the case $\ind(w)=m$. Let $p_k$ be a prime co-divisor of $w$ and write $w=w_1p_k$. By Lemma \ref{lem_union_and_decomposition_for_element_weak_castling}(\romannumeral1), we have $\underline{p_k}u\rlh \widehat{u}\underline{q}$ for some $\widehat{u},q\in \bS$ and $\underline{w_1}\widehat{u}\rlh \widetilde{u}\underline{\widehat{w_1}}$ for some $\widehat{w_1}$, where $\widetilde{w}=\widehat{w_1}q$. By Lemma \ref{lem_union_and_decomposition_for_element_castling}(\romannumeral2) and the fact the $w,u$ are not strongly castlable, one deduces that either $p_k,u$ are not strongly castlable, or $w_1,\widehat{u}$ are not strongly castlable.

\textsc{Case 1.} If $p_k,u$ are not strongly castlable, then we have $q\ddagger u$ by inductive hypothesis, which leads to $\gcd(wu; u, \widetilde{w})\neq 1$. Now we apply similar arguments (replace $\widetilde{u}$ by $\widehat{u}$ and $\widetilde{w}$ by $q$) as previous to obtain
\[
\underline{\underline{P_k}}U = \underline{\underline{P_k}}U_0P_iP_jU_1\rlh U_0^\prime\underline{\underline{P_i}}P_iP_{i+2} U_1 \rlh U_0^\prime P_iP_{i+1}\underline{\underline{P_i}} U_1 \rlh U_0^\prime P_iP_{i+1} U_1^\prime \underline{\underline{P_{k^{\prime\prime\prime}}}},
\]
where $U_0^\prime P_iP_{i+1} U_1^\prime$ is a word of $\widehat{u}$ and $P_{k^{\prime\prime\prime}}$ is the letter of $q$. Denote $\widehat{U}=P_kU_0P_{i+1}U_1^\prime$, which is also a word of $\widehat{u}$. So it is castlable with $P_{k^{\prime\prime\prime}}$ by Proposition \ref{prop_element_trans_independent_of_word}. Write $\underline{\underline{P_k}}\dot{U}\rlh \widehat{U}\underline{\underline{P_{k^{\prime\prime\prime}}}}$ for some word $\dot{U}$ of $u$. Let $W_1$ be any word of $w_1$. If $W_1$ and $\widehat{U}$ are not castlable, then $w_1$ and $\widehat{u}$ are not strongly castlable, which is a case we will handle latter. Now we suppose that $\underline{\underline{W_1}}\widehat{U} \rlh \widehat{\widehat{U}}\underline{\underline{\widehat{W_1}}}$ for some word $\widehat{\widehat{U}}$ and $\widehat{W_1}$. Note that $\underline{\underline{W_1P_k}}\dot{U} \rlh \widehat{\widehat{U}}\underline{\underline{\widehat{W_1}P_{k^{\prime\prime\prime}}}}$. So $\widehat{\widehat{U}}$ is a word of $\widetilde{u}$. Furthermore, we can write
\[
\underline{\underline{W_1}}\widehat{U}= \underline{\underline{W_1}}  P_kU_0P_{i+1} U_1^\prime \rlh \widehat{P_k}\underline{\underline{\grave{W_1}}}  U_0P_{i+1} U_1^\prime \rlh \widehat{P_k} \widehat{U_0}\widehat{P_{i+1}} \widehat{U_1}^\prime \underline{\underline{\widehat{W_1}}}=  \widehat{\widehat{U}}\underline{\underline{\widehat{W_1}}}
\]
for some words $\widehat{P_k},\grave{W_1},\widehat{U_0},\widehat{P_{i+1}},\widehat{U_1}^\prime$. Let $t$ be the element having letter $\widehat{P_k}$. It follows from $\widehat{P_k} \widehat{U_0}\widehat{P_{i+1}} \widehat{U_1}^\prime=\widehat{\widehat{U}}$ that $t|\widetilde{u}$. Moreover, both $\widehat{P_k}\grave{W_1}$ and $W_1  P_k$ are words of $w$, so $t|w$. Now we have $t|\gcd(w,\widetilde{u})$.

\textsc{Case 2.} If $w_1,\widehat{u}$ is not strongly castlable, then inductive hypothesis ensures that $\gcd(w_1,\widetilde{u})\neq 1$ and $\gcd_\ddagger(w_1\widehat{u}; \widehat{u},\widehat{w_1})\neq 1$. So $\gcd(w,\widetilde{u})\neq 1$. Let $r$ be an element such that $r\in \cP$ and $r\ddagger \gcd_\ddagger(w_1\widehat{u}; \widehat{u},\widehat{w_1})$. We write $\widehat{u}=\breve{u}r$ and $\widehat{w_1}=\breve{w_1}r$. It follows from $\underline{p_k}u\rlh \widehat{u}\underline{q}$ and Lemma \ref{lem_union_and_decomposition_for_element_weak_castling}(\romannumeral2) that
\[
\widehat{u}\underline{q} = \breve{u}r\underline{q} \rlh \breve{u} \underline{\breve{q}} \breve{r} \rlh \underline{p_k} \breve{\breve{u}}\breve{r}=\underline{p_k}u
\]
for some $\breve{q},\breve{r},\breve{\breve{u}}\in \bS$. One sees that $\breve{r}\ddagger u$ and $rq=\breve{q}\breve{r}$. Then $\widetilde{w}=\widehat{w_1}q=\breve{w_1}rq=\breve{w_1}\breve{q}\breve{r}$, which implies that $\breve{r}\ddagger \widetilde{w}$. Now $\breve{r}\ddagger \gcd_\ddagger(wu;u,\widetilde{w})$.

The proof is completed by induction.
\end{proof}

\begin{remark} \label{remark_gcd_and_strong_castling}
Note that we also have $\widetilde{u}\underline{\widetilde{w}}\rlh \underline{w}u$. If $\widetilde{u},\widetilde{w}$ are not strongly castlable, then $\gcd(w,\widetilde{u})\neq 1$ and $\gcd_\ddagger(wu; u, \widetilde{w})\neq 1$ still hold.
\end{remark}

\begin{definition}
If $\uline{u}v \rightleftharpoons \widehat{v} \underline{\widehat{u}}$ and $\gcd(u,\widehat{v})=1$, then we say that $u,v$ are castled-free. We denote $\uuuline{u}v \rightleftharpoons \widehat{v} \uuuline{\widehat{u}}$, and call it a free castling. We also put $\fC_1^\prime=\{(u,v)\in \fC:\, u,v \text{ are castled-free}\}$ and $\Gamma_1^\prime=\{((u,v),(\widetilde{v},\widetilde{u}))\in \Gamma^\prime:\, \uuuline{u}v \rightleftharpoons \widetilde{v} \uuuline{\widetilde{u}}\}$.
\end{definition}

By Lemma \ref{lem_not_castlable_implies_divides}, we have $\fC_1^\prime\subseteq \fC_0^\prime$ and $\Gamma_1^\prime \subseteq \Gamma_0^\prime$.

\begin{lemma}  \label{lem_divisor_castling_Thompson_Group}
Let $u,v\in \bS$. Suppose that $w$ is a divisor of $uv$ satisfying $\gcd(w,u)=1$. Then there exist some $v_1|v$ and $\widetilde{u}\in \bS$ such that $\uuuline{w}\widetilde{u} \rlh u\uuuline{v_1}$.
\end{lemma}

\begin{proof}
When $\ind(u)=0$ or $\ind(v)=0$ or $\ind(w)=0$, the proof is trivial. We always assume that $\ind(u),\ind(v),\ind(w)\geq 1$. In the following, induction on $\ind(w)$ is applied.

We first deal with the case $w\in \cP$. Let $W,U,V$ be the words of $w,u,v$, respectively. Note that $w|uv$. By Lemma \ref{lem_divisor_shift_is_castling}, we have a subword-decomposition $UV=YQ_0Z$ with $Q_0$ a letter and $\underline{\underline{Y}}Q_0 \rlh W\underline{\underline{\widetilde{Y}}}$ for some word $\widetilde{Y}$.

\textsc{Case 1.} If $\ind(Y)< \ind(U)$, then $YQ_0$ is a subword of $U$ that starts at the beginning. Write $u_1$ for the element having word $YQ_0$. One has $u_1|u$. Note that $W\widetilde{Y}$ is also a word of $u_1$. It follows that $w|u_1$, which contradicts the fact that $\gcd(w,u)=1$.

\textsc{Case 2.} Now we suppose that $\ind(Y)\geq \ind(U)$. Write $Y=UY_0$. It follows from the castling of $Y$ and $Q_0$ that
\[
\underline{\underline{Y}}Q_0 = \underline{\underline{UY_0}}Q_0 \rlh \underline{\underline{U}} Q \underline{\underline{\widetilde{Y_0}}} \rlh W \underline{\underline{\widetilde{U}\widetilde{Y_0}}} = W \underline{\underline{\widetilde{Y}}}
\]
for some letter $Q$ and some words $\widetilde{U},\widetilde{Y_0}$. Let $\widetilde{u}, q$ be the elements having words $\widetilde{U}, Q$, respectively. One deduces that $\underline{w}\widetilde{u} \rlh u\underline{q}$. Combining Lemma \ref{lem_not_castlable_implies_divides} and the fact $\gcd(w,u)=1$, one concludes that $\uuuline{w}\widetilde{u} \rlh u\uuuline{q}$.

Suppose that the lemma has been proved for $\ind(w)\leq m-1$ with some $2\leq m\leq \ind(v)$. Now we consider the situation $\ind(w)=m$. Write $w=pw_1$ with $p\in \cP$. Then $p|uv$ and $p\nmid u$. By inductive hypothesis, there is a prime $q|v$ and an element $\widetilde{u}\in S$ such that $\uuuline{p}\widetilde{u} \rlh u\uuuline{q}$. Let $v=qv^\prime$. Since $p\widetilde{u}=uq$, one has $uv= uqv^\prime = p\widetilde{u}v^\prime$. It follows that $w_1|\widetilde{u}v^\prime$.

Assume that $r|\gcd(w_1,\widetilde{u})$ for some $r\in \cP$. We write $\widetilde{u}=r\widetilde{u}^\prime$. Then, it follows from the strong castling in $p, \widetilde{u}$ and Lemma \ref{lem_union_and_decomposition_for_element_castling}(\romannumeral3) that
\[
\uuline{p}r \rlh \widetilde{r} \uline{\widetilde{p}},\quad \uuline{\widetilde{p}}\widetilde{u}^\prime \rlh  \widetilde{\widetilde{u}}^\prime\uline{q}
\]
for some primes $\widetilde{r},  \widetilde{p}$ and elements $\widetilde{\widetilde{u}}^\prime$ with $u=\widetilde{r}\widetilde{\widetilde{u}}^\prime$. 
Now $\widetilde{r}|\widetilde{r} \widetilde{\widetilde{u}^\prime}=u$, $\widetilde{r}| \widetilde{r} \widetilde{p}=pr$ and $pr|pw_1=w$, which contradicts the fact that $\gcd(w,u)=1$.

Now we conclude that $\gcd(w_1,\widetilde{u})=1$. Recall that $w_1|\widetilde{u}v^\prime$. By inductive hypothesis, there are an element $v_2|v^\prime$ and some $\widetilde{\widetilde{u}}\in \bS$ such that $\uuuline{w_1}\widetilde{\widetilde{u}} \rlh \widetilde{u}\uuuline{v_2}$. Recall that $\uuuline{p}\widetilde{u} \rlh u\uuuline{q}$. One 
sees that $\uline{w}\widetilde{\widetilde{u}}=\uline{p w_1}\widetilde{\widetilde{u}} \rlh  u\uline{qv_2}$, where $qv_2|v$. Noting that $\gcd(w,u)\neq 1$, we conclude that $\uuuline{w}\widetilde{\widetilde{u}}  \rlh u\uuuline{qv_2}$. The proof is completed.
\end{proof}

\subsection{Verifying the Axioms}
\label{subsection_verifying_axioms_thompson}

In this subsection, we will show that the definition of $\fC^\prime,\fC_0^\prime,\fC_1^\prime$ coincides with $\fC,\fC_0,\fC_1$ defined in previous sections, respectively, and Axioms \uppercase\expandafter{\romannumeral4} and \uppercase\expandafter{\romannumeral5} are satisfied. The underlines occurred in this subsection are still the notations defined in this section.

\begin{lemma} \label{lem_castling_to_lcm_ind_leq}
Let $u,v$ be elements in $\bS$ and $xy^{-1}$ be the fraction of $v^{-1}u$ in lowest terms with numerator $x$ and denominator $y$. Then $\ind(x)\leq \ind(u)$ and $\ind(y)\leq \ind(v)$.
\end{lemma}

\begin{proof}
We prove by induction on $\ind(u),\ind(v)$. For $\ind(u)=0$ or $\ind(v)=0$, the proof is trivial. We assume below $\ind(u),\ind(v)\geq 1$. For $\ind(u)=\ind(v)=1$, we have that $x=y=1$ when $u=v$, or $x,y$ given by $\lcm[u,v]=uy=vx$. In both cases, one has $\ind(x)\leq \ind(u)$ and $\ind(y)\leq \ind(v)$.

Suppose that the lemma has been proved with $\ind(u)+\ind(v)\leq m-1$ with some $m\geq 3$. Now we consider the case $\ind(u)+\ind(v)=m$. Without loss of generality, we assume that $\ind(u)\geq 2$. Let $u=u_1u_2$, where $u_1u_2\neq 1$. Let $x_1y_1^{-1}$ be the fraction of $v^{-1}u_1$ in lowest terms with numerator $x_1$ and denominator $y_1$. By inductive hypothesis, we have $\ind(x_1)\leq \ind(u_1)$ and $\ind(y_1)\leq \ind(v)$. Let $x_2y_2^{-1}$ be the fraction of $y_1^{-1}u_2$ in lowest terms with numerator $x_2$ and denominator $y_2$. By inductive hypothesis, we have $\ind(x_2)\leq \ind(u_2)$ and $\ind(y_2)\leq \ind(y_1)$. Note that $v^{-1}u=v^{-1}u_1u_2=x_1y_1^{-1}u_2=(x_1x_2)y_2^{-1}$. It is also a fraction of $v^{-1}u$. So there is some $c\in \bS$ such that $x_1x_2=xc$ and $y_2=yc$. One concludes that
\begin{align*}
&\ind(x)\leq \ind(x_1)+\ind(x_2)\leq \ind(u_1)+\ind(u_2)=\ind(u),\\
&\ind(y)\leq \ind(y_2)\leq \ind(y_1)\leq \ind(v).
\end{align*}
This completes the proof.
\end{proof}

\begin{lemma} \label{lem_C_1_=_C_1prime}
We have $\fC_1=\fC_1^\prime$ and $\Gamma_1=\Gamma_1^\prime$.
\end{lemma}

\begin{proof}
For any $(u,y)\in \fC_1$, there exists some $v,x\in \bS$ such that $\gcd(u,v)=1$ and $\lcm[u,v]=uy=vx$. Then $v|uy$. By Lemma \ref{lem_divisor_castling_Thompson_Group}, there are elements $\widetilde{u}\in \bS$ and $y_1|y$ such that $\uuuline{v}\widetilde{u}\rlh u \uuuline{y_1}$. So $v\widetilde{u}=uy_1$. One sees that $uy_1$ is a common multiply of $u$ and $v$. Hence $uy=\lcm[u,v]|uy_1$, which leads to $y=y_1$. Now $\uuuline{v}\widetilde{u}\rlh u \uuuline{y}$, i.e., $(u,y)\in \fC_1^\prime$. As a result, we have $\fC_1\subseteq \fC_1^\prime$.

On the other hand, suppose that $(u,y)\in \fC_1^\prime$. Let $\uuuline{u}y\rlh v \uuuline{x}$ for some $v,x\in \bS$. Then $\gcd(u,v)=1$. Let $ab^{-1}$ be the fraction of $v^{-1}u$ in lowest terms with numerator $a$ and denominator $b$. Recalling the definition of least common multiple, we have $\lcm[u,v]=ub=va$. By Lemma \ref{lem_castling_to_lcm_ind_leq}, we have  $\ind(a)\leq \ind(u)$ and $\ind(b)\leq \ind(v)$. Moreover, since $v|ub$ and $\gcd(u,v)=1$, one has $\uuuline{v}\widetilde{u}\rlh u \uuuline{b_1}$ for some $b_1|b$ and $\widetilde{u}\in \bS$. It follows that $\ind(b)\geq \ind(b_1)=\ind(v)$. Now we have $\ind(b)=\ind(v)=\ind(y)$. Furthermore, note that $uy=vx$, which is a common multiple of $u,v$. So $ub=\lcm[u,v]|uy$, which shows that $b|y$. One concludes that $b=y$ and then $a=x$. Now we have $\gcd(u,v)=1$ and $\lcm[u,v]=uy=vx$, i.e., $(u,v)\in \fC_1$.

We conclude that $\fC_1=\fC_1^\prime$ and then $\Gamma_1=\Gamma_1^\prime$.
\end{proof}

\begin{lemma} \label{lem_C=C_prime}
We have $\fC=\fC^\prime$ and $\Gamma=\Gamma^\prime$. Axiom \uppercase\expandafter{\romannumeral4} holds for Thompson's monoid $\bS$.
\end{lemma}

\begin{proof}
By Lemma \ref{lem_C_1_=_C_1prime}, one gets $\fC^\prime\supseteq \fC_0$. And it is not hard to see that $(p,p)\in \fC^\prime$ for all $p\in \cP$. Similarly, the set $\Gamma^\prime$ contains $\Gamma_0$ and the elements $((p,p),(p,p))$ with $p\in \cP$. Besides, recall that $\fC$ is constructed by \eqref{eq_C_castling_composition_1} and \eqref{eq_C_castling_composition_2}. We deduce by Lemma \ref{lem_union_and_decomposition_for_element_weak_castling}(\romannumeral2,\romannumeral3) that $\fC^\prime \supseteq \fC$.

For any $(u,v)\in \fC^\prime$, there is some $\widetilde{u},\widetilde{v}\in \bS$ such that $\uline{u}v\rlh \widetilde{v}\uline{\widetilde{u}}$. Let $U_\sharp,V_\sharp,\widetilde{V}_\sharp,\widetilde{U_\sharp}$ be the maximum words of $u,v,\widetilde{v},\widetilde{u}$, respectively. Then $\underline{\underline{U_\sharp}}V_\sharp\rlh \widetilde{V_\sharp}\underline{\underline{\widetilde{U_\sharp}}}$. Recall that the castling of words are defined by induction of index of the corresponding words. When $\ind(U_\sharp)\geq 2$ or $\ind(V_\sharp)\geq 2$, the castlability of $U_\sharp,V_\sharp$ comes from either type (\uppercase\expandafter{\romannumeral1}) or type (\uppercase\expandafter{\romannumeral2}). Without loss of generality, we deal with type (\uppercase\expandafter{\romannumeral1}) here. That is to say, there are words $U_1,U_2,\widehat{V},\widehat{U_1},\widehat{U_2}$ with $U_1U_2=U_\sharp$ and $\widehat{U_1}\widehat{U_2}=\widetilde{U_\sharp}$ such that
\[
\underline{\underline{U_2}}V_\sharp\rlh \widehat{V}\underline{\underline{\widehat{U_2}}}, \quad \underline{\underline{U_1}}\widehat{V}\rlh \widetilde{V_\sharp}\underline{\underline{\widehat{U_2}}}.
\]
Let $u_1,u_2,\widehat{v},\widetilde{u_1},\widetilde{u_2}$ be elements in $\bS$ having words $U_1,U_2,\widehat{V},\widehat{U_1},\widehat{U_2}$, respectively. Then
\[
\uline{u_2}v\rlh \widehat{v}\uline{\widehat{u_2}}, \quad \uline{u_1}\widehat{v}\rlh \widetilde{v_\sharp}\uline{\widehat{u_2}}.
\]
Here $u_1u_2=u$ and $\widehat{u_1}\widehat{u_2}=\widetilde{u}$. That is to say, the fact that $u,v$ are weakly castlable follows from \eqref{eq_C_castling_composition_1} or \eqref{eq_C_castling_composition_2}. Hence $\fC=\fC^\prime$ and then $\Gamma=\Gamma^\prime$.

By Proposition \ref{prop_element_trans_independent_of_word}, one can deduce that the set $\Gamma$ is a graph of a map $\eta:\fC\rightarrow \fC$. So Axiom \uppercase\expandafter{\romannumeral4} holds.
\end{proof}

Now, the meaning of underlines ``$\uline{\,\,}$" and triple underlines ``$\uuuline{\,\,}$" occurred in this section and that occurred in previous sections coincide, and we do not distinguish them any more. Next, we handle the double underlines ``$\uuline{\,\,}$''.

\begin{lemma}
We have $\fC_0=\fC_0^\prime$ and $\Gamma_0=\Gamma_0^\prime$.
\end{lemma}

\begin{proof}
Combining $\fC_1^\prime \subseteq \fC_0^\prime$ and Lemma \ref{lem_C_1_=_C_1prime}, one deduces that $\fC_0^\prime$ contains $\fC_1$. Since an element $p$ in $\cP$ has only one word, we have $\uuline{p}p\rlh p\uuline{p}$ for $p\in \cP$, i.e., $(p,p)\in \fC_0^\prime$. For any $(u,v)\in \fC_0$, we have that $\uline{u}v\rlh \widetilde{v}\uline{\widetilde{u}}$ and \eqref{eq_C0_castling_decomposition_1}, \eqref{eq_C0_castling_decomposition_2} holds. To prove that $\fC_0^\prime\supseteq \fC_0$, we shall show that $\uuline{u} v \rlh \widetilde{v}\uline{\widetilde{u}}$. Here the double underline is the notation defined in this section.

We use induction on $\ind(u)+\ind(v)$. For $\ind(u)\leq 1$ or $\ind(v)\leq 1$, the proof is trivial. Suppose that the case $\ind(u)+\ind(v)\leq m-1$ with some $m\geq 3$. Now we deal with the case $\ind(u)+\ind(v)=m$. Consider the situation that $\ind(v)\geq 2$. Let $V_\flat$ be the minimum word of $v$. Use an arbitrary proper decomposition $V_\flat=V_1V_2$. Let $v_1,v_2$ be elements in $\bS$ with words $V_1,V_2$, respectively. Then $V_1,V_2$ are actually the minimum words of $v_1,v_2$, respectively. By  \eqref{eq_C0_castling_decomposition_2}, we have $(u,v_1),(\widehat{u},v_2)\in \fC_0$ and
\[
\uline{u} v_1 \rlh \widehat{v_1}\uline{\widehat{u}},\quad \uline{\widehat{u}} v_2  = \widehat{v_2} \uline{\widetilde{u}}
\]
for some $\widehat{v_1},\widehat{v_2},\widehat{u}\in S$ with $\widehat{v_1}\widehat{v_2}=\widetilde{v}$. By inductive hypothesis, we have $(u,v_1),(\widehat{u},v_2)\in \fC_0^\prime$. Therefore,
\[
\uuline{u} v_1 \rlh \widehat{v_1}\uline{\widehat{u}},\quad \uuline{\widehat{u}} v_2  = \widehat{v_2} \uline{\widetilde{u}}.
\]
Recall that $V_1,V_2$ are minimum words. So any word of $u$ can be castled with $V_1$ and any word of $\widehat{u}$ can be castled with $V_2$ by Remark \ref{remark_element_castling}. Let $U,\widehat{U},\widehat{V_1},\widehat{V_2},\widetilde{U}$ be words of $u,\widetilde{u},\widehat{v_1},\widehat{v_2},\widetilde{u}$, respectively, such that
\[
\underline{\underline{U}} V_1 \rlh \widehat{V_1}\underline{\underline{\widehat{U}}},\quad \underline{\underline{\widehat{U}}} V_2  = \widehat{V_2} \underline{\underline{\widetilde{U}}}.
\]
Then $\underline{\underline{U}} V_\flat=\underline{\underline{U}} V_1V_2 \rlh \widehat{V_1}\widehat{V_2}\underline{\underline{\widetilde{U}}}$, which implies that $\uuline{u}v \rlh \widetilde{v}\uline{\widetilde{u}}$, i.e., $(u,v)\in \fC_0^\prime$. For the case $\ind(u)\geq 2$, similar conclusion follows by applying \eqref{eq_C0_castling_decomposition_1}.

On the other hand, combining \ref{lem_union_and_decomposition_for_element_castling}(\romannumeral1,\romannumeral3), one can verify that there are no more elements in $\fC_0^\prime$, i.e., $\fC_0=\fC_0^\prime$. It follows that $\Gamma_0=\Gamma_0^\prime$.
\end{proof}

Finally, we shall prove Axiom \uppercase\expandafter{\romannumeral5} for Thompson's monoid $\bS$.

\begin{lemma}
Let $k,l\geq 1$ and $p,q$ be elements in $\cP$ such that $p^k,q^l$ are weakly castlable. Then $\uline{p^k}q^l \rlh r^l \uline{t^k}$ for some $r,t\in \cP$.
\end{lemma}

\begin{proof}
Let $p=p_i$ and $q=p_j$. For $i=j$, we have $\uline{p_i^k}p_i^l \rlh p_i^l \uline{p_i^k}$. For $i>j$, we have $\uline{p_i^k}p_j^l \rlh p_j^l \uline{p_{i+l}^k}$. For $i<j$, the ordered pair  $p^k,q^l$ are weakly castlable if and only if $i<j-k$. When they are weakly castlable, we have $\uline{p_i^k}p_j^l \rlh p_{j-k}^l \uline{p_i^k}$. The proof is completed.
\end{proof}

Till now, it has been shown that $\bS$ is a natural monoid. We end this section by showing some properties and examples of arithmetics for Thompson's monoid $\bS$.

\begin{lemma} \label{lem_u_trans_p_iff_condition}
Let $k\geq 1$ and $a_0,a_1,\ldots,a_{k-1} \geq 0$ and $u=p_0^{a_0}\ldots p_{k-1}^{a_{k-1}}$. Then $u_k, p_k$ are weakly castlable if and only if
\begin{equation} \label{eq_condition_k-r_geq_sum_+1}
r\geq a_{k-r}+a_{k-r+1}+\ldots a_{k-1}+1,\quad (1\leq r\leq k),
\end{equation}
When they are castlable, we have
\[
\uline{u} p_k  \rightleftharpoons  p_{k-\ind(u)}\, \uline{u}.
\]
\end{lemma}

\begin{proof}
Put $u_r=p_{k-r}^{a_{k-r}}\ldots p_{k-1}^{a_{k-1}}$ for $1\leq r\leq k$. Here $u=u_k$. Then $\ind(u_r)=a_{k-r}+\ldots +a_{k-1}$. Noting that $p_k$ has only one word, we have that $u_R,p_k$ are weakly castlable if and only if $u_R,p_k$ are strongly castlable. For $r=1$, it is not hard to see that $u_1=p_{k-1}^{a_{k-1}}$ and $p_k$ is castlable if and only if $a_{k-1}=0$ if and only if $r\geq a_{k-1}+1$. When they are castlable, one has $\uline{1}p_k \rightleftharpoons p_{k-\ind(u_1)} \uline{1}$. Now we use induction. For some $1\leq R\leq k-1$, suppose that we have proved that $u_R$ is castlable with $p_k$ if and only if $r\geq \ind(u_r)+1$ for all $1\leq r\leq R$. And when they are castlable, we have $\uline{u_R} p_k \rightleftharpoons p_{k-\ind(u_R)}\uline{u_R}$. Now $u_{R+1}$ is castlable with $p_k$ if and only if we additionally have the condition that $p_{k-R-1}^{a_{k-R-1}}$ is castlable with $p_{k-\ind(u_R)}$. The latter condition holds if and only if $(k-\ind(u_R)) - (k-R-1)\geq a_{k-R-1}+1$, which is equivalent to $R+1 \geq  \ind(u_{R+1})+ 1$. If they are castlable, then
\[
\uline{p_{k-R-1}^{a_{k-R-1}}} p_{k-\ind(u_R)}  \rightleftharpoons p_{k-\ind(u_{R})-a_{k-R-1}}\uline{p_{k-R-1}^{a_{k-R-1}}},
\]
which implies
\[
\uline{u_{R+1}} p_k \rightleftharpoons  p_{k-\ind(U_{R+1})}\uline{u_{R+1}}.
\]
By induction, the lemma follows.
\end{proof}

Now we turn back to give an examples about Theorem \ref{thm_prime_multiplicity_from} for $\bS$.

\begin{example}
Consider $u=p_0p_3^2p_5$ in $\bS$. All words of $u$ are listed below with ``$\rightarrow$'' being the partial order ``$\preccurlyeq$''.
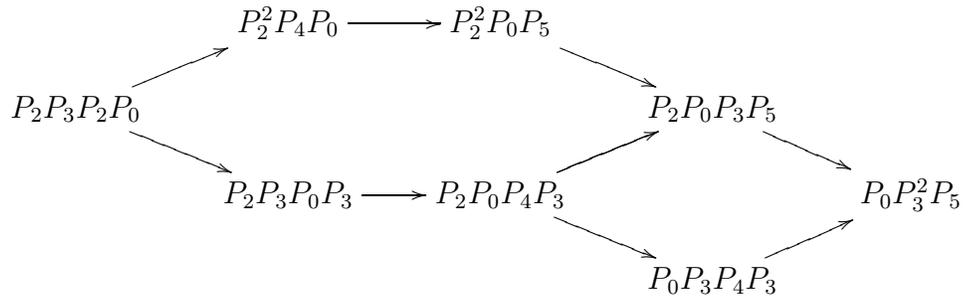
\[
\xymatrix@R=0.5cm{
                &         P_2^2P_4P_0  \ar[r]^{} &P_2^2P_0P_5   \ar[dr]^{}  & &  \\
  P_2P_3P_2P_0 \ar[ur]^{} \ar[dr]_{}     & &   & P_2P_0P_3P_5 \ar[dr]_{}&     \\
                &         P_2P_3P_0P_3  \ar[r]^{} &P_2P_0P_4P_3 \ar[ur]^{} \ar[dr]_{} \ar@{-}[ur]^{}& & P_0P_3^2P_5 \\
               & & & P_0P_3P_4P_3 \ar[ur]^{}&
                 }
\]
We have $\PDM(u)=\{p_0,p_2,p_2\}$ and $\PDM_\ddagger(u)=\{p_0,p_3,p_5\}$.

Consider the irreducible representation $u=p_2p_0p_4p_3$, we have
\[
\uline{p_2}1\rlh 1\uuline{p_2},\quad \uline{p_0}p_1\rlh p_2\uuline{p_0},\quad \uline{p_2}p_2p_0p_5\rlh p_2p_0p_4\uuline{p_3}
\]
and $p_2p_0$, $p_4$ are not strongly castlable. We obtain the prime divisors with multiplicities in this way. Notice that the element $p_2p_0p_5$ has another word $P_3P_2P_0$, and $P_2, P_3P_2P_0$ are not castlable. So $p_2, p_2p_0p_5$ are not strongly castlable. One can not change a simple underline to double underlines in Theorem \ref{thm_prime_multiplicity_from}.

Moreover, we also have
\[
\uuline{p_3}1\rlh 1\uline{p_3},\quad \uuline{p_4}p_3\rlh p_3\uline{p_5},\quad \uuline{p_0}p_4p_3\rlh p_3p_2\uline{p_0}
\]
and $p_2,p_0p_4p_3$ are not strongly castlable. We obtain prime co-divisors with multiplicities in this way. Note that $p_0p_4p_3$ has a word $P_2P_0P_5$ and
\[
\underline{\underline{P_2}}P_2P_0P_5\rlh P_2P_0P_4\underline{\underline{P_3}}.
\]
So $p_2,p_0p_4p_3$ are weakly castlable and $\uline{p_2}p_0p_4p_3\rlh p_2p_0p_4\uline{p_3}$. This example shows that the double underlines can not be replaced by a simple underline in Theorem \ref{thm_prime_multiplicity_from}.
\end{example}

The next example explains the condition ``fully castlable'' in Theorem \ref{thm_Omega=Omega_ddagger}.

\begin{example}
Consider the element $u=p_0p_1^2$. It has only one word and is not fully castlable. We have $\PDM(u)=\{p_0\}$ and $\PDM_\ddagger(u)=\{p_1,p_1\}$. Hence $\Omega(u)=1$ and $\Omega_\ddagger(u)=2$.
\end{example}

\section{Complexity for Castlings}

\label{section_complexity}

\subsection{General Properties}

In this section, we assume that $S$ is a homogeneous monoid. Then $\tau(uv)\leq \tau(u)\tau(v)$ for all $u,v\in S$ by Theorem \ref{thm_divisor_submul}.
It follows that
\[
\tau(u^{m+n}) \leq \tau(u^m)\tau(u^n),\quad (u\in S,\, m,n\geq 0).
\]
The sequence $\{\log \tau(u^n)\}_{n=1}^\infty$ is sub-additive. Thus, the limit $\lim\nolimits_{n\rightarrow \infty} \frac{1}{n}\log \tau(u^n)$ exists.

\begin{definition}
For $u\in S$, define
\[
\tau_0(u)=\lim\limits_{n\rightarrow \infty} \left(\tau(u^n)\right)^{1/n}.
\]
Define
\[
\c{C}(u) = \lim\limits_{n\rightarrow \infty} \left(\frac{\tau(u^n)}{\tau^n(u)}\right)^n = \frac{\tau_0(u)}{\tau(u)}.
\]
Also define
\[
\c{C}(S)=\sup\limits_{1\neq u\in S}\c{C}(u).
\]
We call $\c{C}(S)$ the complexity for castlings in $S$.
\end{definition}

For any $u\in S$, it satisfies that $1\leq \tau_0(u)\leq \tau(u)$. For $p\in \cP$, one has $\tau(p)=2$ and then $\c{C}(p)\geq 1/2$. As a result, we always have $1/2\leq \c{C}(S)\leq 1$. The quantity $\c{C}(S)$ describes the complexity for castlings in $S$. The larger is $\c{C}(S)$, the more divisors are provided during the castlings of elements in $S$. Now we deduce some basic properties below.

\begin{lemma} \label{lem_complexity_1}
For any $u\in S$ and $k\geq 1$, we have $\tau_0(u^k)= (\tau_0(u))^k$.
\end{lemma}

\begin{proof}
Note that, for any $m\geq 1$, we have
\[
\left(\tau\left(\left(u^k\right)^m\right)\right)^{\frac{1}{m}} =  \left(\tau(u^{km})\right)^{\frac{1}{km}\cdot k}.
\]
Letting $m\rightarrow \infty$, then $\tau_0(u^k)= (\tau_0(u))^k$.
\end{proof}

As a result, for $u\in S$ and $k\geq 0$, one has
\[
\c{C}(u^k) = \frac{\tau_0(u^k)}{\tau(u^k)} \geq \frac{(\tau_0(u))^k}{(\tau(u))^k} = (\c{C}(u))^k.
\]

\begin{lemma} \label{lem_tau_0_uv=vu}
For any $u,v\in S$, we have $\tau_0(uv)=\tau_0(vu)$.
\end{lemma}

\begin{proof}
Note that
\[
\left(\tau\left((uv)^k\right))\right)^{1/k} = \left(\tau(u(vu)^{k-1}v)\right)^{1/k}\leq \left(\tau(u)\right)^{1/k}\left(\tau\left((vu)^{k-1}\right)\right)^{1/k} \left( \tau(v)\right)^{1/k}.
\]
Letting $k\rightarrow\infty$, one obtains $\tau_0(uv)\leq \tau_0(vu)$. Similarly, we can deduce that $\tau_0(vu)\leq \tau_0(uv)$.
\end{proof}

\begin{lemma}
For $i=1,2$, let $G_i$ be an fractional group with its integral monoid $S_i$. Then $G_1\times G_2$ is also an fractional group with $S_1\times S_2$ its integral group. Moreover, if both $S_1,S_2$ are homogeneous, then so is $S_1\times S_2$.
\end{lemma}

\begin{proof}
Suppose that $S_1,S_2$ satisfy Axioms \uppercase\expandafter{\romannumeral1}, \uppercase\expandafter{\romannumeral2} and \uppercase\expandafter{\romannumeral3}. It is not hard to see that $S_1\times S_2$ is a monoid with identity $1=(1,1)$. If $(u,v)\in S_1\times S_2$ and $(u,v)\in (S_1\times S_2)^{-1} = S_1^{-1}\times S_2^{-1}$, then $u\in S_1\cap S_1^{-1}=\{1\}$ and $v\in S_2\cap S_2^{-1}=\{1\}$.

For any $(w_1,w_2)\in G_1\times G_2$, let $x_iy_i^{-1}$ be the simplest fraction of $w_i$ with numerator $x_i$ and denominator $y_i$ $(i=1,2)$, respectively. Then $(w_1,w_2)=(x_1,x_2)\cdot (y_1,y_2)^{-1}$. Moreover, suppose that $(w_1,w_2)=(\widetilde{x_1},\widetilde{x_2})\cdot (\widetilde{y_1},\widetilde{y_2})^{-1}$ for some $(\widetilde{x_1},\widetilde{x_2}), (\widetilde{y_1},\widetilde{y_2})\in S_1\times S_2$. Then $w_i=\widetilde{x_i}\widetilde{y_i}^{-1}$ $(i=1,2)$. Since $G_i$ is a fractional group with $S_i$ its integral monoid, one deduces that $\widetilde{x_i}=x_i c_i$, $\widetilde{y_i}=y_i c_i$ for some $c_i\in S_i$ $(i=1,2)$. Now $(\widetilde{x_1},\widetilde{x_2}) = (x_1,x_2)\cdot (c_1,c_2)$ and $(\widetilde{y_1},\widetilde{y_2}) = (y_1,y_2)\cdot (c_1,c_2)$ for $(c_1,c_2)\in S_1\times S_2$.

Moreover, it is not hard to see that, for any given $(w_1,w_2)\in S_1\times S_2$, one has
\begin{align*}
&\{((u_1,u_2),(v_1,v_2))\in (S_1\times S_2)^2:\, (w_1,w_2)=(u_1,u_2)(v_1,v_2)= (u_1v_1,u_2,v_2)\}\\
\quad \quad \quad &= \{(u_1,v_1)\in S_1:\, u_1v_1=w_1\}\times\{(u_2,v_2)\in S_2:\, u_2v_2=w_2\}.
\end{align*}
So Axiom \uppercase\expandafter{\romannumeral3} holds for $S_1\times S_2$.

Now suppose that $S_1,S_2$ also satisfy Axiom \uppercase\expandafter{\romannumeral4}'. Suppose that the elements $(w_1,w_2),(u_1,u_2),(v_1,v_2)\in S_1\times S_2$ satisfy
\begin{align*}
&\lcm[(w_1,w_2),(u_1,u_2)]=\lcm[(w_1,w_2),(v_1,v_2)],\\ &\gcd((w_1,w_2),(u_1,u_2))=\gcd((w_1,w_2),(v_1,v_2)).
\end{align*}
Note that
\begin{align*}
\lcm[(w_1,w_2),(u_1,u_2)] = \left(\lcm[w_1,u_1],\,\lcm[w_2,u_2]\right),\\
\gcd((w_1,w_2),(u_1,u_2)) = \left(\gcd(w_1,u_1),\,\gcd(w_2,u_2)\right).
\end{align*}
One has
\begin{align*}
\lcm[w_1,u_1]=\lcm[w_1,v_1],\quad \gcd(w_1,u_1)=\gcd(w_1,v_1),\\
\lcm[w_2,u_2]=\lcm[w_2,v_2],\quad \gcd(w_2,u_2)=\gcd(w_2,v_2).
\end{align*}
It follows that $u_1=v_1$ and $u_2=v_2$, i.e., $(u_1,v_1)=(u_2,v_2)$. As a result, the monoid $S_1\times S_2$ satisfies Axiom \uppercase\expandafter{\romannumeral4}'.
\end{proof}

For simplicity, we will not use different notations for the divisor functions on different monoids. The notation $\tau$ and $\tau_0$ are always used. From the above, we see that $\tau((u_1,u_2)) = \tau(u_1)\tau(u_2)$ for $(u_1,u_2)\in S_1\times S_2$. In particular, one has \[
\tau\left((u_1,u_2)^k\right)= \tau\left((u_1^k,u_2^k)\right)=\tau(u_1^k)\cdot \tau(u_2^k)
\]
for any $k\geq 1$. Now we state the following result.

\begin{proposition} \label{prop_complexity_2}
Let $S_1,S_2$ be homogeneous monoids. Then for any $(u_1,u_2)\in S_1\times S_2$, we have
\[
\tau_0((u_1,u_2)) = \tau_0(u_1)\cdot \tau_0(u_2),\quad \c{C}((u_1,u_2)) = \c{C}(u_1)\cdot \c{C}(u_2).
\]
\end{proposition}

It follows that
\[
\c{C}(S_1\times S_2)=\sup\limits_{(u,v)\in S_1\times S_2\atop (u,v)\neq (1,1)} \c{C}(u)\c{C}(v) = \max\{\c{C}(S_1),\c{C}(S_2)\}.
\]
The properites in Lemma \ref{lem_complexity_1} and Proposition \ref{prop_complexity_2} are similar to that of entropy of a dynamical system. And the property in Lemma \ref{lem_tau_0_uv=vu} is similar to that of spectral radius of a bounded operator. Next, we consider complexity for natural monoids.

\begin{theorem}\label{thm_tau_0_finitely_many_primes}
Suppose that $S$ is a natural monoid containing finitely many primes. Then $\tau_0(u)=1$ for all $u\in S$. In particular, we have $\c{C}(S)=1/2$.
\end{theorem}

\begin{proof}
Suppose that $\cP=\{p_0,p_1,\ldots,p_{k-1}\}$. By Theorem \ref{thm_finite_primes_implies_fully_castlable}, the monoid $S$ is fully castlable. Recalling Theorem \ref{thm_equivalence_fully_castlable}, we can write
\[
u^n=\lcm\left[p_0^{m_{0,n}},p_1^{m_{1,n}},\ldots ,p_{k-1}^{m_{k-1,n}}\right]
\]
for any $u\in S$ and any $n\geq 1$. Here $m_{0,n},m_{1,n},\ldots,m_{k-1,n}$ are non-negative integers. In particular, one has
\[
\ind(u^n)=m_{0,n}+m_{1,n}+\ldots+m_{k-1,n}=n\cdot \ind(u).
\]
It follows that
\begin{align*}
\tau(u^n)&= (m_{0,n}+1)(m_{1,n}+1)\ldots (m_{k-1,n}+1) \\
&\leq \left(\frac{1}{k}\left(m_{0,n}+m_{1,n}+\ldots+m_{k-1,n}+k\right)\right)^k \leq (n\cdot\ind(u)+1)^k.
\end{align*}
Hence
\[
\tau_0(u)=\lim\limits_{n\rightarrow \infty}(\tau(u^n))^{1/n} = \lim\limits_{n\rightarrow \infty} (n\cdot \ind(u)+1)^{k/n} =1.
\]
In particular, for $u\neq 1$, one has $\c{C}(u)=\frac{\tau_0(u)}{\tau(u)}\leq 1/2$. The conclusion follows.
\end{proof}

If $S$ is an abelian monoid, then similar arguments as above shows that $\tau_0(u)=1$ $(u\in S)$ and $\c{C}(S)=1/2$. For a homogenous monoid $S$ with finitely many irreducible elements, does $\c{C}(S)=1/2$ also hold?

\subsection{Complexity for castlings in Thompson's Monoid}

In this subsection, we will calculate $\tau_0(u)$ for certain kinds of elements $u$ in Thompson's monoid $\bS$, and calculate $\c{C}(\bS)$. Define $\iota$ to be the conjugation on $\bG$ induced by $p_0$, i.e.,
\[
\iota(u) = p_0^{-1} u p_0,\quad (u\in \bG).
\]
It is an automorphism on $G$. Moreover, it satisfies that $\iota(\bS)\subseteq \bS$. In particular, for $m\geq 0$, one has
\[
\iota^m (p_0)=p_0,\quad \iota^m(p_r)=p_{r+m},\,\, (r\geq 1).
\]

\begin{lemma} \label{lem_iota_inverse}
Let $u=p_0^{m_0}p_1^{m_1}\ldots p_k^{m_k}$ for some $k\geq 1$ and $m_0,m_1,\ldots,m_k\geq 0$. Then $u\in \iota(\bS)$ if and only if $m_1=0$.
\end{lemma}

\begin{proof}
When $m_1=0$, we have
\[
\iota(p_0^{m_0} p_1^{m_2}p_2^{m_3}\ldots p_{k-1}^{m_k}) = p_0^{m_0}p_2^{m_2}p_3^{m_3}\ldots p_k^{m_k}=u.
\]
So $u\in \iota(S)$. On the other hand, suppose that $u=\iota(v)$ for some $v\in \bS$. Let $v=p_0^{n_0}p_1^{n_1}\ldots p_k^{n_k}$ for some $k,n_0,n_1,\ldots,n_k\geq 0$. We then have $u=\iota(v)=p_0^{n_0}p_2^{n_1}p_3^{n_2}\ldots p_k^{n_{k-1}}$. One sees that $m_0=n_0$, $m_r=n_{r-1}$ $(2\leq r\leq k)$, and $m_1=0$. The lemma then follows.
\end{proof}

\begin{corollary} \label{cor_tau_u_and_tau_iota_u}
(\romannumeral1) For any $u\in \bS$, we have that $\tau(u)\leq \tau(\iota(u))$.

(\romannumeral2) Let $u=p_1^{m_1}p_2^{m_2}\ldots p_k^{m_k}$. Then $\tau(u)=\tau(\iota(u))$.
\end{corollary}

\begin{proof}
(\romannumeral1) The conclusion follows by noting that $\iota(w)|\iota(u)$ whenever $w|u$.

(\romannumeral2) Note that $\iota(u)=p_2^{m_1}p_3^{m_2}\ldots p_{k+1}^{m_k}$. Since $p_0,p_1$ are not prime divisors of $\iota(u)$, it does not divides any divisor $w$ of $\iota(u)$ either. Then $w=p_2^{n_2}p_3^{n_3}\ldots p_l^{n_l}$ for some $n_2,n_3,\ldots,n_l\geq 0$. By Lemma \ref{lem_iota_inverse}, we have $\iota^{-1}(w)\in \bS$. And $\iota^{-1}(w)|u$. It follows that $\tau(u)\geq \tau(\iota(u))$. This completes the proof.
\end{proof}

\begin{lemma} \label{lem_tau_0_leq_tau_0_cdot_tau}
Let $0\leq k\leq l$ and $m_0,m_1,\ldots,m_l\geq 0$. Let $u=p_0^{m_0}p_1^{m_1}\ldots p_k^{m_k}$ and $v=p_{k+1}^{m_{k+1}}p_{k+2}^{m_{k+2}}\ldots p_l^{m_l}$. Then
\[
\tau_0(u)\leq \tau_0(uv) \leq \tau_0(u)\tau(v).
\]
\end{lemma}

\begin{proof}
Denote $L=\ind(u)$. For $n\geq 1$, one obtains by induction that
\[
(uv)^n = u^n \cdot \iota^{(n-1)L}(v)\iota^{(n-2)L}(v)\ldots \iota^L(v)v.
\]
Combining Corollary \ref{cor_tau_u_and_tau_iota_u}, we have
\[
\tau\left(u^n\right)\leq\tau\left((uv)^n\right) \leq \tau\left(u^n\right) \prod\limits_{r=0}^{n-1}\tau\left(\iota^{rL}(v)\right) =\tau\left(u^n\right)\tau(v)^n.
\]
Taking $n$-th roots on both sides and letting $n\rightarrow\infty$, we conclude that
\[
\tau_0(u)\leq \tau_0(uv) \leq \tau_0(u)\tau(v).
\]
\end{proof}

\begin{example} \label{example_tau_0_primepower}
Consider $u=q^l$ for some $l\geq 0$ and $q\in \cP$. Since $\tau(u^n)=\tau(q^{nl})=nl+1$, we have
\[
\tau_0(q^l) = \lim\limits_{n\rightarrow \infty} (nl+1)^{1/n} =1.
\]
Then $\c{C}(q^l)=\frac{1}{l+1}$.
\end{example}

\begin{example} \label{example_tau_0_two_prime_powers}
Consider $u=p_i^kp_j^l$, where $k,l\geq 1$ and $i\neq j$. By Lemma \ref{lem_tau_0_uv=vu}, one has $\tau_0(p_i^kp_j^l)=\tau_0(p_j^lp_i^k)$. We can suppose without loss of generality that $i<j$. Combining Lemma \ref{lem_tau_0_leq_tau_0_cdot_tau} and Example \ref{example_tau_0_primepower}, one gets
\[
\tau_0(u)\leq \tau_0(p_i^k)\tau(p_j^l) = 1\cdot (l+1)=l+1.
\]
In the following, we show that the equality holds. Note that
\begin{align*}
u^n &= (p_i^k p_j^l)^n = p_i^{kn} \cdot p_{j+(n-1)k}^l p_{j+(n-2)k}^l \ldots p_{j+k}^l p_j^l\\
&= p_i^{kn} \cdot \lcm\left[p_{j+(n-1)k}^l,p_{j+(n-2)k}^l,\ldots,p_{j+k}^l, p_j^l\right].
\end{align*}
Recalling Corollary \ref{cor_tau_lcm_p_m}, one obtains
\[
\tau(u^n)\geq \tau\left(\lcm\left[p_{j+(n-1)k}^l,p_{j+(n-2)k}^l,\ldots,p_{j+k}^l, p_j^l\right]\right) = (l+1)^n.
\]
So
\[
\tau_0(u)\geq \lim\limits_{n\rightarrow \infty} \left((l+1)^n\right)^{1/n}=l+1.
\]
We conclude that $\tau_0(p_i^kp_j^l)=l+1$, where $i<j$.
\end{example}

\begin{example}
Let $k\geq 1$ and consider $u=p_0^k p_1p_3\ldots p_{2k-1}$. By induction, one obtains $u^n = p_0^{nk}p_1p_3\ldots p_{2nk-1}$. Then
\[
\tau(u^n)\geq \tau(p_1p_3\ldots p_{2nk-1}) = \tau(\lcm[p_1,p_2,\ldots,p_{nk}])=2^{nk},
\]
and
\[
\tau(u^n)\leq \tau(p_0^{nk})\tau(p_1p_3\ldots p_{2nk-1}) = (nk+1)2^{nk}.
\]
It follows that $\tau_0(u)=2^k$.
\end{example}

\begin{example}
Consider the element $u=p_0p_1\ldots p_{l-1}$, where $l\geq 2$. One has $\tau(u)=l+1$. By induction, one can verify that
\begin{align*}
(p_0p_1\ldots p_{l-1})^k &= p_0^k\cdot (p_kp_{k+1}\ldots p_{k+l-2})\cdot (p_{k-1}p_k\ldots p_{k+l-3})\cdot \ldots \cdot (p_1p_2\ldots p_{l-1}).
\end{align*}
Put
\begin{align*}
&X_j=P_{j}P_{j+1}\ldots P_{j+l-2},\,\, (1\leq j\leq k),\\
&Y_j=(P_kP_{k+1}\ldots P_{k+l-2})(P_{k-1}P_k\ldots P_{k+l-3})\ldots (P_{j+1}P_{j+2}\ldots P_{j+l-1}), \,\, (0\leq j\leq k-1),\\
&Z_j=(P_{k+l-1}P_{k+l}\ldots P_{k+2l-3})(P_{k+l-2}P_{k+l-1}\ldots P_{k+2l-4}) \ldots  (P_{j+l}P_{j+l+1}\ldots P_{j+2l-2}), \,\, (1\leq j\leq k-1),
\end{align*}
and $x_j,y_j,z_j$ be the corresponding elements in $\bS$, respectively. In the following, we always assume that $1\leq j\leq k-1$. By calculation, we obtain $\underline{\underline{X_j}}Z_j \rlh Y_j\underline{\underline{X_j}}$. Since $X_j,Y_j,Z_j$ are all minimum words, one gets $\uuline{x_j}z_j\rlh y_j\uuline{x_j}$. Moreover, the only prime divisor of $x_j$ is $p_j$, while prime divisors of $y_j$ are exactly $p_k,p_{k-1},\ldots,p_{j+1}$. We deduce that $\gcd(x_j,y_j)=1$. So $\uuuline{x_j}z_j\rlh y_j\uuuline{x_j}$. Combining $y_jx_j=y_{j-1}$ and Theorem \ref{thm_divisor_submul}, we obtain $\tau(y_{j-1})=\tau(y_j)\tau(x_j)$. Note that $y_{k-1}=x_k$. Then
\[
\tau(y_0)=\tau(y_1)\tau(x_1)=\tau(y_2)\tau(x_2)\tau(x_1)= \ldots =\tau(x_k)\ldots\tau(x_2)\tau(x_1) = l^k.
\]
Now
\[
\tau\left((p_0p_1\ldots p_{l-1})^k\right) =\tau(p_0^k\cdot y_0) \geq \tau(y_0) = l^k,
\]
and
\[
\tau\left((p_0p_1\ldots p_{l-1})^k\right) =\tau(p_0^k\cdot y_0) \leq \tau(p_0^k)\tau(y_0) = (k+1)l^k.
\]
Taking $k$-th root and letting $k\rightarrow \infty$ on both sides, one deduces that $\tau(u_l)=l$, and then $\c{C}(u_l)=\frac{l}{l+1}$.
\end{example}

From the above example, we obtain the complexity for castlings of Thompson's monoid $\bS$ immediately.

\begin{theorem}
For $\bS$, we have $\c{C}(\bS)=1$.
\end{theorem}

\begin{proof}
The conclusion follows from
\[
1\geq \c{C}(S) \geq \sup\limits_{l\geq 2} \c{C}(p_0p_1\ldots p_l) = \sup\limits_{l\geq 2} \frac{l}{l+1} =1.
\]
\end{proof}

Unlike natural monoids with only finitely many primes, for most of the elements $u$ in $\bS$, the quantity $\tau_0(u)$ is strictly larger than $1$.

\begin{theorem} \label{thm_tau_0_geq_2}
Let $u\in \bS$. Then $\tau_0(u)=1$ if and only if $u=q^m$ for some $q\in \cP$ and $m\geq 0$. In particular, if $u$ is not a prime power, then $\tau_0(u)\geq 2$.
\end{theorem}

\begin{proof}
It has been shown in Example \ref{example_tau_0_primepower} that $\tau_0(p^m)=1$. On the other hand, suppose that $u\neq 1$ and $u$ is not a prime power. Let $u=p_{j_1}^{m_1}p_{j_2}^{m_2}\ldots p_{j_k}^{m_k}$, where $k\geq 2$, $0\leq j_1<j_2<\ldots <j_k$ and $m_1,m_2,\ldots,k\geq 1$. By Lemma \ref{lem_tau_0_leq_tau_0_cdot_tau}, we conclude that $\tau_0(u)\geq \tau_0(p_{j_1}^{m_1}p_{j_2}^{m_2})$. By Example \ref{example_tau_0_two_prime_powers}, one has $\tau_0(p_{j_1}^{m_1}p_{j_2}^{m_2})\geq m_2+1\geq 2$. The proof is completed.
\end{proof}

We end this section with some questions. Theorem \ref{thm_tau_0_geq_2} shows that, for an element $u$ in Thompson's monoid, either $\tau_0(u)=1$ or $\tau_0(u)\geq 2$. How large is such a gap in a general homogeneous group $S$? Can $\tau_0(u)$ take non-integer value? Or, does $\tau_0(S)$ contain an interval? Moreover, do we have $\tau_0(uv)\geq \tau_0(u)\tau_0(v)$ for all $u,v\in S$? Do we have $\tau_0(uv)\leq \tau(u)\tau_0(v)$, or $\tau_0(uv)\leq \tau_0(u)\tau(v)$ for $u,v\in S$? Are there examples of $S$ with $\c{C}(S)=\alpha$ for any $1/2<\alpha<1$? For a natural monoid, both $\tau_0$ and amenability are related to how a prime is castled with elements in $S$. Do these two properties have connections with each other? Under which kind of conditions we may have ``$\c{C}(S)=1/2$ if and only if $S$ is amenable''? By Theorems \ref{thm_amenable_finitely_many_primes} and \ref{thm_tau_0_finitely_many_primes}, the condition that $S$ is a natural monoid with finitely many primes works. For a natural group $S$ with $\c{C}(S)=1$, it is non-amenable?

\section{Further Remarks}

It would be interesting to make the axioms more elegant, or more general. It is possible to write the axioms with a cancellative monoid $S$, without a group $G$. 
The functions $\beta_p$ defined in \eqref{eq_beta_form} may be helpful to rewrite Axioms \uppercase\expandafter{\romannumeral4}', \uppercase\expandafter{\romannumeral4}, or \uppercase\expandafter{\romannumeral5}, as well as to classify the natural monoids. For Thompson's monoid, it has the property that $\#\cP\setminus \beta_p(\cP)<+\infty$ for any $p\in \cP$. Such condition might give an important class of natural monoids.

It would be interesting to look for concrete examples of homogeneous monoids that are not castlable, and castlable monoids that are not natural. Furthermore, the non-commutative arithmetics in this paper only live on a semigroup $S$. It there an example of a non-abelian $S$ that also admits an addition with $S+S\subseteq S$? To fulfill this, one may weaken Axiom \uppercase\expandafter{\romannumeral2} by only requiring the existence of least common multiples up to some given upper bound, just as the definition of least common co-multiple in this paper. Moreover, if there is a suitable total order on $S$ which well characterize the structure of $S$, one may also consider the problem that counting primes up to some level with respect to this total order.

For two primes $p_i$ and $p_j$ in $\bS$, we have $p_i,p_j$ are castlable if and only of $i-j\neq -1$. This looks similar to the residue in complex analysis. Does Thompson's group $\bG$ has a presentation based on complex functions? Moreover, do certain arithmetic functions on $\bS$ give information on the elements in $\bS$ as a function on $[0,1]$ (such as number of breakpoints, different slopes, etc.) or as a dynamics from $[0,1]$ to itself?

For a castlable monoid, there may be some structures which are weaker than that of tensor products. We wonder whether non-commutative arithmetics has relation with tanglement in theoretical physics. Moreover, it is explained in Section \ref{subsection_fully_castlable_elements} that an element $u$ in a natural monoid can be uniquely written as $u=u_1u_2\ldots u_t$, where $u_j$ is the greatest fully castlable divisor of $(u_1u_2\ldots u_{j-1})^{-1}u$ $(1\leq j\leq t)$. An example in $\bS$ is shown below. Consider
\[
u=\lcm[p_1^2,p_3,p_4^3]\cdot \lcm[p_2,p_3,p_6^2]:= u_1\cdot u_2,
\]
where we regard $u_1$, $u_2$ as a chain of nuclei. When a particle $p_1$ collides with $u$, we have $p_1\cdot u_1=\lcm[p_1^3,p_2,p_3^3]$, and $p_1$ is absorbed into $u_1$. When a particle $p_0$ collides with $u$, we have $p_0\cdot u_1=\lcm[p_0,p_2,p_3^3]\cdot p_1^2$. That is to say, the nucleus $u_1$ becomes $\lcm[p_0,p_2,p_3^3]$, and two particles $p_1^2$ will collides with the second nucleus $u_2$. This seems interesting.

\bigskip

\textbf{Acknowledgements.} The author is especially grateful to Professor Liming Ge, without whom this paper would never appear. The author would like to thank Linzhe Huang, who shares his ideas and checks the proofs. The author also thanks Weichen Gu, Minghui Ma and Dongsheng Wu for helpful discussions. This work is supported by National Natural Science Foundation of China (Grant No. 11701549).



\end{document}